\documentclass[11pt]{amsart}

\usepackage{epsf,amssymb,times,overpic,amscd,diagrams,xcolor}
\usepackage[all]{xy}

\newtheorem{thm}{Theorem}[section]
\newtheorem{prop}[thm]{Proposition}
\newtheorem{lemma}[thm]{Lemma}
\newtheorem{cor}[thm]{Corollary}

\newtheorem{claim}[thm]{Claim}
\newtheorem{fact}[thm]{Fact}

\theoremstyle{definition}
\newtheorem{defn}[thm]{Definition}

\theoremstyle{remark}

\newtheorem{rmk}[thm]{Remark}
\newtheorem{example}[thm]{Example}
\newtheorem{warn}[thm]{Warning}
\newtheorem{conv}[thm]{Convention}
\newtheorem{noti}[thm]{Notation}
\newtheorem{simplification}[thm]{Simplification}

\DeclareMathAlphabet{\mathpzc}{OT1}{pzc}{m}{it}
\newcommand{\hh}{\mathpzc{h}}

\newcommand{\C}{\mathbb{C}}

\newcommand{\R}{\mathbb{R}}
\newcommand{\Z}{\mathbb{Z}}
\newcommand{\Q}{\mathbb{Q}}
\newcommand{\N}{\mathbb{N}}
\renewcommand{\P}{\mathbb{P}}
\newcommand{\bdry}{\partial}
\newcommand{\s}{\vskip.1in}
\newcommand{\n}{\noindent}

\newcommand{\F}{\mathbb{F}}

\newcommand{\op}{\operatorname}

\newcommand{\be}{\begin{enumerate}}
\newcommand{\ee}{\end{enumerate}}

\newcommand{\mb}{\mathcal{M}^{\mbox{\tiny MB}}}
\newcommand{\me}{\mathcal{M}_\epsilon}

\newcommand{\aaa}{\mathfrak{a}}

\numberwithin{equation}{subsection}
\numberwithin{thm}{subsection}

\makeatletter
\def\foo#1\endgraf\unskip#2\foo{\def\row@to@buffer{#1\endgraf\unskip\unskip#2}}
\expandafter\foo\row@to@buffer\foo
\makeatother

\begin{document}

\title[Embedded contact homology and open book decompositions]
{Embedded contact homology and open book decompositions}

\author{Vincent Colin}
\address{Universit\'e de Nantes, 44322 Nantes, France}
\email{Vincent.Colin@univ-nantes.fr}

\author{Paolo Ghiggini}
\address{Universit\'e de Nantes, 44322 Nantes, France}
\email{paolo.ghiggini@univ-nantes.fr}
\urladdr{http://www.math.sciences.univ-nantes.fr/\char126 Ghiggini}

\author{Ko Honda}
\address{University of California, Los Angeles, Los Angeles, CA 90095}
\email{honda@math.ucla.edu} \urladdr{http://www.math.ucla.edu/~honda}

\date{This version: August 18, 2023.}



\thanks{VC supported by the Institut Universitaire de France,
ANR Symplexe, ANR Floer Power and ERC G\'eodycon. PG supported by ANR Floer Power and ERC G\'eodycon.
KH supported by NSF Grants DMS-0805352 and DMS-1105432.}

\begin{abstract}
Given a closed oriented contact $3$-manifold $M$, we prove an equivalence between the embedded contact homology of $M$ and a version of embedded contact homology
``relative to the boundary'', defined on the complement of a tubular neighborhood of a null-homologous knot. This paper can be viewed as the first of a series of papers devoted to proving the  isomorphism between Heegaard Floer homology and embedded contact homology.

The appendix, written jointly with Yuan Yao, gives a complete proof of Morse-Bott gluing for one-level cascades in embedded contact homology.
\end{abstract}

\maketitle

\setcounter{tocdepth}{1}
\tableofcontents

\section{Introduction}
\addtocounter{subsection}{1}

Let $M$ be a closed, oriented, connected $3$-manifold and $N\subset M$ the complement of a tubular neighborhood of a  null-homologous knot.  The goal of this paper is  to associate a specific class of contact forms $\alpha$ to $N$, to introduce relative embedded contact homology groups $ECH(N,\bdry N,\alpha)$ and $\widehat{ECH}(N,\bdry N,\alpha)$, and to prove their  isomorphism with the embedded contact homology groups $ECH(M)$ and $\widehat{ECH}(M)$.

The embedded contact homology group $ECH(M)$ of a closed $3$-manifold $M$, due to Hutchings~\cite{Hu}  partially in collaboration with Taubes~\cite{HT1,HT2}, is defined using a contact form $\alpha$ on $M$ and an adapted almost complex structure $J$ on the symplectization $\R\times M$. The variant $\widehat{ECH}(M)$, called {\em ECH hat}, is defined as the mapping cone of a $U$-map (see Section~\ref{subsection: defn of ECH hat}).  There is currently no direct proof of the fact that these groups are invariants of $M$; the only known proof, due to Taubes~\cite{T1,T2}, is a consequence of the  isomorphism between Seiberg-Witten Floer cohomology and embedded contact homology,  combined with the invariance of Seiberg-Witten Floer cohomology established by Kronheimer and Mrowka~\cite{KrM}.

Embedded contact homology groups can be defined over the integers following \cite{BM} or \cite[Section~9]{HT2}. All results in this article hold over the integers as explained in Proposition~\ref{continuation maps are easy} and Remark~\ref{rmk: integer coefficients}, but we will write detailed proofs only over the field $\F = \Z/2\Z$
for simplicity.
Given a compact $3$-manifold $N$ with $\partial N \simeq T^2$, let $\alpha$ be a contact form on $N$ which is nondegenerate on $int(N)$ and negative Morse-Bott on $\partial N$ (see Definition~\ref{defn: positive negative torus}). In particular, the Reeb orbits on $\partial N$ act as sinks for $J$-holomorphic curves in $\R \times N$, i.e., no  non-trivial $J$-holomorphic curve in $\R \times N$ can have a positive end  at an orbit in $\partial N$.   Then there exist relative embedded contact homology groups $ECH(N,\bdry N,\alpha)$ and $\widehat{ECH}(N,\bdry N,\alpha)$, whose definitions will be given in Section~\ref{section: ECH for manifolds with torus boundary}.
Moreover there is a chain map $U$ on the complex
defining $ECH(N,\bdry N, \alpha)$, and the homology of the cone of $U$ is isomorphic to $\widehat{ECH}(N,\bdry N, \alpha)$.

The embedded contact homology group of a contact manifold $(M,\xi )$ has a natural decomposition as a direct sum of groups $ECH(M, \xi, A)$ indexed by homology classes\footnote{Singular homology groups should always be understood over the integers if no coefficient group is explicitly indicated.} $A \in H_1 (M)$. This decomposition depends on the contact structure $\xi$, although very weakly. For this reason we always specify $\xi$ together with the homology class $A$.

Similarly, the groups $ECH(N,\bdry N,\alpha)$ and $\widehat{ECH}(N,\bdry N,\alpha)$ decompose as direct sums of groups  $ECH(N,\bdry N, \alpha, A)$ and  $\widehat{ECH}(N,\bdry N, \alpha, A)$ indexed by relative homology classes $A \in H_1(N, \partial N)$. The maps $U$ in both $ECH(M, \xi)$
and $ECH(N,\bdry N,\alpha)$ preserve the splitting according to homology classes.
Taking into account the fact that $K$ is null-homologous, excision and the relative homology long exact sequence give an isomorphism $\varpi : H_1 (N,\partial N) \stackrel{\simeq} \longrightarrow  H_1(M)$, and the
equivalence between ECH and relative ECH is compatible with the corresponding decompositions.

The main result of this paper is the following:
\begin{thm}\label{thm: equivalence of ECHs}
Let $N \subset M$ be the complement of a tubular neighborhood  $int(V)$ of a null-homologous knot $K$, where $V \simeq K\times D^2$,  $\xi$ a contact form on $M$ which is transverse to the foliation $K\times \{ *\}$ on $V$ and $\alpha$ a contact form on $N$ for the contact structure $\xi|_N$. If the Reeb vector field $R_\alpha$ of $\alpha$ is nondegenerate on $int(N)$, negative Morse-Bott on $\bdry N$,  foliates $\partial N$ by meridians and  all closed Reeb orbits in $N$ have nonnegative linking number with $K$, then  for all $A \in H_1(N, \partial N; \Z)$,
\begin{enumerate}
\item $ECH(N,\bdry N,\alpha,  A) \simeq ECH(M, \xi, \varpi(A))$ and
\item  $\widehat{ECH}(N,\bdry N,\alpha,  A) \simeq \widehat{ECH}(M, \xi, \varpi(A)).$
\end{enumerate}
 Moreover, the first isomorphism is compatible with the $U$-maps on both sides.
\end{thm}

The prototypical situation to which Theorem~\ref{thm: equivalence of ECHs} applies is the case of an open book decomposition with connected binding. In this case $N$ is the mapping torus of a surface diffeomorphism $\hh : S \stackrel \simeq \to S$ and $V= M - int(N)$ is a tubular neighborhood of the binding. In other words, Theorem~\ref{thm: equivalence of ECHs} allows us to rewrite the embedded contact homology groups of $M$ in terms of the relative embedded contact homology groups on the complement of the binding. We remark here that Yau~\cite{Y} and Wendl~\cite{We,We2} have examined related issues in their work.

Theorem~\ref{thm: equivalence of ECHs}, applied to the open book case, is the first step in the proof of the equivalence of embedded contact homology and Heegaard Floer homology,  a Floer homology theory for three-manifolds defined by Ozsv\'ath  and Szab\'o~\cite{OSz1,OSz2}.  Once we express the embedded contact homology of $M$ purely in terms of $N$ using Theorem~\ref{thm: equivalence of ECHs}, it is easier to define chain maps to and from the hat version of Heegaard Floer homology. In fact, the Giroux correspondence~\cite{Gi2} --- the bijection between open book decompositions up to positive stabilization and isotopy classes of contact structures --- provides a bridge between the contact forms used in the definition of ECH and the Heegaard splittings used in the definition of Heegaard Floer homology. We remark that the proof of the equivalence between Heegaard Floer homology and ECH is independent of the hard part of the Giroux correspondence (i.e., the stabilization equivalence of two open book decompositions which support the same contact structure). The rest of the proof of the equivalence has been carried out in \cite{CGH2, CGH3, CGH4}; see \cite{CGH1} for an overview of the
strategy.

\begin{rmk}
 An alternate proof of the equivalence of Heegaard Floer and embedded contact homologies, passing through Seiberg-Witten Floer homology, has been given by Kutluhan, Lee and Taubes (see \cite{KLT1}--\cite{KLT5}).
\end{rmk}

In Section~\ref{section: sutured applications} we present some independent applications of the
techniques developed here to the embedded contact homology for sutured manifolds
defined in~\cite{CGHH}. More precisely, we prove that ECH of a sutured manifold is invariant of the contact form
and the almost complex structure (Theorem~\ref{invariance of sutured ECH}) and we finish
the proof of \cite[Theorem~1.6]{CGHH} by showing that $\widehat{ECH}(M)$, defined
as the homology of the cone of the $U$-map, is isomorphic to the sutured ECH of the
complement of a ball in $M$ (Theorem~\ref{hat and sutured}). Theorem~\ref{invariance of sutured ECH} has
been independently proved by Kutluhan and Sivek in~\cite{KS}.

\s\n {\em Organization of the paper.} Section~\ref{section: review of ECH} gives a brief review of ECH; in particular we define the groups $ECH(M)$ and $\widehat{ECH}(M)$.  We review some technicalities involving direct limits in Section~\ref{section: cobordism maps and direct limits} and some Morse-Bott theory in the context of ECH in Section~\ref{section: Morse-Bott theory}. In Section~\ref{section: topological constraints} we discuss topological constraints of $J$-holomorphic curves arising from the positivity of intersections in dimension four. In Section~\ref{section: contact forms} we construct contact forms on $D^2 \times S^1$ and $T^2 \times [1,2]$ which are used later. Section~\ref{section: ECH for manifolds with torus boundary} is devoted to the definitions of certain ECH groups for compact manifolds with torus boundary and in particular the variants $ECH(N,\bdry N,\alpha)$ and $\widehat{ECH}(N,\bdry N,\alpha)$ which appear in Theorem~\ref{thm: equivalence of ECHs}. In Section~\ref{section: ECH of solid torus} we calculate some ECH
groups of solid tori which are used in the proof of Theorem~\ref{thm: equivalence of ECHs}. Section~\ref{section: proof of theorem equivalence of ECHs} then completes the proof of Theorem~\ref{thm: equivalence of ECHs}.  Finally, Section~\ref{section: sutured applications} relates some of the versions of ECH defined in Section~\ref{section: ECH for manifolds with torus boundary} to some sutured ECH groups defined in \cite{CGHH}.

\s\n {\em Acknowledgements.} We are indebted to Michael Hutchings
for many helpful conversations and for our previous collaboration
which was a catalyst for the present work. We also thank Dusa
McDuff, Ivan Smith, Jean-Yves Welschinger, and Chris Wendl for illuminating
exchanges. Finally, we thank Michael Hutchings again for pointing out a mistake in
the first version of this article, the anonymous referees for extensive lists of corrections, and Roman Krutowski for a careful reading of the appendix.

 Part of this work was done while KH and PG visited MSRI
during the academic year 2009--2010.  We are extremely grateful to
MSRI and the organizers of the ``Symplectic and Contact Geometry and
Topology'' and the ``Homology Theories of Knots and Links'' programs
for their hospitality; this work probably would never have seen the
light of day without the large amount of free time which was made
possible by the visits to MSRI.

\section{Review of embedded contact homology}
\label{section: review of ECH}

{\em  In this paper all manifolds will be oriented and connected, unless stated otherwise. }

In this section we briefly review the basic definitions of embedded contact homology (from now on abbreviated ECH). For more details the reader is referred to \cite{Hu,Hu2} or to \cite{Hu3}. To avoid orienting the moduli spaces, we will work over $\F=\Z/2\Z$.

\subsection{Generators of the ECH chain complex}

Let $M$ be a closed, oriented and connected $3$-manifold with a contact form $\alpha$.  We will denote by $\xi=\ker\alpha$ the contact structure with contact form $\alpha$. The Reeb vector field $R=R_\alpha$ is {\em nondegenerate} if no Reeb orbit\footnote{In this paper we interchangeably use: ``Reeb orbit", ``closed orbit'', and ``closed Reeb orbit''.  A Reeb orbit which is not necessarily closed will be called a ``Reeb trajectory".} has $1$ as eigenvalue of its linearized first return map. This is a generic condition which can achieved by a generic
$C^{\infty}$-small perturbation of the contact form; see for example \cite[Lemma 7.1]{CH2}. For the rest of the section we will assume that $\alpha$ is nondegenerate.  The linearization of the first return map along a Reeb orbit is a symplectic transformation of the symplectic plane $(\xi ,d\alpha )$. This implies that its eigenvalues are $\{ \lambda, \lambda^{-1} \}$, where $\lambda$ is either real or in the unit circle. Then a Reeb orbit is:
\begin{itemize}
\item {\em hyperbolic} if the eigenvalues of its linearized first return map are real; or
\item {\em elliptic} if they lie on the unit circle.
\end{itemize}
This conditions are mutually exclusive because every orbit is assumed to be nondegenerate.

Let $\mathcal{P}$ be the set of simple orbits of the Reeb vector field $R_\alpha$. The ECH chain complex\footnote{The ECH differential depends on the choice of an adapted almost complex structure $J$ (cf.\ Section~\ref{subsection: moduli spaces}), but the generators only depend on $\alpha$.  Hence we suppress $J$ from the notation for the moment.} $ECC(M,\alpha)$, as a vector space, is generated over $\F$ by finite  sets $\gamma=\{(\gamma_i,m_i)\}$, called {\em orbit sets}, where:
\begin{itemize}
\item $\gamma_i\in \mathcal{P}$ and $\gamma_i\not=\gamma_j$ for $i\not=j$;
\item $m_i$ is a positive integer; and
\item if $\gamma_i$ is a hyperbolic orbit, then $m_i=1$.
\end{itemize}
We will say that $ECC(M,\alpha)$ is {\em constructed from
$\mathcal{P}$}. An orbit set $\gamma$ will also be written
multiplicatively as $\prod \gamma_i^{m_i}$, with the convention that
$\gamma_i^2 = 0$ whenever $\gamma_i$ is hyperbolic. The empty orbit set
$\varnothing$ will be written multiplicatively as $1$.

The homology class of an orbit set $\gamma$ is
$$[\gamma]=\sum_i m_i[\gamma_i]\in H_1(M).$$
If we want to specify the direct summand generated by orbit sets of class $A\in H_1(M)$, then we write $ECC(M,\alpha,A)$.

The {\em action} $\mathcal{A}_\alpha(\gamma_i)$ of an orbit $\gamma_i$ is given by $\int_{\gamma_i}\alpha$, and the action of an orbit set $\gamma$ is given by $$\mathcal{A}_\alpha(\gamma) = \sum_i m_i \mathcal{A}_\alpha(\gamma_i).$$

\subsection{Moduli spaces} \label{subsection: moduli spaces}

We choose an almost complex structure $J$ on $\R\times M$, with $\R$-coordinate $s$, which is {\em adapted to the  symplectization of $\alpha$} (or {\em adapted to $\alpha$}), i.e.,
\begin{enumerate}
\item[(i)] $J$ is $s$-invariant;
\item[(ii)] $J$ takes $\xi$ to itself on each $\{s\}\times Y$;
\item[(iii)] $J$ maps $\bdry_s$ to $R_\alpha$;
\item[(iv)] $J|_\xi$ is $d\alpha$- compatible, i.e., $d\alpha( \cdot ,J \cdot)$ defines an Euclidean metric on $\xi$.
\end{enumerate}

Let $\gamma=\{(\gamma_i,m_i)\}$ and $\gamma'=\{(\gamma'_i,m'_i)\}$
be orbit sets with $[\gamma]=[\gamma']\in H_1(M)$.
The set of holomorphic maps
$$u:(F,j)\to (\R\times M,J),$$
modulo holomorphic reparametrizations, which satisfy:
\begin{enumerate}
\item $(F,j)$ is a closed Riemann surface with a finite number of punctures removed;
\item the neighborhoods of the punctures are mapped asymptotically to cylinders over Reeb orbits;
\item  at the positive end of $\R\times M$, $u$ is asymptotic to $\R\times\gamma_i$ with total multiplicity $m_i$  for each pair $(\gamma_i,m_i)$ (more precisely, if we list the positive ends of $u$ that are asymptotic to some multiple cover of $\R\times \gamma_i$ and the covering degrees are $m_{i1},\dots, m_{ij_i}$, then $m_i=m_{i1}+\dots +m_{ij_i}$); and
\item  at the negative end of $\R\times M$, $u$ is asymptotic to $\R\times\gamma'_i$ with total multiplicity $m'_i$  for each pair $(\gamma'_i,m_i')$;
\end{enumerate}
will be denoted by $\mathcal{M}_J(\gamma,\gamma')$.  We often refer
to an element $u$ of $\mathcal{M}_J(\gamma,\gamma')$ as a {\em $J$-holomorphic map (or curve) from $\gamma$ to $\gamma'$}. We stress the fact that,
according to our definition, the genus, the number of connected components, and the
number of punctures of $F$ are not fixed a priori.  If $*$ is a property of $J$-holomorphic curves, we will denote by $\mathcal{M}_J^{*}(\gamma,\gamma')$ the subset of $\mathcal{M}_J(\gamma,\gamma')$ satisfying $*$.  We can similarly define the ``pointed'' moduli space $\mathcal{M}_J(\gamma,\gamma';pt)$ as the set of holomorphic maps
$$u:(F,j,p)\to (\R\times M,J),$$
modulo holomorphic reparametrizations, where $p\in F$.

\begin{defn} \label{defn: regular}
We say that $J$ is {\em regular} if, for all orbit sets
$\gamma,\gamma'$ and $u\in \mathcal{M}_J(\gamma,\gamma')$ which have
no multiply-covered components, $\mathcal{M}_J(\gamma,\gamma')$ is
transversely cut out near $u$ (i.e., the linearized $\overline\bdry$-operator $D_u$ at $u$ from \cite[Proposition 2.10]{Dr} is surjective).
\end{defn}

Regular  adapted almost complex structures form the complement of a first category set (and
therefore are dense)  in the space of smooth adapted almost complex structures with respect to the $C^\infty$ topology by a result of Dragnev~\cite{Dr}. If  no component of $u$ is multiply-covered and  all components of $u$ are
 transversely cut out, then in a neighborhood of $u$ the moduli space
$\mathcal{M}_J(\gamma,\gamma')$ has the structure of a finite-dimensional manifold
of dimension $\op{ind}(u)$, where the Fredholm index $\op{ind}(u)$ is  the formal dimension of the moduli spaces
computed as in the next paragraph; see \cite[Corollary 1]{Dr}.
Our convention throughout the paper will be that the Fredholm index takes into account the dimensions of the Deligne-Mumford moduli space and the automorphism group of the domain of the map.  In particular, $\op{ind}(u) = \op{ind}(D_u) - 3 \chi(F)$, where $\op{ind}(D_u)$ is the Fredholm index of the linearized Cauchy-Riemann operator at $u$, and  $\chi(F)$ is the Euler characteristic of the domain of $u$.

A $J$-holomorphic map $u : F \to \R \times M$ from $\gamma = \{ (\gamma_i, m_i) \}$ to $\gamma' = \{ (\gamma_i', m_i') \}$ determines partitions $\{ m_{ij} \}$ of $m_i$ and $\{m_{ij}'\}$ of $m_i'$ such that $u$ is positively asymptotic to $m_{ij}$-fold covers $\gamma_i^{m_{ij}}$ of the simple Reeb orbits $\gamma_i$ and negatively asymptotic to $m_{ij}'$-fold covers $(\gamma_i')^{m_{ij}'}$ of the simple Reeb orbits $\gamma_i'$. Let $\tau$ be a trivialization of $\xi$ along each orbit in the orbit sets $\gamma,\gamma'$,  let $\mu_\tau(\delta)$  denote the Conley-Zehnder index of a cover $\delta$ of an orbit in $\gamma$ or $\gamma'$ with respect to $\tau$, and  let $c_1(u^*\xi,\tau)$  denote the relative first Chern class of $u^*\xi$ with respect to $\tau$. Then the Fredholm index $\op{ind}(u)$ is given by the formula
\begin{equation} \label{Fredholm index}
\op{ind}(u) = - \chi(F) + 2 c_1(u^*\xi, \tau) + \sum_{ij} \mu_\tau(\gamma_i^{m_{ij}}) - \sum_{ij} \mu_\tau((\gamma_i')^{m_{ij}'}).
\end{equation}
(See the formula in \cite[Theorem 1.8]{Dr}.)

\subsection{The ECH index}
\label{subsection: ECH index}

The index which appears in the definition of ECH is not the Fredholm
index, but the ECH index, which is more topological in nature. In this subsection
we will review its definition.

Let $\gamma=\{(\gamma_i,m_i)\}$ and $\gamma'=\{(\gamma_i',m_i')\}$
be orbit sets. We denote by  $H_2(M,\gamma,\gamma')$ the relative homology classes of surfaces $Z\in H_2 (M,(\bigcup_i \gamma_i ) \cup (\bigcup_{i'} \gamma_{i'} ))$ such that $\partial Z = \sum m_i [\gamma_i] - \sum m_i' [\gamma_i']$, where
$$\partial : H_2 (M,(\bigcup_i \gamma_i ) \cup (\bigcup_{i'} \gamma_{i'} )) \to H_1((\bigcup_i \gamma_i ) \cup (\bigcup_{i'} \gamma_{i'} ))$$
is the connecting homomorphism of the relative homology exact sequence.  By abuse of notation, $Z$ will also denote an embedded surface with boundary which represents that
homology class. We pick a
trivialization $\tau$ of $\xi$ along each orbit in the orbit sets
$\gamma,\gamma'$ and define $c_1(\xi|_Z, \tau)$ as the first Chern class of
$\xi$ evaluated on $Z$, relative to the trivialization $\tau$ on
$\bdry Z$.

If $\gamma=\{(\gamma_i,m_i)\}_{i=1}^k$ is an orbit set, then we define
the {\em ``symmetric'' Conley-Zehnder index} (so called because of
its motivation from studying symplectomorphisms of a symmetric
product of a surface) as follows:
\begin{equation}
\widetilde\mu_\tau(\gamma)= \sum_{i=1}^k \sum_{j=1}^{m_i}
\mu_\tau(\gamma_i^j),
\end{equation}
where $\gamma_i^j$ is the orbit which multiply covers $\gamma_i$
with multiplicity $j$.

We define the {\em relative intersection pairing} $Q_\tau(Z)$ as
follows: Using the trivialization $\tau$, for each simple orbit
$\gamma_i$ of $\gamma$ or $\gamma'$, fix an identification of a
sufficiently small neighborhood $N(\gamma_i)$ of $\gamma_i$ with
$\gamma_i\times D^2$, where $D^2$ has polar coordinates
$(r,\theta)$. Let $\Sigma$ be an oriented embedded surface and
$f:\Sigma\to [-1,1]\times M$ a smooth map which satisfies the
following:
\begin{enumerate}
\item $f$ maps $\bdry \Sigma$ to $\{-1,1\}\times M$,
$f|_{int(\Sigma)}$ is an embedding, and $f$ is transverse to
$\{-1,1\}\times M$.
\item For all $\varepsilon>0$ sufficiently small, $f(\Sigma)\cap
(\{1-\varepsilon\}\times M)$ consists of $m_i$ disjoint circles of
type $\{r=\varepsilon,\theta=const\}$ in $N(\gamma_i)$ for all $i$
(and similarly for $f(\Sigma)\cap (\{-1+\varepsilon\}\times M)$).
\item The composition of $f$ with the projection $[-1,1]\times M\to
M$ is a representative of the class $Z\in H_2(M,\gamma,\gamma')$.
\end{enumerate}
We then choose two maps $f_1,f_2$ satisfying (1)--(3) above, such
that they are disjoint on $\{-1+\varepsilon,1-\varepsilon\}\times M$
and transverse on $[-1+\varepsilon,1 -\varepsilon]\times M$.  Then
$Q_\tau(Z)$ is the signed intersection number of $f_1$ and $f_2$
 in $[-1+\varepsilon,1-\varepsilon]\times M$.

We are now in a position to define the ECH index.

\begin{defn}[{\cite[Definition~1.5]{Hu}}]
The ECH index $I(\gamma,\gamma',Z)$ is given by:
\begin{equation}\label{eqn: ECH index formula}
I(\gamma,\gamma',Z)= c_1(\xi|_Z, \tau)+Q_\tau(Z) +\widetilde\mu_\tau(\gamma)-\widetilde\mu_\tau(\gamma').
\end{equation}
\end{defn}

 The ECH index depends only on the relative homology class $Z\in H_2 (M,\gamma ,\gamma' )$ and not on a particular surface representing it. Moreover
the ECH index is independent also of the choice of trivialization. If $Z' \in H_2 (M,\gamma ,\gamma' )$ is another relative homology class, then (\cite{Hu3})
$$I(\gamma,\gamma',Z') - I(\gamma,\gamma',Z)= \langle Z'-Z,c_1 (\xi )+PD (\sum_i m_i [\gamma_i] )\rangle,$$
where $\sum_i m_i [\gamma_i]$ is the total homology class of $\gamma$ in $H_1 (M)$ and $PD$ is the Poincar\'e duality map.

\begin{rmk}
A finite energy holomorphic map $u$ with asymptotics $\gamma$ and $\gamma'$ defines a relative homology class $Z \in H_2(M, \gamma, \gamma')$.  Hence we can write $I(u)=I(\gamma,\gamma',Z)$.
\end{rmk}

The ECH index and the Fredholm index satisfy the following index inequality, which is one of the basic tools of ECH.

\begin{thm}[{\cite[Theorem 4.15]{Hu2}}] \label{thm: index inequality}
If $u$ is simply-covered, then $\op{ind}(u) \le I(u)$.
\end{thm}

\subsection{The ECH differential}

In this subsection we define the differential $\bdry$ for the ECH chain complex, after recalling some properties of $J$-holomorphic maps with small ECH index. In the following we will say that a map $u : F \to \R \times M$ is the ``disjoint union'' of maps $u_i : F_i \to \R \times M$ (with $1 \le i \le k$) if $F= F_1 \sqcup \ldots \sqcup F_k$ and the images are pairwise disjoint.  Here each $F_i$ can still be disconnected.  A {\it trivial cylinder} over a (not necessarily simple) orbit $\gamma$ with period $T$ is the $J$-holomorphic map $u : \R \times S^2 \to \R \times M$, $u(s,t) = (Ts, \gamma(Tt))$. By abuse of notation, we will always denote the trivial cylinder over $\gamma$ by $\R \times \gamma$.

\begin{lemma}[{\cite[Proposition 7.15]{HT1}}]
Let $J$ be a regular almost complex structure adapted to $\alpha$. Then:
\begin{enumerate}
\item A $J$-holomorphic map $u$ with $I(u)=0$ is a disjoint union of branched covers of trivial cylinders over simple Reeb orbits. (Such curves are called {\em connectors}.)
\item A $J$-holomorphic map $u$ with $I(u)=1$ (resp.\ $2$) from $\gamma$ to $\gamma'$ is a disjoint union of a connector and an embedding $u'$ with $I(u')=ind(u')=1$ (resp.\ $2$).
\end{enumerate}
\end{lemma}

In this paper a ``branched cover'' will always refer to a ``branched cover with possibly empty branch locus''.

The ends of a $J$-holomorphic map $u$ from $\gamma$ to $\gamma'$ determine partitions of the multiplicities of the elliptic orbits. It turns out that, when $I(u)=1$ or $I(u)=2$, these partitions must coincide with preferred partitions called the {\em outgoing} and {\em incoming} partitions for positive and negative ends, respectively.
The incoming and outgoing partitions can be computed from the dynamics of the linearized Reeb flow. For their definition see \cite[Section 4.1]{Hu} or \cite[Definition 4.14]{Hu2}. For the relation between these partitions and the ECH index see \cite[Theorem 4.15]{Hu2}, for example. In this article we will not need the precise definition of the incoming or the outgoing partition, except for the following fact, which follows directly from \cite[Definition 4.14]{Hu2}.

\begin{fact}\label{basic fact about partitions}
Let $\gamma$ be a simple elliptic orbit and suppose that its linearized Reeb flow rotates by an angle $2 \pi \theta$. If $0< \theta < {1\over m}$, then the incoming partition of $(\gamma, m)$ is $(m)$ and the outgoing partition is $(1, \ldots, 1)$. On the other hand, if $- {1\over m} < \theta < 0$, then the incoming partition of $(\gamma, m)$ is $(1, \ldots, 1)$ and the outgoing partition is $(m)$.
\end{fact}

The boundary  operator in the ECH chain complex is defined by a count of $J$-holomorphic maps with
index $I=1$ for a regular almost complex structure $J$. In order to make the dependence
 on $J$ explicit we write the complex as $ECC(M, \alpha, J)$. However, when $J$ is clear
from the context, it will be dropped from the notation.

\begin{defn}
Let $J$ be a regular almost complex structure adapted to $\alpha$. Then the boundary map $\partial : ECC(M, \alpha, J) \to ECC(M, \alpha, J)$ is defined as:
$$\bdry\gamma= \sum_{\gamma'} \langle\bdry \gamma,\gamma'\rangle ~\gamma',$$
where $\langle\bdry \gamma,\gamma'\rangle$ is the (mod 2) count of curves $u\in \mathcal{M}_J^{I=1}(\gamma,\gamma')/\R$ such that every connector component of $u$ is a trivial cylinder  over a simple orbit.
\end{defn}

The map $\bdry$ was shown to satisfy $\bdry^2=0$ by Hutchings  and Taubes~\cite{HT1,HT2}. The homology of the chain complex
$(ECC(M,\alpha,J),\bdry)$ is the {\em embedded contact homology}
group $ECH(M,\alpha,J)$.  It is independent of the choice of contact
form $\alpha$, the  contact structure $\xi$, and adapted almost
complex structure $J$, by the work of Taubes~\cite{T2}.  Hence we
are justified in writing $ECH(M)$.

\subsection{Definition of $\widehat{ECH}(M)$} \label{subsection: defn of ECH hat}

In this section we define a map $U: ECH(M) \to ECH(M)$ and a variant $\widehat{ECH}(M)$ of $ECH(M)$, called the {\em ECH hat group} in analogy with well-known constructions in Heegaard Floer homology. An a priori different group, also called $\widehat{ECH}(M)$, was defined in \cite{CGHH} using sutured ECH (in analogy with the sutured Floer homology of Juh\'asz~\cite{Ju}). In Section~\ref{section: sutured applications} we will prove that the two approaches yield isomorphic groups.

\begin{defn}
Let $J$ be a regular almost complex structure and  $z\in \R\times M$ a generic point  so that the evaluation map
$$ev:\mathcal{M}_J^{I=2}(\gamma,\gamma';pt)\to \R\times M, \quad (u,p)\mapsto u(p)$$
is transverse to $z$. We define the map $U: ECC(M,\alpha,J)\to ECC(M,\alpha,J)$ as:
$$U\gamma = \sum_{\gamma'} \langle U\gamma,\gamma'\rangle \ \gamma',$$
where $\langle U\gamma,\gamma'\rangle$ is the (mod 2) count of holomorphic maps $u \in \mathcal{M}_J^{I=2}(\gamma,\gamma')$ which pass through the point $z$ and such that every connector component of $u$ is a trivial cylinder  over a simple Reeb orbit.
\end{defn}

The same techniques used to show that $\partial^2=0$ also show that $U$ is a chain map; see \cite[Section 2.5]{HT5} for more details on the $U$-map. Then $\widehat{ECH}(M,\alpha,J)$ is defined as the homology of the mapping cone of $U$.

\section{Cobordism maps and direct limits} \label{section: cobordism maps and direct limits}

In this section we review the work of Hutchings  and Taubes~\cite{HT3} on maps on ECH induced by exact symplectic cobordisms, which in turn makes it possible to define continuation maps and take direct limits in ECH.

\subsection{Maps induced by cobordisms} \label{subsection: maps induced by cobordisms}

Given a contact $3$-manifold $(M,\alpha)$ with $\alpha$ nondegenerate, let $ECC^L(M,\alpha)$ be the subcomplex of $ECC(M,\alpha)$ generated by orbit sets $\gamma$ of action $\mathcal{A}_\alpha(\gamma) < L$, and $ECH^{L}(M,\alpha)$ be the resulting homology group. Given $L<L'$, the inclusion of chain complexes $ECC^{L}(M,\alpha)\subset ECC^{L'}(M,\alpha)$ induces a map
$$i_{L,L'}: ECH^{L}(M,\alpha) \to ECH^{L'}(M,\alpha)$$
on the level of homology.  The following is an immediate consequence of the definition of a direct limit:
$$ECH(M,\alpha)=\lim_{L\to\infty} ECH^L(M,\alpha).$$

Let $(M_1,\alpha_1)$ and $(M_2,\alpha_2)$ be contact $3$-manifolds. An {\em exact symplectic cobordism $(X,\omega)$ from\footnote{This is the convention from symplectic field theory~\cite{EGH} and is opposite from the one used in Heegaard Floer homology, for example.} $(M_1,\alpha_1)$ to $(M_2,\alpha_2)$} is an exact symplectic manifold with boundary $\bdry X= M_1-M_2$ and symplectic form $\omega=d\alpha$, where $\alpha$ restricts to $\alpha_1$ on $M_1$ and $\alpha_2$ on $M_2$.

Given an exact symplectic cobordism $(X, \omega)$, we form its {\em completion} $(\hat{X}, \hat{\omega})$ by attaching  the half positive symplectization of $(M_1, \alpha_1 )$ along $M_1 \subset \partial X$ and the half negative symplectization of $(M_2,\alpha_2 )$ along $M_2 \subset \partial X$.

\begin{defn}
Let $(\hat{X},\hat{\omega})$ be the completed symplectic cobordism with an almost complex structure $J$ which is compatible with $\hat{\omega}$ and is adapted to $\alpha_1$ and $\alpha_2$ at the positive and negative ends. Then the image of an embedding
$$\phi:(\R\times U,d(e^s\alpha_0),J_0) \hookrightarrow (\hat{X},\hat{\omega},J)$$
is called a {\em product region} if $\phi_*(d(e^s \alpha_0))=\hat{\omega}$, $\phi_* J_0=J$,  $J_0$ is adapted to $\alpha_0$ and, at the ends of $\R\times U$, $\phi(s,x)= (s+C_i,\phi_i(x))$, $i=1,2$, for some embedding $\phi_i :U\to M_i$ and constant $C_i$.
\end{defn}

The main technical result of \cite{HT4} is the following (the first item in (i) is a slight improvement due to Cristofaro-Gardiner \cite[Theorem~5.1]{Cr}):

\begin{thm}[Hutchings-Taubes {\cite[Theorem~1.9]{HT4}}] \label{thm: Hutchings Taubes cobordism map}
Let $(M_1,\alpha_1)$ and $(M_2,\alpha_2)$ be contact $3$-manifolds and let $(X,\omega)$ be an
exact symplectic cobordism from $(M_1,\alpha_1)$ to
$(M_2,\alpha_2)$. Suppose the contact forms $\alpha_1$, $\alpha_2$
are nondegenerate. Then for each positive real number $L$ there
exists a map:
$$\Phi^L(X,\omega): ECH^{L}(M_1,\alpha_1)\to ECH^L(M_2,\alpha_2).$$
Moreover, the following are satisfied:
\begin{enumerate}
\item[(i)] Let $J$ be a regular almost complex structure on $\hat{X}$ which is $\hat{\omega}$-compatible and is adapted to $\alpha_i$ at the positive and negative ends. Then $\Phi^L(X,\omega)$ is induced from a (noncanonical) chain map
$$\hat{\Phi}^L(X,\omega,J): ECC^{L}(M_1,\alpha_1,J|_{M_1})\to ECC^L(M_2,\alpha_2,J|_{M_2}),$$
which is {\em supported on the $J$-holomorphic curves}, i.e.,
\begin{itemize}
\item $\langle \hat{\Phi}^L(X,\omega,J)(\gamma),\gamma'\rangle =0$ if there is no $I=0$ $J$-holomorphic building from $\gamma$ to $\gamma'$ in $\hat{X}$.
\item If the only $J$-holomorphic building in $\hat{X}$ from $\gamma$ to $\gamma'$ is a union of covers of product cylinders contained in a product region, then $\langle \hat{\Phi}^L(X,\omega,J)(\gamma),\gamma' \rangle = 1$.
\end{itemize}

\item[(ii)] The map $\Phi^L(X,\omega)$  only depends on $L$
and $(X,\omega)$, and not on any auxiliary almost complex structure
$J$ on $(\hat{X},\hat{\omega})$. Moreover it depends on $\omega$ only through its
homotopy class as an exact symplectic form.

\item[(iii)] If $L<L'$, then the following diagram commutes:
\begin{equation}
\begin{diagram}
ECH^{L}(M_1,\alpha_1) & \rTo^{\Phi^L(X,\omega)} &
ECH^{L}(M_2,\alpha_2)  \\
\dTo^{i_{L,L'}} &
&  \dTo_{i_{L,L'}} \\
ECH^{L'} (M_1,\alpha_1) & \rTo^{\Phi^{L'}(X,\omega)} &
ECH^{L'} (M_2,\alpha_2)
\end{diagram}
\end{equation}
Hence the maps pass to the direct limit:
$$\Phi(X,\omega): ECH(M_1,\alpha_1) \to ECH(M_2,\alpha_2). $$

\item[(iv)] Suppose $(X,\omega)$ is the composition of exact
symplectic cobordisms $(X_1,\omega_1)$ from $(M_1,\alpha_1)$ to
$(M',\alpha')$ and $(X_2,\omega_2)$ from $(M',\alpha')$ to
$(M_2,\alpha_2)$, and $\alpha'$ is nondegenerate. Then
$$\Phi^L(X,\omega) = \Phi^L(X_2,\omega_2)\circ
\Phi^L(X_1,\omega_1).$$

\item[(v)] If $c>0$, then the following diagram commutes:
\begin{equation}
\begin{diagram}
ECH^{L} (M_1,\alpha_1) & \rTo^{\Phi^L(X,\omega)} & ECH^L(M_2,\alpha_2)\\
\dTo^s & & \dTo^s\\
ECH^{cL} (M_1, c\alpha_1) & \rTo^{\Phi^{cL}(X,c\omega)} &
ECH^{cL} (M_2, c\alpha_2),
\end{diagram}
\end{equation}
where $s$ is the canonical rescaling isomorphism.

\item[(vi)] If $X=[0,a]\times M$ and $\omega=d(e^s\alpha)$ where $\alpha$ is nondegenerate, then
$$\Phi^L(X,\omega): ECH^{L}(M,e^a\alpha)\to ECH^L(M,\alpha)$$
is equal to the composition
$$ECH^{L}(M,e^a\alpha)\xrightarrow{s}
ECH^{e^{-a}L}(M,\alpha)\xrightarrow{i_{e^{-a}L,L}} ECH^L(M,\alpha).$$
\end{enumerate}
\end{thm}
 \begin{rmk} The map involved in this result is borrowed from Seiberg-Witten theory via Taubes' isomorphism, where one counts solutions of the perturbed Seiberg-Witten equations on the cobordism. As we take a perturbation parameter $r$ to be large, these solutions concentrate near a holomorphic building. It is however not known yet how to reconstruct the count of solutions from just knowing the limit holomorphic building. This explains why there is no direct definition of cobordism maps by a count of holomorphic buildings and also why there is no direct proof of invariance for ECH.
\end{rmk}
\begin{defn}
A contact form $\alpha$ is called {\em $L$-nondegenerate} if all Reeb orbits of action less than $L$ are nondegenerate and there is no orbit set of action exactly $L$.
\end{defn}

The action-truncated $ECH$ groups $ECH^{L}(M,\alpha)$ make sense for contact forms $\alpha$ which are $L$-nondegenerate and Theorem~\ref{thm: Hutchings Taubes cobordism map}(i), (ii), and (iv) hold for $L$-nondegenerate contact forms.

All exact cobordisms considered in this paper will be of the following type:

\begin{defn} \label{defn: interpolating cobordism}
An {\em interpolating cobordism} from $(M, \alpha_1)$ to $(M, \alpha_0)$ is an exact symplectic cobordism $([0,1] \times M, \lambda)$ from $\alpha_1$ to $\alpha_0$ such that $\lambda$ is of the form
$$\lambda = \Phi^*(f \alpha),$$
where $\alpha$ is the pullback to $[0,1] \times M$ of a $1$-form (also called $\alpha$) on $M$, $f : [0,1] \times M \to \R$ is a positive function with $\frac{\partial f}{\partial t}>0$, and $\Phi: [0,1]\times M\stackrel\sim\to [0,1]\times M$ is a diffeomorphism taking $\{i\}\times M$ to itself for $i=0,1$.
\end{defn}

 In this article, interpolating cobordisms are  all constructed as follows: Let $\alpha_0$, $\alpha_1$ be isotopic contact forms on $M$ and let $\{\phi_t: M\stackrel\sim \to M\}_{t\in[0,1]}$ be an isotopy such that:
\begin{itemize}
\item $\phi_t^*(f_t\alpha_0)=\alpha_t$ for all $t\in[0,1]$;
\item $\{f_t\}$ and $\{\alpha_t\}$ are $1$-parameter families of functions and $1$-forms on $M$; and
\item $\phi_0=id$, and  $f_0=1$.
\end{itemize}
Then define $\Phi : [0,1] \times M \to [0,1] \times M$ by $\Phi(t, \mathbf{x})= \phi_t(\mathbf{x})$,  $f: [0,1] \times M \to \R$ by $f(t, \mathbf{x})=f_t(\mathbf{x})$ and $\lambda_{\phi}:= \Phi^*(f \alpha_0)$.  If ${\bdry f_t\over \bdry t}>0$, then
 $$([0,1] \times M, \lambda_\phi)$$
is an interpolating cobordism.  Interpolating cobordisms do not necessarily exist between any two isotopic $\alpha_0$ and $\alpha_1$, but one can always construct them at the small price of scaling one of the two forms by a constant.

\begin{lemma}\label{homotopy of interpolating cobordisms}
Let $([0,1] \times M, \lambda_ \phi)$ and $([0,1] \times M, \lambda_\phi')$ be interpolating cobordisms from $(M, \alpha_1)$ to $(M, \alpha_0)$  defined by contact isotopies $\phi$ and $\phi'$ respectively.
If the isotopies $\{\phi_t\}$ and $\{\phi'_t\}$ are homotopic relative to the endpoints, then $\lambda_\phi$ and $\lambda_\phi'$ are homotopic as exact symplectic forms.
\end{lemma}

\begin{proof}
Define $\Phi(t, \mathbf{x})= \phi_t(\mathbf{x})$ and $\Phi'(t, \mathbf{x})= \phi_t'(\mathbf{x})$. Without loss of generality we can write $\lambda_0:=\lambda_\phi = \Phi^*(f \alpha)$ and $\lambda_1: = \lambda_\phi'= (\Phi')^*(f' \alpha)$ with the same form $\alpha$ in both definitions. Let $ \{\Phi_s \}$ be a homotopy between $\Phi$ and $\Phi'$ such that:
\begin{itemize}
\item $\Phi_0 = \Phi$ and $\Phi_1= \Phi'$;
\item $\Phi_s(0, \mathbf{x}) = \mathbf{x}$ and $\Phi_s(1, \mathbf{x})= \phi_1(\mathbf{x}) = \phi_1'(\mathbf{x})$ for all $s \in [0,1]$.
\end{itemize}
Also define $F_s(t, \mathbf{x})= (1-s)f(t, \mathbf{x}) + s f'(t, \mathbf{x})$. Then
$$\lambda_s= \Phi_s^*(F_s \alpha)$$
is a homotopy of exact symplectic forms because
$\dfrac{\partial F_s}{\partial t}>0$ for all $s \in [0,1]$.
\end{proof}

\begin{lemma}\label{tired with all these}
Let $\alpha$ be a contact form, $L$, $L'>0$ real numbers, $\phi_t : M \to M$, $t\in[0,1]$, an isotopy such that $\phi_0=id$, and $f$, $f': M \to \R^+$ smooth functions such that $Lf' < L'f$. If $f\alpha$ and $f'\alpha$ are $L$- and $L'$-nondegenerate, respectively, then there is a map
$$ECH^{L}(M, \phi_1^*(f \alpha)) \to ECH^{L'}(M,f' \alpha).$$
Moreover, this map depends only on the homotopy class of $\{\phi_t\}$ relative to the endpoints and  has the following properties:
\begin{itemize}
\item[(a)] if $f=f'$ and $\phi_t \equiv id$, $t\in[0,1]$, then the map is induced by the inclusion of chain complexes, and
\item[(b)] if $L''>0$, $f'' : M \to \R^+$ is another function such that $L'f'' < L''f'$, and $\phi_t : M \to M$, $t\in[1,2]$, is  an extension of the isotopy, then the following triangle commutes:
$$ \xymatrix{ ECH^{L}(M, \phi_2^*(f \alpha)) \ar[rr] \ar[dr] & &  ECH^{L''}(M,f'' \alpha). \\
& ECH^{L'}(M, \phi_1^*(f' \alpha)) \ar[ur] & }$$
\end{itemize}
\end{lemma}

\begin{proof}
The inequality $Lf' < L'f$ implies that there is an interpolating cobordism with
$L' \phi_1^*(f \alpha)$ at the positive end and $L f' \alpha$ at the negative end.
We define the map
$ECH^{L}(M, \phi_1^*(f \alpha)) \to ECH^{L'}(M, f' \alpha)$ by the composition
$$\xymatrix{
ECH^{L}(M, \phi_1^*(f \alpha)) \ar[r] &  ECH^{L'L}(M, L'\phi_1^*(f\alpha)) \ar[d] & \\
& ECH^{L'L}(M, L f' \alpha) \ar[r] & ECH^{L'}(M, f' \alpha),}$$
where the map $ECH^{L'L}(M, L'f \alpha) \to ECH^{L'L}(M, L f' \alpha)$ is the map induced by an interpolating cobordism from $L' \phi_1^*(f \alpha)$ to $L f' \alpha$ and the horizontal maps are rescaling isomorphisms. The resulting map depends only on the homotopy class of $\{\phi_t\}$ relative to the endpoints by Lemma~\ref{homotopy of interpolating cobordisms}. The properties of these maps are an immediate consequence of Theorem \ref{thm: Hutchings Taubes cobordism map}.
\end{proof}

\subsection{Direct limits}

One consequence of Theorem~\ref{thm: Hutchings Taubes cobordism map}
is the following theorem, whose statement and proof were communicated to the authors
by Michael Hutchings:

\begin{thm}[Hutchings-Taubes] \label{thm: direct limit}
Let $M$ be a closed oriented $3$-manifold with a nondegenerate contact form $\alpha$ and let $\{f_i\}_{i=1}^\infty$ be a sequence of smooth positive functions such that $1\geq f_1\geq f_2\geq \dots$ and $f_i\alpha$ is $L_i$-nondegenerate for an increasing sequence of positive real numbers $L_i$ such that $\lim \limits_{i \to \infty} L_i= + \infty$. Then there is a canonical isomorphism
$$ ECH(M,\alpha)\simeq \lim_{i\to \infty} ECH^{L_i}(M,f_i\alpha).$$
\end{thm}

\begin{proof}
We have a map
$$f: ECH(M, \alpha) \to \lim \limits_{i \to \infty} ECH^{L_i}(M, f_i \alpha),$$
obtained by taking the direct limit of the cobordism maps
$$ECH^{L_i}(M, \alpha) \to ECH^{L_i}(M, f_i \alpha).$$
Choose an increasing sequence of natural numbers $c_i$ such that $L_{c_i}f_i > L_i$. Then there are maps
$$ECH^{L_i}(M, f_i \alpha) \to  ECH^{L_{c_i}}(M, \alpha)$$
by Lemma~\ref{tired with all these}. These maps form a directed system, and taking the direct limit we obtain a map
$$g: \lim \limits_{i \to \infty} ECH^{L_i}(M, f_i \alpha) \to ECH(M, \alpha).$$
The verification that the maps $f$ and $g$ are inverse of each other is a straightforward application of Lemma~\ref{homotopy of interpolating cobordisms}.
\end{proof}

We can now quantify when it makes sense to take direct limits of a
sequence of contact forms $\alpha_i$ for isotopic contact structures.
In this case we can write $\phi_i^*(\alpha_i)=f_i\alpha$ for some
positive function $f_i$ and diffeomorphism $\phi_i$ isotopic to the
identity.

\begin{defn} \label{defn: commensurate}
Let $\alpha$ be a contact form on $M$. A sequence
$\{\alpha_i\}_{i=1}^\infty$ of contact forms on $M$ is {\em
commensurate to $\alpha$} if there is a constant $0<c<1$,
diffeomorphisms $\phi_i$ of $M$ isotopic to the identity, and
functions $f_i: M\to \R^{>0}$ such that $\phi_i^*\alpha_i=f_i\alpha$
and $c< |f_i|_{C^0}< {1\over c}$.
\end{defn}

A corollary of Theorem~\ref{thm: direct limit} is the following:

\begin{cor} \label{cor: direct limit of commensurate contact forms}
Let $\{\alpha_i\}$ be a sequence of contact $1$-forms on $M$ which
is commensurate to $\alpha$ on $M$ with constant $0<c<1$. If $L_i\to
\infty$ is a sequence which satisfies $L_{i+1}> {1\over c^3} L_i$
for all $i$, then the groups $ECH^{L_i}(M,\alpha_i)$ form a directed system with
the maps defined in Lemma~\ref{tired with all these} and we have:
$$ECH(M) = \lim_{i\to\infty} ECH^{L_i}(M,\alpha_i).$$
\end{cor}

\begin{proof}
Define $L_i' = c^{2i} L_i$ and $g_i = c^{2i}f_i$. Then $\lim \limits_{i \to \infty} L_i' = + \infty$ and $1 > g_1 > \ldots > g_i > \ldots$, so we can apply Theorem~\ref{thm: direct limit} to the sequences $L_i'$ and $g_i$.
\end{proof}

\section{Morse-Bott theory}
\label{section: Morse-Bott theory}

In this section we discuss a special case of Morse-Bott theory as
it applies to our context. In particular, we explain how to use Theorem~\ref{thm:
direct limit} to justify the Morse-Bott arguments which populate
this paper. For a  more detailed discussion of Morse-Bott theory in
contact homology, the reader is referred to
Bourgeois~\cite{Bo1,Bo2}.

\subsection{Morse-Bott contact forms}
\label{subsection: Morse-Bott contact forms}

Let $\alpha$ be a {\em Morse-Bott contact form} on $M$. For the
purposes of this paper, this means that all the orbits either are
isolated and nondegenerate, or come in $S^1$-families and are
nondegenerate in the normal direction.  (In general, there is also the case where the Reeb orbits come in two-dimensional families, i.e., are the fibers of a circle bundle; however this will not occur here.) We denote a Morse-Bott family
of simple orbits by $\mathcal{N}$ and the Morse-Bott torus corresponding to $\mathcal{N}$ by $T_{\mathcal{N}}= \cup_{x\in \mathcal{N}} x$.

Let $\{v_1,v_2\}$ be an oriented basis for $\xi$ at some point $p \in T_{\mathcal{N}}$ so that $v_1$ is transverse to $T_{\mathcal{N}}$ and $v_2$ is tangent to $T_{\mathcal{N}}$.
The derivative  of the first return map $\xi_p\to\xi_p$ of the Reeb flow is given by the matrix
$\begin{pmatrix} 1 & 0
\\ a & 1 \end{pmatrix}$
with respect to the basis $\{v_1,v_2\}$. (Here a vector $v=a_1v_1+a_2v_2$ is written as a column vector.) The Morse-Bott condition implies that $a \ne 0$.

\begin{defn} \label{defn: positive negative torus}
$T_{\mathcal N}$ is called a {\em positive} Morse-Bott torus if $a>0$ and a {\em negative} Morse-Bott torus if $a<0$.
\end{defn}

Let us identify a sufficiently small neighborhood of a Morse-Bott torus $T_{\mathcal N}$ with   $T^2 \times [-\nu,\nu]$  with coordinates $(\theta,t,y)$ so that the Reeb vector field is   a positive constant times $\bdry_t$  along $T_{\mathcal N}=\{y=0\}$.  For a positive Morse-Bott torus the Reeb vector field rotates in a counterclockwise manner as $y$ goes from $\nu$ to $-\nu$ (i.e., in the  same direction as a positive contact structure), while for a negative Morse-Bott torus it rotates in a clockwise manner.

On each $\mathcal{N} \simeq S^1$, we pick a Morse function $\overline{g}_{\mathcal N} :\mathcal{N} \to \R$ with two critical points. After perturbing $\alpha$ using these functions, each Morse-Bott family gives rise to an elliptic orbit $e$ and a hyperbolic orbit $h$.

  We choose specific $\alpha$ and perturbations $\alpha_\epsilon$ 
as follows:  Fix a real constant $L>0$ such that no Reeb orbit of $\alpha$ has $\alpha$-action equal to $L$ and let ${\mathcal N}_1, \ldots, \mathcal{N}_n$ be the Morse-Bott families consisting of simple orbits with $\alpha$-action less than $L$.   On the small neighborhood $T^2\times[-\nu,\nu]$ of $T_{\mathcal{N}_i}$, we set
\begin{equation}\label{eqn: alpha and alpha epsilon}
\alpha=Cdt+ \delta(f dt+yd\theta), \quad \alpha_\epsilon=Cdt +\delta(f_\epsilon dt+yd\theta),
\end{equation}
where $\epsilon>0$, $\delta>0$ are small,  $C>0$ is the action of the Reeb orbits of $\mathcal{N}_i$ and:
\begin{itemize}
\item[(P1)] $f(y,\theta)=\pm{1\over 2} y^2$ and $f_\epsilon(y,\theta)=\pm({1\over 2} y^2+ \epsilon\phi(y)\overline{g}_{\mathcal{N}}(\theta))$, where the sign $\pm$ depends on whether we have a negative or positive Morse-Bott torus.
\item[(P2)] $\overline{g}_{\mathcal{N}}:\R/\Z\to \R$ is a perfect Morse function with maximum at ${1\over 4}$ and minimum at $-{1\over 4}$. More specifically, we assume that $\overline{g}_{\mathcal{N}}'(\theta)=0$ on $\theta=\pm{1\over 4}$, is linear with positive slope on $[-{1\over 4},-{1\over 5}]$, is nondecreasing on $[-{1\over 5},-{1\over 6}]$, and is equal to $1$ on $[-\tfrac{1}{6}, \tfrac{1}{6}]$; and $\overline{g}_{\mathcal{N}}(\theta)$ is an odd function about $\theta=0$.
\item[(P3)] $\phi:[-\nu,\nu]\to [0,1]$ is an even bump function with support on $[-a,a]$ and is equal to $1$ on $[-b,b]$, where $\nu> a>b>0$ are sufficiently small.
\end{itemize}
In particular, (P1) implies:
\begin{itemize}
\item[(P4)] as $\epsilon\to 0$, $f_\epsilon\to f$ in $C^\infty$.
\end{itemize}


\begin{prop} \label{prop: perturbation of alpha}
Let $\alpha$ be a Morse-Bott contact form.   After a small modification of $\alpha$ near the $\mathcal{N}_i$ which we still call $\alpha$, for every $L>0$ there exist $\delta>0$, $\nu>0$, and $\epsilon>0$ small such that:
\begin{enumerate}
\item   $\alpha_{\epsilon}$ is $L$-nondegenerate and satisfies Equation~\eqref{eqn: alpha and alpha epsilon} and Conditions~(P1)--(P4);
\item each $\mathcal{N}_i$ is perturbed into a pair of nondegenerate Reeb orbits $e_i$ and $h_i$ of $\alpha_{\epsilon}$-action less than $L$;
\item all multiples $e_i^k$ and $h_i^k$ of $\alpha_{\epsilon}$-action less than $L$ have Conley-Zehnder indices $1$ and $0$ if $\mathcal{N}_i$ is positive and $-1$ and $0$ if $\mathcal{N}_i$ is negative; and
\item all other orbits which are created have $\alpha_{\epsilon}$-action greater than $L$.
\end{enumerate}
Here the Conley-Zehnder indices are computed with respect to the trivialization $\tau$ induced from $\mathcal{N}_i$.
\end{prop}

Strictly speaking, we make the slight modification of $\alpha$ so that it satisfies the conditions of Lemma~\ref{lemma: cx str}.

Let $\mathcal{P}'$ be the set of simple nondegenerate orbits of $R_\alpha$ and let $\mathcal{P}_{MB}= \mathcal{P}'\cup (\cup_i \mathcal{N}_i)$ be the set of all simple Reeb orbits of $R_\alpha$, where $\mathcal{N}_i$ denotes a Morse-Bott family of simple orbits. An {\em orbit set $\gamma$ for the Morse-Bott contact form $\alpha$} is an orbit set constructed from
$\mathcal{P}=\mathcal{P}'\cup(\cup_i\{ h_i,e_i\})$, where $h_i$ is treated as a hyperbolic orbit  (in particular its multiplicity cannot be greater than one) and $e_i$ is treated as an elliptic orbit.

\subsection{Morse-Bott buildings}
\label{subsection: Morse-Bott regularity}

Let $J$ be an almost complex structure on $\R\times M$ which is adapted to the Morse-Bott contact form $\alpha$.   We also assume the following:
\begin{enumerate}
\item[(*)] For each Morse-Bott torus $T_{\mathcal{N}}=T^2$, $J$ is invariant in the $s$-, $t$-, and $\theta$-directions on $\R\times T^2\times[-\nu,\nu]$ and the projection of $J|_{\ker \alpha}$ to $(\R/\Z)\times[-\nu,\nu]$ with coordinates $(\theta,y)$ is the standard complex structure ${\bdry\over\bdry y}\mapsto{\bdry\over \bdry \theta}$.
\end{enumerate}
(Strictly speaking, we require $\alpha$ and $J$ to satisfy the conditions of Lemma~\ref{lemma: cx str}; they can be arranged by a small modification near the Morse-Bott torus.)

\begin{rmk}
In \cite{BEHWZ}, Morse-Bott compactness was proved for slightly different perturbations of $\alpha$, namely for $f_\epsilon\alpha$.  Morse-Bott compactness still holds in our case.
\end{rmk}

Although the notation is a bit cumbersome, consider the moduli space
$$\mathcal{M}_J(\gamma_1^+,\dots,\gamma_{i_1}^+; \mathcal{N}^+_1,\dots,\mathcal{N}^+_{i_2}; \gamma_1^-,\dots,\gamma_{i_3}^-; \mathcal{N}^-_1,\dots,\mathcal{N}^-_{i_4}),$$
abbreviated $\mathcal{M}_J(\gamma^+,\mathcal{N}^+,\gamma^-,\mathcal{N}^-)$, of $J$-holomorphic maps $u$ in $\R\times M$ which have positive ends at orbits $\gamma_1^+,\dots,\gamma_{i_1}^+,\widetilde\gamma_1^+,\dots, \widetilde\gamma_{i_2}^+$ and negative ends at orbits $\gamma_1^-,\dots,\gamma_{i_3}^-,\widetilde\gamma_1^-,\dots, \widetilde\gamma_{i_4}^-$, where $\gamma_i^\pm$ covers a simple orbit in $\mathcal{P}'$ with multiplicity $l_i^{\pm}\ge 1$ and $\widetilde\gamma_i^\pm$ covers a simple orbit in the Morse-Bott family $\mathcal{N}^\pm_i$ with multiplicity $k_i^{\pm} \ge 1$.

We say that $J$   satisfying (*)  is {\em Morse-Bott regular} if, for all data
$\gamma^+,\mathcal{N}^+,\gamma^-,\mathcal{N}^-$ and $u\in
\mathcal{M}_J(\gamma^+,\mathcal{N}^+,\gamma^-,\mathcal{N}^-)$ which
have no multiply-covered components, the moduli space
$\mathcal{M}_J(\gamma^+,\mathcal{N}^+,\gamma^-,\mathcal{N}^-)$ is
transversely cut out (and hence is a manifold) near $u$.    Since it suffices to perturb $J$ outside of the sufficiently small neighborhood $\R\times T^2\times[-\nu,\nu]$, a generic $J$ satisfying (*) is regular.

We now give the definition of a Morse-Bott building. See \cite[Section 11.2]{BEHWZ} for a similar definition.

\begin{defn} \label{Morse-Bott building}
Let $\gamma$ and $\gamma'$ be orbit sets constructed from $\mathcal{P}$. A {\em Morse-Bott building} $\tilde{u}$ consists of a set $\{u_i : F_i \to \R \times M, i=1,\dots,n\}$ of holomorphic maps with possibly disconnected domains $F_i$ and a set $\{\delta_{i,j},i=0,\dots,n, j=1,\dots,j_i\}$ of gradient flow lines in $\cup_k{\mathcal N}_k$ such that the following hold:
\begin{itemize}
\item[(a)] For $i=1, \ldots, n-1$, the negative ends $\mathcal{E}_{i,j}^-$ of $u_i$ are paired with positive ends $\mathcal{E}_{i+1,j'}^+$ of $u_{i+1}$. Paired ends $(\mathcal{E}_{i,j}^-,\mathcal{E}_{i+1,j'}^+)$ are asymptotic to $k_{i,j}$-fold covers of simple orbits $(\gamma_{i,j}^-,\gamma_{i+1,j'}^+)$ in the same Morse-Bott family and $\delta_{i,j}$ is a gradient flow line  from $\gamma_{i,j}^-$ to $\gamma_{i+1,j'}^+$.  (Here $\delta_{i,j}$ can be viewed as a $k_{i,j}$-fold unbranched cover of a cylinder connecting $\gamma_{i,j}^-$ to $\gamma_{i+1,j'}^+$.)
\item[(b)] Positive ends $\mathcal{E}_{1,j}^+$ of $u_1$ and negative ends $\mathcal{E}_{n,j}^-$ of $u_n$ which are asymptotic to Reeb orbits in $\cup_k{\mathcal N}_k$ are augmented by gradient flow lines $\delta_{0,j}$ and $\delta_{n,j}$ connecting the orbit from/to  a critical point of the appropriate Morse function $\overline{g}_{\mathcal{N}_k}$ determined by $\gamma$ or $\gamma'$.
\item[(c)] A nondegenerate orbit is considered as a Morse-Bott family consisting of a single point and in this case the gradient flow line has length zero.
\end{itemize}
Given two orbit sets $\gamma$ and $\gamma'$ constructed from $\mathcal{P}$, the set of Morse-Bott buildings $\tilde{u}$ from $\gamma$ to $\gamma'$ will be denoted by $\mathcal{M}^{MB}_J(\gamma,\gamma')$.
\end{defn}

The collection of maps $u_i$ will be called the {\em holomorphic part} of the building. The restriction of any map $u_i$ to a connected component of its domain will be called an {\em irreducible holomorphic component} of $\tilde{u}$.

\begin{defn} \label{defn: Morse-Bott simply covered}
A Morse-Bott building $\tilde{u}$ from $\gamma$ to $\gamma'$ is {\em simply-covered} if every multiply-covered irreducible holomorphic component of $\tilde{u}$ is either:
\begin{itemize}
\item[(i)] a branched cover of a trivial cylinder over a simple orbit in $\mathcal{P}$; or
\item[(ii)] an unbranched cover of a trivial cylinder over a simple orbit in $\mathcal{P}_{MB}-\mathcal{P}$.
\end{itemize}
\end{defn}
 Note that this definition allows connectors over the orbits $e$ and $h$ of every Morse-Bott torus, but not connectors over any other Morse-Bott orbit, which would necessarily break a gradient flow line. This second type of connectors would make gluing more complicated.

\subsection{ECH and Fredholm indices}

In this subsection we define the ECH and Fredholm indices of a Morse-Bott building.

\begin{defn} \label{defn: ECH index in Morse-Bott setting}
The ECH index $I(\gamma, \gamma', Z)$ in the Morse-Bott setting is defined, as in the nondegenerate case, as
$$I(\gamma, \gamma', Z) = c_1(\xi|_Z, \tau)+Q_\tau(Z) +\widetilde\mu_\tau(\gamma)- \widetilde\mu_\tau (\gamma'),$$
where the symmetric Conley-Zehnder indices of $\gamma$ and $\gamma'$ are computed with the convention that $\mu_\tau(e_i^j) = 1$ for all $j$ and $\mu_\tau(h_i) = 0$ if ${\mathcal N}_i$ is a positive Morse-Bott family and $\mu_\tau(h_i) = 0$ and $\mu_\tau(e_i^j) = -1$ for all $j$ if ${\mathcal N}_i$ is a negative Morse-Bott family. Here $\tau|_{{\mathcal N}_i}$ is the trivialization defined by ${\mathcal N}_i$.
\end{defn}

\begin{rmk}
The ECH index computed with this definition coincides with the limit
of ECH indices computed with respect to nondegenerate perturbations  
$\alpha_\epsilon$  of the Morse-Bott contact form $\alpha$ as
$\epsilon \to 1$.
\end{rmk}

As in the nondegenerate case, a Morse-Bott building $\tilde{u}$ from $\gamma$ to $\gamma'$ determines a relative homology class $Z \in H_2(M, \gamma, \gamma')$ which is obtained from projecting the holomorphic part to $M$ and gluing the annuli corresponding to the gradient trajectories. In view of this construction, we will often write $I(\tilde{u})$ for $I(\gamma, \gamma',Z)$.

We can also define the Fredholm index of a Morse-Bott building as follows. To a building
$\tilde{u}$ we associate a map $u_\# : F_\# \to  \R \times M$ by cutting the ends of the
holomorphic components of $u$ and connecting them with cylinders corresponding to
the gradient trajectories. Then the Fredholm index of a Morse-Bott building $\tilde{u}$
which is positively asymptotic to Reeb orbits $\gamma_i^{m_{ij}}$ and negatively
asymptotic to Reeb orbits $(\gamma_i')^{m_{ij}'}$ is:
\begin{equation} \label{eqn: Fredholm index in Morse-Bott}
\op{ind}(\tilde{u})= - \chi(F_\#) + 2 c_1(u_\#^* \xi, \tau) + \sum_{ij}
\mu_\tau(\gamma_i^{m_{ij}}) -  \sum_{ij} \mu_\tau((\gamma_i')^{m_{ij}'}),
\end{equation}
with the same convention for the Conley-Zehnder indices of $h_i$, $e_i$ and their iterates as in Definition~\ref{defn: ECH index in Morse-Bott setting}. (See \cite[Corollary~5.4]{Bo2}.)

\subsection{Morse-Bott chain complex} \label{subsection: Morse-Bott chain complex}

In this subsection we introduce a Morse-Bott version of the ECH chain complex.  Due to technical difficulties concerning non-simply-covered Morse-Bott buildings, we will develop an ECH Morse-Bott theory only for special Morse-Bott contact forms, which we call {\em nice}.

\begin{defn} \label{defn: niceness} $\mbox{}$
\begin{enumerate}
\item A Morse-Bott building $\tilde u$ is {\em nice} if its holomorphic part has at most one irreducible component which is not a connector. This irreducible component will be called the {\em principal part} of $\tilde{u}$.
\item A Morse-Bott building $\tilde u$ is {\em very nice} if  it is nice and every irreducible component besides the principal part is an unbranched cover of a trivial cylinder.
\item A Morse-Bott contact form $\alpha$ on $M$ is {\em nice} if, for a generic almost complex structure $J$, all $J$-holomorphic Morse-Bott buildings of ECH index $I=1$ in the symplectization of $(M, \alpha)$ are nice.
\end{enumerate}
\end{defn}

\begin{rmk}
We will consider contact forms on manifolds with torus boundary which are nondegenerate on the interior and Morse-Bott on the boundary. Such contact forms are automatically nice (cf.\ Lemma~\ref{fastidio}).  It is not clear whether nice contact forms with $\cup_i \mathcal{N}_i\not=\varnothing$ exist on closed manifolds.
\end{rmk}

Now we describe the relation between moduli spaces of $J$-holomorphic Morse-Bott buildings for a Morse-Bott contact form $\alpha$ and moduli spaces of holomorphic maps for generic perturbations of $\alpha$ following \cite{Bo2}. Our statement will be weaker than that of \cite{Bo2} because we are going to state only
what can be proved without resorting to abstract perturbations.

Let $J_0$ be a Morse-Bott regular almost complex structure on $\R \times M$ adapted to $\alpha$, and let $J_{\epsilon}$ be almost complex structures on $\R \times M$  adapted to the contact forms $\alpha_{\epsilon}$ in Proposition \ref{prop: perturbation of alpha} such that:  
\begin{enumerate}
\item[(**)] For each Morse-Bott torus $T_{\mathcal{N}}=T^2$, $J_\epsilon$ is invariant in the $s$- and $t$-directions on $\R\times T^2\times[-\nu,\nu]$ and the projection of $J_\epsilon|_{\ker \alpha}$ to $(\R/\Z)\times[-\nu,\nu]$ with coordinates $(\theta, y)$ is the standard complex structure ${\bdry\over\bdry y}\mapsto{\bdry\over \bdry \theta}$.
\end{enumerate}
In particular, $\lim \limits_{\epsilon \to 0} J_\epsilon = J_0$ in the $C^{\infty}$-topology.

\begin{thm}\label{thm: Morse-Bott perturbation of moduli spaces}
Let $\alpha$ be a Morse-Bott contact form on $M$.   Fix $L>0$. Then there exist $\delta$, $\nu$, $\epsilon$, and $\alpha_{\epsilon}$ as in Proposition~\ref{prop: perturbation of alpha}, $J_0$ Morse-Bott regular satisfying (*), and $\alpha_\epsilon$-adapted regular $J_\epsilon$ satisfying (**) as in the previous paragraph such that for all orbit sets $\gamma, \gamma' \in {\mathcal P}$ with action less than $L$ the following holds:
\begin{enumerate}
\item For all sequences $\epsilon_i \to 0$ and $u_i \in {\mathcal M}_{J_{\epsilon_i}}(\gamma, \gamma')$, there is a subsequence $u_{i_k}$ which converges to a Morse-Bott building in ${\mathcal M}_{J_0}^{MB}(\gamma, \gamma')$.
\item If $\tilde{u}$ is a very nice, simply-covered Morse-Bott building, then there is a $J_\epsilon$-holomorphic map $u_\epsilon\in {\mathcal M}_{J_\epsilon}(\gamma, \gamma')$ which is ``close to breaking'' into $\tilde{u}$ and a curve $u_{\mathcal{T}^\epsilon_0}$ corresponding to a gradient trajectory $\mathcal{T}^\epsilon_0$ of $f_\epsilon$ along $y=0$.
\item If $\op{ind}(\tilde{u})=1$, then the mod $2$ algebraic count of $[u_\epsilon]\in {\mathcal M}_{J_\epsilon}(\gamma, \gamma')/\R$ that are  ``close to breaking'' into $\tilde{u}$ and $u_{\mathcal{T}^\epsilon_0}$ is one.
\item If $\op{ind}(\tilde{u})=2$ and $\tilde u$ passes through a generic point $z \in \R \times M$, then the mod $2$ algebraic count of $u_\epsilon \in {\mathcal M}_{J_\epsilon}(\gamma, \gamma')$ that are ``close to breaking" into $\tilde{u}$ and $u_{\mathcal{T}^\epsilon_0}$ and passing through $z$ is one.
\end{enumerate}
\end{thm}

\begin{proof}
(1) follows from Morse-Bott SFT compactness~\cite{BEHWZ,Bo2}.  The proofs of (2) and (3) are given in the Appendix;  (4) is similar.
\end{proof}

\begin{lemma} \label{lemma: counting Morse-Bott buildings}
Let $J$ be a Morse-Bott regular almost complex structure and let $\tilde{u} \in
\mathcal{M}_J^{MB} (\gamma,\gamma')$ be a  very nice Morse-Bott building with
$I(\tilde{u})=1$. Then $\tilde{u}$ is simply-covered and $\op{ind}(\tilde{u})=1$.
\end{lemma}

\begin{proof}
Assume that $\tilde{u}$ has no trivial cylinders. In the general case, removing the trivial cylinders of $\tilde u$ might decrease the ECH index by \cite[Theorem~5.1]{Hu2} and positivity of intersection, but the same argument holds.

We first consider the case when the principal part $u$ of $\tilde{u}$ is nonempty.  Suppose that $u$ is a $k$-fold branched cover of a nontrivial simply-covered $J$-holomorphic curve $v$. Let $\tilde{v}$ be the Morse-Bott building obtained by augmenting $v$ with gradient trajectories.  If the functions $\overline{g}_{\mathcal N}$ are chosen generically, then $\op{ind}(\tilde{v}) > 0$ by the regularity of $J$.  Since $\tilde{v}$ is a very nice simply-covered $J$-holomorphic building, by Theorem~\ref{thm: Morse-Bott perturbation of moduli spaces}(2), we can perturb it to a $J_\varepsilon$-holomorphic map $v_\varepsilon$ for $\varepsilon$ small.
Then $I(v_\varepsilon) \ge \op{ind}(v_\varepsilon) >0$ by the ECH index inequality (Theorem \ref{thm: index inequality}), so $I(\tilde{v})>0$.  Consider the $J_\varepsilon$-holomorphic curve $v^k_\varepsilon$ given by $k$ translated copies of $v_\varepsilon$. Since both $\tilde{u}$ and $v_\varepsilon^k$  represent the same relative homology class in $H_2(M, \gamma, \gamma')$, we have $I(\tilde{u})= I(v^k_\varepsilon)$. Since $I(v^k_\varepsilon) \ge k I(v_\varepsilon)$ by \cite[Theorem~5.1 and Proposition~5.6] {Hu2}, it follows that $I(\tilde{u}) \ge k$. Hence $k=1$ and $u$ is simply-covered.

Next let $u_\epsilon$ be the simply-covered $J_\epsilon$-holomorphic map which corresponds to $\tilde{u}$ under an arbitrarily small generic perturbation of the Morse-Bott contact form by Theorem \ref{thm: Morse-Bott perturbation of moduli spaces}(2). Clearly $I(u_\epsilon)=I(\tilde{u})=1$, so $\op{ind}(u_\epsilon) =1$ and $\op{ind}(\tilde u)=1$. This implies the lemma when the principal part of $\tilde{u}$ is not empty.

If the principal part is empty, then $\tilde{u}$ consists of a gradient trajectory on a Morse-Bott family and a  gradient trajectory has ECH index one.
\end{proof}

\begin{lemma}\label{cirget lounge}
Let $\alpha$ be a nice Morse-Bott contact form.  If we fix a regular almost complex structure $J_0$ adapted to $\alpha$, then, for any orbit sets $\gamma$ and $\gamma'$ and any $\epsilon>0$ sufficiently small, there is a bijection
$$\mathcal{M}_{J_0}^{MB,I=1,vn}(\gamma,\gamma')/\R\simeq \mathcal{M}_{J_\epsilon}^{I=1,tn}(\gamma,\gamma')/\R.$$
Here the modifier $vn$ stands for ``very nice'' and the modifier $tn$ means that all the connectors are trivial cylinders.
\end{lemma}

\begin{proof}
By Theorem~\ref{thm: Morse-Bott perturbation of moduli spaces} and Lemma~\ref{lemma: counting Morse-Bott buildings}, every  very nice $I=1$ $J_0$-holomorphic Morse-Bott building can be deformed into an $I=1$ $J_\epsilon$-holomorphic map,  all of whose connectors are trivial cylinders.

It remains to show that every sequence $v_i$ of $J_{\epsilon_i}$-holomorphic maps with $I(v_i)=1$ and trivial cylinders as connectors converges to a  very nice $J_0$-holomorphic Morse-Bott building $\tilde{u}$ as $\epsilon_i\to 0$, after possibly passing to a subsequence. Suppose without loss of generality that the domains of the maps $v_i$ are connected. (Indeed, since $I(v_i)=1$, discarding the possible trivial cylinders does not change $I(v_i)$ by \cite[Theorem~5.1]{Hu2} and positivity of intersection.)  By Theorem \ref{thm: Morse-Bott perturbation of moduli spaces} (1), the sequence $v_i$ converges to a Morse-Bott building $\tilde{u}$ with $I(\tilde{u})=1$. Since $\alpha$ is a nice Morse-Bott contact form, the holomorphic part of $\tilde{u}$ has at most one irreducible component which is not a connector. Assume there is a nontrivial principal part $u_0$; the case of $u_0=\varnothing$ is simpler and is left to the reader.  We consider the very nice Morse-Bott building $\tilde{u}'$ obtained by augmenting the Morse-Bott ends of $u_0$ with gradient flow trajectories to the critical
points of the Morse functions on the Morse-Bott tori. Then $I(\tilde{u}') = I(\tilde{u})=1$ because they represent the same relative homology class, and therefore
Lemma~\ref{lemma: counting Morse-Bott buildings} implies that $u_0$ is simply covered.

We claim that every other irreducible component is a trivial cylinder over an orbit in $\mathcal{P}$.  Arguing by contradiction, suppose there  are nontrivial connectors
that  are connected to $u_0$ by  one or more finite length gradient flow  trajectories. We will show that $\op{ind}(\tilde{u})>1$, which contradicts the fact that $\tilde{u}$ is the limit of curves $v_i$ with $\op{ind}(v_i)=1$.  To this end we  consider the Morse-Bott building $\tilde{u}'$  defined above.  We recall that $\tilde{u}'$ is  very nice, simply-covered, and $I(\tilde{u}')=I(\tilde{u}) =1$.

 The ends of the building $\tilde{u}$ satisfy the incoming/outcoming partitions because it is the limit of $J_{\varepsilon_i}$-holomorphic maps $v_i$, while the ends of the building $\tilde{u}'$ satisfy the incoming/outgoing partitions because $\tilde{u}'$ can be deformed to $J_{\varepsilon_i}$-holomorphic maps $v_i'$ for $i \gg 0$ by Theorem~\ref{thm: Morse-Bott perturbation of moduli spaces}(2).

We make now the simplifying hypothesis that $u_0$ has ends only at one Morse-Bott torus. (The general case is more complicated only in the notation.) Then the ends of $\tilde{u}$ and $\tilde{u}'$ differ only for the multiplicity of $e$. We denote by $n_+$ and $n_-$ the positive and negative multiplicities of $e$ in $\tilde{u}$, respectively, and by $n_+'$, $n_-'$ the corresponding multiplicities in $\tilde{u}'$. Moreover, we denote by $\mu(e, n_\pm)$ and $\mu(e, n_\pm')$ the contributions of ends at $e$ to the Fredholm indices of $\tilde{u}$ and $\tilde{u'}$ respectively. (We recall that these contributions are determined by the total multiplicities because $\tilde{u}$ and $\tilde{u'}$ satisfy the incoming/outgoing partition conditions.) We observe that
$n_\pm \ge n_\pm'$ and $n_+ - n_+' = n_- - n_-'$.

Let $F$ be the domain of $v_i$ and $F'$ the domain of $v_i'$ for $i \gg 0$.
Then, by the Fredholm index formula \eqref{eqn: Fredholm index in Morse-Bott}, we have
$$\op{ind}(\tilde{u}) - \op{ind}(\tilde{u}') = - (\chi(F) - \chi(F')) + (\mu(e, n_+) - \mu(e, n_-)) - (\mu(e, n_+') - \mu(e, n_-')).$$
The term $\chi(F) - \chi(F')$ is the sum of the Euler characteristics of the connector components of $\tilde{u}$, and therefore $- (\chi(F) - \chi(F')) >0$ if $\tilde{u}$ is not very nice. Now we claim that the term  $(\mu(e, n_+) - \mu(e, n_+)) - (\mu(e, n_+') - \mu(e, n_+'))$ is always nonnegative. To see this, first we compute the contributions of the ends at $e$ to the Fredholm index.
If the Morse-Bott torus is positive, then
 $$\mu(e, n_+) = n_+, \quad \mu(e, n_-) = \left \{ \begin{array}{l} 0 \text{ if } n_-=0, \\ 1 \text{ if } n_- >0. \end{array} \right.$$

On the other end, if the Morse-Bott torus is negative, then
$$ \mu(e, n_+) = \left \{ \begin{array}{l} 0 \text{ if } n_+=0, \\ - 1 \text{ if } n_+ >0, \end{array} \right. \quad \mu(e, n_-) = -n_-.$$
Similar formulae hold for $\mu(e, n_\pm')$.

Now we focus on the case of a positive Morse-Bott torus. (The case of a negative one is completely symmetric.) Then  $(\mu(e, n_+) - \mu(e, n_-)) - (\mu(e, n_+') - \mu(e, n_-'))= n_+ - n_+' \ge 0$ if $n_-, n_-'>0$ or $n_- = n_-'=0$. On the other hand, if
$n_- >0$ but $n_-'=0$, we have $(\mu(e, n_+) - \mu(e, n_-)) - (\mu(e, n_+') - \mu(e, n_-'))= n_+ - n_+' -1$. However, in this case, $n_+ - n_+' = n_--n_-' >0$, so  $(\mu(e, n_+) - \mu(e, n_-)) - (\mu(e, n_+') - \mu(e, n_-')) \ge 0$.

This proves that, if $\tilde{u}$ is not very nice, $\op{ind}(\tilde{u}) > \op{ind}(\tilde{u}')$.
This is a contradiction because $\op{ind}(\tilde{u})=1$, as $\tilde{u}$ is a limit of $\op{ind}=1$ maps $v_i$, and $\op{ind}(\tilde{u}) > 0$ since $J_0$ is a Morse-Bott regular almost complex structure.
\end{proof}

\begin{defn} \label{defn: Morse bott complex}
Let $\alpha$ be a nice Morse-Bott contact form and $J$ a Morse-Bott regular almost
complex structure adapted to the symplectization of $\alpha$. Then the
Morse-Bott chain complex $(ECC_{MB} (M,\alpha,J),
\bdry_{MB})$ is generated by orbit sets constructed from
$\mathcal{P}$ and the differential counts  very nice Morse-Bott buildings with
$I(\tilde{u})=1$.  We denote by $ECC_{MB}^L(M,\alpha,J)$ the subcomplex generated by orbit sets of action less than $L$.
\end{defn}

\begin{prop}\label{prop: from generic to MB}
Let $\alpha$ be a nice Morse-Bott contact form. If no Reeb orbit of $\alpha$ has action equal to a fixed $L>0$, then there is an isomorphism of chain complexes
$$ECC^{L}(M,\alpha_{\epsilon},J_{\epsilon})\simeq ECC^{L}_{MB}(M,\alpha,J_0),$$
for all sufficiently small $\epsilon>0$. In particular, $\bdry_{MB}^2=0$.
\end{prop}

\begin{proof}
The isomorphism follows from Lemma \ref{cirget lounge}.
\end{proof}

\subsection{Comparison with the nondegenerate case} \label{subsection: direct limit arguments}

In this subsection we use a direct limit argument to prove the isomorphism between ECH of a nondegenerate contact form and Morse-Bott ECH of a nice Morse-Bott form $\alpha$.

Let $L_i\to \infty$ be an increasing sequence such that each $L_i$ is positive and there is no Reeb orbit of $\alpha$ with action equal to $L_i$. Let $\mathcal{N}_1,\dots,\mathcal{N}_{n(i)}$ be the Morse-Bott families consisting of simple orbits with $\alpha$-action $< L_i$. (In many useful cases $\lim \limits_{i \to + \infty} n(i) = +\infty$.)

The following lemma provides a sequence of perturbing functions and is an immediate corollary of Proposition~\ref{prop: perturbation of alpha} and Theorem \ref{thm: Morse-Bott perturbation of moduli spaces}.

\begin{lemma} \label{bourgeois bis}
Let $\alpha$ be a nice Morse-Bott form and let $L_i$ be a sequence of positive constants such that $L_i \to +\infty$ and no Reeb orbit of $\alpha$ has action equal to $L_i$. There exist sequences of positive numbers $\epsilon_i\to 0$ and functions $g_i: M \to \R^{\geq 0}$ such that $f_i=1+g_i$ and:
\begin{enumerate}
\item $g_i$ is supported in disjoint neighborhoods of $T_{\mathcal{N}_1}\cup \dots \cup T_{\mathcal{N}_{n(i)}}$;
\item the support of $g_i$ is disjoint from all nondegenerate Reeb orbits of $\alpha$ of $\alpha$-action $< L_i$;
\item on a sufficiently small neighborhood $T^2\times [-\varepsilon,\varepsilon]$ of $T_{\mathcal{N}_j}$, $j=1,\dots,n(i)$,   there exist precisely two simple orbits of $f_i\alpha$ of action $\leq L_i$ corresponding to elliptic and hyperbolic orbits of the perturbed Morse-Bott family;
\item $\lim \limits_{i \to + \infty} f_i =1$ in the $C^k$-topology for $k\gg 0$;
\item for every $i$, the contact form $f_i \alpha$ satisfies Conditions~(1)--(4) of Proposition~\ref{prop: perturbation of alpha} and the conclusion of   Theorem~\ref{thm: Morse-Bott perturbation of moduli spaces} for orbits of action $\leq L_i$; and
\item $f_i \alpha$ (resp.\ $f_{i+1}\alpha$) has no Reeb orbits with $f_i \alpha$-action (resp.\ $f_{i+1}\alpha$-action) in the interval $[a_i^{-2} L_i, a_i^2 L_i]$, where $a_i=(1+\epsilon_ic_0)^2$ for some constant $c_0>0$.
\end{enumerate}
\end{lemma}

\begin{warn}\label{just a silly remark}
For all $i$, Morse-Bott theory (and in particular Proposition~\ref{prop: from generic to MB})
gives injections $ECC^{L_i}(M, f_i \alpha) \to ECC^{L_{i+1}}(M, f_{i+1} \alpha)$.
However, the maps induced in homology by these injections {\em a priori} could be
different from the canonical maps given in Lemma~\ref{tired with all these}, and it
is with respect to the latter that the direct limit must be taken.  ({\em A posteriori,} they are shown to be the same in the proof of Theorem~\ref{thm: main result of Morse-Bott}.)
\end{warn}

Observe that $a_i^{-1} f_i < f_{i+1} < a_if_i$ for all $i$. Then Lemma~\ref{tired with all these} gives maps
$$\Phi_+: ECH^{L_i}(M, f_i \alpha) \to  ECH^{a_i L_i}(M, f_{i+1} \alpha),$$
$$\Phi_-: ECH^{a_i L_i}(M,f_{i+1}\alpha) \to ECH^{a_i^2 L_i}(M,f_i\alpha),$$
$$\Phi_-': ECH^{a_i^{-1} L_i}(M,f_{i+1}\alpha) \to ECH^{L_i}(M,f_i\alpha).$$

\begin{lemma}\label{chesialavoltabuona}
The map $\Phi_+$ is an isomorphism.
\end{lemma}

\begin{proof}
By Theorems~\ref{thm: Hutchings Taubes cobordism map}(ii) and (iv), the composition
$$ECH^{L_i}(M, a_i f_{i+1} \alpha) \to  ECH^{L_i}(M, f_{i} \alpha) \to ECH^{L_i}(M,a_i^{-1} f_{i+1} \alpha)$$
is equal to the cobordism map induced by a piece of symplectization. Then by Theorem~\ref{thm: Hutchings Taubes cobordism map}(vi) it is a composition of a scaling with an inclusion. From this and Lemma~\ref{bourgeois bis}(9), it follows easily that $\Phi_+ \circ \Phi_-' =id$. Similarly, $\Phi_-\circ \Phi_+=id$.  Hence $\Phi_+$ is an isomorphism.
\end{proof}

Let $([0,1] \times M, d \lambda_i)$ be an interpolating cobordism from  $f_i \alpha$ to $a_i^{-1}f_{i+1} \alpha$ and $(\R \times M, d \widehat{\lambda}_i)$ its completion. Let $\widetilde J_i$ be a regular almost complex structure on $(\R \times M, d \widehat{\lambda}_i)$ which is $d \lambda_i$-compatible and adapted to the symplectizations of  $f_i \alpha$ and $a_i^{-1}f_{i+1} \alpha$ at the ends. We denote the moduli space of $\widetilde J_i$-holomorphic {\em buildings} in $(\R \times M, d \widehat{\lambda}_i)$ from $\gamma$ to $\gamma'$ by ${\mathcal M}^{b}_{\widetilde J_i}(\gamma, \gamma')$.

The following lemma, stated without proof, is a consequence of the Morse-Bott compactness theorem~\cite{Bo2}  and the triviality of $I<0$ moduli spaces in symplectizations.

\begin{lemma}\label{rivopietroso}
If $\epsilon_i>0$ is sufficiently small, then there is a regular almost complex structure  ${\widetilde J_i}$ such that, if $\gamma$ and $\gamma'$ have  $f_i \alpha$-actions less than $L_i$, then the moduli spaces ${\mathcal M}_{\widetilde J_i}^{b,I=0}(\gamma, \gamma')$ and ${\mathcal M}^{b,I=0}_{\widetilde J_i}(\gamma', \gamma)$ are empty if $\gamma \ne \gamma'$ and consist of  branched covers of trivial holomorphic cylinders if $\gamma = \gamma'$.
\end{lemma}

By Morse-Bott theory there is an identification of complexes
$$e: ECC^{L_i}(M, f_i \alpha,J_i) \stackrel{\simeq} \longrightarrow ECC^{a_iL_i}(M, f_{i+1} \alpha,J_{i+1}).$$
In fact, $ECC^{L_i}(M, f_i \alpha)$ and $ECC^{a_iL_i}(M, f_{i+1}\alpha)$ are generated by the same orbit sets and the moduli spaces of $I=1$ holomorphic curves (modulo $\R$-translations) have the same cardinality, by Lemma \ref{bourgeois bis} and Proposition \ref{prop: from generic to MB}. Let $e_*$ be the map induced by $e$ on homology.

\begin{prop} \label{continuation maps are easy}
$e_*= \Phi_+$.
\end{prop}

\begin{proof}
Let
$$\widehat{\Phi}_+ : ECC^{L_i}(M, f_i \alpha,J_i) \to ECC^{a_i L_i}(M, f_{i+1} \alpha,J_{i+1})$$
be a (noncanonical) chain map which induces $\Phi_+$ and is given by Theorem~\ref{thm: Hutchings Taubes cobordism map} and Lemma~\ref{tired with all these}. Theorem~\ref{thm: Hutchings Taubes cobordism map}(i) and Lemma~\ref{rivopietroso} imply that $\widehat{\Phi}_+$ is a diagonal map. Note that $ECC^{L_i}(M, f_i \alpha)$ and $ECC^{a_iL_i}(M, f_{i+1} \alpha)$ are generated by the same orbit sets.  The reason why we cannot conclude that $\widehat{\Phi}_+=e$ by Theorem~\ref{thm: Hutchings Taubes cobordism map}(i) is that some of the $I=0$ holomorphic cylinders in the interpolating cobordism from $(M, a_i f_i \alpha)$ to $(M,f_{i+1}\alpha)$ are, strictly speaking, not contained in product regions.

For $\F$-coefficients we can use the following algebraic trick to finish the proof: Identify $ECC^{a_iL_i}(M, f_{i+1} \alpha)$ with $ECC^{L_i}(M, f_i \alpha)$ via  $e^{-1}$. Then
$$(e^{-1}\circ\widehat{\Phi}_+) \circ (e^{-1}\circ \widehat{\Phi}_+) = e^{-1}\circ \widehat{\Phi}_+$$ over $\F$.  Since $\Phi_+$ and $e_*$ are isomorphisms, it follows that $e^{-1}_*\circ\Phi_+=id$ and $\Phi_+ = e_*$.
\end{proof}

 Now we give a sketch of the proof  of Proposition~\ref{continuation maps are easy} which applies to  integer coefficients.  The uninterested reader can jump directly to Theorem~\ref{thm: main result of Morse-Bott}. Given a pair $(\lambda,J)$ consisting of a nondegenerate contact form $\lambda$ and a compatible $J$, Taubes~\cite{T2} first perturbs $(\lambda,J)$ into an {\em $L$-flat pair} $(\lambda',J')$ before identifying $ECH^L(\lambda',J')$ with Seiberg-Witten Floer cohomology. A pair $(\lambda',J')$ is {\em $L$-flat} if near each Reeb orbit of length $<L$ it satisfies the conditions in \cite[Equation~(4-1)]{T2}, and  $L$-flat perturbations are constructed in \cite[Proposition~2.5 and Appendix]{T2}. (See \cite[Section~5.c, Part 2]{T2} for the reasons for introducing the $L$-flat condition.)

The following lemma is a slight rephrasing of \cite[Lemma~3.4(d)]{HT4} and will not be proved:

\begin{lemma} \label{lemma: canonical bijection}
If $(\lambda^t,J^t,L^t)$, $t\in[0,1]$, is a $1$-parameter family and $(\lambda^t,J^t)$ is $L^t$-flat, $\lambda^t$ is $L^t$-nondegenerate, and $J^t$ is $L^t$-regular (i.e., Definition~\ref{defn: regular} holds for all $\gamma,\gamma'$ with $\mathcal{A}_{\lambda^t}(\gamma)< L^t$) for all $t\in[0,1]$, then the ECH cobordism map
$$ECH^{L^0}(M,\lambda^0)\to ECH^{L^1}(M,\lambda^1)$$ is induced by the isomorphism
$$ECC^{L^0}(M,\lambda^0,J^0)\stackrel\sim\to ECC^{L^1}(M,\lambda^1,J^1)$$ given by the canonical bijection of generators.
\end{lemma}

Setting $\lambda^0=f_i\alpha$ and $\lambda^1=f_{i+1}\alpha$, it is easy to find an extension $\lambda^t$, $t\in[0,1]$, of the form $f_{i+1}^t\alpha$, where $f^0_{i+1}=f_i$, $f^1_{i+1}=f_{i+1}$, and $f_{i+1}^t$ satisfies the conditions of Lemma~\ref{bourgeois bis} with $f_{i+1}$ replaced by $f_{i+1}^t$.  By choosing $f_{i+1}^t$ to be sufficiently close to $1$ and applying Lemma~\ref{cirget lounge}, there exist an extension $L^t$, $t\in[0,1]$, of $L^0=L_i$ and $L^1=a_iL_i$ and an extension $J^t$, $t\in[0,1]$, of $J^0=J_i$ and $J^1=J_{i+1}$, such that $J^t$ is adapted to $\lambda^t$ and is $L^t$-regular.

Next we fix a Riemannian metric on $M$, with respect to which we measure distances. Assume for simplicity that there is a unique Morse-Bott torus $T_{\mathcal{N}}$. Let $\gamma_e$ and $\gamma_h$ be the elliptic and hyperbolic orbits of $\lambda^t$ which are obtained by perturbing $T_{\mathcal{N}}$, where we assume that $\gamma_e\sqcup \gamma_h$ is independent of $t\in[0,1]$. For each $\varepsilon>0$ sufficiently small, we construct an $L^t$-flat family ($t\in[0,1]$) of perturbations $(\lambda^{t,\varepsilon},J^{t,\varepsilon})$ of $(\lambda^t,J^t)$ which are supported on an $\varepsilon$-neighborhood of $\gamma_e\sqcup \gamma_h$. Moreover, $(\lambda^{t,\varepsilon},J^{t,\varepsilon})$ converges (uniformly in $t\in[0,1]$) to $(\lambda^t,J^t)$ in the $C^0$-topology as $\varepsilon\to 0$.  The proof is a $1$-parameter version of the construction of $L$-flat perturbations in \cite[Proposition~2.5 and Appendix]{T2} and will be omitted.

\begin{claim} \label{claim: regularity}
For $\varepsilon>0$ sufficiently small, $J^{t,\varepsilon}$ is $L^t$-regular for all $t\in[0,1]$.
\end{claim}

\begin{proof}[Proof of Claim~\ref{claim: regularity}]
We may assume that $J^{t,\varepsilon}$, $t\in[0,1]$, is a generic $1$-parameter family of almost complex structures. Arguing by contradiction, there exist orbit sets $\gamma$, $\gamma'$ and sequences $\varepsilon_j\to 0$, $t_j\in[0,1]$, and $u_j:F_j\to \R\times M$, where:
\begin{enumerate}
\item $u_j$ is a somewhere injective $J^{t_j,\varepsilon_j}$-holomorphic curve from $\gamma$ to $\gamma'$;
\item $\gamma$ and $\gamma'$ are constructed from the nondegenerate orbits of $\alpha$ together with $\gamma_e$ and $\gamma_h$ and $\mathcal{A}_{\lambda^{t_j,\varepsilon_j}}(\gamma), \mathcal{A}_{\lambda^{t_j,\varepsilon_j}}(\gamma')< L^{t_j}$;
\item $u_j$ is not a connector and  $I(u_j)=\op{ind}(u_j)=0$.
\end{enumerate}

\begin{claim} \label{claim: C zero SFT compactness}
After passing to a subsequence, there exists an SFT limit $u_j\to u_\infty$, where $I(u_\infty)=\op{ind}(u_\infty)=0$ and $u_\infty$ is not a connector.
\end{claim}

A sketch of Claim~\ref{claim: C zero SFT compactness} is given in Section~\ref{subsection: blocking lemma}.  Since $u_\infty$ is a $J^{t_0}$-holomorphic curve and $J^{t}$ is $L_{t}$-regular for all $t\in[0,1]$ by Lemma~\ref{cirget lounge}, we have a contradiction.  This implies Claim~\ref{claim: regularity}.
\end{proof}

Claim~\ref{claim: regularity} and Lemma~\ref{lemma: canonical bijection} then imply Proposition~\ref{continuation maps are easy}  for integer coefficients.

By passing to direct limits, we obtain the main result of Morse-Bott theory.
\begin{thm} \label{thm: main result of Morse-Bott}
Let $\alpha$ be a nice Morse-Bott form and $J$ a generic almost complex structure
adapted to the symplectization of $\alpha$. Then we have
$$ ECH_{MB}(M,\alpha,J) \simeq ECH(M).$$
\end{thm}

\begin{proof}
Choose sequences of functions  $f_i: M \to \R$ and constants $L_i \to + \infty$ which
satisfy the hypotheses of Lemmas \ref{bourgeois bis} and \ref{rivopietroso}.
Then
\begin{equation} \label{eqn: Morse-Bott direct limit}
ECH(M) = \lim_{i\to \infty} ECH^{L_i}(M,f_i\alpha)
\end{equation}
by Corollary \ref{cor: direct limit of commensurate contact forms} and
\begin{equation} \label{belin}
ECC^{L_i}(M,f_i\alpha, J_i) \simeq ECC^{L_i}_{MB}(M,\alpha, J)
\end{equation}
for all $i$ by Proposition \ref{prop: from generic to MB}.
Also, tautologically, $$ ECC_{MB}(M,\alpha, J) =  \lim_{i\to \infty} ECC^{L_i}_{MB}(M,
\alpha, J).$$
In order to take the direct limit on both sides of Equation~\eqref{belin} on the level of
homology, we need the commutativity of the following diagram for all $i$:
\begin{diagram}
ECH^{L_i}_{MB}(M,\alpha) & \rTo^{\simeq} &
ECH^{L_i}(M, f_i \alpha)  \\
\dTo & &  \dTo \\
ECH^{L_{i+1}}_{MB} (M,\alpha) & \rTo^{\simeq} &
ECH^{L_{i+1}} (M, f_{i+1} \alpha)
\end{diagram}
where the rightmost vertical arrow is the natural map defined in Lemma~\ref{tired with all these} from interpolating cobordisms. This map coincides with $\Phi_+$ followed by the map induced by the inclusion
$$ECC^{a_iL_i}(M, f_{i+1} \alpha) \hookrightarrow ECC^{L_{i+1}}(M, f_{i+1} \alpha).$$
Therefore the diagram commutes by Proposition~\ref{continuation maps are easy}.
\end{proof}

\section{Topological constraints on holomorphic curves} \label{section: topological constraints}

\subsection{The winding number}

In this subsection we recall the {\em winding number} from \cite[p.\ 290]{HWZ1}: Given a
contact manifold $(M,\xi)$ with $\xi =\ker\alpha$, an $\alpha$-adapted almost
complex structure $J$ on $\R\times M$, and a $J$-holomorphic curve
$u: F \to \R \times M$ between orbits sets, the {\em winding number}
$\op{wind}_\pi(u)$ is an algebraic count of the zeros of the
section:
$$s:F\to \op{Hom}_{\C}(TF,u^*\xi).$$ Here $s$ is obtained by
composing
$$TF\stackrel{u_*}\to T(\R\times M)
\stackrel{(\pi_M)_*}\longrightarrow TM\stackrel{\pi} \to \xi, $$
where $\pi_M:\R\times M\to M$ is the projection onto the second
factor and $\pi$ is the projection along the Reeb vector field
$R_\alpha$.

In \cite{HWZ1}, Hofer-Wysocki-Zehnder prove that
$\op{wind}_\pi(u)$ is finite.  (This is analogous to the elementary complex analysis fact that the number of zeros of a holomorphic function $f:D^2\subset \C\to \C$, counted with multiplicities, is equal to the winding number of $f|_{\bdry D^2}$.)  An immediate corollary is the following lemma:

\begin{lemma} \label{lemma: winding number}
The map $u_M = \pi_M \circ u$ is transverse to $R_\alpha$ away from a
finite number of points on $F$. In particular it is an immersion outside a
finite number of points on $F$.
\end{lemma}

Throughout the section we will use the notation $u_M = \pi_M \circ u$.

\subsection{Blocking Lemma} \label{subsection: blocking lemma}

In this subsection we discuss the topological restrictions that a torus foliated by Reeb trajectories imposes on the $J$-holomorphic curves.

Let $\alpha$ be a contact form on $M$ and $T\subset M$ an oriented torus which is linearly foliated by Reeb trajectories of $\alpha$. The foliation can either have closed leaves or dense leaves. We denote by $\P^+ H_1(T; \R)$ the quotient of $H_1(T; \R) - \{ 0 \}$ by multiplication of positive real numbers.  The Reeb flow on $T$ will then have a well-defined ``slope'' $s \in \P^+ H_1(T; \R)$.

Let $\langle,\rangle$ be the intersection pairing on $H_1(T;\R)$.  We then make the following definition:

\begin{defn}
If $\delta \in H_1(T; \Z)$, then we write $\delta \cdot s >0$ (resp.\ $\delta\cdot s=0$) if $\langle \delta,\gamma\rangle>0$ (resp.\ $=0$) for any representative $\gamma  \in H_1 (T;\R)$ of $s  \in \P^+ H_1(T; \R)$.  \end{defn}

Note that if $\delta\cdot s=0$, then $\delta$ represents the slope $s$ or $-s$.

Let $T^2\times[-\varepsilon,\varepsilon]$ be a neighborhood of the Morse-Bott torus $T=T^2\times\{0\}$ with coordinates $(\theta,t,y)$. We assume that the normal vector to $T$ points in the direction of $\bdry_y$.  Let $u : F \to \R \times M$ be a $J$-holomorphic curve such that:
\begin{enumerate}
\item[(C$_1$)] $F$ is a compact Riemann surface with boundary $\bdry F$; and
\item[(C$_2$)] $u_M(\bdry F)\cap (T^2\times[-\varepsilon,\varepsilon])  =\varnothing$.
\end{enumerate}
Then $u_M(F) \cap T$ only has a finite number of singularities by Lemma~\ref{lemma: winding number} and we denote by $\delta \in H_1(T; \Z)$ the homology class of $u_M(F) \cap T$, where the smooth part of $u_M(F) \cap T$ is oriented as the boundary of $u_M(F)\cap (T^2\times[-\varepsilon,0])$.

\begin{lemma}[Positivity of intersections in dimension three] \label{lemma: positive slope}
Let $T \subset M$ be an oriented torus which is linearly foliated by Reeb trajectories of slope $s$. If $u : F \to \R \times M$ is a $J$-holomorphic curve satisfying (C$_1$) and (C$_2$) and $\delta= [u_M(F) \cap T]\in H_1(T; \Z)$, then $\delta \cdot s \ge 0$. Moreover, $\delta \cdot s = 0$ if and only if $u_M(F) \cap T= \varnothing$.
\end{lemma}

\begin{proof}
We will prove this lemma in the harder case when $T$ is foliated by orbits of irrational slope, leaving the rational slope case to the reader.

By Lemma~\ref{lemma: winding number},  $u_M(F) \cap T$, if not empty, is the union of a finite set of points and curves which are immersed outside a finite number of singularities.

Assume first that $u_M(F) \cap T$ has a one-dimensional component. By abuse of notation, we do not distinguish between the homology class $\delta$ and its representative $u_M(F) \cap T$. A generic finite length Reeb trajectory $\gamma$ on $T$ intersects $\delta$ in finitely many points away from the singularities and isolated points. In fact, $\delta \cap \gamma = \pi_M(u(F) \cap (\R \times \gamma))$ and $u(F) \cap (\R \times \gamma)$ is a finite set by the intersection theory of holomorphic curves in dimension four; see \cite[Theorem 7.1]{MW}.  Since all Reeb trajectories are dense in $T$, we can choose $\gamma$ arbitrarily long so that its endpoints are close to each other and far away from $\delta$.  Hence we can complete $\gamma$ to a homologically nontrivial closed curve $\overline{\gamma}$ without introducing extra intersection points with $\delta$. Then the positivity of intersections in dimension four implies that $\delta \cdot [\overline{\gamma}] >0$.  In particular, $\delta\not=0\in H_1(T;\Z)$.  Since we can make the slope of $\overline{\gamma}$ as close as we want to $s$ by
taking $\gamma$ sufficiently long and $s$ is not an integral homology class, we conclude that $\delta \cdot s > 0$.  (Recall that if $\delta\cdot s=0$ then $\delta$ and $s$ or $-s$ are parallel.)

Assume now that $u_M(F) \cap T$ is a finite set. We claim that $u_M(F) \cap T= \varnothing$. Suppose that $u_M(F) \cap T \ne \varnothing$ by contradiction. Repeating the construction from the previous paragraph with a finite length Reeb trajectory $\gamma$ (resp.\ $\gamma'$) which passes through a point in $u_M(F) \cap T$ (resp.\  is disjoint from $u_M(F)\cap T$), we obtain $\overline\gamma$ and $\overline\gamma'$. Then $[u(F)] \cdot [\R \times \overline{\gamma}] >0$ by the positivity of intersections, while $[u(F)] \cdot [\R \times \overline{\gamma}'] = 0$ because they are disjoint. Since $\R \times \overline{\gamma}$ and $\R \times \overline{\gamma}'$ are homologous, we have a contradiction. This completes the proof of the lemma.
\end{proof}

The following is an immediate consequence of Lemma~\ref{lemma: positive slope}.

\begin{lemma}[Blocking Lemma] \label{lemma: blocking lemma}
Let $T \subset M$ be an oriented torus which is linearly foliated by Reeb trajectories of slope $s$ and let $u : F \to \R \times M$ be a finite energy $J$-holomorphic map, where $F$ is a closed Riemann surface with a finite number of punctures removed. Then:
\begin{enumerate}
\item If $u$ is homotopic, by a compactly supported homotopy, to a map whose image is disjoint from $\R \times T$, then $u_M(F) \cap T = \varnothing$.
\item If $T'$ is a torus which is parallel to and disjoint from $T$, $u$ has no end that limits to a Reeb orbit that intersects the half-open region between $T$ and $T'$ which includes $T'$ but not $T$, and the homology class $[u_M(F)\cap T']$  is nonzero and has slope $\pm s$, then $u$ has an end which is asymptotic to a Reeb orbit in $T$.
\end{enumerate}
\end{lemma}

We now sketch the proof of Claim~\ref{claim: C zero SFT compactness}.

\begin{proof}[Sketch of proof of Claim~\ref{claim: C zero SFT compactness}]
The consideration that needs slight care is that as $\varepsilon_j\to 0$, $J^{t_j,\varepsilon_j}\to J^{t_\infty}$ only in the $C^0$-topology.  Let $N_{\varepsilon_j}(\gamma_e\sqcup \gamma_h)$ be an $\varepsilon_j$-neighborhood of $\gamma_e\sqcup \gamma_h$ and let $F_j'= u_j^{-1}(\R\times (M-N_{\varepsilon_j}(\gamma_e\sqcup \gamma_h)))$. By the Gromov-Taubes compactness theorem~\cite{T4}, which requires no a priori bound on the genus and is local, there exists a limit $u_\infty$ of $u_j|_{F'_j}$ as currents, after passing to a subsequence.  This implies that the homology class $[u_j(F_j)]$ can be taken to be independent of $j$. The argument from \cite[Theorem~10.1]{Hu} then gives a bound on the genus of $F_j$.

We can then either appeal to the $C^0$-Gromov compactness theorem of Ivashko\-vich-Shevchishin~\cite{IS} or argue as follows using the Blocking Lemma.  We make the simplifying assumption that $\gamma$ and $\gamma'$ do not contain $\gamma_h$ and that $u_j$ does not intersect neighborhoods of $\R\times \gamma_h$ and leave the harder general case to the reader.

We claim that $-\chi(F_j')$ is bounded above. Since we have a genus bound for $F_j$, it suffices to show that the number $\# \bdry F_j'$ of boundary components of $F_j'$ is bounded above. Let $V'_j= N_{\varepsilon_j}(\gamma_e)$ and let $T'_j=\bdry V'_j$ with the boundary orientation.  Choose an oriented identification $T'_j\simeq \R^2/\Z^2$ such that the meridian has slope $0$ and the longitude is determined by the Morse-Bott family and has slope $\infty$. We may assume that $T'_j$ is foliated by Reeb orbits of slope $s'_j\gg 0$ and that there exists a torus $T_j''\subset M- V_j'$ which is parallel to $T'_j$ and is foliated by Reeb orbits of rational slope $s''$, where $s''$ is independent of $j$ and $s'_j> s''>0$.

Let $V_j''\subset M$ be the solid torus bounded by $T_j''$ and let $F_j''= u_j^{-1}(\R\times (M-V''_j))$.  Let $\pi_M: \R\times M\to M$ be the projection onto the second factor. By Lemma~\ref{lemma: positive slope}, if $C$ is a component of $\bdry F_j''$, then $\pi_M\circ u_j(C)\cdot s'<0$. Since $[u_j(F_j)]$ is fixed and $s''$ is rational, $\#\bdry F_j''$ must then be bounded above.  On the other hand, let $V_j^{(0)}\subset V_j'$ be a sufficiently small neighborhood of $\gamma_e$, $T_j^{(0)}=\bdry V_j^{(0)}$, and $F_j^{(0)}=u_j^{-1}(\R\times (M-V^{(0)}_j))$. Since $[u_j(F_j)]$ is fixed, $\#\bdry F_j^{(0)}$ is also bounded above by the positivity of intersections in dimension four and the asymptotics of $u_j$ near their ends.

To obtain the bound on $\#\bdry F_j'$, it then suffices to show that $u_j^{-1}(\R\times (V''_j- V'_j))$ and $u_j^{-1}(\R\times (V'_j-V^{(0)}_j))$ have no disk components $D$ with $\pi_M\circ u_j(\bdry D)\subset T'_j$.  By Lemma~\ref{lemma: positive slope}, $\pi_M\circ u_j(\bdry D)$ represents a nonzero homology class in $T'_j$. On the other hand, the inclusion $T^2\times\{1\}\to T^2\times[0,1]$ induces an isomorphism on homology, which is a contradiction.  This proves the bound on $\#\bdry F_j'$ and $-\chi(F_j')$.

We then apply the SFT compactness theorem to $u_j|_{F_j'}$ to obtain $u_\infty: F_\infty \to \R\times M$.  If $C$ is a component of $\bdry F_\infty$, then $u_\infty(C)\subset \gamma_e$, which in turn implies that $u_\infty$ is a constant.  Hence $\bdry F_\infty=\varnothing$. The punctures of $F_\infty$ are either removable or limit to orbits in $\gamma,\gamma'$.  Finally, since $[u_j(F_j)]$ is not the class given by a connector, $u_\infty$ is also not a connector.
\end{proof}

\subsection{Trapping Lemma}

In this subsection we analyze some topological restrictions on $J$-holomorphic curves with ends at a Morse-Bott torus.

We fix coordinates $(\theta, t, y)$ on a neighborhood $T^2\times[-\varepsilon,\varepsilon]$ of a torus $T=T^2\times\{0\}$ and consider contact forms of the type
$\alpha = g(\theta, t, y) d \theta + f(\theta, t, y) dt$ such that:
\begin{itemize}
\item $f (\partial_y g) - (\partial_y f) g >0$,
 \item $f|_{y=0} =1$,
\item $\partial_\theta f|_{y=0} = \partial_t f|_{y=0} = \partial_y f|_{y=0} =0$,
\item $\partial_\theta g|_{y=0} = \partial_t g|_{y=0} = 0$ and $\partial_y g|_{y=0} =1$,
\item $\partial_{y}^2 f|_{y=0} \neq 0$.
\end{itemize}

These conditions imply that $T$ is a Morse-Bott torus and that the Reeb vector field  $R$ is given by $\bdry_t$ on $T$.

We recall that the asymptotic operator of a closed Reeb orbit $\gamma$ describes the action of the linearized Reeb flow on sections of the (pull-back of the) contact structure $\gamma^*\xi$ along the orbit. More precisely, the linearized Reeb flow gives a symplectic connection $\nabla^R$ for $\gamma^*\xi$ and the asymptotic operator is $J\nabla^R$, where $J$ is an almost complex structure on $\xi$; see \cite{HWZ2} for  more details on the asymptotic operator
and Section~\ref{subsection: Morse-Bott contact forms} for the linearized Reeb vector
field.)

If we choose a generic almost complex structure $J$ adapted to the symplectization of $\alpha$ such that  $\bdry_t J|_{y=0}=0$, then  there is a unitary trivialization of $\xi$ along $T$  such that the asymptotic operator of an end of a holomorphic map converging to a Reeb orbit on $T$ has the form
\begin{equation}
A= - J_0 \frac{d}{dt} + J_0 \begin{pmatrix} 0 & 0 \\ a & 0 \end{pmatrix}, \end{equation}
where $a >0$ if $T$ is a positive Morse-Bott torus, $a < 0$ if $T$ is a
negative Morse-Bott torus, and $J_0=\begin{pmatrix}  0& -1 \\ 1&0 \end{pmatrix}$.  This unitary trivialization is obtained by projecting $(\partial_y, \partial_\theta)$ to $\xi$ along $\partial_t$.

\begin{lemma}\label{lemma: spectral}
The eigenvalues of the asymptotic operator $A$ are $\lambda_0=0$, $\lambda_a =
- a$ and $\lambda_{n}, \lambda_{-n}$, for $n \in \N$, which are the positive and the
negative solutions of the equation $\lambda(\lambda+ a) = n^2$. The
eigenfunctions that correspond to the eigenvalues $\lambda_0$ and $\lambda_{a}$ are
$$f_0(t) = \left ( \begin{matrix}0 \\ 1 \end{matrix} \right ) \quad \text{and} \quad
f_a (t)= \left ( \begin{matrix}1 \\ 0 \end{matrix} \right ).$$
The eigenvalues $\lambda_{\pm n}$ for $n \ge 1$ are degenerate with multiplicity $2$
and their eigenfunctions have winding number $\pm n$.
\end{lemma}

\begin{proof}
The asymptotic operator is sufficiently simple that we can determine its spectrum by an
explicit computation: the eigenfunctions $\xi$ of $A$ are the $2 \pi$-periodic solutions
of the differential equation
\begin{equation} \label{eqn: system}
\dot \xi = \left ( \begin{matrix} 0 & - \lambda \\ \lambda + a & 0 \end{matrix}
\right ) \xi.
\end{equation}
If $\lambda =0$ or $\lambda = - a$, which are the only cases when the matrix in Equation~\eqref{eqn: system} cannot be diagonalized, the eigenfunctions are
$$f_0(t) = \left ( \begin{matrix}0 \\ 1 \end{matrix} \right ) \quad \text{and} \quad f_a (t)= \left ( \begin{matrix}1 \\ 0 \end{matrix} \right ).$$
If $\lambda (\lambda + a) < 0$, then Equation~\eqref{eqn: system} can be diagonalized over the real numbers, and it is easy to see that it has no periodic solutions. If
$\lambda (\lambda + a) > 0$ a direct computation shows that solutions of Equation~\eqref{eqn: system} are of the form $\xi(t) = \Phi_{\lambda}(t) \xi_0$, where
$$\Phi_{\lambda}(t) = \left ( \begin{matrix} \cos (\sqrt{\lambda(\lambda+a)} t) &
- \frac{\lambda}{\sqrt{\lambda(\lambda+a)}} \sin (\sqrt{\lambda(\lambda+
a)} t) \\ \frac{\lambda+ a}{\sqrt{\lambda(\lambda+a)}} \sin
(\sqrt{\lambda(\lambda+a)} t) & \cos (\sqrt{\lambda(\lambda+a)} t)
\end{matrix} \right ).$$
Then $\Phi_{\lambda}(2 \pi)$ has eigenvalue $1$ if and only if $\lambda (\lambda + a) = n^2 \in \N$, in which case $\Phi_{\lambda}(2 \pi)$ is the identity.
\end{proof}

If $u$ is a $J$-holomorphic map with an end $\mathcal{E}$ which is asymptotic to a Morse-Bott torus $T$, we say that $\mathcal{E}$ is {\em one-sided} if its projection to $M$ does not intersect $T$.

\begin{lemma}[Trapping Lemma] \label{lemma: trapping}
Let $\alpha$ be a contact form, $T$ an $\alpha$-Morse-Bott torus, and $\mathcal{E}$ a one-sided end of a $J$-holomorphic map which is asymptotic to a Reeb orbit in $T$. If $T$ is positive (resp.\ negative), then $\mathcal{E}$ is a positive (resp.\ negative) end.
\end{lemma}

\begin{proof}
Suppose $T$ is positive. By \cite[Theorem 1.3]{HWZ2}, a positive (resp.\ negative) end $\mathcal{E}$ of a $J$-holomorphic curve approaches a Reeb orbit of $T$ along an eigenfunction of the asymptotic operator with negative (resp.\ positive) eigenvalue. By Lemma~\ref{lemma: spectral}, the eigenfunction has a  nonpositive eigenvalue if and only if it has nonpositive winding number.  On the other hand, if $\mathcal{E}$ is one-sided, then the asymptotic eigenfunction must have winding number zero.  Hence $\mathcal{E}$ must be a positive end.  The case for $T$ negative is similar.
\end{proof}

\section{Construction of contact forms} \label{section: contact forms}

In this section we construct some contact forms on $T^2 \times [1,2]$ and $D^2 \times S^1$ which will be used in the proof of the main theorem.

\subsection{Contact forms on $T^2\times[1,2]$} \label{subsec: contact forms on toric annuli}

Let $(\vartheta,t,y)$ be coordinates on
$$T^2\times[1,2]\simeq (\R^2/\Z^2)\times[1,2].$$
Slopes of essential curves on $T^2$ will be measured with respect to $(\vartheta,t)$,  i.e.~with respect to the basis of $H_1(T^2)$ given by the homology classes of the curves $\vartheta \mapsto (\vartheta, *)$ and $t \mapsto (*, t)$. Let
\begin{equation} \label{eqn: contact form on T2 times 1,2}
\alpha_{f,g}= g(y)d\vartheta+ f(y)dt
\end{equation}
be a contact form on $T^2\times[1,2]$, where $f,g$ are functions on $[1,2]$.  We write $f'=\frac{df}{dy}$ and $g'=\frac{dg}{dy}$.

The following is a straightforward calculation:

\begin{lemma} \label{lemma: form alpha f g}
The form $\alpha_{f,g}$ is a contact form if and only if
\begin{equation}
\label{eqn: contact condition for T2 times 1,2} fg'-f'g>0.
\end{equation}
The kernel $\ker\alpha_{f,g}$ is spanned by $\{\bdry_y,-f\bdry_\vartheta+g\bdry_t\}$.  Assuming $\alpha_{f,g}$ is a contact form, the Reeb vector field is given by:
\begin{equation}
\label{eqn: Reeb} R_{\alpha_{f,g}} ={1\over fg'-gf'}(-f'\bdry_\vartheta+g'\bdry_t),
\end{equation}
\end{lemma}

In words, Equation~\eqref{eqn: contact condition for T2 times 1,2} says that the curve in $\R^2$ parametrized by $(f,g)$ is transverse to the radial rays and rotates in the counterclockwise direction.

Later in the article will need the following family of contact forms on $T^2 \times [1,2]$.

\begin{example}\label{alpha delta}
Given a (small) positive irrational parameter $\delta$ we consider pairs of
functions $(f_\delta,g_\delta)$ such that the following hold  (cf.\ Figure~\ref{fig: new f and g part 2}):
\begin{enumerate}
\item $(f_\delta,g_\delta)$ satisfies Equation~\eqref{eqn: contact condition for T2 times 1,2}.
\item $0\leq {f_\delta'(y)\over g_\delta'(y)}\leq \delta$; ${f_\delta'(y)\over g_\delta'(y)}$ is increasing on $(1,{3\over 2})$ and is decreasing on $({3\over 2},2)$, and is equal to $\delta$ at $y={3\over 2}$.
\item $(f_\delta(y),g_\delta(y))=(f_\delta(1)+(y-1)^2,g_\delta(1)+(y-1))$ near $y=1$.
\item $(f_\delta(y),g_\delta(y))=(f_\delta(2)-c_\delta (y-2)^2,g_\delta(2)+ c_\delta (y-2))$ near $y=2$.
\item $(f_\delta(1),f_\delta(1))$ is independent of $\delta$ and all the $(f_\delta(2),f_\delta(2))$ lie on the same line through the origin.
\item The constants $c_\delta$ are chosen so that any two contact forms $\alpha_{\delta}$ and $\alpha_{\delta'}$ are constant multiples of one another near $y=2$.
\end{enumerate}

\begin{figure}[ht]
\begin{overpic}[height=4.5cm]{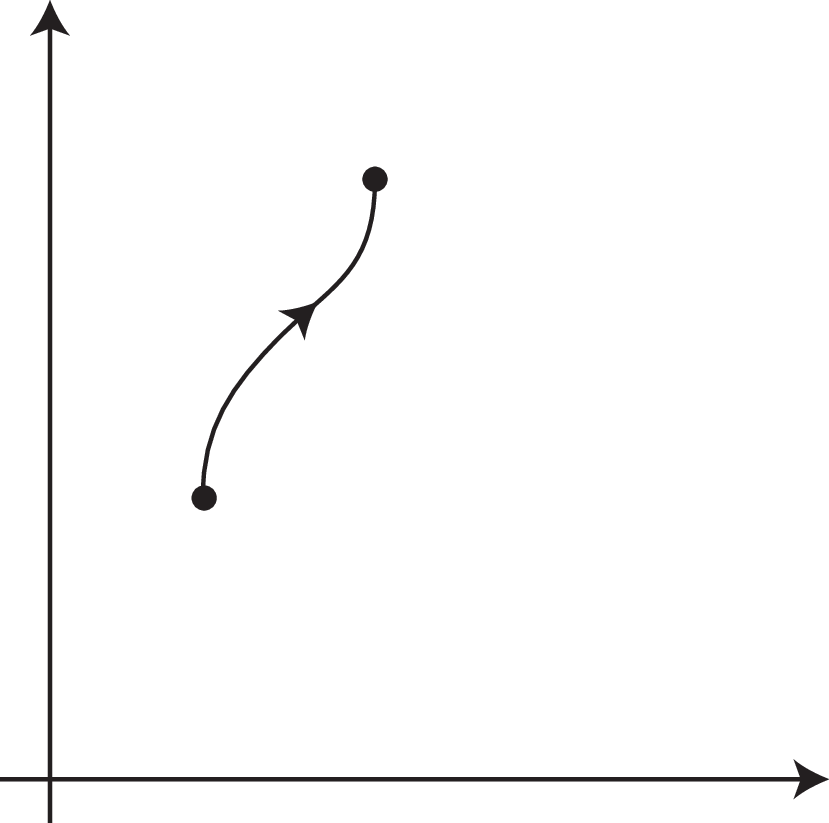}
\put(30,35){\tiny $(f(1),g(1))$} \put(50,-2) {\tiny $f$} \put(0,50)
{\tiny $g$} \put(50,75){\tiny $(f(2),g(2))$}
\end{overpic}
\caption{Trajectory of $(f_\delta(y),g_\delta(y))$.} \label{fig: new f and g part 2}
\end{figure}
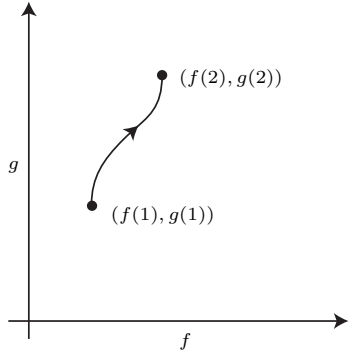

The contact form $\alpha_{f_\delta, g_\delta}$ will also be called $\alpha_\delta$.
Its Reeb vector field $R_{\alpha_\delta}$ has Morse-Bott tori whose
Reeb orbits have rational slope in the interval $[-\infty,-{1\over
\delta}]$; each rational slope occurs twice, once on the interval
$[1,{3\over 2}]$ and once on the interval $[{3\over 2},2]$.  Note
that the Reeb orbits in the two Morse-Bott tori of infinite slope
have parallel directions and are in ``elimination position'', i.e., assuming that $(f_\delta,g_\delta)$ is extended slightly to $T^2\times [1-\varepsilon,2+\varepsilon]$ so that the Reeb orbits have positive slope on $y\in [1-\varepsilon,1)\cup (2,2+\varepsilon]$, one could deform the pair $(f_\delta,g_\delta)$ relative to $\{y=1-\varepsilon,2+\varepsilon\}$ to make the slope of the Reeb vector field always positive; during the deformation we would see the two
Morse-Bott tori of infinite slope coming close to each other and finally canceling.  Also,
by taking $\delta$ to be sufficiently small, all the Reeb orbits in
$int(T^2\times[1,2])$ can be made to have arbitrarily large action.
\end{example}

\subsection{Contact forms on $D^2 \times S^1$} \label{subsec: contact form on V}

Let $(\rho,\phi,\theta)$\footnote{We are making a distinction between symbols $\vartheta$ and $\theta$.} be cylindrical coordinates on the solid torus
$$D^2 \times S^1=\{\rho\leq 1\}\times (\R/2\pi\Z).$$
Let $T_{\overline{\rho}} = \{ \rho = \overline{\rho} \}\subset D^2 \times S^1$ for $\overline{\rho} \in (0,1]$.

\begin{conv} \label{convention on solid torus}
Slopes of essential curves on the torus $T_{\overline\rho}$ are measured with respect to $(\theta,\phi)$ instead of $(\phi,\theta)$.
\end{conv}

We consider contact forms which can be written as:
\begin{equation} \label{eqn: contact form on S1 times D2}
\alpha_{f,g}= g(\rho)d\theta +f(\rho) d\phi.
\end{equation}
Here we need to choose $(f(\rho),g(\rho))$ so that $\alpha_{f,g}$ is
smooth on all of $D^2\times S^1$, which means that $f(0)=0$ and the
derivatives of odd degree of both $f$ and $g$ at $\rho=0$ vanish. We write $f'= \frac{df}{d \rho}$ and $g'= \frac{dg}{d \rho}$.
The analog of Lemma~\ref{lemma: form alpha f g} is the following:

\begin{lemma}\label{lemma: contact forms on V}
The form $\alpha_{f,g}$ is a contact form if and only if
\begin{align} \label{eqn: contact condition on solid torus}
& f'g-fg'>0 \quad \text{for } \rho >0, \text{ and } \\
&  \lim_{\rho \to 0} \dfrac{f'g-fg'}{\rho} >0.
\end{align}
The kernel $\ker\alpha_{f,g}$ is spanned by $\{\bdry_\rho,-f\bdry_\theta+ g\bdry_\phi\}$.  Assuming $\alpha_{f,g}$ is a contact form, the Reeb vector field is given by:
\begin{equation}
R_{\alpha_{f,g}}= {1\over f'g-fg'} (f' \bdry_\theta-g' \bdry_\phi).
\end{equation}
In particular, $R_{\alpha_{f,g}}$ is parallel to $\bdry_\theta$ at
$\rho=0$.
\end{lemma}

Each torus $T_{\overline{\rho}}$ is linearly foliated by the Reeb flow of $\alpha_{f,g}$.

Since they will be useful later, we present a pair of constructions of contact forms on $D^2\times S^1$ of the form given in Equation~\eqref{eqn: contact form on S1 times D2}.

\begin{example}\label{esempio scemo}
Given $\nu>0$ and $C>1$, let $(f(\rho),g(\rho))=(\nu\rho^2,C-\rho^2)$.  This gives a smooth
contact form on $D^2\times S^1$ and the Reeb vector field on
$T_{\rho}$ has slope $-{g'\over f'}= {1\over \nu}$ for all $\rho>0$.  In particular, if $\nu$ is irrational, then the only simple closed orbit of
$R_{\alpha_{f,g}}$ is the core curve $\{ \rho=0 \}$.
\end{example}

\begin{example} \label{esempio utile}
The following contact forms, which generalize those in Example
\ref{esempio scemo}, will be used later in the paper.
We define $\alpha$ on $D^2 \times S^1$ so that
the following hold:
\begin{enumerate}
\item $(f,g)$ satisfies Equation~\eqref{eqn: contact condition on solid torus}.
\item $(f(\rho),g(\rho)) =(\rho^2,C-\rho^2)$ near $\rho=0$, where $C>0$ is a large constant.
\item
$(f(\rho),g(\rho))=(f(1)-(\rho-1)^2,g(1)-(\rho-1))$ near $\rho=1$\footnote{Here $(f,g)|_{\rho=1}=(f_{\delta},g_{\delta})|_{y=2}$. This allows us to extend $\alpha_\delta$ to $D^2 \times S^1$ for all sufficiently small $\delta>0$ by writing $(f_\delta,g_\delta)$ as a suitable constant multiple of $(f_{\delta_0},g_{\delta_0})$. This is possible because
of Condition~(6) in the definition of $\alpha_\delta$. Observe that the Reeb orbits of $\alpha_\delta$ and $\alpha_{\delta_0}$ agree on $V$, modulo parametrization.}.
\item $-\frac{g'}{f'}$ monotonically increases from $1$ to $+\infty$ as $\rho$ goes from $0$ to $1$.
\end{enumerate}
The profile of the functions $(f, g)$ is shown in  Figure~\ref{fig: new f and g
part 3}.

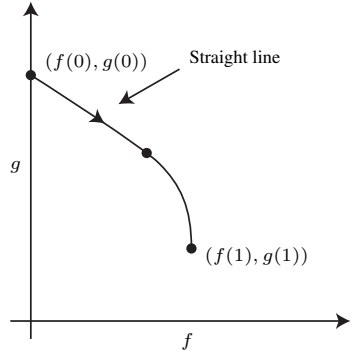
\begin{figure}[ht]
\begin{overpic}[height=4.5cm]{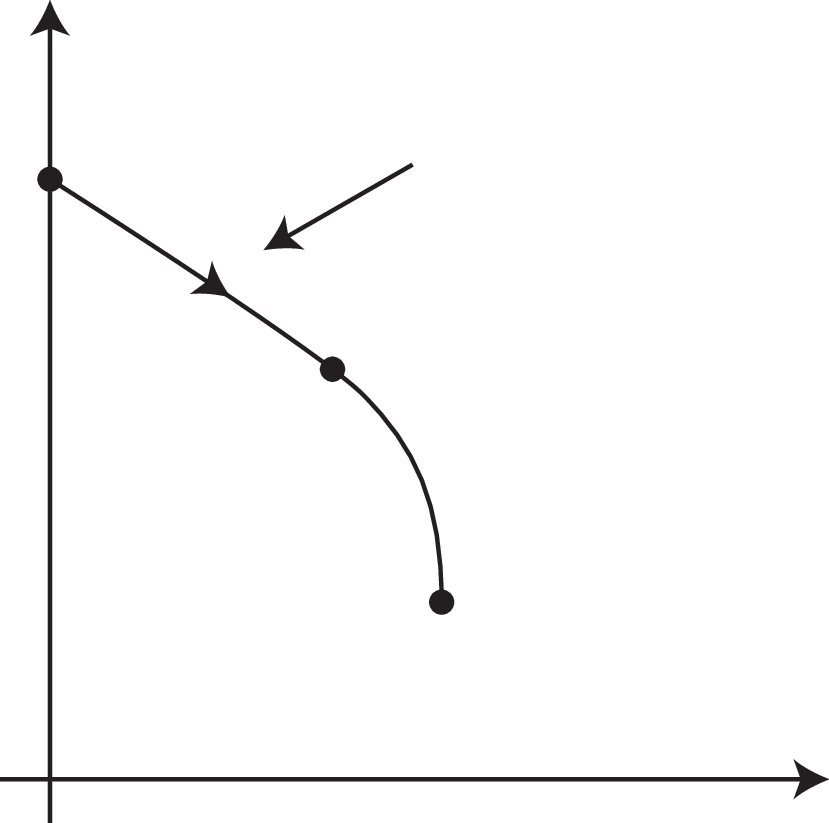}
\put(52.5,81){\tiny Straight line } \put(50,-2) {\tiny $f$}
\put(0,50) {\tiny $g$} \put(57,23){\tiny $(f(1),g(1))$}
\put(10,80){\tiny $(f(0),g(0))$}
\end{overpic}
\caption{Trajectory of $(f(\rho),g(\rho))$. Here the
arrow points in the positive $\rho$-direction.} \label{fig: new f
and g part 3}
\end{figure}

On each torus $T_\rho \subset D^2\times S^1$, the Reeb vector field
$R_{\alpha}$ gives a foliation by Reeb orbits of slope $r$ in
the interval $[1,\infty]$, where there is a unique $\rho$ for each
slope $r \in (1,\infty]$.
\end{example}

\section{ECH for manifolds with torus boundary}
\label{section: ECH for manifolds with torus boundary}
In this section we define several ECH groups on a compact manifold
$M$ with torus boundary $T=\bdry M$.  We fix
 an oriented identification
$T\simeq \R^2/\Z^2$ so that we can refer to slopes of essential
curves on $T$.  Let $\alpha$ be a contact form on $M$ such that $T$
is foliated by Reeb orbits of slope $r$. If $r$ is rational, we
assume that $T$ is Morse-Bott. All ECH groups on $M$ and $int(M)$
are computed using a $C^\infty$-small perturbation of $\alpha$
 so that all Reeb orbits in $int(M)$ are
nondegenerate. Let $J$ be a Morse-Bott regular almost complex
structure on $\R\times M$ adapted to $\alpha$.

\subsection{Definitions} \label{subsection: defns of ECH with t2 bdry}

We introduce several ECH groups:

\s\n 1. $ECH(int(M),\alpha)$. The ECH chain group $ECC(int(M),\alpha)$ is generated
by orbit sets whose simple orbits lie in the interior of $M$. In particular, we are discarding
the Morse-Bott family of orbits on $T$ if $r$ is rational. The differential $\bdry$ is the
usual one, i.e., counts holomorphic curves of ECH index $I(\gamma,
\gamma',Z)=1$ in $\R\times int(M)$ whose connector components are trivial cylinders.
Since $int(M)$ is not closed, we need to verify that $ECC(int(M),\alpha)$ is indeed a
chain complex.

\begin{lemma}
\label{lemma: bdry squared is zero for mfld with torus bdry}
 $\bdry$ is defined and $\bdry^2=0$.
\end{lemma}

\begin{proof}
We claim that the SFT compactness theorem holds in $\R \times int(M)$.  This implies that the arguments used in \cite{HT1, HT2} to prove $\partial^2=0$ will then carry over to our setting. Let $u_n$ be a sequence of $J$-holomorphic maps with image in $\R \times int(M)$. After passing to a subsequence, $u_n$ converges to a building $u_{\infty}$ such that all its components have image in $\R \times M$.  By the Blocking Lemma, no component of $u_{\infty}$ can intersect $\partial M$.

We claim that no component of $u_{\infty}$ can have an end at a Reeb orbit in $\partial M$: indeed, if there is a component with a positive (resp.\ negative) end at a Reeb orbit in $\bdry M$, then there is another component of $u_{\infty}$ with a negative (resp.\ positive) end at a Reeb orbit in $\partial M$. By the Trapping Lemma, this is impossible if the image of $u_{\infty}$ is contained in $\R \times M$.
\end{proof}

\n 2a. $ECH(M,\alpha)$ for $r$ irrational.  This is defined to be $ECH(int(M),\alpha)$.

\s\n 2b. $ECH(M,\alpha)$ for $r$ rational. Let $\mathcal{N}$ be the set of simple Reeb
orbits on $T$. The set $\mathcal{N}$ comes with distinguished orbits $e,h$ which
become elliptic and hyperbolic after a suitable perturbation. Writing $\mathcal{P}$ for
the set of simple orbits in $int(M)$, $ECC(M,\alpha)$ is the chain complex which is
generated by orbit sets constructed from $\mathcal{P}\cup\{h,e\}$ and whose
differential counts Morse-Bott buildings of ECH index $1$ in $\R\times M$.

\begin{lemma}\label{fastidio}
If $\alpha$ is nondegenerate on $int(M)$, then it is nice.
\end{lemma}

\begin{proof}
Suppose that $\partial M$ is a negative Morse-Bott torus; the positive case is analogous. Let ${\mathcal N}$ be the Morse-Bott family corresponding to $\partial M$. If $\alpha$ is not nice, then there is a Morse-Bott building $\tilde{u}$ in $\R \times M$  with ECH index $I(\tilde{u})=1$ whose holomorphic part $u$ has more than one  non-connector irreducible component. Assume that there are exactly two  non-connector components $u_1$ and $u_2$  (this is mostly to simplify notation; the general case is treated in the same way). By the Trapping Lemma, the only ends of $u_1$ and $u_2$ that limit to $\bdry M$ are negative ends.  We form the Morse-Bott buildings $\tilde{u}_1$ and $\tilde{u}_2$ by augmenting the ends of $u_1$ and $u_2$ at $\partial M$ with gradient flow lines and denote the union of these two buildings  by $\tilde{u}'$.

We claim that $I(\tilde{u})= I(\tilde{u}')$. Indeed, all the ends of $u_1$ and $u_2$ that limit to orbits on $\partial M$ are connected to critical points in ${\mathcal N}$ by gradient flow lines, with possible interruptions by connectors. Hence $\tilde{u}$ and $\tilde{u}'$ have the same ends in the ECH sense and define the same relative homology class. This implies that $I(\tilde u)=I(\tilde u')$.

On the other hand,
 let $u_i$, $i=1,2$, be a $k_i$-th cover of a $J$-holomorphic curve $v_i$, and define very nice, simply-covered buildings $\tilde{v}_i$.  By Theorem~\ref{thm: Morse-Bott perturbation of moduli spaces}(2), we can perturb $\tilde{v}_1$ and $\tilde{v}_2$ to $J_\epsilon$-holomorphic maps $v_{1, \varepsilon}$ and $v_{2, \varepsilon}$, respectively. We denote by $v_{i, \varepsilon}^{k_i}$ the $J_\epsilon$-holomorphic map made of $k_i$ parallel copies of $v_{i, \varepsilon}$.
Then, by~\cite[Theorem~5.1]{Hu2},
$$I(\tilde{u}) \ge I(v_{1, \varepsilon}^{k_1}) + I(v_{2, \varepsilon}^{k_2}) \ge
k_1 I(v_{1, \varepsilon}) + k_2 I(v_{2, \varepsilon}).$$
Since $I(v_{i, \varepsilon}) >0$ for $i=1,2$, this is a contradiction.
\end{proof}

 Lemma~\ref{fastidio} implies that $\partial^2=0$, since it guarantees that the Morse-Bott gluing is done at a different end from the gluing of connectors (i.e., the obstruction bundle gluing of Hutchings-Taubes~\cite{HT1,HT2}) and the two kinds of gluing can be done independently.

\s\n 3. $ECH^{\flat}(M, \alpha)$. The chain complex
$ECC^{\flat}(M,\alpha)$ is generated by orbit sets which are
constructed from $\mathcal{P}\cup\{e\}$. As in the case of
$ECC(M,\alpha)$, if $\mathcal{N}$ is a negative Morse-Bott family,
no Morse-Bott building $\tilde u$ in $\R\times M$ besides trivial
cylinders can have $e$ at the positive end. Hence the differential
can be defined by counting Morse-Bott buildings of ECH index $1$ in
$\R\times M$, whose orbit sets are constructed from $\mathcal{P}\cup
\{e\}$.

The verification of $\bdry^2=0$ needs one extra consideration:
An index $2$ family of $J$-holomorphic curves in $\R\times M$
can break into a  Morse-Bott building $\tilde u$ which involves $h$ at the
negative end, followed by a holomorphic cylinder from $h$ to $e$. (Note that,
by the Trapping Lemma, these holomorphic cylinders are the only nontrivial holomorphic curves which go from $h$ to $e$ and so there are no other cases to consider.)

This type of breaking could be a problem because orbit sets containing $h$
are not in the chain complex $ECC^{\flat}(M,\alpha)$. However, since there are
{\em two} gradient trajectories from $h$ to $e$  with ECH index $I=1$ and no other holomorphic curve  (or building) with a positive end at $h$,
the Morse-Bott building $\tilde u$ can be glued onto each of the two gradient
trajectories.  This proves that families breaking at $h$ always come in
pairs, and therefore $\bdry^2=0$ holds even when we discard orbit sets which contain
$h$.

If $\mathcal{N}$ is a positive Morse-Bott family, then $e$ can only
be at the positive end of a $J$-holomorphic curve in $\R \times M$,
and the proof of $\partial^2=0$ remains the same with the obvious
modifications.

\s\n 4. $ECH^{\sharp}(M, \alpha)$. The chain complex
$ECC^{\sharp}(M,\alpha)$ is generated by orbit sets which are
constructed from $\mathcal{P}\cup\{h\}$, and its differential counts
ECH index 1 Morse-Bott buildings which are asymptotic to orbit sets
constructed from $\mathcal{P}\cup\{h\}$.

 \begin{rmk} The differentials of the ECH groups defined in this section
preserve the total homology class of the generators. Then we can define subgroups $ECH(M, \alpha, A)$ for every $A \in H_1(M)$. Similar notations will be used for the variants of this group.
\end{rmk}

\subsection{Well-definition}\label{subsec: well definition of ECH with boundary}

In this subsection we prove that $ECH(M,\alpha)$ is independent of
the choice of $\alpha$, provided the slope $r$ is irrational. The
verification in the other cases will be omitted; we will be careful to use the
invariance of ECH groups for manifolds with torus boundary only in the case where
it is proved. The main result proved in this subsection is the following:

\begin{prop}\label{prop: ECH of two contact forms in M}
Let $\alpha_1$ and $\alpha_2$ be contact forms on $M$  which agree on $\partial M$ to first order (and in particular the Reeb vector fields  and the characteristic foliations of $\alpha_1$ and $\alpha_2$ at $\partial M$ are equal)  and define contact structures $\xi_i = \ker \alpha_i$ which are isotopic relative to the boundary. If $\partial M$ is foliated by Reeb orbits of irrational slope, then there is an isomorphism $$ECH(M, \alpha_1) \simeq ECH(M, \alpha_2).$$
\end{prop}

The strategy of the proof is to extend $(M, \alpha_i)$, $i=1,2$, to closed contact manifolds and to use the invariance of ECH for closed manifolds.
Lemma~\ref{claim: extension of contact form wo orbits} constructs the contact forms which are used to extend $(M, \alpha_i)$.
 Then Lemma~\ref{lemma: mahler} shows that, up to some action $L$,
the ECH groups of $(M, \alpha_i)$ are isomorphic to the ECH groups of their extension.
Finally Lemmas~\ref{lemma: identity form a cobordism}, \ref{pronto soccorso} and \ref{lemma: ECH of two contact forms in M - commutativity} establish some compatibility properties for
the continuation maps between the extended forms, so that the proposition can finally be proved by a direct limit argument.

 \begin{lemma}\label{crude bound}
Let $\alpha= g(\rho) d \theta + f(\rho) d \phi$ be a contact form on $D^2 \times S^1$
with cylindrical coordinates $(\rho, \phi, \theta)$. Denote $v(\rho) = (f(\rho), g(\rho))$
and let $|v(\rho)|$ be the norm of $v(\rho)$ and $\zeta(\rho)$ the angle between $v(\rho)$ and $v'(\rho)$, both measured with respect to the standard Euclidean structure on $\R^2$. Then, if the torus $T_\rho$ is foliated by closed Reeb orbits, for every Reeb orbit $\gamma$ on $T_\rho$ we have
\begin{equation}\label{eqn: crude bound}
{\mathcal A}(\gamma) \ge |v(\rho)| |\sin \zeta(\rho)|.
\end{equation}
\end{lemma}
\begin{proof}
Let $J$ be the standard complex structure, $\cdot$ the standard inner product, and $|\cdot|$ the standard Euclidean norm on $\R^2$. For every $\rho \in (0,1]$ we trivialize the tangent bundle of the torus $T_\rho$ by $(\partial_{\phi}, \partial_{\theta})$ and measure the slope of curves on $T_\rho$ with respect to $(\phi,\theta)$.\footnote{In the proof we are using a different convention from that of Convention~\ref{convention on solid torus}.}

By Lemma~\ref{lemma: contact forms on V}, $R$ is tangent to $T_\rho$ for all $\rho \in (0,1]$ and can be written as:
$$R= {(-g',f')\over (-g',f')\cdot (f,g)},$$
with respect to $(\bdry_\phi,\bdry_\theta)$. If we write $v=(f,g)$, then $Jv'=(-g',f')$ and
$$|R|= \left| { (-g',f')\over (-g',f')\cdot (f,g)}\right| = \left| { Jv'\over Jv'\cdot v}\right|={1\over |v| |\sin\zeta|},$$
where $\zeta(\rho)$ is the angle between $v(\rho)$ and $v'(\rho)$.  Note that $\mbox{slope}(R)=\mbox{slope}(Jv')=-{f'\over g'}$.

Let $\rho \in (0,1)$  be such that $R$ has rational slope on $T_\rho$ and let $w$ be the shortest integer vector with that slope. Then $T_\rho$ is foliated by Reeb orbits and each Reeb orbit $\gamma$ has action ${\mathcal A}(\gamma) = \frac{|w|}{|R|}$. Since $|w| \ge 1$, we have the bound
\begin{equation*}
{\mathcal A}(\gamma) \ge \tfrac{1}{|R|} = |v| |\sin \zeta|.
\end{equation*}
\vskip-.25in
\end{proof}

\begin{lemma} \label{claim: extension of contact form wo orbits}
Given $L>0$  and $r>0$ irrational, there is a contact form $\alpha( r, L)=g(\rho)d\theta+f(\rho)d \phi$ on  $V= D^2 \times S^1$
with cylindrical coordinates $(\rho, \phi, \theta)$ such that:
\begin{itemize}
 \item[(a)] on $\partial V$ the Reeb vector field  $R$ of $\alpha(r, L)$ has slope $- \frac 1r$ and the characteristic foliation has infinite slope; and
\item[(b)] all the closed orbits of  $R$ have $\alpha( r, L)$-action larger than $L$.
\end{itemize}
\end{lemma}

\begin{figure}[ht]
\begin{overpic}[height=5cm]{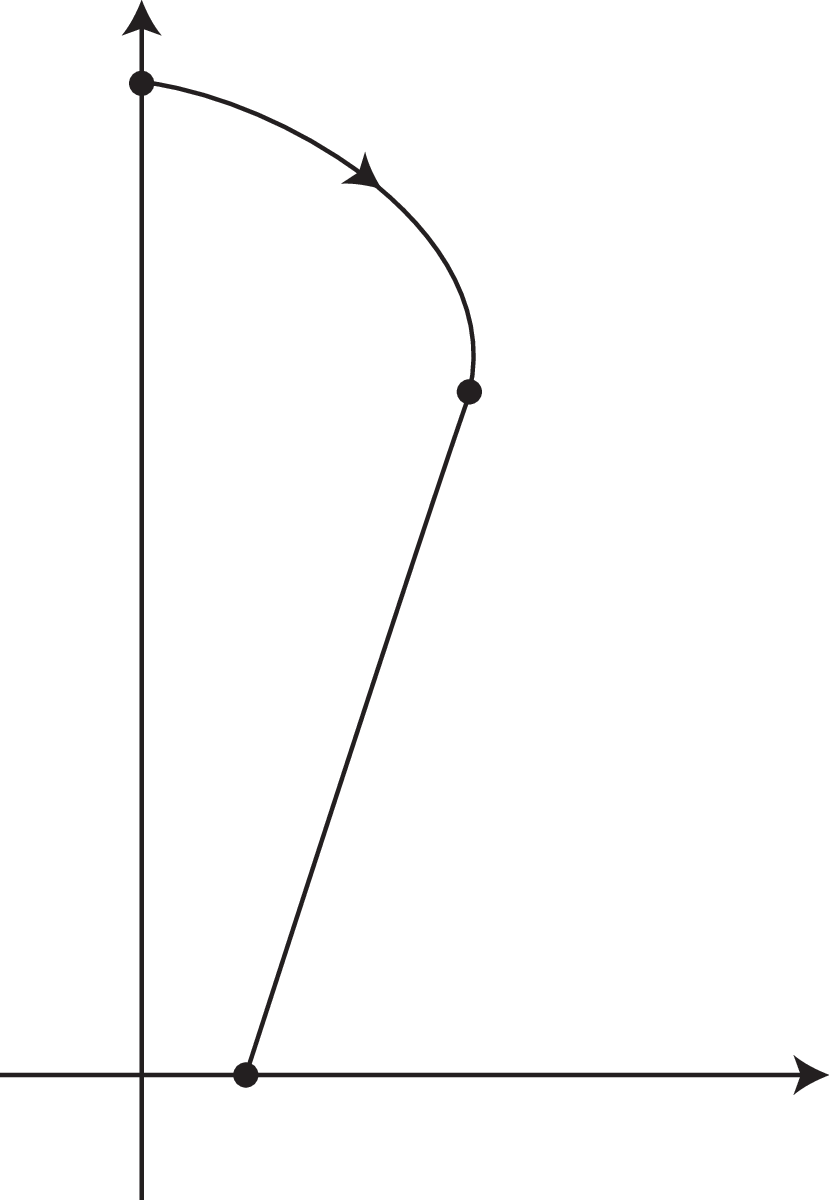}
\put(36,40){\tiny long line segment}
\end{overpic}
\caption{Trajectory of $(f(\rho),g(\rho))$. The arrow is in the direction of increasing $\rho$.} \label{fig: extend to S1 times S2}
\end{figure}

\begin{proof}
 We describe $\alpha(r, L)$ by describing the vector $v(\rho) = (f(\rho), g(\rho))$.
We construct $v(\rho)=(f(\rho), g(\rho))$ ``backwards'', starting with larger $\rho$, subject to the condition ${d|v|\over d\rho}<0$. The profile of $v(\rho)$ is given in Figure \ref{fig: extend to S1 times S2}.

\s\n (1) For $\rho \in [{3\over 4}, 1]$, define $v(\rho)$ so that it parametrizes a ``long''\footnote{The segment is chosen so that Equation~\eqref{eqn: compare sines} from (2) is satisfied.} segment and $R$ is constant, has slope $-{1\over r}$, and satisfies $|R|={1\over K}$.  Since $r$ is irrational, there are no Reeb orbits on $T_{\rho}$ for $\rho \in [{3\over 4}, 1]$.

\s\n (2) Fix an irrational slope $-{1\over r'} > -{1\over r}$ so that all integer vectors with slope between $-{1\over r'}$ and $-{1\over r}$ have norm greater than $\frac{2L}{K}$.  For $\rho \in [{1\over 2},{3\over 4}]$, define $v(\rho)$ so that $|R(\rho)|< \frac 2 K$ and $\mbox{slope}(R)=\mbox{slope}(Jv')$ decreases monotonically from $-{1\over r'}$ to $-{1\over r}$ as $\rho$ increases. One can achieve this by making $v(\rho)$ vary sufficiently slowly for $\rho \in [{1\over 2}, {3\over 4}]$. Hence, if $\gamma$ is a Reeb orbit of $T_\rho$ with $\rho \in [{1\over 2},{3\over 4}]$, then
$$\mathcal{A}(\gamma)\geq \tfrac K 2 \tfrac{2L}{K} = L.$$

Let $\overline{\zeta}$ be the clockwise angle from a line of slope $-{1\over r'}$ to a line of slope $-{1\over r}$. By taking the ``long'' segment to be sufficiently long, we may assume that
\begin{equation} \label{eqn: compare sines}
|\sin\overline{\zeta}| > KL|\sin(\zeta(\tfrac{3}{4}))|.
\end{equation}

\s\n (3) For $\rho \in [{1\over 4}, {1\over 2}]$, define $v(\rho)$ so that $\mbox{slope}(Jv')$ decreases monotonically from ${1\over r''}>0$ to $-{1\over r'}$ as $\rho$ increases and $|\sin \zeta(\rho)| \ge |\sin \overline{\zeta}|$. We can achieve these properties by changing $v(\rho)$ slowly with respect to the slope of $v'(\rho)$. Then, by Equations~\eqref{eqn: crude bound} and \eqref{eqn: compare sines},
\begin{equation} \label{eqn: bound for A}
\mathcal{A}(\gamma)\geq |v(\rho)|\cdot KL |\sin (\zeta(\tfrac{3}{4}))|\geq KL |v(\tfrac{3}{4})|\cdot |\sin (\zeta(\tfrac{3}{4}))|\geq KL \tfrac{1}{K}=L,
\end{equation}
where $\gamma$ is a Reeb orbit of $T_\rho$, $\rho\in [{1\over 4}, {1\over 2}]$.

\s\n (4) Finally, define $v(\rho)$ for $\rho \in [0, {1\over 4}]$ which parametrizes a segment of slope ${1\over r''}$ and satisfies $f(0)=0$.  $\mathcal{A}(\gamma)\geq L$ follows from Equation~\eqref{eqn: bound for A}.
\end{proof}

\begin{rmk}\label{rmk: cobordisms between extensions}
We will always assume that, when $L_0 < L_1$, each radial ray in the $fg$-plane intersects the curve $(f_0(\rho),g_0(\rho))$ defining $\alpha( r, L_0)$ before or at the same time as the curve $(f_1(\rho), g_1(\rho))$ defining $\alpha( r, L_1)$. Then there exist a diffeomorphism $\sigma :  D^2 \times S^1 \to D^2 \times S^1$ such that $\sigma(\rho,\phi,\theta)=(\sigma_0(\rho),\phi,\theta)$ and a function $h : [0,1] \to \R^{\geq 0}$ such that
$\alpha( r, L_1) = e^{h(\rho)}  \sigma^*( r, \alpha(L_0)).$
\end{rmk}

Let $(M, \alpha_i)$, $i=1,2$, be contact manifolds as in Proposition \ref{prop: ECH of two contact forms in M}.  We can choose coordinates
$(\vartheta,t,y) \in (\R^2/\Z^2)\times [-\varepsilon,0]$ on a small collar of $\partial M$  such that $\partial M$ corresponds to $T^2 \times \{ 0 \}$ and the contact forms $\alpha_i$ can be written as
$$\alpha_i = g_i(\vartheta, t, y) d \vartheta + f_i(\vartheta, t, y) dt$$
with $\frac{\partial f_i}{\partial \vartheta} = \frac{\partial f_i}{\partial t} = \frac{\partial g_i}{\partial \vartheta} = \frac{\partial g_i}{\partial t} =0$ at $t=0$ (i.e., along $\partial M$). Note that we have used the assumption that $\alpha_1$ and $\alpha_2$ coincide to first order along $\partial M$ to conclude that they can be put in this form with the same choice of coordinates.
Moreover, we assume that these coordinates have been chosen so that, on $\partial M$, the Reeb vector fields of $\alpha_1$ and $\alpha_2$ have {\em negative irrational} slope $-r$ and that the slopes of the characteristic foliations of $\xi_i=\ker\alpha_i$ are   nonnegative  and sufficiently close to zero.\footnote{Close enough that Claim~\ref{claim: short orbits in M} applies.}  Here the slope is measured with respect to $(\vartheta,t)$.

For $L'>0$ sufficiently large we embed $(M, \alpha_i)$ into a
closed contact manifold $(M', \alpha_i'(L'))$ such that:
\begin{enumerate}
\item $M' = M \cup V$, where $\bdry M$ and $\bdry V$ are glued
by the identifications $\rho = 1-y$, $\phi = 2 \pi t$, $\theta = 2 \pi \vartheta$; and
\item $\alpha_i'(L') |_{M}= \alpha_i$ and  $\alpha_i'(L') |_{V}$ is a $C^1$-small perturbation of  $\alpha(r, L')$ near the boundary.
\end{enumerate}

If the perturbation of the form $\alpha(r, L')$ is small enough in the $C^1$ topology, it does not create any closed Reeb orbit of action less than $L'$.
Since the size of the perturbation which is necessary to glue $\alpha_i$ with $\alpha(r, L')$
essentially depends on the slope of the characteristic foliation of $\alpha_i$ on $\partial M$, we can claim the following.
\begin{claim}\label{claim: short orbits in M}
All closed Reeb orbits of $(M', \alpha_i(L'))$ of action less than $L'$ are contained in $M$.
\end{claim}

The next lemma identifies some ECH groups for $(M, \alpha_i)$ with ECH groups for $(M', \alpha_i'(L'))$.

\begin{lemma} \label{lemma: mahler}
For all $L \le L'$, if we choose the almost complex structure on the symplectization of $(M',\alpha_i'(L'))$ to extend the almost complex structure picked on the symplectization of $(M, \alpha_i)$, then there are isomorphisms
$$ECC^{L}(M, \alpha_i)\simeq ECC^{L}(M', \alpha_i'(L'))$$
of chain complexes.
\end{lemma}

\begin{proof}
By Lemma \ref{claim: extension of contact form wo orbits}, there is an isomorphism
$$ECC^{L}(M,\alpha_i)\simeq ECC^{L}(M', \alpha_i'(L'))$$
{\em as vector spaces}. To prove that the isomorphism holds as chain complexes, it suffices to show that every holomorphic curve in $\R\times M'$ which is positively asymptotic to an orbit set of $R_{\alpha_i'(L')}$ of $\alpha_i'(L')$-action less than $L$ (which is equal to an orbit set of $R_{\alpha_i}$ of $\alpha_i$-action less than $L$) has image in $\R \times M$. Let $u$ be a holomorphic map in $\R\times M'$ connecting the orbit set $\gamma$ of $R_{\alpha_i}$ in $M$ with $\mathcal{A}_{\alpha_i}(\gamma) < L$ to the orbit set $\gamma'$ of $R_{\alpha_i'(L')}$ in $M'$. Since $\mathcal{A}_{\alpha_i}(\gamma) < L$, $\gamma'$ must be contained in $M$. Hence the homology class of $u_{M'}\cap \bdry V$ in $H_1(\bdry V)$ is a multiple of the class of the meridian of $V$. On the other hand, inside $V$ there is a concentric torus $V'$ on which the Reeb orbits are meridians. (This torus corresponds to the vertical tangency of the curve in Figure \ref{fig: extend to S1 times S2}.) Then Blocking Lemma (2) implies that $u$ must be asymptotic to some orbits in $V'$. This is  not possible since all the ends of $u$ limit to orbits of action less than $L$.  Hence the image of $u$ is contained in $\R \times M$ by Blocking Lemma (1).
\end{proof}

The induced identification
$$ECH^{L}(M, \alpha_i) \simeq ECH^{L}(M', \alpha_i'(L'))$$
is independent of $L'$ in the following sense:  Let $L \le L_0 \le L_1$ be positive numbers such that no Reeb orbit in $(M, \alpha_i)$ (for either $i=1$ or $i=2$) has action $L$. By Remark \ref{rmk: cobordisms between extensions} and Lemma~\ref{tired with all these} there are maps
$$\Psi_i^{L,L_0,L_1} : ECH^L(M', \alpha_i'(L_1)) \to ECH^L(M', \alpha_i'(L_0))$$
induced by interpolating cobordisms $(W, \mu_i)$ from $(M',\alpha_i'(L_1))$ at the positive end to $(M', \alpha_i'(L_0))$ at the negative end.
Then we have the following:

\begin{lemma}\label{lemma: identity form a cobordism}
The maps $\Psi^{L,L_0,L_1}_i$ restrict to the identity on $ECH^{L}(M, \alpha_i)$.
\end{lemma}

\begin{proof}
The cobordism $W$ is topologically trivial, i.e., $W \simeq [0,1] \times M'$, and we can assume that $(W, \mu_i)$ restricts to a piece of symplectization on $[0,1] \times M$.  We choose the almost complex structure $J$ to be $\R$-invariant on $\R\times M$. As before, all orbit sets of $\alpha_i'(L_j)$-action less than $L$ for $j=0,1$ are contained in $M$. Then the Blocking Lemma\footnote{This situation is slightly more general than that for which the Blocking Lemma has been stated and proved  because we are in a cobordism}. However the lemma is still valid and the proof is unchanged. and the argument of Lemma~\ref{lemma: mahler} imply that all $J$-holomorphic maps between orbit sets of action less than $L$ are contained in $\R \times M$. If those $J$-holomorphic maps have  ECH index zero, then they are  branched covers of trivial cylinders because $([0,1] \times M, \mu_i|_{[0,1] \times M})$ is a piece of symplectization. Hence the map induced on $ECH^{L}(M, \alpha)$ is the identity by Theorem \ref{thm: Hutchings Taubes cobordism map}(i).
\end{proof}

We will use the identifications $ECH^{L}(M,\alpha_i)\simeq ECH^{L}(M', \alpha_i'(L'))$ to define a map
$$\Phi: ECH(M,\alpha_1) \to ECH(M,\alpha_2).$$
This involves two steps: the construction of maps
$$\Phi_L : ECH^L(M, \alpha_1) \to  ECH^{\kappa L}(M, \alpha_2)$$
for some $\kappa>1$ and the taking of direct limits.

Let $f: M \to \R$ be a smooth positive function such that $\phi^*(\alpha_2)=f\alpha_1$ for some diffeomorphism $\phi$ of $M$ which is isotopic to the identity and restricts to the identity on $\bdry M$. Then choose $\kappa >1$ such that $\frac{1}{\kappa} \le f \le \kappa$. Given $L'>L$, we consider the contact forms $\alpha_i'(\kappa L')$, $i= 1,2$, on $M'$ constructed in Lemma~\ref{claim: extension of contact form wo orbits}. Then there is an interpolating
cobordism $(X, \lambda_{L'})$ from $(M', \alpha_1'(\kappa L'))$ at the positive
end to $(M', \kappa^{-1}\alpha_2'(\kappa L'))$ at the negative end. Moreover we can
assume that $(X,\lambda_{L'})$ restricts to a piece of symplectization on a small neighborhood of $[0,1] \times V$.

We define $\Phi_L$ by imposing the commutativity of the following diagram:
\begin{equation}
\begin{diagram} \label{diagram: definition of Phi_L}
ECH^{L} (M',\alpha_1'(\kappa L')) & \rTo^{\Phi^L} &
ECH^{\kappa L} (M', \alpha_2'(\kappa L'))  \\
\dTo^{\simeq} & &  \dTo_{\simeq} \\
ECH^{L}(M,\alpha_1) &
\rTo^{\Phi_{L}}  & ECH^{\kappa L} (M,\alpha_2),
\end{diagram}
\end{equation}
where the vertical maps are the isomorphisms coming from Lemma~\ref{lemma: mahler} and the top map is induced by the interpolating cobordisms $(X, \lambda_{L'})$ via Lemma~\ref{tired with all these}.

\begin{rmk}\label{cobordisms dont leak}
Using the Blocking Lemma one can prove that the map $\Phi_L$ is supported, in the
sense on Theorem~\ref{thm: Hutchings Taubes cobordism map}(i), by holomorphic
curves in $\R \times M$. See the proof of Lemma~\ref{lemma: mahler} for the details.
\end{rmk}

\begin{lemma} \label{pronto soccorso}
$\Phi_L$ is independent of the choice of $L'$ in Diagram~\eqref{diagram: definition of Phi_L}.
\end{lemma}

\begin{proof}
Suppose $L \le L_0 \le L_1$, $\alpha_1$ has no orbit sets of action $L$, and $\alpha_2$ has no orbit sets of action $\kappa L$. Then the diagram
\begin{equation}
\begin{diagram} \label{diagram: composition of ECH maps 2}
ECH^{L}(M',\alpha_1'(\kappa L_1)) & \rTo^{\Phi^{L}}  &
ECH^{\kappa L} (M',\alpha_2'(\kappa L_1)) \\
\dTo^{\Psi_1^{L,\kappa L_0,\kappa L_1}} & &  \dTo_{\Psi_2^{L,\kappa L_0,\kappa L_1}} \\
ECH^{L} (M',\alpha_1'(\kappa L_0)) & \rTo^{\Phi^L} &
ECH^{\kappa L} (M',\alpha_2'(\kappa L_0))
\end{diagram}
\end{equation}
commutes by Theorem \ref{thm: Hutchings Taubes cobordism map} since the compositions of cobordisms $(X, \lambda'_{L_0}) \circ (W, \mu_1)$ and $(W, \mu_2) \circ
(X, \lambda'_{L_1})$ are homotopic by Lemma \ref{homotopy of interpolating cobordisms}. The maps $\Psi_i^{L,\kappa L_0,\kappa L_1}$ induce the identity on $ECH(M, \alpha_i)$ by Lemma
\ref{lemma: identity form a cobordism}, so the maps on the top and bottom of Diagram~\eqref{diagram: composition of ECH maps 2} define the same map
$\Phi_L : ECH^L(M, \alpha_1) \to ECH^{\kappa L}(M, \alpha_2)$.
\end{proof}

\begin{lemma} \label{lemma: ECH of two contact forms in M - commutativity}
Let $\alpha_1$ and $\alpha_2$ be contact forms as in Proposition \ref{prop: ECH of two contact forms in  M}.  If $L_i$ is an increasing sequence of positive real numbers such that $\alpha_1$ has no orbit set of action $L_i$ and $\alpha_2$ has no orbit set of action $\kappa L_i$ for all $i$,\footnote{ This condition can be fulfilled due to the fact that the action spectrum is discrete for a generic contact form.} then the maps
$$\Phi_{L_i} : ECH^{L_i}(M, \alpha_1) \to ECH^{\kappa L_i}(M, \alpha_2)$$
define a morphism of directed systems.
\end{lemma}

\begin{proof}
For all  $L<L'$ as above, the diagram
\begin{equation}
\begin{diagram} \label{diagram: composition of ECH maps 1}
ECH^{L} (M,\alpha_1) & \rTo^{\Phi_L} &
ECH^{\kappa L} (M,\alpha_2)  \\
\dTo & &  \dTo \\
ECH^{L'}(M, \alpha_1) &
\rTo^{\Phi_{L'}}  & ECH^{\kappa L'} (M,\alpha_2),
\end{diagram}
\end{equation}
where the vertical arrows are maps induced by the inclusions of chain
complexes, commutes by Lemmas~\ref{lemma: identity form a cobordism} and
\ref{pronto soccorso}.
\end{proof}

By taking the direct limit of the maps $\Phi_{L_i}$ from Lemma~\ref{lemma: ECH of two contact forms in M - commutativity}, we obtain a linear map
$$\Phi : ECH(M, \alpha_1) \to ECH(M, \alpha_2).$$
Since the roles of $\alpha_1$ and $\alpha_2$ are interchangeable, the same arguments can be used to define a map $\Phi' : ECH(M, \alpha_2) \to ECH(M, \alpha_1)$ as a direct limit of maps $\Phi'_{L_j'}$.

\begin{proof}[Proof of Proposition \ref{prop: ECH of two contact forms in M}]
We prove that $\Phi$ and $\Phi'$ are inverses of each other. We identify the composition $\Phi'_{\kappa L} \circ \Phi_L$ (after a proper rescaling) with the map induced by an interpolating cobordism which is homotopic to a piece of symplectization. Then $\Phi'_{\kappa L} \circ \Phi_L = i_{L, \kappa^2 L}$, where $i_{L, \kappa^2 L}$ is
the inclusion map.  By taking the direct limit, we obtain $\Phi' \circ \Phi= id$. The proof of $\Phi \circ \Phi'=id$ is similar.
\end{proof}

\begin{rmk}
We sketch a possible strategy to prove the invariance of the group $ECH(M, \alpha)$ when the Reeb vector field of $\alpha$ defines a foliation on $\partial M$ with closed leaves. This result will not be used in the rest of the article.

When $\partial M$ is foliated by closed orbits of the Reeb vector field of $\alpha$ we would like to view $ECH(M,\alpha)$ as a direct limit of ECH groups of nondegenerate contact forms as in Equation~\eqref{eqn: Morse-Bott direct limit}.  We pick $L>0$ and slightly extend $(M,\alpha)$ to $(M_{\varepsilon},\alpha_\varepsilon)$ so that:
\begin{itemize}
\item $M_{\varepsilon}=M\cup (T^2\times[0,\varepsilon))$ where $\bdry M= T^2\times\{0\}$;
\item $\alpha_{\varepsilon} |_M = \alpha$;
\item $\partial M_{\varepsilon}$ is foliated by Reeb trajectories of $\alpha_{\varepsilon}$ with irrational slope; and
\item there are no Reeb orbits of $\alpha_{\varepsilon}$ on $M_{\varepsilon} - M$ with action $\leq L$.
\end{itemize}

We now consider the chain complexes $ECC^L(M_\varepsilon,f_i\alpha_{\varepsilon})$, where $f_i : M_{\varepsilon} \to \R$ is as in Lemma \ref{bourgeois bis} for $i\gg 0$. Then
\begin{equation}
ECC^L(M, \alpha) \simeq ECC^L(M_{\varepsilon}, f_i \alpha_{\varepsilon})
\end{equation}
by Proposition \ref{prop: from generic to MB}. We then write the ECH group $ECH(M,\alpha)$ as the direct limit of groups $ECH^{L}(M_\varepsilon,f_i\alpha_{\varepsilon})$ as in Corollary~\ref{cor: direct limit of commensurate contact forms}.   We extend $(M_\varepsilon,f_i\alpha_{\varepsilon})$ to a closed manifold by using Lemma~\ref{claim: extension of contact form wo orbits} and apply the (analogs of the) results of this section to define the ECH cobordism maps.
\end{rmk}

\subsection{Variants of ECH relative to the boundary}
\label{subsection: variants of ECH of an open book decomposition}

The goal of this subsection is to define the homology groups
$ECH(M,\bdry M,\alpha)$ and $\widehat{ECH}(M,\bdry M,\alpha)$ which
appear in the statement of Theorem~\ref{thm: equivalence of ECHs}.
They are variants of $ECH(M,\alpha)$ and in many ways can be viewed
as ECH groups relative to the boundary of $M$, hence the notation.

Let $M$ be a manifold with $\partial M \simeq T^2$. Let $\alpha$ be a contact form on $M$ which is nondegenerate on $int(M)$ and such that $\partial M$ is a {\em negative} Morse-Bott torus. Then the ECH groups introduced in Section~\ref{subsection: defns of ECH with t2 bdry} are defined for $(M, \alpha)$ In the rest of  this section we make the further assumption that  there exists a properly embedded oriented surface $(\Sigma, \bdry \Sigma ) \subset (M,\bdry M)$ with connected boundary such that an orbit of the Morse-Bott torus has algebraic intersection number one with $\Sigma$.

As before, we pick two orbits on $\partial M$ and label them $h$ and $e$. There is a perturbation of $\alpha$ near $\bdry M$ which makes $h$ hyperbolic and $e$ elliptic; $h$ corresponds to the maximum and $e$ to the minimum of the perturbing Morse function.

Let $\mathcal{P}$ be the set of simple Reeb orbits of $\alpha$ in the
interior of $M$. Let $ECC^{\flat}_j(M,\alpha)$ be the chain complex
generated by orbit sets $\gamma$ constructed from $\mathcal{P} \cup
\{ e \}$,  whose algebraic intersection number $\langle [\gamma],\Sigma\rangle$ is $j$.
By construction, $ECC_j^{\flat}(M,\alpha)$ is a direct summand of
$ECC^{\flat}(M, \alpha)$ and its differential  is the restriction of the differential for
$ECC^{\flat}(M, \alpha)$.

In the same way we write $ECC_j(M,\alpha)$ for the chain complex generated by orbit
sets $\gamma$ constructed from $\mathcal{P} \cup \{ e,h \}$,  whose algebraic intersection number $\langle [\gamma],\Sigma\rangle$ is $j$.  By construction, $ECC_j(M,\alpha)$
is a direct summand of $ECC(M, \alpha)$ and its differential is the restriction
of the differential for $ECC(M, \alpha)$.

\begin{lemma}\label{lemma: general elections}
There are inclusions of chain complexes:
$$ECC^{\flat}_j(M,\alpha)\to ECC^{\flat}_{j+1}(M,\alpha), $$
$$ECC_j(M,\alpha) \to ECC_{j+1}(M,\alpha)$$
given by the map $\gamma\mapsto e\gamma$,
where we are using multiplicative notation for orbit sets.
\end{lemma}

\begin{proof}
Let  $\gamma$ be an orbit set in $M$ and $u$ a holomorphic map with image in
$\R \times M$ which is positively asymptotic to $e\gamma$.
Then $u$ has an irreducible component which is mapped to the trivial cylinder over
$e$. In fact, by the Trapping Lemma, $u$ cannot have nontrivial positive ends that limit to orbits on $\partial M$ because $M$ is a negative Morse-Bott
torus.   Also, one can check that, $Z' \in H_2(M, e \gamma, e \gamma')$ is obtained  by adding a trivial cylinder over $e$ to $Z \in H_2(M, \gamma, \gamma')$, then $I(e\gamma, e\gamma', Z')=1$ whenever $I(\gamma,\gamma', Z)=1$. This is a consequence of \cite[Proposition 7.1]{Hu}, since the associated partitions satisfy the admissibility conditions (Equations~(23) and (24) in \cite[Definition 4.7]{Hu}). It is crucial in the verification of the admissibility condition that, in the Morse-Bott situation, the outgoing partition for $e$ with multiplicity $n$ is $(n)$ and the incoming partition is $(1, \ldots ,1)$ for all $n$, together with the fact that every $J$-holomorphic map in $\R \times M$ with a positive end to $e$ is a connector.  Hence $\partial^\flat(e \gamma)= e \partial^\flat(\gamma)$ and
$\partial (e \gamma)= e \partial (\gamma)$.
\end{proof}

The homology of the chain complex $ECC_j^{\flat}(M,\alpha)$ will be
written as $ECH^{\flat}_j(M,\alpha)$ and that of  the chain complex
$ECC_j(M,\alpha)$ will be written as $ECH_j(M,\alpha)$.

\begin{defn}\label{defn: ECH relative to the boundary}
We define
$$ECH(M,\bdry M,\alpha)=\lim_{j\to\infty} ECH^{\flat}_j(M,\alpha),$$
$$\widehat{ECH}(M,\bdry M,\alpha)=\lim_{j\to\infty} ECH_j(M,\alpha).$$
\end{defn}

\begin{rmk}
The groups $ECH(M,\bdry M,\alpha)$ and $\widehat{ECH}(M,\bdry
M,\alpha)$ can also be interpreted as the homology of the chain complexes obtained
by taking the quotient of the chain complexes $ECC^{\flat}(M,\alpha)$ and
$ECC(M,\alpha)$ respectively by
the subcomplexes generated by all elements of the form $e\gamma-\gamma$, where $\gamma$ is any orbit set constructed
from ${\mathcal P} \cup \{ e \}$ in the case of $ECH(M,\bdry M,\alpha)$ or from
${\mathcal P} \cup \{ e, h \}$ in the case of $\widehat{ECH}(M,\bdry M,\alpha)$.  This alternative definition, unlike Definition~\ref{defn: ECH relative to the boundary}, does not need the assumption that the Reeb orbits on the boundary have intersection one with a properly embedded surface.
\end{rmk}
 \begin{rmk}
The differentials in $ECH(M,\bdry M,\alpha)$ and $\widehat{ECH}(M,\bdry
M,\alpha)$ preserve the total relative homology class of the generators. Then we can define subgroups $ECH(M, \partial M, \alpha, A)$ and $\widehat{ECH}(M,\bdry
M,\alpha, A)$ for every $A \in H_1(M, \partial M)$.
\end{rmk}

\section{ECH of the solid torus} \label{section: ECH of solid torus}

\subsection{Overview of the computation}

In this section we calculate various versions of ECH of the solid torus with certain boundary conditions and specific contact structures. We will write $V = D^2 \times S^1$ and use Convention~\ref{convention on solid torus}  to compute the slope of essential curves in $\partial V$ and in boundary-parallel tori contained in $V$.

The following lemma constructs the contact forms used in the main theorem. Let $V_0 \subset \ldots \subset V_i \subset \ldots \subset V$ be an exhaustion by concentric solid tori, $T_i = \partial V_i$, and ${\mathcal T} = \cup_i T_i$.  Let $(\rho, \phi, \theta)$ be the cylindrical coordinates on $V=D^2\times S^1$ from Section~\ref{subsec: contact form on V}. We assume that $T_i =\{ \rho = \rho_i \}$. We will choose $V_i$ so that the Reeb flow foliates $T_i = \partial V_i$ by orbits of {\em irrational} slope $r_i$.

\begin{lemma} \label{careful perturbation}
There exists a contact form $\alpha_V$ on $V=D^2 \times S^1$ which is an arbitrarily $C^{\infty}$-small perturbation of the contact form $\alpha$ from Example~\ref{esempio utile} and which satisfies the following:
\begin{itemize}
\item[(a)] the Reeb orbits of $\alpha_V$ in $int(V)$ are nondegenerate;
\item[(b)] $\alpha_V$ and $\alpha$ agree to infinite order along $\partial V$ and along $\mathcal{T}$. In particular, the Reeb flow of $\alpha_V$ foliates the tori $T_i$ by orbits of irrational slope $r_i$ and $\partial V$ by orbits of infinite slope; and
\item[(c)] for every $i$, all orbits in $V - V_i$ have slope greater than $r_i$.
\end{itemize}
\end{lemma}

\begin{proof}
Let $L_i\to\infty$, $i=1,2,\dots$, be an increasing sequence of real numbers and let $d$ be a metric on $C^{\infty}(V)$ inducing the $C^{\infty}$-topology.\footnote{For example we can take $d(f,g) = \sum \limits_{k=0}^{\infty} 2^{-k} \dfrac{\| f-g \|_{C^k}}{1+\| f-g \|_{C^k}}$.} Fix $\varepsilon >0$ sufficiently small.

We claim that for $i=1,2,\dots$ there exists a function\footnote{The functions $f_i$, $g_i$, and $f$ introduced in this proof are, of course, unrelated to the functions $f$ and $g$ defining $\alpha$ in Example~\ref{esempio  utile}.} $f_i:V\to\R$ which satisfies the following:
\begin{itemize}
\item[(i)] $e^{f_i}\alpha$ is $L_i$-nondegenerate;
\item[(ii)] $d(f_i, f_{i-1}) < 2^{-i} \varepsilon$; and
\item[(iii)]  $\op{supp} (f_i - f_{i-1}) \subset int(V)- ({\mathcal O}_{i-1} \cup {\mathcal T})$,
\end{itemize}
where $\mathcal{O}_i$ is the union of all simple Reeb orbits of $e^{f_i}\alpha$ with action less than $L_i$.
Here we are setting $f_0=0$, $\mathcal{O}_0=\varnothing$, and $L_0=0$. We define $f_i$ inductively:  We choose $g_i$ such that $f_i = f_{i-1}+g_{i-1}$ satisfies (i)--(iii). In fact, as shown for example in the proof of \cite[Lemma~7.1]{CH2}, the functions
$g_i$ can be chosen arbitrarily close to $0$ in the $C^{\infty}$-topology and with
support in arbitrarily small neighborhoods of the Reeb orbits of action in
$[L_{i-1}, L_i]$. The claim then follows.  The sequence $f_i$ is a Cauchy sequence, so we define
$f = \lim \limits_{i \to \infty} f_i$ and $\alpha_V = e^f \alpha$. The contact form $\alpha_v$ satisfies (a) and (b).

It remains to prove (c).  But this is immediate since the slope in Example~\ref{esempio utile} is strictly increasing with the radius on the region $V - V_i$ and we are performing a $C^\infty$-small perturbation so that this property is preserved.
\end{proof}

$\partial V$ is a positive $\alpha_V$-Morse-Bott torus. We can perturb $\alpha_V$ so that the Morse-Bott family for $\partial V$ becomes a pair of nondegenerate Reeb orbits $e'$ and $h'$, where $e'$ is an elliptic orbit corresponding to the maximum of the perturbing function and $h'$ is a hyperbolic orbit corresponding to the minimum. The following is the main result of this section:

\begin{thm} \label{thm: ECH of  solid torus is Z}
Let $\alpha_V$ be a contact form on $V$ constructed in Lemma~\ref{careful perturbation}. Then:
\begin{enumerate}
\item $ECH(int(V),\alpha_V)\simeq \F$, generated by $\varnothing$.
\item $ECH^{\sharp}(V, \alpha_V) \simeq 0$.
\item $ECH(V, \alpha_V) \simeq 0$.
\item $ECH^{\flat}(V, \alpha_V) \simeq \F[e']$, where $\F[e']$ is the polynomial ring generated by $e'$ over $\F$.
\end{enumerate}
\end{thm}

\begin{rmk}
Proposition \ref{prop: ECH of two contact forms in M} does
not cover contact forms whose Reeb flow
has rational slope on $\partial V$,  so we cannot claim that the computation in Theorem
\ref{thm: ECH of  solid torus is Z} is independent of the contact form. However, the
computation for the contact forms $\alpha_V$ constructed in Lemma \ref{careful
perturbation} will be sufficient for the proof of Theorem \ref{thm: equivalence of ECHs}.
\end{rmk}

The proof of (1) proceeds as follows: In Section
\ref{subsec: irrational slope} we compute $ECH(V_i, \alpha_V|_{V_i})$. Since the slope of
the Reeb flow of $\alpha_V$ on $T_i = \partial V_i$ is irrational, we can use Proposition
 \ref{prop: ECH of two contact forms in M} to replace the contact forms
$\alpha_V|_{V_i}$ with different forms for which the computation is easy. We also lift the
relative grading on the ECH groups given by the ECH index to an absolute grading
which is compatible with the maps induced by the interpolating cobordisms. In Section \ref{subsec: infinite slope} we prove that the inclusions $V_i \subset V_{i+1}$ induce inclusions of chain complexes $ECC(V_i, \alpha_V|_{V_i}) \subset ECC(V_{i+1}, \alpha_V|_{V_{i+1}})$ as a consequence of the Blocking Lemma. This implies that
$$ECH(int(V), \alpha_V) = \lim \limits_{i \to \infty} ECH(V_i, \alpha_V|_{V_i}).$$
We then use the absolute grading to conclude the proof: the degrees of the generators of
$ECH(V_i, \alpha_V|_{V_i})$ that are different from $\varnothing$ go to infinity as $i\to \infty$, so only $\varnothing$ survives in the direct limit.

The proofs of (2)--(4) are given in Section~\ref{subsec: completion of proof of thm solid torus} and use (1) and some results on holomorphic curves in $\R \times V$ due to Taubes and Wendl.

\subsection{ECH (V, $\alpha$) when the Reeb flow has irrational slope on the boundary} \label{subsec: irrational slope}

In this subsection we compute $ECH(V, \alpha)$ for contact forms $\alpha$ whose
Reeb flow foliates $\partial V$ by orbits of irrational slope  and whose underlying contact structure gives the standard contact neighborhood of a transverse knot. For this boundary condition we
have proved the invariance of $ECH$, so by Proposition \ref{prop: ECH of two contact
forms in M} we can choose a particularly simple contact form to do the computation.

Let $r>0$ be an irrational number.  Pick a contact form $\alpha_r$ on
$V \simeq D^2\times S^1$ as in Example~\ref{esempio scemo}
so that the following hold:
\begin{itemize}
\item the boundary $\bdry V$ and all the concentric tori $T_\rho$, $\rho\in(0,1]$, are foliated by Reeb orbits of irrational slope $r$;
\item the contact structure $\ker\alpha_r$ is transverse to all the fibers $\{pt\}\times S^1$.
\end{itemize}
There is only one simple closed orbit, namely the core $c=\{0\}\times S^1$. The orbit $c$ is elliptic and all its multiple covers $c^n$ are nondegenerate due to the irrationality of $r$. Note that $[c^n]=n[S^1]\in H_1(V)$, so we immediately have the following lemma:

\begin{lemma}\label{banalotto}
$ECH(int(V), \alpha_r; n[S^1]) \simeq \F$, generated by $c^n$, if $n \ge 0$ (where $c^n=\varnothing$ if $n=0$) and $ECH(int(V), \alpha_r; n[S^1])=0$ if $n<0$.
\end{lemma}

In order to plug this computation into the direct limit in the proof of Theorem~\ref{thm: ECH of  solid torus is Z}, we define an absolute grading on the ECH
groups of the solid torus in a way which is compatible with the cobordism maps.
For simplicity we will consider only contact forms
$\alpha$ which satisfy the following assumption:
\begin{itemize}
\item[($\bigstar$)] the core of $V$ is an elliptic Reeb orbit $c$ , all of whose multiple covers are nondegenerate.
\end{itemize}
The contact forms $\alpha_r$ in Lemma~\ref{banalotto}, as well as the contact forms
$\alpha_V|_{V_i}$ of Lemma~\ref{careful perturbation} satisfy this assumption.

Let $\xi = \ker \alpha$. We chose a trivialization $\tau$ of $\xi$ such that its restriction to the core orbit $e$ is homotopic to the pullback of a basis of $T_0D^2$ and the
linearized Reeb flow at $e$ is a rotation by angle $2 \pi \theta$ with $\theta \in \R - \Q$.

\begin{lemma}\label{absolute grading exists}
Let $\alpha$ be a contact form on a solid torus $V$ which satisfies ($\bigstar$). Then there is an absolute grading $I$ on $ECC(int(V), \alpha)$ such that:
\begin{enumerate}
\item $I(c^n) = \sum \limits_{k=1}^n (2 \lfloor k\theta \rfloor +1)$,
\item if $\gamma_1$, $\gamma_2$ are two orbit sets and $Z$ is a surface from $\gamma_1$ to $\gamma_2$, then
$$I(\gamma_1, \gamma_2, Z) = I(\gamma_1) - I(\gamma_2).$$
\end{enumerate}
\end{lemma}

\begin{proof}
Given an orbit set $\gamma$ with $[\gamma] = n[S^2]$, we choose a $\tau$-trivial surface $Z$ from $\gamma$ to $e^n$ and define
\begin{equation}
I(\gamma) := \widetilde{\mu}_{\tau}(\gamma) + c_1(\xi|_Z, \tau)+ Q_{\tau}(Z).
\end{equation}
Since $H_2(V)=0$, $I(\gamma, c^n, Z)$ is independent of $Z$ by \cite[Lemma 2.5(a)]{Hu}.  Hence $I(\gamma)$ is well-defined.

(1) follows from the calculation $\widetilde{\mu}_{\tau}(c^n) = \sum \limits_{k=1}^n (2 \lfloor k \theta \rfloor +1)$ using \cite[Formula 2.3]{Hu2} and (2) follows from the additivity of the ECH index.
\end{proof}

\begin{lemma}\label{absolute grading is invariant}
Let $\alpha_1$ and $\alpha_2$ be contact forms on $V$ which coincide on $\partial V$ to first order and  define contact structures which are isotopic relatively to the boundary. If both $\alpha_1$ and $\alpha_2$ satisfy ($\bigstar$)  and their Reeb flows foliate $\partial V$ by
orbits of irrational slope, then the isomorphism $ECH(V, \alpha_1) \simeq ECH(V, \alpha_2)$ from Proposition \ref{prop: ECH of two contact forms in M} preserves the absolute grading $I$.
\end{lemma}

\begin{proof}
 We denote by $I_1$ and $I_2$ the absolute grading on the groups $ECH(V, \alpha_1)$ and $ECH(V, \alpha_2)$ respectively.
We know from Remark~\ref{cobordisms dont leak} that  the isomorphism $ECH(V, \alpha_1) \stackrel{\simeq} \longrightarrow ECH(V, \alpha_2)$ is supported by holomorphic buildings  in a completed interpolating cobordism $([0,1] \times V, \lambda)$ from $(V, \alpha_1)$ to $(V, \alpha_2)$\footnote{To add some confusion, what is called $V$ here corresponds to $M$ in Section~\ref{subsec: well definition of ECH with boundary}.}. Moreover by \cite[Theorem 5.1]{Cr}, those buildings have total $ECH$ index $I=0$ for a version of the ECH index in cobordisms; see \cite{Hu2} for its definition.  Then the lemma holds if
\begin{equation} \label{non ne posso piu}
I(\gamma_1, \gamma_2, Z)= I_ 1 (\gamma_1)- I_ 2(\gamma_2)
\end{equation}
for all surfaces $Z$ in $[0,1] \times V$ connecting an orbit set $\gamma_1$ for $\alpha_1$ to an orbit set $\gamma_2$ for $\alpha_2$.

Since $H_2(V)=0$,  we can assume that $Z$ is the union of a surface $Z_1$ from $\gamma_1$ to $c^n$ (for some $n$), the surfaces $Z_0^n$ consisting of $n$ copies of the
cylinder $Z_0$ over the core orbit $c$, and a surface $Z_2$ from $c^n$ to $\gamma_2$.
Moreover we can assume that $Z_1$ and $Z_2$ project to surfaces in $V$, so that
$I(\gamma_1, c^n, Z_1) = I_1(\gamma_1) - I_1(c^n)$ and $I_2(c^n, \gamma_2, Z_2) =
I_2(c^n) - I_2(\gamma_2)$. Then
$$I(\gamma_1, \gamma_2, Z) = I_1(\gamma_1) -I_1(c^n) + I(c^n, c^n, Z_0^n) + I_2(c_n) - I_2(\gamma_2)$$
and consequently Equation \eqref{non ne posso piu} holds if and only if
$$I(c^n, c^n, Z_0^n) = I_1(c^n) -  I_2(c_n)$$
for every $n \ge 0$. This is however the case because
$$c_1(T([0,1] \times V)|_ {Z_0^n},\tau)= Q_{\tau}( Z_0^n) =0.$$
\end{proof}

By combining Proposition \ref{prop: ECH of two contact forms in M} and Lemmas~\ref{banalotto}--\ref{absolute grading is invariant} we obtain:

\begin{lemma} \label{lemma: description of ECH of alpha_r}
If $\alpha$ is a contact form on $V$ satisfying ($\bigstar$)
and $\partial V$ is foliated by Reeb orbits of irrational slope $r>0$, then
$$ECH(V,\alpha,n[S^1]) \simeq \left \{
\begin{array}{l} \F \quad \text{in degree }
I= \sum_{k=1}^n (2\lfloor kr \rfloor +1), \text{ for } n>0; \\
\F \quad \text{in degree } I=0, \text{ for } n=0; \text{ and } \\
0 \quad \text{for } n<0. \end{array} \right.$$
\end{lemma}

\subsection{Computation of $ECH(int(V), \alpha_V)$} \label{subsec: infinite slope}

The goal of this subsection is to compute $ECH(int(V), \alpha_V)$, where $\alpha_V$ is a contact form constructed in Lemma \ref{careful perturbation}.

\begin{lemma} \label{lemma: V_i includes into V_i+1}
The inclusions $V_i\subset V_j$ for $i <j$ induce inclusions of chain
complexes
\begin{equation} \label{eqn: map induced by inclusion of chain maps}
ECC(V_i,\alpha_V|_{V_i})\to ECC(V_j,\alpha_V|_{V_j}).
\end{equation}
Moreover, the inclusions $V_i\subset V$ induce inclusions of chain complexes
$$ECC(V_i,\alpha_V|_{V_i})\to ECC(int(V), \alpha_V).$$
\end{lemma}

\begin{proof}
Let $\gamma$ be an orbit set whose orbits are contained in $V_i$. We will prove that
every $J$-holomorphic map $u : F \to \R \times V$ which has $\gamma$ at its positive
end has image in $\R \times V_i$. Let $\gamma'$ be the orbit set at the negative end
of $u$. We first prove that all the orbits of $\gamma'$ must be contained in
$V_i$. Arguing by contradiction, suppose
$\gamma'=\gamma'_{in}\gamma'_{out}$, where the orbits of
$\gamma'_{in}$ are in $V_i$ and the orbits of
$\gamma'_{out} \neq \varnothing$ are in $V-V_i$.
The Reeb vector field determines a homology class $s_i \in H_1(T_i; \R)$, up to multiplication by a positive constant, which has slope $r_i$ using Convention~\ref{convention on solid torus}. We can also regard
$[\gamma'_{out}]$ as a homology class in $H_1(T_i; \R)$ and the slope of
$[\gamma'_{out}]$ is larger than $r_i$ because every Reeb orbit in $V -V_i$ has slope larger than $r_i$ by Lemma~\ref{careful perturbation}. This implies that
$[\gamma'_{out}] \cdot s_i >0$.

 Denote by $u_V$ the composition of $u$ with the projection $\R \times V \to V$ and let $\delta \in H_1(T_i; \R)$ be the homology class of the intersection $u_V(F) \cap T_i$, oriented as the boundary of $u_V$ restricted to $V_i$.  (Recall that the tori $T_i$ are foliated by Reeb orbits of irrational slope, so that $u$ has no ends at $T_i$.) Then $\delta = -[\gamma'_{out}]$, so $\delta \cdot s_i<0$. This contradicts the positivity of intersections in dimension three (Lemma~\ref{lemma: positive slope}) and therefore all orbits in $\gamma'$ are contained in $V$. Hence   Lemma~\ref{lemma: blocking lemma}(1) (Blocking Lemma)
implies that $u(F) \subset \R \times V_i$.
\end{proof}

With all these preliminary steps in place, the computation of $ECH(int(V), \alpha_V)$ is straightforward.

\begin{prop}\label{prop: computation of ECH of int(V)}
$ECH(int(V), \alpha_V) \simeq \F$ and is generated by $\varnothing$.
\end{prop}

\begin{proof}
By Lemma~\ref{lemma: V_i includes into V_i+1} we have
\begin{equation}\label{eqn: direct limit for int(V)}
ECH(int(V),\alpha_V)=\lim_{i\to \infty} ECH(V_i,\alpha_V|_{V_i}).
\end{equation}
Moreover, all the generators of $ECH(V_i,\alpha_V|_{V_i})$ in Lemma \ref{lemma: description of ECH of alpha_r} that are different from $\varnothing$ have degree $I>\lfloor 2 r_i \rfloor +1$. Since $r_i\to \infty$ and the inclusions
$$ECH(V_i, \alpha_V|_{V_i}) \to ECH(V_j, \alpha_V|_{V_j})$$
are degree-preserving, every generator different from $\varnothing$ eventually is mapped to zero in the directed system. Hence $ECH(int(V), \alpha_V) \simeq \F$ and is generated by $\varnothing$.
\end{proof}

\subsection{Finite energy foliations}

In this subsection we study finite energy foliations of $\R \times V$ and $\R \times T^2 \times [1,2]$ which have been constructed by Wendl~\cite{We, We2} and Taubes~\cite{T3}. Finite energy foliations were introduced in \cite{HWZ}; here we recall their definition.

\begin{defn}
A {\em finite energy foliation} of a symplectic cobordism $(W, \omega)$ with an adapted almost complex structure $J$ is a codimension two foliation of $W$ such that every leaf is the image of an embedded $J$-holomorphic map with finite energy.
\end{defn}

Here we are using the notion of energy from \cite[Section~6.1]{BEHWZ}.  The ends of a finite energy $J$-holomorphic map in $W$ are asymptotic to cylinders over Reeb orbits.

The purpose of considering finite energy foliations is twofold: they constrain holomorphic curves by the positivity of intersections and contribute to the ECH differential via the
Morse-Bott construction. The foliation on $\R \times V$ will be used in the proof of Theorem~\ref{thm: ECH of  solid torus is Z} (2)--(4) and the foliation on $\R \times T^2 \times [1,2]$ will be used in the proofs of Lemmas~\ref{lemma: more precise version of bdry zero} and \ref{lemma: computation of U_0}.

\subsubsection{Automatic transversality} \label{subsub: automatic transversality}

For certain moduli spaces of $J$-holomorphic maps in dimension four, transversality holds for topological reasons and there is no need to perturb the almost complex structure. In this subsection we describe such {\em automatic transversality} results of Wendl~\cite{We3}. We need to discuss automatic transversality, since the finite energy foliations that we consider are constructed for very symmetric, and therefore nongeneric, almost complex structures.

Let $F = \overline{F} - \mathbf{z}$, where $\overline{F}$ is a closed oriented surface and $\mathbf{z} = \{z_1, \ldots , z_r \}$ is a finite set of punctures.  Following Wendl~\cite{We3}, we fix a partition  $\mathcal{P}= \{\mathbf{z}_C^+,  \mathbf{z}_C^-, \mathbf{z}_U^+,  \mathbf{z}_U^- \}$ of $\mathbf{z}$. We use the superscript $+$ (resp.\ $-$) to indicate the punctures which correspond to the positive (resp.\ negative) ends and define $\mathbf{z}_C = \mathbf{z}_C^+ \cup  \mathbf{z}_C^-$, $\mathbf{z}_U = \mathbf{z}_U^+ \cup  \mathbf{z}_U^-$.

To any puncture $z \in \mathbf{z}_C$ we associate an orbit $\gamma_z$ (which can either be non-degenerate or belong to a Morse-Bott family) and to any
puncture $z \in \mathbf{z}_U$ we associate a Morse-Bott family ${\mathcal N}_z$.
We write
$${\mathcal M}^{\mathcal{P}} = \mathcal{M}(\{ \gamma_z \}_{z \in \mathbf{z}^+_C}, \{\mathcal{N}_z \}_{z \in \mathbf{z}^+_U}, \{ \gamma_z \}_{z \in \mathbf{z}^-_C}, \{\mathcal{N}_z \}_{z \in \mathbf{z}^-_U})$$
for the moduli space of holomorphic maps $u : (F,j) \to (\R \times M,J)$,  which are positively asymptotic to the orbits $\gamma_z$ for  $z \in \mathbf{z}^+_C$ and to the Morse-Bott families $\mathcal{N}_z$ for $z \in \mathbf{z}^+_U$ and are negatively asymptotic to the orbits $\gamma_z$ for  $z \in \mathbf{z}^-_C$ and to the Morse-Bott families $\mathcal{N}_z$ for $z \in \mathbf{z}^-_U$. Here we range over all complex structures $j$ on $F$ and quotient by automorphisms of the domain.

 Ends which correspond to punctures in $\mathbf{z}_C$ are called {\em constrained ends} and ends which correspond to punctures in $\mathbf{z}_U$ are called {\em unconstrained ends}. The definition of ${\mathcal M}^{\mathcal{P}}$ motivates this terminology: constrained ends are asymptotic to a specify orbit, while unconstrained ends are asymptotic to ends which can move in a Morse-Bott family.

The virtual dimension of  ${\mathcal M}^{\mathcal{P}}$ at $u$ will be denoted by $\op{ind} (u, \mathcal{P})$.
Fix $\delta >0$ arbitrarily small. For every puncture $z\in \mathbf{z}$ we define
$$c_z = \left \{ \begin{array}{ll} \delta & \text{if} \ z \in \mathbf{z}_C, \\ - \delta & \text{if} \ z \in \mathbf{z}_U. \end{array} \right.$$
Choose a symplectic trivialization $\tau$ of $\xi|_{\gamma_z}$ which is complex linear with respect to $J$. Let $A_{\gamma_z}$ be the asymptotic operator of $\gamma_z$.  With respect to the trivialization $\tau$, $A_{\gamma_z}$ can be written in the form $-J{d\over dt} +S(t)$, where $J=\begin{pmatrix} 0& -1 \\ 1& 0\end{pmatrix}$, $t$ is the direction of $\gamma_z$, and $S(t)$ is a  loop of symmetric matrices. Also let $A^{\mathcal{P}}_{\gamma_z} = A_{\gamma_z} \pm c_z  \op{Id}$ be the {\em perturbed} asymptotic operator of $\gamma_z$, where we choose the positive (resp.\ negative) sign if $z\in \mathbf{z}^+$ (resp.\ $z\in \mathbf{z}^-$). This is equivalent to turning on negative (resp.\ positive) exponential weights at positive unconstrained (resp.\ constrained) ends and negative constrained (resp.\ unconstrained) ends.

The perturbed asymptotic operator $A^{\mathcal{P}}_{\gamma_z}$ yields a path of symplectic matrices $\Phi^{\mathcal{P}}_z$, and we define $\mu_{\tau}(\gamma_z,\mathcal{P}) = \mu(\Phi^{\mathcal{P}}_z)$. We say that a puncture $z$ is {\em even} if $\mu_{\tau}(\gamma_z,\mathcal{P})$ is even and we denote by $\# \Gamma_0(u, \mathcal{P})$ the number of even punctures of $(u,\mathcal{P})$. By the properties of the Conley-Zehnder index the set of even punctures, and therefore $\# \Gamma_0(u, \mathcal{P})$, does not depend on the trivialization $\tau$.

\begin{thm} [{\cite[Equation (1.1)]{We3} and \cite[Remark 1.2]{We3}}]
\label{automatic transversality}
Let $u : F \to \R \times M$ be a $J$-holomorphic map and $\mathcal{P}$ a partition of the ends of $u$. Then
\begin{equation}
\op{ind}(u, \mathcal{P}) = - \chi(F) + 2c_1(u^*\xi, \tau) + \sum_{z \in \mathbf{z^+}}
\mu_{\tau}(\gamma_z,\mathcal{P}) - \sum_{z \in \mathbf{z^-}} \mu_{\tau}(\gamma_z,\mathcal{P}).
\end{equation}
Moreover, if $u$ is an immersion, then it is a regular point of ${\mathcal M}^{\mathcal{P}}$ if
\begin{equation}
\op{ind}(u, \mathcal{P}) > 2g(F)-2 + \# \Gamma_0(u, \mathcal{P}).
\end{equation}
\end{thm}

The following lemma computes $\mu_{\tau}(\gamma_z,\mathcal{P})$ in terms of the Conley-Zehnder index of a nondegenerate perturbation of the Reeb orbit.

\begin{lemma} \label{index demistified}
Suppose $\delta>0$ is sufficiently small.
\begin{enumerate}
\item If $\gamma_z$ is a nondegenerate orbit, then $\mu_{\tau}(\gamma_z,\mathcal{P}) = \mu_{\tau}(\gamma_z)$.
\item If $\gamma_z$ belongs to a Morse-Bott family ${\mathcal N}$ and $\gamma_{min}$ and $\gamma_{max}$ are the nondegenerate Reeb orbits corresponding to a minimum and a maximum of a Morse function on ${\mathcal N}$, then:
\begin{itemize}
\item $\mu_{\tau}(\gamma_z,\mathcal{P}) =\mu_{\tau}(\gamma_{min})$ if $z\in \mathbf{z}_C^+\cup \mathbf{z}_U^-$; and
\item $\mu_{\tau}(\gamma_z,\mathcal{P}) =\mu_{\tau}(\gamma_{max})$ if $z\in \mathbf{z}_U^+\cup \mathbf{z}_C^-$.
\end{itemize}
\item $\# \Gamma_0(u, \mathcal{P})$ is the total number of:
\begin{itemize}
\item ends at even nondegenerate orbits;
\item constrained positive ends and unconstrained negative ends at positive Morse-Bott tori; and
\item unconstrained positive ends and constrained negative ends at negative Morse-Bott tori.
\end{itemize}
\end{enumerate}
\end{lemma}

\begin{proof}
(1) is immediate.

(2) Let $T=T_{\mathcal{N}}\subset M$ be the torus corresponding to ${\mathcal N}$ and let $g : M \to \R$ and $\overline{g}_{\mathcal{N}}:\mathcal{N}\to \R$ be $C^{\infty}$-small functions satisfying (P1)--(P4) from Section~\ref{subsection: Morse-Bott contact forms}. We denote the Morse-Bott form $\alpha_0$ and its Reeb vector field by $R_0$. Then the Reeb vector field of the perturbed contact form $(1+g) \alpha_0$ is $R= (1+g)^{-1} R_0 + X$, where $X\in \xi=\ker\alpha_0$ is a solution of
$$ i_X d \alpha_0 = (1+g)^{-2} (dg-dg(R_0)\alpha_0).$$
If we choose an almost complex structure $J$ on $\xi$ and a metric $h$ on $M$ which is
compatible with $J$ and $\alpha_0$ in the sense that $R_0$ is a unit vector field which
is orthogonal to $\xi$ and $h|_{\xi \otimes \xi} = d \alpha_0(\cdot, J \cdot)$, then
$$X= - (1+g)^{-2}J (\nabla g- h(\nabla g,R_0) R_0).$$

Let $\gamma$ be an orbit in ${\mathcal N}$ which corresponds to a critical point of $\overline{g}$ so that $\gamma$ is also a Reeb orbit for $R$. We can associate two asymptotic operators to $\gamma$: the operator $A_{\gamma}$, when we regard $\gamma$ as a Reeb orbit of $R_0$, and the operator $A'_{\gamma}$ when we regard $\gamma$ as a Reeb orbit of $R$.

Let $\tau$ be the period of $\gamma$ as an orbit of $R$ and assume for simplicity that the period of $\gamma$ as an orbit of $R_0$ is $1$. Then $\tau$ is equal to the value of $(1+g)$ at any point of $\gamma$. If $\nabla$ is a symmetric connection, the asymptotic operators can be written as
$$A_{\gamma} = -J(\nabla_t - \nabla R_0), \quad A'_{\gamma} = -J(\nabla_t - \tau \nabla R);$$
see \cite[Page 370]{We3}.
Since $dg =0$ and $\nabla g=0$ along $\gamma$, we have
\begin{align*}
\nabla R = (1+g)^{-1} \nabla R_0 & - (1+g)^{-2} \nabla (J (\nabla g-h(\nabla g,R_0) R_0)),\\
\nabla (J(\nabla g-h(\nabla g,R_0) R_0)) & =(\nabla J)(\nabla g-h(\nabla g,R_0) R_0) \\
& \quad +J\nabla(\nabla g-h(\nabla g,R_0) R_0))\\
& =  J\nabla(\nabla g-h(\nabla g,R_0) R_0))=J Hg,
\end{align*}
along $\gamma$, where $Hg$ is the Hessian of $g$ restricted to the $\xi$-directions.  Hence
$$A'_\gamma = -J(\nabla_t - \nabla R_0 + (1+g)^{-1} JHg) = A_\gamma + (1+g)^{-1} Hg.$$
If $\overline{g}$ has a minimum at $\gamma$, then $Hg \ge 0$ along $\gamma$ and $A'_{\gamma}$ has the same Conley-Zehnder index as $A_{\gamma}+ \delta$. On the other hand, if $\overline{g}$ has a maximum at $\gamma$, then $A'_{\gamma}$ has the same Conley-Zehnder index as $A_{\gamma} - \delta$.

(3) is immediate from (2).
\end{proof}

\subsubsection{Foliations on $\R \times V$ and $\R \times T^2 \times [1,2]$} \label{subsection: foliations}

We first describe the finite energy foliation on $\R \times V$.  The following is proven in Wendl~\cite{We} (see pp.~594--600, especially the removal of singularities argument on p.~599; the gist of the proof is to reduce the $J$-holomorphic curve equation to an ODE \cite[Equations~(37a) and (37b)]{We}).

\begin{prop}\label{prop: Wendl V}
Let $\alpha$ be a contact form on $V$ as in Example~\ref{esempio utile} and $J_0$ a ``cylindrically symmetric'' almost complex structure on $\R \times V$ (i.e., $J_0$ depends only on the radial coordinate $\rho$ of $V$) which is adapted to $\alpha$. Then there is a finite energy foliation ${\mathcal Z}_0$ of $\R \times V$ such that:
\begin{enumerate}
\item $\R\times int(V)$ is foliated by $J_0$-holomorphic planes which are positively asymptotic to the Reeb orbits on $\partial V$; and
\item $\R\times \bdry V$ is foliated by trivial cylinders over Reeb orbits of $\bdry V$.
\end{enumerate}
Any orbit of $\bdry V$ is the limit of a unique $1$-dimensional $\R$-invariant family of non-cylindrical leaves and the projections of the leaves to $int(V)$ foliate $int(V)$ by meridian disks.
\end{prop}

We will use a finite energy foliation of $\R \times V$ in the proof of Theorem~\ref{thm: ECH of  solid torus is Z} (2)--(4). However, the contact form used there is a small perturbation $\alpha_V$ of $\alpha$, and for this reason we need to show that ${\mathcal Z}_0$ persists if $\alpha$ and $J_0$ are deformed.

\begin{prop}\label{foliation on V}
 If $\alpha_V$ is the $C^\infty$-small perturbation of $\alpha$ from Lemma~\ref{careful perturbation}, then there is a finite energy foliation ${\mathcal Z}_1$ of $(\R \times V, d(e^s \alpha_V))$
which is isotopic to ${\mathcal Z}_0$ by the lift to $\R \times V$ of an isotopy of $V$
relative to the boundary.
\end{prop}

\begin{proof}
A leaf $u$ of ${\mathcal Z}_0$, considered as a $J_0$-holomorphic map with a {\em constrained} end, has  Fredholm index one and is automatically transverse by Theorem \ref{automatic transversality}.  Indeed, by Lemma \ref{index demistified}, the index of $u$, as a $J_0$-holomorphic map with constrained end, is equal to the index of a $J_\varepsilon$-holomorphic plane $u_\varepsilon$ which limits to a hyperbolic orbit $h$ (i.e., the minimum of the Morse-Bott family) on the boundary for a perturbed contact form. If $\tau$ is the trivialization of $\xi$ along $h$ given by $\xi\cap T\bdry V$, then $\chi(D^2)=1$, $c_1(u_\varepsilon^*\xi,\tau)=1$, $\mu_\tau(h)=0$, and therefore $\op{ind}(u)= \op{ind}(u_\varepsilon)=1$. The same leaf $u$, considered as a $J_0$-holomorphic map with an {\em unconstrained} end, has  Fredholm index two and is also automatically transverse.

Let ${\mathcal M}_0$ be the $2$-dimensional moduli space of $J_0$-holomorphic planes which are leaves of ${\mathcal Z}_0$. By the unconstrained automatic transversality, if we perturb $\alpha$ and the almost complex structure $J_0$ slightly, then each leaf of ${\mathcal Z}_0$ is deformed to a $J$-holomorphic curve for the new almost complex structure $J$ and the space $\mathcal{M}_1$ of deformed $J$-holomorphic curves is diffeomorphic to ${\mathcal M}_0$. On the other hand, the constrained automatically transversality implies that for each Reeb orbit in $\partial V$ there is exactly one $\R$-invariant family of $J$-holomorphic maps in ${\mathcal M}_1$ positively asymptotic to that orbit.

The maps in ${\mathcal M}_1$ are embeddings because embeddedness is an open condition and the exponential decay estimates imply that no self-intersection can be created near infinity.  Moreover,  the relative intersection number of their images is zero and by the positivity of intersections, their images are pairwise disjoint, so they define a finite energy foliation ${\mathcal Z}_1$ of $\R \times V$.
\end{proof}

Now we discuss a finite energy foliation ${\mathcal Z}_2$ on a  completed interpolating cobordism $(\R\times T^2 \times [1,2], \lambda)$ between two contact forms satisfying Equation~\eqref{eqn: contact form on T2 times 1,2}.
In the case of a symplectization this foliation was constructed by Wendl~\cite{We}.

We assume that every slice $\{ s \} \times T^2 \times [1,2]$ is a contact type hypersurface;  Then we can write $\lambda = e^s \alpha_s$, where $\alpha_s$ is a contact form  on $\{ s \} \times T^2 \times [1,2]$ given by Equation~\eqref{eqn: contact form on T2 times 1,2} for pairs of functions $(f_s,g_s)$ which depend on $s$ and $y$.  The forms $\alpha_s$ will define a $2$-plane field $\xi$ and a vector field $R$ on $\R\times T^2 \times [1,2]$ which restrict to the contact structure and the Reeb vector field on each slice $\{ s \} \times T^2 \times [1,2]$.  In particular, $R$ is tangent to the tori $\{ s \} \times T^2 \times \{ y \}$. Moreover we assume that  $\alpha_s$ is constant in $s$ near $\R \times T^2 \times \{1,2\}$ and that $R$ is parallel to $\partial_t$ when $y =1,2$ and not parallel to it otherwise. Finally, we assume that the tori $ \{s\} \times T^2 \times \{ 1 \}$ and $ \{s\} \times T^2 \times \{ 2 \}$ are foliated by Morse-Bott families ${\mathcal N}_1$ and ${\mathcal N}_2$ respectively  for each $s$, where ${\mathcal N}_1$ is negative and ${\mathcal N}_2$ is positive.

We take an almost complex structure $J$ on $\R\times T^2\times[1,2]$ with coordinates $(s,\vartheta,t,y)$ so that the following hold:
\begin{itemize}
\item $J$ is adapted to $\lambda$;
\item $J$ is invariant in the $s$-direction on the cylindrical ends of the cobordism;
\item $J$ is invariant in the $\vartheta,t$-directions;
\item $J(\partial_s) =R$; and
\item $J$ sends $\bdry_y\in \xi$ to the tangent space to $\{ s \} \times T^2\times\{y\}$.
\end{itemize}
 For the existence of such an almost complex structure we need to verify
that the plane distribution generated by $\partial_s$ and $R$ is $d \lambda$-symplectic, and that $\partial_y$ belongs to its $d \lambda$-orthogonal. The first property is guaranteed if $\alpha_s$ varies sufficiently slowly is $s$, while the second property follows from the fact that $\alpha_s(\partial_y)=0$ everywhere. Finally, the symmetries of $J$ reflect the symmetries of the forms $\alpha_s$.
\begin{lemma} \label{lemma: foliation in T^2 x [1,2]}
Let $(\R\times T^2 \times [1,2], \lambda)$ be an exact symplectic cobordism with an
adapted almost complex structure $J$ as above. Then there is a $2$-dimensional
family ${\mathcal Z}_2$ of holomorphic cylinders $Z_{s,\vartheta}$ on
$\R\times T^2\times[1,2]$,  for $(s,\vartheta)\in \R\times
\R/\Z$,  which foliate $\R\times int(T^2\times[1,2])$ and
project to cylinders $\vartheta=const$ in $int(T^2\times[1,2])$. Each cylinder $Z_{s,\vartheta}$ is positively asymptotic to a Reeb orbit in $\mathcal{N}_2$ and negatively asymptotic to a Reeb orbit in $\mathcal{N}_1$.
\end{lemma}

\begin{proof}
Let us write $v = J(\partial_y)$. Our conditions on $J$ and $R$ imply that $\partial_t =  a(s,y)v + b(s, y)R$ with $b(s,y) \ne 0$ everywhere and $a(s,y)=0$ only when $y=1$ or $y=2$, in which case $\left .\frac{\partial a}{\partial y} \right|_{y=0,1}\ne 0$. Then $J(\partial_t)= -a(s,y) \partial_y -b(s,y) \partial_s$. The vector fields $\partial_t$ and $Y(s,y)=a(s,y)\partial_y+b(s,y)\partial_s$ span a $J$-invariant $2$-plane distribution on $\R \times T^2 \times [1,2]$. Since $a$ and $b$ do not depend on $t$ and $\vartheta$, this distribution is integrable and every integral submanifold in $\R \times T^2 \times [1,2]$ is the product of $\R/\Z$ with coordinate $t$ and an integral curve of $Y$ on the strip $\R \times [1,2]$.

The functions $a$ and $b$ are bounded in $\R \times [0,1]$ because $\left. \frac{\partial a}{\partial s} \right |_{|s| \gg 0}= \left. \frac{\partial b}{\partial s} \right |_{|s| \gg 0}=0$. This implies that $Y$ is complete. Moreover, the maximal integral curves of $Y$ on $\R\times (1,2)$ project diffeomorphically onto $(1,2)$ and have vertical asymptotes for $y \to 1$ and $y \to 2$ because $a(s,y) \ne 0$ when $y \ne 1,2$.
\end{proof}

\begin{lemma} \label{lemma: Z_2 is regular}
Let $u_{s, \vartheta}: \R\times S^1 \to \R\times T^2\times[1,2]$ be a $J$-holomorphic map which parametrizes the holomorphic cylinder $Z_{s, \vartheta}$. Then $(u_{s, \vartheta},\mathcal{P})$ satisfies automatic transversality if $\#\mathcal{P}_U\geq 1$.
\end{lemma}

\begin{proof}
By Theorem~\ref{automatic transversality},
$$\op{ind}(u_{s, \vartheta}, \mathcal{P}) = \mu_{\tau}(\gamma_2,\mathcal{P},+) -\mu_{\tau}(\gamma_1,\mathcal{P},-),$$
where $\gamma_i \in {\mathcal N_i}$, so $\op{ind}(u_{s, \vartheta}, \mathcal{P}) = 2 - \# \Gamma_0 (u_{s, \vartheta}, \mathcal{P})$ by Lemma \ref{index demistified}. Hence the condition for automatic transversality in Theorem~\ref{automatic transversality} holds if $\# \Gamma_0 (u_{s, \vartheta}, \mathcal{P}) <2$. Both the constrained negative end at ${\mathcal N}_1$ and the constrained positive end at ${\mathcal N}_2$ are even and the lemma follows.
\end{proof}

\subsubsection{Constraints on holomorphic curves}

Finite energy foliations constrain $J$-holomorphic maps with the same asymptotics. The following lemma describes an instance of this phenomenon. A similar situation has also been considered in Wendl~\cite{We4}.

\begin{lemma} \label{fef constrains curves}
Let $P$ be a compact oriented surface and $\alpha$ a Morse-Bott contact form on $S^1 \times P$ such that $S^1 \times \partial P$ is  a union of Morse-Bott tori and $\{\vartheta\}\times \bdry P$ is a union of Reeb orbits for each $\vartheta\in S^1$. If $\R \times S^1 \times P$ has a finite energy foliation ${\mathcal Z}$  on which $\R \times S^1$ acts freely and transitively and such that every leaf projects diffeomorphically to $int(P)$, then every somewhere injective finite-energy $J$-holomorphic map $u : F\to \R \times S^1 \times P$ with no ends at a Reeb orbit in $S^1 \times int(P)$ is a leaf of ${\mathcal Z}$.
\end{lemma}

\begin{proof}
Let $Z_{s, \vartheta}$ be the leaves of $\mathcal{Z}$ parametrized by $(s, \vartheta) \in \R \times S^1$. Suppose first that there is a leaf $Z_{s_0, \vartheta_0}$ such that
$u(F) \cap Z_{s_0, \vartheta_0} \ne \varnothing$ and which is asymptotic to different Reeb orbits than $u$. The intersection points in $u(F) \cap Z_{s_0, \vartheta_0}$ are isolated and positive. However $u(F) \cap Z_{s_0', \vartheta_0} = \varnothing$ if $s_0'$ is sufficiently large, a contradiction. Hence there exists some $\vartheta_0\in S^1$ such that $u(F) \subset \cup_{s \in \R} Z_{s, \vartheta_0}$ and the leaves $Z_{s, \vartheta_0}$ are asymptotic to the same Reeb orbits as $u$. If $u(F)$ is not contained in a leaf, this forces the intersection $u(F) \cap Z_{s_0, \vartheta_0}$ to be one-dimensional for some $s_0 \in \R$. This is too large an intersection, and the unique continuation for $J$-holomorphic maps
\cite[Theorem 2.3.2]{McDS} implies that $u(F)$ is a leaf of ${\mathcal Z}$.
\end{proof}

\begin{rmk} \label{rmk: in cobordisms too}
  The proof of Lemma \ref{fef constrains curves} goes through unchanged for the foliation ${\mathcal Z}_2$ constructed in Lemma \ref{lemma: foliation in T^2 x [1,2]}, even though the curves $Z_{s, \vartheta}$ and $Z_{s', \vartheta}$ are not translations of one another unless $(\R\times T^2 \times [1,2], \lambda)$ is a symplectization. In fact, they still project to the same annulus in $T^2 \times [1,2]$ and, given any point in that annulus, their preimages $x \in Z_{s, \vartheta}$ and  $x' \in Z_{s', \vartheta}$ become arbitrarily far apart in the $s$-coordinate when $|s'-s| \to + \infty$. These  properties of the foliation ${\mathcal Z}_2$ are sufficient to make the proof  of Lemma~\ref{fef constrains curves} work.
\end{rmk}

\subsection{Completion of proof of Theorem~\ref{thm: ECH of  solid torus is Z}} \label{subsec: completion of proof of thm solid torus}

In this subsection we prove (2)--(4) of Theorem~\ref{thm: ECH of  solid torus is Z}.

(2) The inclusion $ECC(int(V), \alpha_V) \subset ECC^{\sharp}(V,\alpha_V)$ is an inclusion of chain complexes since no $J$-holomorphic curve in $\R \times V$ with all positive ends in $int(V)$ can have a negative end on $\partial V$ by the Trapping Lemma. Moreover, the map
$$ECC^{\sharp}(V,\alpha_V) \to ECC(int(V), \alpha_V),$$
$$\gamma  \mapsto 0, \quad h' \gamma \mapsto \gamma,$$
where $\gamma$ is an orbit set constructed from orbits in $int(V)$, induces an isomorphism of complexes
$$ECC^{\sharp}(V,\alpha_V)/ ECC(int(V), \alpha_V) \simeq ECC(int(V), \alpha_V).$$
This is due to the fact that $h'$ is a hyperbolic orbit and appears with exponent at most one in a generator of $ECC^{\sharp}(V,\alpha_V)$. From this we have an exact triangle
$$\xymatrix{
ECH(int(V), \alpha_V) \ar[rr] & & ECH(int(V), \alpha_V) \ar[dl] \\
& ECH^{\sharp}(V,\alpha_V) \ar[ul] &
}$$
which splits according to homology classes in $H_1(V)$. Then Proposition \ref{prop: computation of ECH of int(V)} implies that $ECH^{\sharp}(V,\alpha_V,n[S^1])=0$ when $n \ne 0$.

It remains to show that $ECH^{\sharp}(V, \alpha_V, n[S^1])\simeq 0$ for $n=0$.  Its chain complex $ECC^{\sharp}(V, \alpha_V,0)$ is generated by $h'$ and $\varnothing$.  We claim that $\bdry h'=\varnothing$.  By Proposition \ref{foliation on V}, there is a finite energy foliation ${\mathcal Z}_1$ on $(\R \times V, d(e^s \alpha_V))$, whose leaves (in $\R\times int(V)$) are $J$-holomorphic planes which are positively asymptotic to the Morse-Bott family on $\bdry V$. This foliation constrains the $J$-holomorphic curves that limit to orbits on $\bdry V$ at the positive ends. Indeed, by Lemma~\ref{fef constrains curves}, every holomorphic curve which is positively asymptotic to a simple Reeb orbit on $\partial V$ and has no negative ends must be a plane in ${\mathcal Z}_1$. The leaves of ${\mathcal Z}_1$ also contribute to the differential of $ECC(V, \alpha_V)$ since they are automatically transverse by Theorem \ref{automatic transversality}. Hence $\bdry h'=\varnothing$, which implies the vanishing of $ECH^{\sharp}(V, \alpha_V,0)$.

(3) We define a filtration $\mathcal{F}$ on $ECC(V,\alpha_V)$ as follows: Given an orbit set $(e')^m \gamma$, where $\gamma$ does not have any $e'$-terms, we set
$$\mathcal{F}((e')^m\gamma)=m.$$
This defines an ascending filtration of chain complexes: since $J$-holomorphic maps to $\R \times V$ can have only positive ends at $e'$ by the Trapping Lemma, the differential of
$ECC(V,\alpha_V)$ cannot increase the exponent of $e'$. The $E^1$-term of the associated
spectral sequence is isomorphic to $ECH^{\sharp}(V,\alpha_V)$ at each filtration level.
By (1), $ECH^{\sharp}(V,\alpha_V)=0$, and the spectral sequence converges to $0$.

(4) The restriction of ${\mathcal F}$ to $ECC^{\flat}(V,\alpha_V)$ induces a filtration on $ECC^{\flat}(V,\alpha_V)$ which we still denote by ${\mathcal F}$. The $E^1$-term of the spectral sequence for $\mathcal{F}$ is isomorphic to
$$\bigoplus_{m=0}^\infty ECH(int(V),\alpha_V)\cdot(e')^m.$$
Since $ECH(int(V),\alpha_V)\simeq \F\{\varnothing\}$ by Theorem~\ref{thm: ECH of  solid torus is Z}, the $E^1$-term of the spectral sequence is $\F[e']$. All higher differentials vanish  for degree reasons: recall that ECH has a $\Z /2$ grading in which generators with no hyperbolic orbits have even grading.  Hence $E^1 = E^{\infty}$ is the graded group of the induced filtration on $ECH^{\flat}(V,\alpha_V)$.  Since the filtration $\mathcal{F}$ on $ECC^{\flat}(V,\alpha_V)$ is bounded below and exhaustive, the spectral sequence converges by \cite[Theorem~5.5.5]{W} and therefore $ECH^{\flat}(V,\alpha_V) \simeq \F[e']$.

\section{Proof of Theorem~\ref{thm: equivalence of ECHs}}
\label{section: proof of theorem equivalence of ECHs}

In this section we prove Theorem~\ref{thm: equivalence of ECHs}. The proof was greatly influenced by Michael Hutchings, who encouraged us to look for an appropriate filtration.

\subsection{Intuitive idea behind Theorem~\ref{thm: equivalence of ECHs}}\label{subsec: euristic argument}

We briefly explain the intuitive idea behind Theorem~\ref{thm: equivalence of ECHs}.  We recall that $M$ denotes a connected, closed, oriented three-manifold and $K$ is a null-homologous knot in $M$. Suppose for the moment that the contact form $\alpha$ on $M$, in a neighborhood $V \simeq D^2 \times S^1$ of $K$, is given by Example~\ref{esempio scemo}.  In other words, the concentric tori $T_\rho \subset V$, $\rho\not=0$, are foliated by Reeb orbits of irrational slope
${1\over \nu}$.  We would like to take the limit as $\nu \to 0$; in the limit $\bdry V$ is foliated by Reeb orbits of slope $\infty$.  Let us write $N=M - int(V)$. There should be a
one-to-one correspondence, modulo $\R$-translations, between holomorphic curves $u$ in $\R\times M$ of ECH index $1$ which intersect the binding $k$ times, and holomorphic curves $u'$ in $\R\times N$ of ECH index $1$ which have negative ends at an elliptic orbit $e$ of slope $\infty$ with total multiplicity $k$.  Also, as we take $\delta\to 0$, the Conley-Zehnder index of the binding,
measured with respect to the longitudinal framing on $V$, i.e., the framing
given by a Seifert surface $\Sigma$ for $K$, goes to $\infty$. This suggests that we
should be able to effectively ignore the binding if we could take
the limit.

The actual proof --- at least the one we could find --- is considerably more complicated, and uses three ingredients: (i) the calculation of ECH on the solid torus from Section~\ref{section: ECH of solid torus}, (ii) some understanding of holomorphic curves that project to a neighborhood of $K$, and (iii) a filtration on $ECC(M)$.

\subsection{Description of the contact forms}
\label{subsec: description}

We start with a description of the contact forms and their Reeb orbits on $M$ that we use in the proof of
Theorem \ref{thm: equivalence of ECHs}. We fix a neighborhood $V \simeq D^2 \times S^1$ of $K$ and decompose $M$ as
$$M= N \cup (T^2 \times [1,2]) \cup V.$$
Since $K$ is an oriented null-homologous knot, there is a properly embedded oriented surface $S \subset N$ whose boundary $\partial S \subset \partial V$ is a longitude for $K$. On $V$ we choose cylindrical coordinates $(\rho, \phi, \theta)$ such that $\bdry V=\{\rho=1\}$ and $\partial S= \{ \rho=1, \phi= \phi_0 \}$. On $T^2 \times [1,2]\simeq (\R^2/\Z^2)\times[1,2]$ we choose coordinates $(\vartheta, t, y)$ such that $(\vartheta, t, 2)$ is identified with $(\rho,\phi,\theta)=(1,2\pi t,2\pi \vartheta) \in \partial V$. We identify a neighborhood of $\partial N$ in $N$ with $T^2 \times [0,1]$ so that $\bdry N= T^2\times\{1\}$ and the coordinates $(\vartheta, t, y)$ on $T^2\times[0,1]$ extend those on $T^2 \times [1,2]$; similarly  we identify a neighborhood of $\bdry V$ in $V$ with $T^2\times[2,\tfrac{5}{2}]$.

We will work with an increasing sequence $L_i \to + \infty$ and a sequence of Morse-Bott contact forms $\alpha_i$ on $M$ such that:
\begin{itemize}
\item $\alpha_i|_N$ is a fixed Morse-Bott contact form $\alpha$ which is nondegenerate on $int(N)$ and its Reeb vector field is positively transverse to $S$;
\item $\alpha_i |_{T^2 \times [1,2]}$ is a contact form $\alpha_{\delta_i}$ as in Example~\ref{alpha delta} which is chosen so that all the Reeb orbits in $T^2 \times (1,2)$ have action larger than $L_i$; and
\item $\alpha_i|_V = c_{\delta_i} \alpha_V$ for a fixed contact form $\alpha_V$ constructed as in Lemma \ref{careful perturbation} and a decreasing sequence $c_{\delta_i}$ which is bounded above by $1$ and bounded below by a positive constant.
\end{itemize}
We also assume the following technical condition:
\begin{itemize}
\item there is a decreasing sequence $\epsilon_i \to 0$ such that $\alpha_i$ agrees with $\alpha_{i+1}$ on $N \cup (T^2 \times [1, 1+ \epsilon_i])$ and with a constant positive multiple of $\alpha_{i+1}$ on $V \cup (T^2 \times [2- \epsilon_i, 2])$.
\end{itemize}
We will refer to $T^2 \times (1,2)$ as the {\em no man's land}.

The contact form $\alpha$ on $N$ can be constructed using the techniques developed in \cite{CH} and \cite{CGHH}. The construction is described in Section~\ref{subsubsection: construction of the contact form on Nh} in the special case where $K$ is the binding of an open book decomposition of $M$ and $N$ is the mapping torus of a diffeomorphism of $S$.

The contact forms $\alpha_i$ are Morse-Bott and all the Morse-Bott tori are of the form $T^2\times\{y\}$ with $y\in[1,2]$. In particular, $\partial N = T^2 \times \{ 1 \}$ is foliated by a {\em negative} Morse-Bott family ${\mathcal N}_1$ and $\partial V = T^2 \times \{ 2 \}$ by a {\em positive} Morse-Bott family ${\mathcal N}_2$. Both families have infinite slope,  i.e., the Reeb orbits on both tori are meridians of $K$.

We construct $L_i$-nondegenerate contact forms $\alpha_i' = f_i \alpha_i$, where the perturbing functions $f_i$ are as in Section~\ref{subsection: direct limit arguments}. We choose $f_i$ so that the Morse-Bott family ${\mathcal N}_1$ corresponding to $\partial N$ is perturbed into an elliptic orbit $e$ and a hyperbolic orbit $h$, the
Morse-Bott family ${\mathcal N}_2$ corresponding to $\partial V$ is perturbed into a
hyperbolic orbit $h'$ and an elliptic orbit $e'$, no new closed orbits with action less than
$L_i$ are created, and $f_i \equiv 1$ in a neighborhood of all nondegenerate
Reeb orbits of $\alpha_i$  with action less than $L_i$.

For all $i$ we choose regular almost complex structures $J_i$ adapted to $\alpha_i$ and $J_i'$ adapted to $\alpha_i'$ such that all the $J_i$ are fixed on the contact structure  outside $T^2\times[1-\epsilon_i,2+\epsilon_i]$ and $J_i'$ is an arbitrarily small perturbation of $J_i$.

We will also consider interpolating cobordisms $(\R \times M, \widehat{\lambda}_i)$ from $(M, \alpha_i)$ to a rescaling of $(M, \alpha_{i+1})$ and $(\R \times M, \widehat{\lambda}_i')$ from $(M, \alpha_i')$ to a rescaling of $(M, \alpha_{i+1}')$.  By construction, $\lambda_i'$ is an arbitrarily small perturbation of $\lambda_i$. We fix
compatible almost complex structures $\widehat{J}_i$ on $(\R \times M, \widehat{\lambda}_i)$ and $\widehat{J}'_i$ on $(\R \times M, \widehat{\lambda}_i')$ such that they are both regular and $\widehat{J}'_i$ is an arbitrarily small perturbation of $\widehat{J}_i$.

We assume that the perturbing functions are close enough to $1$ that the following hold:
\begin{itemize}
\item[({\bf MB}$_1$)] For $ k=1,2$, if $\gamma_+$ and $\gamma_-$ are generators of $ECC^{L_i}(M, \alpha_i')$ and $u\in {\mathcal M}_{J_i'}^{I= k}(\gamma_+, \gamma_-)$, then there is a corresponding $u_\infty \in {\mathcal M}_{J_i}^{MB, I= k}(\gamma_+, \gamma_-)$.
\item[({\bf MB}$_0$)] If $\gamma_+$ and $\gamma_-$ are generators of $ECC^{L_i}(M, \alpha_i')$ and $ECC^{L_{i+1}}(M, \alpha_{i+1}')$, respectively, and $ u\in {\mathcal M}_{\widehat{J}_i'}^{I=0}(\gamma_+, \gamma_-)$, then there is a corresponding $u_\infty\in{\mathcal M}_{\widehat{J}_i}^{MB,I=0}(\gamma_+, \gamma_-)$.
\end{itemize}
 Recall from Definition~\ref{Morse-Bott building} that ${\mathcal M}_{J}^{MB}(\gamma_+, \gamma_-)$ denotes the set of Morse-Bott $J$-holomorphic buildings from $\gamma_+$ to $\gamma_-$.

For reference we enumerate the main properties of the Reeb vector fields of the contact forms $\alpha_i$ and their perturbations $\alpha_i'$:
\begin{enumerate}
\item $\alpha_i$ is Morse-Bott and $\alpha_i'$ is $L_i$-nondegenerate.
\item $R_{\alpha_i}$ is positively transverse to $S \subset N$ and the meridian disks in $int(V)$.
\item $\alpha_i|_{N}= \alpha$ and $\alpha_i|_{V}= c_{\delta_i} \alpha_V$, where the sequence $c_{\delta_i}$ is decreasing, bounded above by $1$ and bounded below by a positive constant and the contact form $\alpha_V$ is constructed as in Lemma~\ref{careful perturbation}.
\item $\alpha_i$ and $\alpha_{i+1}$ coincide on $N \cup (T^2 \times [1, 1+ \epsilon_i])$ and are constant multiples of one another on $V \cup (T^2 \times [2- \epsilon_i, 2])$, where $\epsilon_i \to 0$ is a decreasing sequence.
\item The Reeb orbits of $\alpha_i$ in the no man's land come in Morse-Bott families of large negative slope and their action is bounded below by $L_i$.
\item There are concentric solid tori $V_0\subset V_1\subset \dots \subset V$ such that $\bdry V_j$, $j=0,1,\dots$, is foliated by dense Reeb orbits of irrational slope $r_j>0$ with $\lim \limits_{j \to \infty} r_j = + \infty$ for any contact form $\alpha_i$.
\item $\bdry N$ is foliated by a negative Morse-Bott family $\mathcal{N}_1$ of Reeb orbits of $\alpha_i$ of slope $\infty$. After perturbation, ${\mathcal N}_1$ becomes a pair of orbits $e$ and $h$. Their Conley-Zehnder indices with respect to the framing coming from $\partial N$  (given by $T(\bdry N) \cap \xi$) are $\mu(e)=-1$ and $\mu(h)=0$.
\item $\bdry V$ is foliated by a positive Morse-Bott family $\mathcal{N}_2$ of slope $\infty$. After perturbation, ${\mathcal N}_2$ becomes a pair of orbits $e'$ and $h'$. Their Conley-Zehnder indices with respect to the framing coming from $\partial V$ are $\mu(e')=1$ and $\mu(h')=0$.
\end{enumerate}

\subsection{Construction of the contact forms}

In this subsection we construct the contact forms $\alpha_i$ when $K$ is the binding of an open book decomposition. In this case $N$ is the mapping torus of a diffeomorphism $\hh : S \to S$ such that $\hh|_{\partial S}= id$. This means that
$$N = (S \times[0,1])/(x,1)\sim (\hh(x),0),$$ where $x \in S$
and $t$ is the coordinate for $[0,1]$. Using the coordinates $(\theta,\phi)$ from Section~\ref{subsec: description} we identify the isotopy classes of simple closed curves in $\partial N$ (and in all parallel tori) with rational numbers so that the meridian has slope $\infty$ and $\partial S$ has slope $0$.

\begin{rmk}
The above slope convention is the same as the usual surgery convention for performing surgery along the binding.
\end{rmk}

\subsubsection{Construction of the contact form on $N$} \label{subsubsection: construction of the contact form on Nh}

We take a $1$-form $\beta$ on $S$ such that $\omega = d \beta$ is a positive
area form on $S$ and $\beta=cyd\theta$ in a
neighborhood $N(\bdry S)\subset S$ of $\bdry S$. Here $c>0$ is a
small constant and $N(\bdry S)$ is identified with
$[1-\delta,1]\times \R/\Z$ with coordinates $(y,\theta)$.

We assume that the diffeomorphism $\hh:S\stackrel\sim\to S$ satisfies $\hh|_{N(\bdry S)}=id$. Let $\op{Symp}(S,\partial S, \omega)$ be the group of symplectomorphisms of $(S,\omega)$ which restrict to the identity on a neighborhood of $\bdry S$. By Moser's lemma, there is an isotopy of $\hh$ relative to $\partial S$ so that the resulting diffeomorphism --- also called $\hh$ by abuse of notation --- is in $\op{Symp}(S, \partial S, \omega)$.

\begin{lemma}[Giroux]  \label{lemma: giroux}
Given $\hh\in\op{Symp}(S, \partial S, \omega)$, there exists an
isotopy $\hh_t$, $t\in [0,1]$, in $\op{Symp}(S,\partial S, \omega)$ so
that $\hh_0=\hh$ and $\hh_1 ^* \beta -\beta= d f$ for some positive function $f$ on $S$.
\end{lemma}

\begin{proof}
Let $\mu =\hh^* \beta -\beta$ and let $Y$ be the vector field which
satisfies $i_Y\omega=-\mu$. By the Cartan formula, we compute that
$\mathcal{L}_Y \omega =i_Y d\omega +d(i_Y \omega)=-d\mu =0$ and
$\mathcal{L}_Y \mu =i_Y d\mu +d(i_Y \mu )=0$. Hence the flow
$\phi_t$ of $Y$ preserves $\omega$ and $\mu$. Moreover, $\phi_t$ is
equal to the identity near $\partial S$, where we have $\mu =0$.

Now let $\hh_t =\hh\circ \phi_t$. We then compute that:
\begin{eqnarray*}
\frac{d}{dt} \hh_t^* \beta  &=& \phi_t^* (\mathcal{L}_Y \hh^*\beta)=
d(\phi_t^* (i_Y \hh^* \beta))+\phi_t^* (i_Y d(\hh^*\beta))\\
&=& dg_t+ \phi_t^* (i_Y \omega) = dg_t-\phi_t^*\mu = dg_t-\mu,
\end{eqnarray*}
where $g_t=\phi_t^* (i_Y \hh^* \beta)$. Hence
\begin{equation}\label{equation: in Giroux lemma}
\frac{d}{dt} \hh_t^* \beta =dg_t+\beta -\hh^* \beta.
\end{equation}
By integrating Equation~\eqref{equation: in Giroux lemma}, we obtain
$\hh_1 ^* \beta -\beta= d f$, where $f=\int_0^1 g_t dt +C$ for a sufficiently large
constant $C$.
\end{proof}

By Lemma \ref{lemma: giroux} we assume that $\hh \in
\op{Symp}(S, \partial S, \omega)$ satisfies $\hh^* \beta- \beta = df$.
Next we construct a contact form on $N$ whose corresponding Reeb
vector field is transverse to the fibers and has first return map
$\hh$.

\begin{lemma}\label{lemma:construction}
Let $\hh$ be a diffeomorphism in $\op{Symp}(S,\partial S, \omega)$
which satisfies $\hh^*\beta-\beta=df$ for some function $f$ on $S$.
Then there is a contact form $\alpha =f_tdt+\beta_t$ on $N$, where
$f_t$ is a family of positive functions on $S$ and $\beta_t$ is a family
of $1$-form on $S$, so that the corresponding Reeb vector field
$R_{\alpha}$ is transverse to all the fibers $S\times\{t\}$ and $\hh$
is the first return map of $R_\alpha$.
\end{lemma}

For a more complete discussion of the realizability of surface
symplectomorphisms as the first return map of a Reeb vector field,
we refer the reader to \cite{CHL}.

\begin{proof}
Consider the $1$-form $\alpha =f_t dt+\beta_t$ on
$S\times[0,1]$, where $f_t$ is to be determined, $\beta_0=\beta$,
$\beta_1=\hh^*\beta$, and
$$\beta_t =\chi (t)\beta_1 +(1-\chi (t))\beta_0$$
interpolates between $\beta_0$ and $\beta_1$. Here we take $\chi
:[0,1] \rightarrow [0,1]$, so that $\chi(0) =0$, $\chi(1)=1$,
${d\chi\over dt}(t)=\dot{\chi} (t)\geq 0$, and $\chi$ is constant
near $0$ and $1$.

Using the condition $\hh^*\beta-\beta=d_Sf$, we verify that the
$1$-form $\dot\beta_t$ is exact on $S$:
$$\dot\beta_t= \dot{\chi} (t) (\beta_1 -\beta_0)=\dot\chi(t) (d_Sf)= d_S
(\dot{\chi} (t)f).$$ Here $d_{S}$ is the exterior derivative on $S$. We then take $f_t= \dot{\chi} (t)f +c$, where $c$ is an
arbitrary positive constant such that $f_t>0$ (and is different from the $c$ in $\beta=cyd\theta$ from the beginning of Section~\ref{subsubsection: construction of the contact form on Nh}). Then $\dot\beta_t =d_{S} f_t$. Since $\chi$ is constant near $t=0$ and
$t=1$, $f_t$ is also constant, and so is $\beta_t$. In particular,
we have $\hh^*f_1=f_0$.

We now compute that
$$d\alpha=d_{S} f_t\wedge dt+d_S\beta_t+dt\wedge\dot{\beta}_t = d_{S} f_t\wedge dt+\omega +dt\wedge d_S f_t= \omega.$$
Hence $\alpha$ is a contact form, its Reeb vector field is parallel to $\partial_t$ on $S\times[0,1]$, and its first return map is $\hh$.
\end{proof}

Now we make a slight modification to $\alpha$ so that $\bdry N$ becomes a {\em negative} Morse-Bott family --- one
that behaves like a sink for $J$-holomorphic maps in $\R \times N$.

On $T_1=\bdry N$, the germ of $\alpha$ is given by $f(y) dt+
g(y)d\theta$, where $f(y)=C$ and $g(y)=cy$.  Here $c>0$ is a small
constant and $C>0$ is a large constant.   We extend $\alpha$ to
$T^2\times[1,1+\varepsilon]$ by extending $(f(y),g(y))$ to
$y\in[1,1+\varepsilon]$ as follows:
\begin{enumerate}
\item $(f(y),g(y))$ satisfies Equation~\eqref{eqn: contact condition for T2 times 1,2}.
\item $(f(y),g(y))$, $y\in[1,1+\varepsilon]$, is close to
$(f(1),g(1))$.
\item $(f(y),g(y))=(f(1+\varepsilon)+
(y-(1+\varepsilon))^2,g(1+\varepsilon)+(y-(1+\varepsilon)))$ near
$y=1+\varepsilon$.
\end{enumerate}
See Figure~\ref{fig: new f and g}. In particular, Condition (3)
implies that $(f'(1+\varepsilon),g'(1+\varepsilon))$ is parallel to
$(0,1)$.
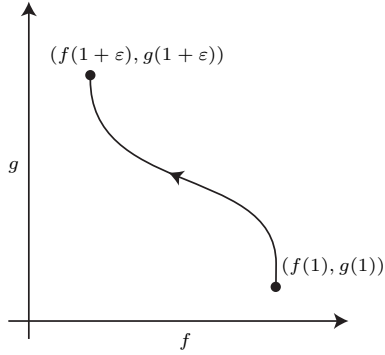
\begin{figure}[ht]
\begin{overpic}[height=4.5cm]{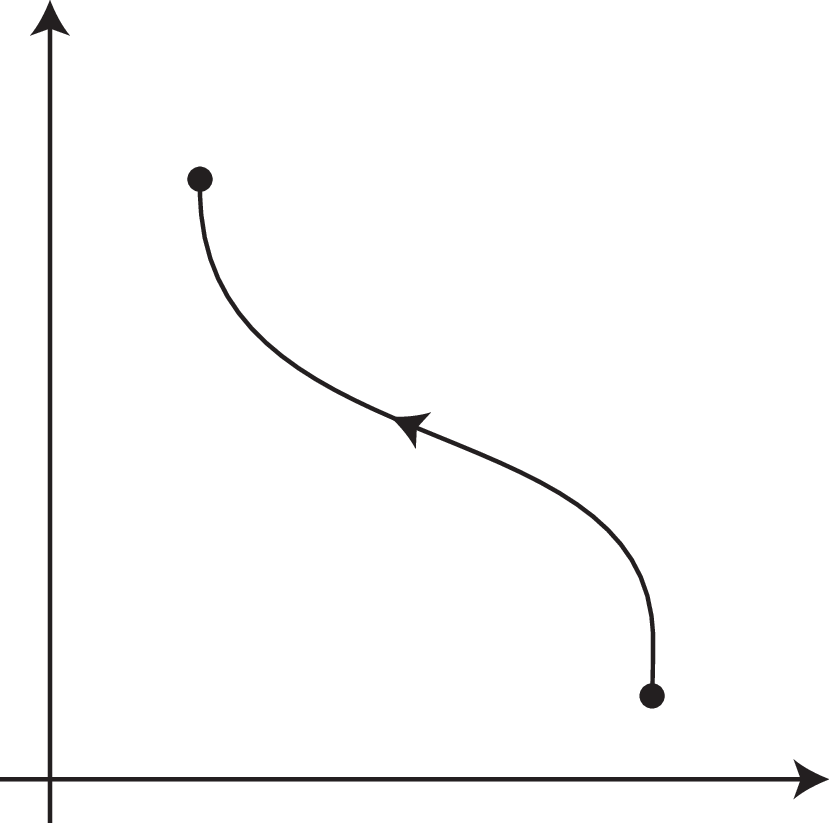}
\put(80,20){\tiny $(f(1),g(1))$} \put(50,-2) {\tiny $f$} \put(0,50)
{\tiny $g$} \put(12,82){\tiny $(f(1+\varepsilon),g(1+\varepsilon))$}
\end{overpic}
\caption{Trajectory of $(f(y),g(y))$. The $f$-axis and $g$-axis do
not necessarily intersect at $(0,0)$ in this figure.} \label{fig:
new f and g}
\end{figure}
Hence $T_{1+\varepsilon}$ is foliated by a Morse-Bott family of Reeb
orbits of slope $\infty$. We write $\alpha$ for the extension of
$\alpha$ to $N\cup (T^2\times[1,1+\varepsilon])$.

We now consider the deformation retract
$$\phi: N\cup (T^2\times[1,1+\varepsilon]) \stackrel\sim \to N,$$
obtained by flowing along the vector field $X=-a(y)\bdry_y$, where
$a(y)=1$ on $T^2\times[1,1+\varepsilon]$ and damps out to zero on
$T^2\times[1-\varepsilon,1]$.  Finally, we perturb $\phi_*\alpha$ on
$N$ so that all Reeb orbits in $int(N)$ become nondegenerate, while
keeping $\partial N$ Morse-Bott. {\em The resulting form will be
called $\alpha$ in the rest of the paper.}

\subsubsection{Extension to $M$.}\label{subsubsection: Extension to M}

The contact form $\alpha$ has the form
$$\alpha = (b+ (y-1)) d\theta + (a + (y-1)^2)dt$$
in some collar $T^2 \times [1- \epsilon, 1]$ of $\partial N$. Here $\epsilon$ is different from the $\varepsilon$ in Section~\ref{subsubsection: construction of the contact form on Nh}.

Choose a decreasing sequence of irrational numbers $\delta_i \to 0$ and a contact form $\alpha_{\delta_i}$ on $T^2 \times [1,2]$ for each $i$ as in Example~\ref{alpha delta}
with $f(1) =a$ and $g(1)=b$. Then $\alpha$ on $N$ and $\alpha_{\delta_i}$ on $T^2 \times [1,2]$ glue to a smooth contact form on $N \cup (T^2 \times [1,2])$. Moreover, there is an increasing sequence $L_i \to + \infty$ such that all Reeb orbits of $\alpha_{\delta_i}$ in $T^2 \times (1,2)$ have action greater than $L_i$.

Fix a contact form $\alpha_{f,g}$ on $V \simeq D^2 \times S^1$ as in Example~\ref{esempio utile}. For each $i$, a multiple of $\alpha_{f,g}$ glues smoothly to the contact form $\alpha_{\delta_i}$ on $T^2 \times [1,2]$. Let $c_{\delta_i}$ be the scaling factor. Then $\alpha_{\delta_i}$ glues smoothly also to $c_{\delta_i}\alpha_V$, where $\alpha_V$ is the contact form obtained by applying the construction of Lemma~\ref{careful perturbation} to $\alpha_{f,g}$. By putting all three pieces together we obtain the contact forms $\alpha_i$ on $M$.

\subsection{The filtrations ${\mathcal F}_i$.}

For each $i$ we define a filtration ${\mathcal F}_i$ on $ECC^{L_i}(M, \alpha_i')$. We first identify $ECC^{L_i}(M, \alpha_i')$, as a vector space, with a subspace of
$$ECC(V, \alpha_V) \otimes ECC(N, \alpha).$$
This is possible because the Reeb orbits of $\alpha_i'$ in the no man's land have actions greater than $L_i$ and those in $V$ coincide with the Reeb orbits of $\alpha_V$, up to reparametrization. The generators of $ECC^{L_i}(V, \alpha_i')$ will be denoted by $\gamma \otimes \Gamma$, where $\gamma \in ECC(V, \alpha_V)$ and $\Gamma \in ECC(N, \alpha)$.
Choose an identification
$$\eta: H_1(V;\Z)\stackrel\sim\to\Z$$
so that the homology class of the null-homologous knot $K$ is mapped to $1$. Then we define the ascending filtration ${\mathcal F}_i: ECC^{L_i}(M,\alpha_i')\to \Z^{\geq 0}$ as follows:
$${\mathcal F}_i \left(\sum_n \gamma_n\otimes \Gamma_n\right)= \max_n \eta([\gamma_n]).$$
 We define ${\mathcal F}_i^p$ as
${\mathcal F}_i^p = \{ x \in ECC^{L_i}(M,\alpha_i') : {\mathcal F}_i(x) \le p \}$.
Note that these filtrations are uniformly bounded below because ${\mathcal F}_i^p=0$
for $p<0$.

\begin{lemma}\label{lemma: filtration constraint on curves}
Let $u:F\to \R\times M$ be a $J_i'$-holomorphic map which is asymptotic to $\gamma\otimes \Gamma$ at the positive end and to $\gamma'\otimes \Gamma'$ at the negative end. Then
$${\mathcal F}_i(\gamma\otimes\Gamma)\geq {\mathcal F}_i(\gamma'\otimes\Gamma').$$
\end{lemma}

\begin{proof}
By ({\bf MB}$_1$) there is a $J_i$-holomorphic Morse-Bott building from $\gamma \otimes \Gamma$ to $\gamma' \otimes \Gamma'$.
Let $\overline{u} : \overline{F} \to \R \times M$ be the holomorphic part of this
building --- which may be disconnected because $\alpha_i$ is not necessarily nice ---
and denote the projection to $M$ by $\overline{u}_M$.

We will use the tori $T_n=\bdry V_n$ in $V$ from Lemma \ref{careful perturbation} to constrain the ends of $\overline{u}$. We recall that $T_n$ is foliated by dense Reeb
orbits of irrational slope $r_n$ with $r_n \to + \infty$. Let $\delta_n$ be the homology class of $\overline{u}_M(\overline{F}) \cap T_n$, oriented as the boundary
of $\overline{u}_M(\overline{F})\cap V_n$. If $n$ is sufficiently large, then  all the orbits in $\gamma$ and $\gamma'$  that are not in the Morse-Bott family on $\bdry V$ are contained in $V_n$.  Hence the sequence $\delta_n$ is constant for $n\gg 0$ and $\eta(\delta_n) = \eta(\gamma')- \eta(\gamma)$.

Regarding both $\delta_n$ and $r_n$ as homology classes in $H_1(T_n; \R)$ and orienting $T_n$ as the boundary of $V_n$, for $n\gg 0$ we obtain $\delta_n \cdot r_n \ge 0$ by the positivity of intersections in dimension three (Lemma~\ref{lemma: positive slope}). Taking the limit $n \to \infty$ and using the fact that the sequence $r_n$ converges
to the slope of the Reeb vector field on $\partial V$, we obtain $\eta(\delta_n) \le 0$ for $n\gg 0$. This implies ${\mathcal F}_i(\gamma\otimes\Gamma)\geq {\mathcal F}_i(\gamma'\otimes\Gamma')$.
\end{proof}

\begin{cor}\label{cor: differential respects filtration}
The differential of $ECC^{L_i}(M, \alpha_i')$ respects the filtration ${\mathcal F}_i$.
\end{cor}

For each $i$ the filtration ${\mathcal F}_i$ induces a spectral sequence
$E^r({\mathcal F}_i)$ which converges to $ECH^{L_i}(M, \alpha_i')$.
The terms $E^0({\mathcal F}_i)$ correspond to
the graded complexes associated to ${\mathcal F}_i$ and can be identified
(as vector spaces) with subspaces of $ECC(V, \alpha_V) \otimes ECC(N, \alpha_N)$.
 The differential $\partial_0$ on
$E^0({\mathcal F}_i)$  is the filtration-preserving component of the differential
on $ECC^{L_i}(M, \alpha_i')$. Every sheet
$E^r({\mathcal F}_i)$ has a grading coming from ${\mathcal F}_i$, and
the component in degree $p$ of  $E^r({\mathcal F}_i)$ will be denoted
by $E^r_p({\mathcal F}_i)$.

\subsection{Description of the differential on $E^0({\mathcal F}_i)$} \label{subsection: step 1}

In this subsection we compute the differential $\partial_0$ on $E^0({\mathcal F}_i)$ using Morse-Bott techniques. This is possible, in spite of the fact that the contact forms $\alpha_i$ are not necessarily nice, because of the following lemma.

\begin{lemma} \label{lemma: constraints on ends}
Let $\tilde{u}$ be a Morse-Bott building from $\gamma \otimes \Gamma$ to $\gamma' \otimes \Gamma'$ in the symplectization of $(M, \alpha_i)$ and let $u$ be its holomorphic part.  If $u$ has a positive end at $\partial N$ or a negative end at $\partial V$, then ${\mathcal F}_i(\gamma' \otimes \Gamma')< {\mathcal F}_i(\gamma \otimes \Gamma)$.
\end{lemma}

\begin{proof}
We recall that $K$ denotes the core of $V$ and that $S \subset N$ is a properly embedded surface such that $\partial S$ defines a longitude of $K$. In the case of an open book
decomposition $K$ is the binding and $S$ is a page.

Let $\overline{U}\simeq T^2 \times [1-\epsilon, 2+ \epsilon] \subset M$ be a small neighborhood of the no man's land $T^2 \times [1,2]$ such that $u$ has no ends at Reeb orbits intersecting $\overline{U}$, except at orbits in ${\mathcal N}_1$ or ${\mathcal N}_2$. Assume without loss of generality that the ends of $u$ limit to distinct orbits $\eta_1,\dots,\eta_n$. Then we let $U_k$ be a small tubular neighborhood of $\eta_k$ for $k=1,\dots,n$ and let $U = \overline{U} -(U_1 \cup \ldots \cup U_n)$. Let $B_k = -\partial U_k$, $k=1,\dots,n$, $B_0=(\bdry \overline U)\cap V$, and $B_{n+1}=(\bdry \overline U)\cap N$. We orient each $B_k$, for $k=0, \ldots, n+1$, using the boundary orientation of $U$.

On each $B_k$, $k=0,\dots,n+1$, we choose an oriented basis of curves $(\mu_k,\nu_k)$ as follows: On $B_0$ and $B_{n+1}$ we choose $\mu_0$ and $\mu_{n+1}$ so that they are longitudes of $K$ coming from $S$ and $\nu_0$ and $\nu_{n+1}$ so that they are meridians of $K$. On each $B_k$, $k=1, \ldots, n$, we choose $\nu_k$ so that it is the longitude of the Reeb orbit in $U_k$ induced by the Morse-Bott torus (which is either $\partial N$ or $\partial V$) and $\mu_k$ so that it is a meridian of $U_k$.  The curves $\nu_k$, $k=0,\dots,n+1$, are oriented by the vector field $\bdry_t$ and the curves $\mu_k$, $k=0, \ldots, n+1$, are oriented by $\mu_k \cdot \nu_k =1$.

By abuse of notation we identify the oriented curves $\mu_k$ and $\nu_k$ with their homology classes in $H_1(U; \Z)$. With this convention $\nu_0 = \nu_1=\dots = \nu_{n+1}$ and  $\mu_0+\mu_1 +\dots +\mu_{n+1}=0$.  Moreover these relations generate the kernel of the  map
$$\bigoplus_{k=0}^{n+1} H_1(B_k; \Z) \to H_1(U; \Z)$$
 induced by the inclusion.
Let $C=\op{Im}(u_M) \cap U$. Then $\partial C = \delta_0 + \ldots + \delta_{n+1}$, where $\delta_k \subset B_k$ is given the orientation induced by $C$.
We will view $\delta_k$ either as an element of $H_1(B_k; \Z)$ or as an element of $H_1(U; \Z)$. Then $\delta_0 + \ldots + \delta_{n+1}=0$ in $H_1(U; \Z)$. For each $k$ we write $\delta_k = a_k \mu_k + b_k\nu_k$.

By the Trapping Lemma and the positivity of intersections in dimension three, we have $\delta_k \cdot \nu_k \ge 0$, $k=1,\dots,n$, because the curves $\nu_k$ can be represented by Reeb orbits. (Here we are using a variation of Lemma \ref{lemma: positive slope} which is an immediate consequence of the positivity of intersections in dimension four.) Then, for all $k= 1, \ldots, n$, $a_k \ge 0$; moreover, if $\delta_k$ corresponds either to a positive end at $T_1$ or to a negative end at $T_2$, then $a_k >0$. The relations in $H_1(U; \Z)$ among the
curves $\mu_k$ and $\nu_k$ imply that $a_0= \ldots = a_{n+1}$, so $a_0 >0$ if $u$ has either a positive end at $\partial N$ or a negative end at $\partial V$. Then
$$a_0 = \eta(\gamma) - \eta(\gamma') = {\mathcal F}_i(\gamma \otimes \Gamma) - {\mathcal F}(\gamma' \otimes \Gamma')>0$$
and this proves the lemma.
\end{proof}

\begin{cor} \label{cor: Morse-Bott is possible}
Let $\gamma \otimes \Gamma$ and $\gamma' \otimes \Gamma'$ be generators of
$ECC^{L_i}(M, \alpha_i')$. If ${\mathcal F}_i(\gamma \otimes \Gamma) =
{\mathcal F}_i(\gamma' \otimes \Gamma')$ and $\tilde{u}$ is a Morse-Bott building
with $I(\tilde{u}) =1$ in the symplectization of $(M, \alpha_i)$ from $\gamma \otimes
\Gamma$ to $\gamma' \otimes \Gamma'$, then the holomorphic part of $\tilde{u}$
 has at most one nontrivial irreducible component.
\end{cor}

\begin{proof}
 Let $u$ be the holomorphic part of $\tilde{u}$. By Lemma~\ref{lemma: constraints on ends}, all ends of $u$ at $\partial N$
are negative and all ends of $u$ at $\partial V$ are positive. Then the structure of
$\tilde{u}$ is simple enough that the argument of Lemma~\ref{fastidio} implies
that $u$ has a unique irreducible component which is not a connector.
\end{proof}

{\em Corollary~\ref{cor: Morse-Bott is possible} implies that, for the purpose of
computing the differential $\partial_0$ of $E^0({\mathcal F}_i)$, we can use Morse-Bott
theory as if the contact forms $\alpha_i$ were nice.}

In order to describe the differential concisely we introduce the following notation. Given two  orbit sets $\gamma'=\prod\gamma_i^{m_i'}$ and
$\gamma=\prod \gamma_i^{m_i}$  (in multiplicative notation),
we set $\gamma/\gamma'= \prod\gamma_i^{m_i-m_i'}$ if $m_i'\leq m_i$ for all $i$; otherwise we set $\gamma/\gamma'=0$. We also call $T_1= \partial N$ and $T_2 = \partial V$.

We now prove the following lemma, which describes the differential
$\bdry_0$ on $E^0$ in some detail:

\begin{lemma} \label{lemma: more precise version of bdry zero}
After identifying $E_0({\mathcal F}_i)$, as a vector space, with a subspace of
$ECC(V, \alpha_V) \otimes ECC(N, \alpha)$, the differential $\bdry_0$ is given by:
\begin{equation} \label{equation for bdry 0, second version}
\bdry_0(\gamma \otimes\Gamma) =  (\bdry_V \gamma)\otimes\Gamma +
(\gamma/e') \otimes h\Gamma + (\gamma/h')\otimes e\Gamma +\gamma
\otimes (\bdry_{N} \Gamma).
\end{equation}
Here $\gamma$ is an orbit set of $V$; if $h$ divides $\Gamma$, then
$h\Gamma$ is understood to be $0$; and $\bdry_X$ is the differential
on the subset $X\subset M$.
\end{lemma}

\begin{proof}
Corollary~\ref{cor: Morse-Bott is possible} and Proposition~\ref{prop: from generic
to MB} imply that $\partial_0$ on $E^0({\mathcal F}_i)$ can be computed by counting
 $I=1$ very nice Morse-Bott buildings in the symplectization of $(M, \alpha_i)$ which do not decrease the filtration level.

The differential $\bdry_0$ does not count holomorphic curves which cross $\R\times T_1 = \R \times \partial N$ or $\R\times T_2 = \R \times \partial V$:  Indeed, if $u$ is a holomorphic curve which contributes to $\bdry_0$ and $u_M$ its projection to $M$, then the homology classes $[\op{Im}(u_M)\cap T_{1\pm \varepsilon}]\in H_1(T_{1\pm \varepsilon})$ and $[\op{Im}(u_M)\cap T_{2\pm \varepsilon}]\in H_1( T_{2\pm \varepsilon})$ (for $\varepsilon>0$ small) have slope $\infty$, and we apply the Blocking Lemma (Lemma~\ref{lemma: blocking lemma}(2)). This still allows for the
possibility of curves that are negatively asymptotic to orbits of $T_1$ or positively
asymptotic to orbits in $T_2$. (Curves which are positively asymptotic to orbits of $T_1$
or negatively asymptotic to orbits of $T_2$ are ruled out by Lemma \ref{lemma:
constraints on ends} because they have been shown to decrease the filtration level.) Such curves are contained in $\R\times V$,  $\R\times T^2\times[1,2]$, or $\R\times N$ by a combination of the Trapping Lemma  (Lemma~\ref{lemma: trapping}) and the Blocking Lemma  (Lemma~\ref{lemma: blocking lemma}).

Curves in $\R\times V$ contribute to the term $(\bdry_V \gamma)\otimes \Gamma$, while curves in $\R \times N$ contribute to $\gamma \otimes \partial_N(\Gamma)$.  Note that there are two cylinders from $e'$ to $h'$ and two cylinders from $h$ to $e$ corresponding to gradient trajectories on $\mathcal{N}_2$ and $\mathcal{N}_1$; these give $\bdry_0 (e'\otimes 1)=0$ and $\bdry_0 (1\otimes h)=0$.

Next we consider curves in $\R\times int(T^2\times[1,2])$. By Lemma~\ref{fef constrains curves}, the only somewhere injective curves in $\R\times int(T^2\times[1,2])$ are the cylinders $Z_{s,\theta}$ defined in Lemma \ref{lemma: foliation in T^2 x [1,2]}. (Remember that we are ignoring the curves which are asymptotic to the orbits in $int(T^2\times[1,2])$ because they have action larger than $L_i$.)  By Lemma~\ref{lemma: Z_2 is regular}, the cylinders $Z_{s,\theta}$ satisfy automatic transversality as long as at least one of the ends is treated as unconstrained. Branched covers of $Z_{s,\theta}$ of degree $>1$ are not counted in the differential since they have $I>1$  (after augmenting them with cylinders corresponding to gradient trajectories). Modulo translations in the $s$-direction, there is a unique $I=1$ Morse-Bott building from $h'$ to $e$, which gives the term $(\gamma/h') \otimes e \Gamma$, and a unique $I=1$ Morse-Bott building from $e'$ to $h$, which gives the term $(\gamma/e') \otimes h \Gamma$ (adding trivial cylinders to these buildings does not change their ECH index because they satisfy the admissibility conditions (Equations~(23) and (24)) from \cite[Proposition 7.1]{Hu}).
\end{proof}

\subsection{Direct limit} \label{subsection: direct limit}

In this subsection we use a direct limit argument to exclude the Reeb orbits in the no
man's land from the complex computing $ECH(M)$. The limit will be compatible with the
filtrations ${\mathcal F}_i$, so the end result will be a spectral sequence $E^r$
converging to $ECH(M)$. The following lemma is immediate from Corollary~\ref{cor:
direct limit of commensurate contact forms} and the construction of the contact forms
$\alpha_i'$.

\begin{lemma} \label{lemma: direct limit 1}
For an appropriate choice of contact forms $\alpha'_i$ and action thresholds $L_i$, we have
$$ECH(M) = \lim_{i\to\infty} ECH^{L_i}(M,\alpha_i').$$
\end{lemma}

The direct limit is taken with respect to  maps
$$\Phi_i : ECH^{L_i}(M, \alpha_i') \to ECH^{L_{i+1}}(M, \alpha_{i+1}')$$
induced by interpolating cobordisms via Lemma~\ref{tired with all these}.

\begin{lemma}\label{continuation maps are nice}
The map $\Phi_i$ is induced by a noncanonical chain map
$$\hat{\Phi}_i : ECC^{L_i}(M, \alpha_i') \to ECC^{L_{i+1}}(M, \alpha_{i+1}')$$
$$\gamma \otimes \Gamma \mapsto \gamma \otimes \Gamma + \mathbf{r}(\gamma \otimes \Gamma),$$
where ${\mathcal F}_{i+1}(\mathbf{r}(\gamma \otimes \Gamma)) < {\mathcal F}_{i+1} (\gamma \otimes \Gamma)$.
\end{lemma}

\begin{proof}
The map $\Phi_i$ is induced by an interpolating cobordism from $\alpha_i'$ to (a rescaling of) $\alpha_{i+1}'$.  We degenerate this cobordism into a two-level cobordism so that the top level interpolates from  $\alpha_i' = f_i \alpha_i$ to $f_{i+1} \alpha_i$ and the bottom level interpolates from $f_{i+1} \alpha_i$ to $\alpha_{i+1}'= f_{i+1} \alpha_{i+1}$. Then $\Phi_i = \Phi_i' \circ \Phi_i''$ by Theorem~\ref{thm: Hutchings Taubes cobordism map}, where
\begin{align*}
& \Phi_i'': ECH^{L_i}(M, \alpha_i') \to ECH^{L_i}(M, f_{i+1}\alpha_i), \\
& \Phi_i'  : ECH^{L_i}(M, f_{i+1}\alpha_i) \to ECH^{L_{i+1}}(M, \alpha_{i+1}').
\end{align*}
The maps $\Phi_i'$ and $\Phi_i''$ are induced by noncanonical chain maps $\hat{\Phi}_i'$ and $\hat{\Phi}_i''$.  By Proposition~\ref{continuation maps are easy} we can assume that $\hat{\Phi}_i''$ is the identity map.

Next we claim that the filtration-nondecreasing part of $\hat{\Phi}_i'$ only counts trivial cylinders. Let $([0,1] \times M, \lambda_i')$ be an interpolating cobordism from $f_{i+1}\alpha_i$ to $\alpha_{i+1}'$ and $(\R \times M, \widehat{\lambda}_i')$ its completion.  By Theorem~\ref{thm: Hutchings Taubes cobordism map}, $\hat{\Phi}'_i$ is ``supported'' on the $I=0$ holomorphic buildings of $(\R \times M, \widehat{\lambda}_i')$. We are assuming that $\widehat{\lambda}_i'$ is sufficiently close to $\widehat{\lambda}_i$, where $([0,1] \times M, \lambda_i)$ is an interpolating cobordism from $\alpha_i$ to $\alpha_{i+1}$ and $(\R \times M, \widehat{\lambda}_i)$ is its completion. Hence, by ({\bf MB}$_0$), if $\langle \hat{\Phi}'_i(\gamma \otimes \Gamma), \gamma' \otimes \Gamma' \rangle \ne 0$, then there is a Morse-Bott building in $(\R \times M, \widehat{\lambda}_i)$ connecting $\gamma \otimes \Gamma$ to $\gamma' \otimes \Gamma'$. Since the $2$-form $d \lambda_i$ agrees with a symplectization on a neighborhood of $\R \times (N \cup V)$, we can repeat
the argument of Lemma~\ref{lemma: filtration constraint on curves} to show that ${\mathcal F}_i(\gamma \otimes \Gamma) \ge {\mathcal F}_{i+1}(\gamma' \otimes \Gamma')$. Moreover, if ${\mathcal F}_i(\gamma \otimes \Gamma) = {\mathcal F}_{i+1}(\gamma' \otimes \Gamma')$, then the holomorphic buildings in $(\R \times M, \widehat{\lambda}_i)$ cannot cross the no man's land by Lemma \ref{lemma: foliation in T^2 x [1,2]} and Remark \ref{rmk: in cobordisms too}. Therefore they are contained in the part of the cobordism $(\R\times M,\widehat\lambda_i)$ which is diffeomorphic to a symplectization. This implies the claim.
\end{proof}

\begin{lemma}\label{lemma: juliette et chocolat}
The chain maps $\hat{\Phi}_i: ECC^{L_i}(M, \alpha_i') \to ECC^{L_{i+1}}(M, \alpha_{i+1}')$ induce chain maps $E^r({\mathcal F}_i) \to E^r({\mathcal F}_{i+1})$. The direct limits
$$E^r({\mathcal F}) = \lim \limits_{i \to \infty} E^r({\mathcal F}_i)$$
form a spectral sequence converging to $ECH(M)$. The page $E^0({\mathcal F})$ can
be identified, as a vector space, with $ECC(V, \alpha) \otimes ECC(N, \alpha)$ and the
differential $\partial_0$ on $E^0({\mathcal F})$ is described by Equation~\eqref{equation for bdry 0, second version}.
\end{lemma}

\begin{proof}
By Lemma \ref{continuation maps are nice} the continuation maps $\hat{\Phi}_i$ are morphisms of chain complexes. Since the construction of the spectral sequence associated to a filtered complex is functorial (see \cite[Proposition~5.9.2]{W}), the
maps  $\hat{\Phi}_i$ induce a morphism of spectral sequences
$$E^r({\mathcal F}_i) \to E^r({\mathcal F}_{i+1}).$$
We define $E^r({\mathcal F}) = \lim \limits_{i \to \infty} E^r({\mathcal F}_i)$. Since
direct limit is an exact functor from the category of directed systems of abelian groups
to the category of  abelian groups (see for example \cite[Theorem~2.6.15]{W}), the limits $E^r({\mathcal F})$ still form a spectral sequence.

We claim now that $E^\infty({\mathcal F}) = \lim \limits_{i \to \infty} E^\infty({\mathcal F}_i)$. First we recall the definition of the $E^\infty$ term of a spectral sequence:
on $E^1$ there is a sequence of subgroups
$$ \{ 0 \} =B^1 \subset B^2 \ldots \subset B^r \subset \ldots \subset Z^r \subset \ldots
\subset Z^2 \subset Z^1 =E^1$$
such that $E^r \simeq Z^r / B^r$; then we define
$$Z^\infty = \bigcap_{r \ge 1} Z^r,  \quad B^\infty = \bigcup_{r \ge 1} B^r \quad \text{and}
\quad E^\infty = Z^\infty / B^\infty.$$
By going through the construction of the spectral sequence, one can see that
$$B^r({\mathcal F}) = \lim_{i \to \infty} B^r({\mathcal F}_i) \quad \text{and} \quad
Z^r({\mathcal F}) = \lim_{i \to \infty} Z^r({\mathcal F}_i)$$ because the direct limit is
an exact functor. (The description of $B^r$ and $Z^r$ given in \cite[Exercise~5.9.1]{W}
can be useful to prove this.)

Then, in order to prove the claim, it is enough to prove that
\begin{align} \label{mittag}
&\lim_{i \to \infty} \left ( \bigcup_{r \ge 1} B^r({\mathcal F}_i) \right ) = \bigcup_{r \ge 1} \left ( \lim_{i \to \infty} B^r({\mathcal F}_i) \right) \\
\label{leffler}& \lim_{i \to \infty} \left ( \bigcap_{r \ge 1} Z^r({\mathcal F}_i) \right ) = \bigcap_{r \ge 1} \left ( \lim_{i \to \infty} Z^r({\mathcal F}_i) \right ).
\end{align}
Equation \eqref{mittag} is not problematic because direct limits commute with countable unions. In fact countable unions can themselves be seen as direct limits,
and direct limits commute as a consequence of their universal property (\cite[Exercise~20]{La}).  On the other hand, in general, direct limits do not commute with infinite intersections, so we need more work to prove Equation~\eqref{leffler}.

The spectral sequence of a filtered complex has a grading coming from the filtration:
we can decompose $E^r({\mathcal F}_i)= \bigoplus E^r_p({\mathcal F}_i)$,
$B^r({\mathcal F}_i)= \bigoplus B^r_p({\mathcal F}_i)$ and $Z^r({\mathcal F}_i)= \bigoplus Z^r_p({\mathcal F}_i)$. Since ${\mathcal F}_i^p=0$ if $p<0$, it follows from the construction of the spectral sequence that $Z_p^\infty({\mathcal F}_i) = Z_p^r ({\mathcal F}_i)$ provided that $r \ge p$. (Again \cite[Exercise~5.9.1]{W}
can be useful here). Taking the direct limit, we obtain that
$\lim \limits_{i \to \infty} E^\infty_p({\mathcal F}_i) = E^\infty_p({\mathcal F})$ and this proves the claim.

The filtrations ${\mathcal F}_i$ induce filtrations on $ECH^{L_i}(M, \alpha_i')$; taking direct limits we obtain a filtration on $ECH(M)$ whose the graded group is the limit of the graded groups of the filtrations on $ECH^{L_i}(M, \alpha_i')$ (again because direct limit is an exact functor). Since the filtrations ${\mathcal F}_i$ are bounded below and exhaustive, the classical convergence theorem \cite[Theorem 5.5.5]{W} implies that
$E^r({\mathcal F}_i)$ converges to $ECH^{L_i}(M, \alpha_i')$ (i.e. $E^\infty({\mathcal F}_i)$ is isomorphic to the graded group of $ECH^{L_i}(M, \alpha_i')$). Taking a direct limit, we then conclude by that $E^r({\mathcal F})$ converges to $ECH(M)$.
\end{proof}

Here the notation $E^r({\mathcal F})$ does not mean that the spectral sequence comes
from some filtration ${\mathcal F}$, but only remembers the fact that it is the direct limit of the spectral sequences induced by the filtrations ${\mathcal F}_i$  --- in fact $E^r({\mathcal F})$ is a spectral sequence of a filtration because a direct limit of filtered complexes is a filtered complex; however the limit complex defining $E^r({\mathcal F})$ is too abstract to be useful. This notation will be useful in the next section, when we will introduce another spectral sequence.

We now rewrite the differential $\partial_0$ in a way which highlights the roles played by the orbits $h$ and $h'$; this will be used extensively in the following subsections. By factoring out the terms $h'$ and $h$, we can write the
differentials $\bdry_V$ and $\bdry_N$ as:
\begin{equation} \label{eqn: decomposition of boundary}
\left \{
\begin{array}{l} \partial_V \gamma = \partial^{\flat}_V \gamma \\
\partial_V (h' \gamma) = h' \partial_V^{\flat} \gamma +
\partial'_V(h' \gamma) \end{array} \right. \qquad \left \{
\begin{array}{l} \partial_N \Gamma = \partial^{\flat}_N
\Gamma + h \partial_N' \Gamma \\ \partial_N (h \Gamma)= h
\partial^{\flat}_N \Gamma
\end{array} \right.
\end{equation}
where $\gamma \in ECC^{\flat}(V, \alpha_V)$, $\Gamma \in ECC^{\flat}(N, \alpha)$, $\bdry^\flat_V$ and $\bdry^\flat_N$ are the differentials for the chain complexes $ECC^{\flat}(V, \alpha_V)$ and $ECC^{\flat}(N, \alpha)$, and the terms $\partial'_V(h' \gamma)$ and $\partial'_N \Gamma$ do not contain $h'$.

\subsection{The map $\sigma_*$}\label{subsection: explicit map}

In this subsection we define an explicit map
$$\sigma_* : ECH(N, \partial N, \alpha) \to ECH(M)$$
and in the next one we will prove that it is an isomorphism.  It will be easy to see that $\sigma_*$ preserves the decomposition by (relative) homology classes; namely, if $\varpi : H_1(N, \partial N) \to H_1(M)$ is the isomorphism described in the introduction, $\sigma_*$ maps $ECH(N, \partial N, A)$ to $ECH(M, \varpi(A))$ for
every $A \in H_1(N, \partial N)$.

We introduce the following notation, which will be used in this and in the following sections. Given a set of Reeb orbits $e_1, \ldots, e_n, h_1, \ldots, h_m$, where $e_1, \ldots, e_n$ are elliptic and $h_1, \ldots, h_m$ are hyperbolic, we denote
$${\mathcal R}[e_1, \ldots, e_n, h_1, \ldots, h_m] := \F[e_1, \ldots, e_n, h_1, \ldots, h_m]/(h_1^2, \ldots, h_m^2);$$
i.e., in ${\mathcal R}[e_1, \ldots, e_n, h_1, \ldots, h_m]$ the elliptic orbits are free variables and the hyperbolic orbits are nilpotent variables of order two.
Whenever we use the notation  ${\mathcal R}[e_1, \ldots, e_n, h_1, \ldots, h_m]$ in this paper, we will assume $\{e_1, \ldots, e_n \} \subseteq \{e, e'\}$ and $\{h_1, \ldots, h_m \} \subseteq \{h,h'\}$.

Define $ECC^\natural(N, \alpha)$ as $ {\mathcal R}[h'] \otimes ECC^{\flat}(N, \alpha)$ with differential
$$\partial^\natural (\gamma \otimes \Gamma) = \gamma \otimes \partial^\flat \Gamma + \gamma/h' \otimes (1+e) \Gamma.$$

\begin{lemma} \label{lemma: uqam}
$ECH^\natural(N, \alpha) \simeq ECH(N, \partial N, \alpha)$.
\end{lemma}

\begin{proof}
$ECC^\natural(N, \alpha)$ can be identified with the cone of the multiplication map $\cdot(1+e)$ on $ECC^{\flat}(N, \alpha)$. Hence there is an exact triangle
\begin{equation} \label{eqn: exact triangle e+1}
\xymatrix{ECC^{\flat}(N, \alpha) \ar[rr]^{\cdot (1+e)} & & ECC^{\flat}(N, \alpha). \ar[dl] \\ & ECC^\natural(N, \alpha) \ar[ul] & }
\end{equation}
The map $\cdot (1+e)$ is injective on homology since $\Gamma$ and $e \Gamma$ belong to different singular homology classes for all $\Gamma \in ECC^\flat(N, \alpha)$.
Then the exact triangle implies that
$$ECH^\natural (N, \alpha) \simeq \frac{ECH^\flat(N, \alpha)}{(1+e) ECH^\flat(N, \alpha)} \simeq ECH(N, \partial N, \alpha).$$
\end{proof}

We denote by $ECC^{\natural, L}_{\le k}(N, \alpha)$ the subcomplex of $ECC^\natural(N, \alpha)$ generated by orbit sets $\gamma \otimes \Gamma$ which have linking number  less than or equal to $k$ with $K$ and action less than $L$. We fix an increasing sequence $L_k' \to + \infty$ and let $c=\sup_k \mathcal{A}_{\alpha'_k}(e')$.  Then for every $k$ we choose $i_k$ so that $L_{i_k} \ge k L_k' + ck^2$.

In the following, we will rename $L_{i_k}=L_k$, $\alpha_{i_k}'= \alpha_k'$ and ${\mathcal F}_{i_k}= {\mathcal F}_k$. Also, the composition $\hat{\Phi}_{i_{k+1}-1} \circ \ldots \circ \hat{\Phi}_{i_k}$ will be renamed as
$$\hat\Phi_k: ECC^{L_k}(M,\alpha'_k)\to ECC^{L_{k+1}}(M,\alpha'_{k+1}).$$

For any integer $k$ we define
$$\sigma_k : ECC^{\natural, L_k'}_{\le k}(N, \alpha) \to ECC^{L_k}(M, \alpha_k')$$
$$\gamma \otimes \Gamma\mapsto \sum \limits_{i=0}^{\infty} (e')^i \gamma \otimes (\partial_N')^i \Gamma,$$
where $\bdry_N'$ is defined by Equation~\eqref{eqn: decomposition of boundary} and $\gamma=1$ or $h'$.

These maps are well-defined because the map $\partial_N'$ is nilpotent. In fact, $\partial_N'$  decreases the linking number with the binding, so $(\partial_N')^{k+1} =0$ on $ECC^{\flat}_{\le k}(N, \alpha)$.

\begin{rmk} This, and the analogous construction in Section~\ref{subsec: U map}, are the only places where we use the hypothesis that the Reeb flow be transverse to a fixed Seifert surface for $K$. In fact, while we could deduce the nilpotency of $\partial_N'$ from an action argument, by choosing to work with the action we would lose the estimate on the nilpotency order of $\partial_N'$ and, consequently, on the action of $\sigma_k (\gamma \otimes \Gamma)$. However, in view of the heuristic argument described in Section~\ref{subsec: euristic argument}, we suspect that this hypothesis is actually not necessary.
\end{rmk}

\begin{lemma}
The maps $\sigma_k$ are chain maps and form a directed system, i.e., the following diagram commutes:
\begin{equation} \label{diagramma insignificante}
\xymatrix{ ECC^{\natural, L_k'}_{\le k}(N, \alpha) \ar[r]^-{\sigma_k} \ar[d]^{\iota_k} &
ECC^{L_k}(M, \alpha_k') \ar[d]^{\hat \Phi_k} \\
ECC^{\natural, L_{k+1}'}_{\le k+1}(N, \alpha) \ar[r]^-{\sigma_{k+1}} &
ECC^{L_{k+1}}(M, \alpha_{k+1}').}
\end{equation}
Here $\iota_k$ is the inclusion.
\end{lemma}

\begin{proof}
(1) We first show that $\sigma_k$ is a chain map.  Since $\sigma_k$ takes values in the lowest level for the filtration ${\mathcal F}_k$  (recall $\gamma =1$ or $h'$), we have $\partial (\sigma_k(\Gamma)) = \partial_0 (\sigma_k(\Gamma))$, where $\partial_0$ is given by Equation~\eqref{equation for bdry 0, second version}. Using the decomposition of $\partial_N$ in Equation~\eqref{eqn: decomposition of boundary} and $\partial_V((e')^i \gamma)= (e')^i \gamma/h'$ for $\gamma=1$, $h'$, we obtain:
\begin{align*}
&\partial_0 (\sigma_k (\gamma \otimes \Gamma)) = \partial_0 \left ( \sum
\limits_{i=0}^{\infty} (e')^i \gamma \otimes
(\partial_N')^i \Gamma \right ) \\
& = \sum \limits_{i=0}^{\infty} (e')^i \gamma /h' \otimes  (\partial_N')^i \Gamma + \sum
\limits_{i=0}^{\infty} (e')^i \gamma \otimes \left ( \partial^\flat_N  (\partial_N')^i \Gamma
+h (\partial_N')^{i+1} \Gamma \right )  \\
&\qquad\quad + \sum \limits_{i=1}^{\infty} (e')^{i-1} \gamma
\otimes h (\partial_N')^i \Gamma + \sum \limits_{i=0}^{\infty} (e')^i \gamma /h' \otimes e  (\partial_N')^i \Gamma.
\end{align*}
Rearranging the sum and using the fact that $\partial_N'$ commutes with $\partial_N^{\flat}$ and with the multiplication by $(1+e)$ gives:
\begin{align*}
\partial_0 (\sigma_k(\gamma \otimes \Gamma)) &= \sum \limits_{i=0}^{\infty} \left (
(e')^i \gamma \otimes (\partial_N')^i \partial^\flat_N \Gamma + (e')^i \gamma/h' \otimes
(\partial_N')^i ( (1+e) \Gamma) \right ).
\end{align*}
Hence $\partial(\sigma_k(\gamma\otimes \Gamma))=\partial_0 (\sigma_k(\gamma \otimes \Gamma)) =\sigma_k(\partial^{\natural}(\gamma \otimes \Gamma))$.
\s\n
(2) Diagram~\eqref{diagramma insignificante} commutes because we have shown in Lemma
\ref{continuation maps are nice} that
the continuation maps are induced by the  identity at the chain level on the lowest
filtration level.
\end{proof}

Taking homology first and then direct limits in Diagram~\eqref{diagramma insignificante}, we obtain a map
$$\sigma_* : ECH(N, \partial N, \alpha)\simeq ECH^\natural(N,\alpha) \to ECH(M).$$
The maps $\sigma_k$ also induce maps
$$\sigma^0 : ECC^{\natural}(N, \alpha) \to E^0({\mathcal F}),$$
$$\gamma \otimes \Gamma\mapsto \sum \limits_{i=0}^{\infty} (e')^i \gamma \otimes (\partial_N')^i \Gamma$$
and
$$\sigma^r : ECH(N, \partial N) \simeq ECH^\natural(N,\alpha)\to E^r({\mathcal F}),\quad  r>0.$$

\subsection{Computation of $E^1({\mathcal F})$.} \label{subsection: step 2}

In this subsection we compute the term $E^1({\mathcal F})$ of the spectral sequence that converges to $ECH(M)$ and prove the first half of Theorem~\ref{thm: equivalence of ECHs}.

Recall from Lemma~\ref{lemma: juliette et chocolat} that $E^0({\mathcal F})\simeq ECC(V, \alpha) \otimes ECC(N, \alpha)$ as a vector space and the
differential $\partial_0$ is given by Equations~\eqref{equation for bdry 0, second version} and ~\eqref{eqn: decomposition of boundary}. If we write
$$C_{k,k'}= (h')^{k'} ECC^{\flat}(V, \alpha) \otimes h^k ECC^{\flat}(N, \alpha),$$
then
$$E^0({\mathcal F})\simeq ECC(V, \alpha) \otimes ECC(N, \alpha) = C_{0,0} \oplus C_{0,1} \oplus C_{1,0} \oplus C_{1,1}.$$

We can organize all components of the differential $\partial_0$
besides $\bdry^\flat_V\otimes 1$ and $1\otimes \bdry^\flat_N$ in the
following diagram:
\begin{equation} \label{diagramma 1} \xymatrix{
C_{0,1} \ar[rrr]^{1 \otimes h\partial'_N + \cdot / e' \otimes h}
\ar[d]_{\partial'_V \otimes 1 + \cdot / h' \otimes e} & & &
C_{1,1} \ar[d]^{\partial'_V \otimes 1 + \cdot / h' \otimes e} \\
C_{0,0} \ar[rrr]^{1 \otimes h\partial_N' + \cdot / e' \otimes h} & &
& C_{1,0} } \end{equation}

\subsubsection{The filtration $\mathcal{G}$} \label{subsub: filtration G}

We introduce a filtration ${\mathcal G}$ of length $3$ on
$$(E^0({\mathcal F}), \partial_0) = (ECC(V, \alpha) \otimes ECC(N, \alpha), \partial_0),$$
which is defined as follows:
$${\mathcal G}^0 = C_{1,0}, \quad {\mathcal G}^1 = C_{0,0} \oplus C_{1,1},
\quad {\mathcal G}^2 = C_{0,1}.$$
This filtration induces a spectral sequence $E^r({\mathcal G})$ which converges to
$E^1({\mathcal F})$. The groups $E^r({\mathcal G})$ have two gradings: one inherited
from the grading on $E^0({\mathcal F})$ (which, in turn, is induced by the filtrations
${\mathcal F}_i)$ and one induced by the filtration ${\mathcal G}$. We will denote the
homogeneous components of $E^r({\mathcal G})$ by $E^r_{pq}({\mathcal G})$, where
$p$ is the degree inherited from $E^0({\mathcal F})$ and $q$ is the degree induced by
${\mathcal G}$.  We also write $E^r_{p}({\mathcal G})$, in which case $p$ is the degree inherited from $E^0({\mathcal F})$.

\subsubsection{Determination of $(E^1({\mathcal G}), \partial_{01})$}

The graded complex associated to ${\mathcal G}$ is
$$(E^0({\mathcal G}), \partial_{00}) \simeq ( {\mathcal R}[h',h] \otimes ECC^{\flat}(V, \alpha)
\otimes ECC^{\flat}(N, \alpha), 1 \otimes \partial^{\flat}_V \otimes 1 + 1 \otimes 1
\otimes \partial^{\flat}_N).$$
Then $(E^0({\mathcal G}), \partial_{00})$ is a product complex and its homology can be
computed by the K\"unneth formula:
$$E^1({\mathcal G}) =  {\mathcal R}[h',h] \otimes ECH^{\flat}(V, \alpha) \otimes ECH^{\flat}(N, \alpha).$$
Taking into account the grading inherited from $E^0({\mathcal F})$ and the
computation of $ECH^{\flat}(V, \alpha)$ from Theorem~\ref{thm: ECH of  solid torus is Z} (4), we obtain
$$E^1_{p}({\mathcal G})\simeq \left \{ \begin{array}{ll}
 {\mathcal R}[e', h',h] \otimes ECH^{\flat}(N, \alpha) & \text{when } p=0, \\
0 & \text{when } p>0. \end{array} \right.$$
Then $E^1_p({\mathcal F}) =0$ for $p>0$ and standard properties of spectral
sequences immediately imply the following lemma.

\begin{lemma}\label{lemma: homology is E^1}
There is an isomorphism $E^1_0({\mathcal F}) \simeq ECH(M)$ which is induced by
the direct limit of the inclusion maps $E^0_0({\mathcal F}_i) \hookrightarrow ECC^{L_i}(M,
\alpha_i')$.
\end{lemma}

The differential $\partial_{01}$ on $E^1({\mathcal G})$ is induced by the components of $\partial_0$
between consecutive filtration levels. By Proposition \ref{foliation on V} and Lemma \ref{fef constrains curves}, the only
$J$-holomorphic map in $\R \times V$ with an end at $h'$ is a disk in the foliation
${\mathcal Z}_1$, which has ECH index $I=1$. Therefore $\partial'_V(h'(e')^i)= (e')^i$.
Then the differential $\partial_{01}$ on $E^1_{0, \bullet}({\mathcal G})$ is described by
the following commutative diagram:
\begin{equation} \label{boundary partial_01} \xymatrix{
h'  {\mathcal R}[e'] \otimes ECH^{\flat}(N, \alpha) \ar[rrr]^{1 \otimes h\partial'_N + \cdot / e' \otimes h}
\ar[d]_{\cdot / h' \otimes (1+e)} & & &
h'  {\mathcal R}[e'] \otimes hECH^{\flat}(N, \alpha)\ar[d]^{\cdot / h' \otimes (1+e)} \\
 {\mathcal R}[e'] \otimes ECH^{\flat}(N, \alpha) \ar[rrr]^{1 \otimes h\partial_N' + \cdot / e' \otimes h} & &
&   {\mathcal R}[e'] \otimes hECH^{\flat}(N, \alpha).} \end{equation}

\subsubsection{Homological algebra lemma}

The following elementary lemma in homological algebra will be used in the proof of Theorem~\ref{cafepi}.

\begin{lemma} \label{montreal}
Let $A$ be an abelian group and $f,g:A\to A$ commuting morphisms. Consider the chain complex
$$ {\mathcal C}_\bullet = \left( 0 \longrightarrow C_2 \stackrel{\partial_2}
\longrightarrow C_1 \stackrel{\partial_1} \longrightarrow C_0 \longrightarrow 0 \right )=
\left (0 \longrightarrow A \stackrel{\binom{f}{g}} \longrightarrow A^2 \xrightarrow{(g,-f)} A
\longrightarrow 0 \right ).$$
If $f$ has a right inverse $s:A\to A$ (i.e., $f \circ s=id$) such that $g \circ s=s\circ g$, then
$$H_2({\mathcal C}_\bullet) \simeq \ker f \cap \ker g, \quad H_1({\mathcal C}_\bullet)
\simeq {\ker f}/{g(\ker f)}, \quad H_0 ({\mathcal C}_\bullet)=0.$$
\end{lemma}

\begin{proof}
$H_2({\mathcal C}_\bullet) \simeq \ker f \cap \ker g$ is immediate and $H_0 ({\mathcal C}_\bullet)=0$ follows from the surjectivity of $f$.

Next consider $H_1(\mathcal{C}_\bullet)$. By definition, $\ker \bdry_1=\{(x,y)\in A^2~|~ g(x)=f(y)\}$ and $\op{Im}(\bdry_2)= \{(f(z),g(z))\in A^2~|~ z\in A\}$. If we define the map
$$\phi: A\to A^2, \quad x\mapsto (x,g\circ s(x))=(f\circ s(x),g\circ s(x)),$$
then we can write $\op{Im}(\bdry_2)=\op{Im}(\phi)\oplus g(\ker f)$ and $\ker(\bdry_1)= \op{Im}(\phi)\oplus \ker f$.  The details are left to the reader. Hence $H_1({\mathcal C}_\bullet)
\simeq {\ker f}/{g(\ker f)}$.
\end{proof}

\subsubsection{Completion of proof of Theorem~\ref{thm: equivalence of  ECHs}(1)}

We use a comparison theorem for spectral sequences (e.g., \cite[Exercise~A3.41]{eisenbud}) to prove Theorem~\ref{cafepi}, establishing Theorem~\ref{thm: equivalence of ECHs}(1).

\begin{thm} \label{cafepi}
The map $\sigma_*: ECH(N, \partial N, \alpha) \to ECH(M)$ is an isomorphism.
\end{thm}

\begin{proof}
Since $\sigma_k$ takes values in the lowest level of the filtration ${\mathcal F}_k$, $\sigma_*$ factors through the map
$$\sigma^1 : ECH(N, \partial N, \alpha) \simeq ECH^\natural(N,\alpha) \to E^1_0({\mathcal F}).$$
By Lemma~\ref{lemma: homology is E^1} it suffices to show that $\sigma^1$ is an isomorphism.

Recall the filtration $\mathcal{G}$ on $E^0(\mathcal{F})$ from Section~\ref{subsub: filtration G}. On $ECC^\natural(N, \alpha)$ we define an analogous filtration ${\mathcal G}^\natural$ such that
$${\mathcal G}^\natural(\gamma \otimes \Gamma)= \left \{ \begin{array}{l} 2 \quad
\text{if } \gamma=h',  \text{ and } \\
1 \quad \text{if } \gamma= \varnothing.
\end{array} \right.$$
This filtration induces a spectral sequence $E^r({\mathcal G}^\natural)$ such that $E^1_q ({\mathcal G}^\natural) \simeq ECH^\flat(N, \alpha)$ for $q= 1,2$ and $d_1$ is the multiplication by $(1+e)$. This is simply a reformulation of Exact Triangle~\eqref{eqn: exact triangle e+1} in the language of spectral sequences. The map $\sigma^0$ is compatible with the filtrations ${\mathcal G}^\natural$ and ${\mathcal G}$ and induces a map
$$\overline{\sigma} : E^1({\mathcal G}^\natural) \to E^1({\mathcal G}).$$

We now compute the homology of $(E^1({\mathcal G}), \partial_{01})$ using Lemma~\ref{montreal}. We set
$$A=  {\mathcal R}[e']\otimes ECH^{\flat}(N, \alpha),\quad f = 1 \otimes \partial_N' + \cdot / e' \otimes 1,\quad \mbox{and} \quad g=1 \otimes (1+e),$$
where $fg=gf$ by Diagram~\eqref{boundary partial_01}. Define the map
$$s :  {\mathcal R}[e']\otimes ECH^\flat(N, \alpha) \to  {\mathcal R}[e'] \otimes ECH^\flat(N, \alpha),$$
$$(e')^k\otimes \Gamma\mapsto (e')^k \sum_{i=1}^{\infty} (e')^i \otimes (\partial_N')^{i-1} \Gamma,$$
where $\Gamma$ denotes an element of $ECH^{\flat}(N, \alpha)$ and not an orbit set as usual. Then $s$ is well-defined since $\partial'_N$ is nilpotent. Moreover $fs=id$ and $gs=sg$. Then $E^2_{00}({\mathcal G})= E^2_{02}({\mathcal G})=0$ because the map $g$ is injective. Next consider $E^2_{01}({\mathcal G}) = {\ker f}/g(\ker f)$. An element of $\ker f$ has the form
$$(e')^n \otimes \Gamma_n + (e')^{n-1} \otimes \Gamma_{n-1} + \dots + 1 \otimes \Gamma_0,$$
where $\Gamma_i\in ECH^\flat(N, \alpha)$ and $\Gamma_{i+1} = \partial_N' \Gamma_i$, $i=0,1,\dots$. Hence the map
$$\overline{\sigma}: ECH^{\flat}(N, \alpha) \to  {\mathcal R}[e'] \otimes ECH^{\flat}(N, \alpha),$$
$$\Gamma\mapsto \sum_{i=0}^\infty (e')^i\otimes (\bdry_N')^i\Gamma,$$
is an isomorphism with $\ker f$.  The diagram
$$\xymatrix{
ECH^\flat(N, \alpha) \ar[r]^-{\overline{\sigma}} \ar[d]_{\cdot (1+e)} & \ker f
\ar[d]^{\cdot (1+e)=g} \\ ECH^\flat(N, \alpha) \ar[r]^-{\overline{\sigma}} &
\ker f}$$
commutes because $\partial_N'(e \Gamma) = e \partial'_N(\Gamma)$ for all $\Gamma
\in ECH^{\flat}(N, \alpha)$ by the Trapping Lemma.
Hence $\overline{\sigma}$ induces an isomorphism
$$E^2({\mathcal G}^\natural) \simeq ECH^\flat(N, \alpha)/ (\Gamma+ e \Gamma)\stackrel\sim\to E^2({\mathcal G}) \simeq {\ker f}/{g(\ker f)}.$$
By the comparison theorem for spectral sequences, $\sigma^1$ is an isomorphism. This completes the proof of Theorem~\ref{cafepi}.
\end{proof}

\subsection{The $U$-map}\label{subsec: U map}

In this subsection we prove that $\sigma_*$ intertwines the map $U$ on $ECH(M)$ with the map induced by $\partial'_N$ on $ECH(N, \partial N, \alpha)$. This will allow us to deduce Theorem \ref{thm: equivalence of ECHs}(2) from algebraic considerations. Let $L_k$ and $L_k'$ be as in Section~\ref{subsection: explicit map}.

We define the map
$$U^\natural : ECC^\natural(N, \alpha) \to ECC^\natural(N, \alpha),$$
$$\gamma \otimes \Gamma\mapsto \gamma \otimes \partial_N' \Gamma.$$
Since $U^\natural(ECC_{\le k}^{\natural, L_k'}(N, \alpha)) \subseteq ECC_{\le k}^{\natural, L_k'}(N, \alpha)$, we can define
$$U^{\natural}_k : ECC_{\le k}^{\natural, L_k'}(N, \alpha) \to ECC_{\le k}^{\natural, L_k'}(N, \alpha)$$
as the restriction of $U^\natural$ to $ECC_{\le k}^{\natural, L_k'}(N, \alpha)$.

We also define the chain complex $$\widehat{ECC}^\natural (N, \alpha) =  {\mathcal R}[h'] \otimes ECC(N, \alpha)$$ with differential
$$\widehat{\partial}^\natural (\gamma \otimes \Gamma) =\gamma \otimes \partial_N \Gamma + \gamma /h' \otimes (1+e) \Gamma.$$

The following lemma is similar to Lemma \ref{lemma: uqam} and its proof will be omitted.

\begin{lemma}\label{aeroport}
$\widehat{ECH}^\natural (N, \alpha) \simeq \widehat{ECH}(N, \partial N, \alpha)$.
\end{lemma}

The decomposition of the differential $\partial_N$ described in Equation~\eqref{eqn: decomposition of boundary} implies the following lemma.

\begin{lemma}\label{trudeau}
$\widehat{ECC}^{\natural}(N, \alpha)$ is isomorphic to the cone of $U^\natural$. If $L'_k \to \infty$ is an increasing sequence and $\widehat{ECC}^{\natural, L'_k}_{\le k}(N, \alpha)$ is the cone of $U^{\natural}_k$, then
$$\lim \limits_{k \to \infty} \widehat{ECC}^{\natural, L'_k}_{\le k}(N, \alpha) \simeq \widehat{ECC}^{\natural}(N, \alpha).$$
\end{lemma}

Let $z$ be a generic point in the interior of $\R\times V$. We denote by $U_k$ the
$U$-map on $ ECC^{L_k} (M, \alpha_k')$ defined with respect to $z$.

\begin{lemma}\label{lemma: computation of U_0}
The map $U_k$ preserves the filtration ${\mathcal F}_k$ for each $k$. On the lowest filtration level, generated by orbit sets $\gamma \otimes \Gamma$ such that $\gamma \in  {\mathcal R}[e',h']$, $U_k$ is given by:
\begin{equation} \label{U in lowest level}
U_k(\gamma\otimes\Gamma)=\gamma/e'\otimes \Gamma.
\end{equation}
\end{lemma}

\begin{proof}
Fix $k$. By Lemma~\ref{lemma: filtration constraint on curves}, the map $U_k$ preserves the filtration ${\mathcal F}_k$. Moreover, by Lemma
\ref{lemma: constraints on ends} (see also Corollary \ref{cor: Morse-Bott is possible}),
curves which contribute to $U_k$ and do not decrease the filtration level do not cross
$\R \times T_i$ (for $i=1,2$).  This implies  that $U_k(\gamma \otimes \Gamma)= U_k(\gamma) \otimes \Gamma$ when $\gamma \in  {\mathcal R}[e',h']$, and $U_k(\gamma)$ counts index $I=2$ curves in $V$ passing through $z$. We will use the ECH index and the Fredholm index to constrain such curves.

Let $u$ be an $I=2$ $J_k'$-holomorphic map in $\R \times V$ with $\gamma_+ = (e')^{a_+}(h')^{b_+}$ at the positive end and $\gamma_- = (e')^{a_-}(h')^{b_-}$ at the negative end; of course $b_\pm \in \{ 0,1 \}$. If we denote by $D_{e'}$ and $D_{h'}$ the meridian disks of $V$ with boundary on $e'$ and $h'$ respectively, and by $Z \in H_2(V, \gamma_+, \gamma_-)$ the relative homology class determined by $u$, we have $Z= (\alpha_+- \alpha_-)[D_{e'}] + (\beta_+ - \beta_-) [D_{h'}]$.

We compute $I(\gamma_+, \gamma_-, Z)$ using Equation \eqref{eqn: ECH index formula}. On $e'$ and $h'$ we consider the trivialization $\tau$ induced by $\partial V$.  The Conley-Zehnder indices are $\mu_\tau((e')^i) = 1$ for $i =1, \ldots, k$ and $\mu_\tau(h')=0$ by Definition \ref{defn: ECH index in Morse-Bott setting}, because they are on a slight perturbation of a positive Morse-Bott torus. The relative Chern class is $c_1(\xi|_{[D_{e'}]}, \tau) = c_1(\xi|_{[D_{h'}]}, \tau)=1$. Putting everything together,
$$I(\gamma_+, \gamma_-, Z) = 2(a_+ - a_-) + (b_+ - b_-).$$
$I(\gamma_+, \gamma_-, Z)=2$ then implies $e_+-a_- =1$ and $b_+ - b_- =0$, because
$b_+ - b_- \in \{ -1,0,1 \}$. We call $b=b_+=b_-$.

Negative ends at $e'$ cannot be contained in $\R \times V$ by the Trapping Lemma \ref{lemma: trapping}. (While the Trapping Lemma was proved for orbits on a Morse-Bott torus, it still holds for $e'$ which is a slight elliptic perturbation.) Therefore $u$ consists of a cover of a trivial cylinder over $e'$ of degree $a_-$, together with a $J_k'$-holomorphic map $u : F \to \R \times V$ with positive asymptotics to $e' (h')^b$, negative asymptotics to $ (h')^b$ and representing the relative homology class $[D_{e'}]$. Since $\op{ind}(u)=2$, the
index formula \eqref{Fredholm index} implies that $\chi(F)=1$. This leaves only two possibilities: either $u$ consists of a Fredholm index $2$ plane which is positively asymptotic to $e'$ together with a trivial cylinder over $h'$, or it consists of a Fredholm index one cylinder from $e'$ to $h'$ together with a  Fredholm index one plane which is positively asymptotic to $h'$. The second configuration cannot pass through a generic point $z$ and therefore has to be discarded. The problem of computing $U_k$ in the lowest filtration level is thus reduced to the count of $J_k'$-holomorphic planes in $\R \times V$ asymptotic to $e'$ and passing through a generic point.

If we degenerate the contact forms $\alpha_k'$ toward the Morse-Bott contact forms $\alpha_k$ and the almost complex structures $J_k'$ toward the almost complex structures $J_k$, the $J_k'$-holomorphic curves described above converge to very nice $J_k$-holomorphic Morse-Bott buildings because the topology of the domain does not allow the creation of branched covers of trivial cylinders (with nonempty branch locus) connected to Morse trajectories. Then by Theorem \ref{thm: Morse-Bott perturbation of moduli spaces}(4) the count of $I=2$ $J_k'$-holomorphic planes on $\R \times V$ which are positively asymptotic to $e'$ and pass through a generic point $z$ is the same as the count of Morse-Bott buildings consisting of a $J_k$-holomorphic plane on $\R \times V$ which passes through a generic point $z$ and is positively asymptotic to an orbit of $\partial V$, augmented by a Reeb trajectory from $e'$ to that orbit.

By Lemma \ref{fef constrains curves}, the principal part of such a Morse-Bott building must be a leaf of the finite energy foliation ${\mathcal Z}_1$. Since there is a unique leaf through any point, this proves that $U_k(\gamma \otimes \Gamma)= \gamma/e' \otimes \Gamma$.
\end{proof}

\begin{cor}\label{sigma commutes with U}
The following diagram commutes for each $k$:
\begin{equation}\label{eqn: sigma commutes with U}
\xymatrix{
ECC^{\natural, L_k'}_{\le k}(N, \alpha) \ar[r]^{\sigma_k} \ar[d]_{U_k^{\natural}} &
ECC^{L_k}(M, \alpha_k') \ar[d]^{U_k} \\
ECC^{\natural, L_k'}_{\le k}(N, \alpha) \ar[r]^{\sigma_k} & ECC^{L_k}(M,\alpha_k').}
\end{equation}
\end{cor}

\begin{proof}
Since $\sigma_k$ takes values in the lowest level of the filtration ${\mathcal F}_k$, we
can use Equation~\eqref{U in lowest level} to compute $U_k \circ \sigma_k$. Then, for
$\gamma \otimes \Gamma \in ECC^{\natural, L_k'}_{\le k}(N, \alpha)$, we have
\begin{align*}
U_k(\sigma_k(\gamma \otimes \Gamma)) & = U_k \left ( \sum \limits_{i=0}^{\infty} (e')^i
\gamma \otimes (\partial_N')^i \Gamma \right ) = \sum \limits_{i=1}^{\infty} (e')^{i-1}
\gamma \otimes (\partial_N')^i \Gamma,\\
\sigma_k(U^\natural_k(\gamma \otimes \Gamma)) & = \sigma_k(\gamma \otimes
\partial_N' \Gamma) =\sum \limits_{i=0}^{\infty} (e')^i
\gamma \otimes (\partial_N')^{i+1} \Gamma.
\end{align*}
Hence $U_k\circ\sigma_k= \sigma_k\circ U^\natural_k$.
\end{proof}

\begin{proof}[Proof of Theorem \ref{thm: equivalence of ECHs}(2)]
By Lemma \ref{trudeau}, Diagram~\eqref{eqn: sigma commutes with U}, and the naturality property of mapping cones, there is a chain map
$$\widehat{\sigma}_k : \widehat{ECH}^{\natural, L_k'}_{\le k} (N, \alpha) \to \widehat{ECC}^{L_k}(M, \alpha_k').$$
for each $k$. Taking homology (with the help of Lemma \ref{aeroport}) and direct limits over $k$, we obtain a map
$$\widehat{\sigma}_* : \widehat{ECH}(N, \partial N, \alpha)\simeq \widehat{ECH}^\natural(N,\alpha) \to \widehat{ECH}(M).$$ This map
fits into the $U$-map exact sequences by properties of mapping cones:
$$\xymatrix{
\ldots \ar[r]^-{U^\natural} \ar[d]_{\sigma_*}& ECH(N, \partial N) \ar[r] \ar[d]_{\sigma_*} &
\widehat{ECH}(N, \partial N) \ar[r] \ar[d]_{\widehat{\sigma}_*} &  ECH(N, \partial N)
\ar[d]_{\sigma_*} \ar[r]^-{U^\natural} & \ldots \ar[d]_{\sigma_*} \\
\ldots \ar[r]^-U & ECH(M) \ar[r] & \widehat{ECH}(M) \ar[r] & ECH(M) \ar[r]^-U & \ldots
}$$
The five lemma then implies that $\widehat{\sigma}_*$ is an isomorphism.  Moreover $\widehat{\sigma}_*$ preserves the decompositions
of $\widehat{ECH}(N, \partial N, \alpha)$ and $\widehat{ECH}(M)$ according to (relative) homology classes.
\end{proof}

 \begin{rmk}\label{rmk: integer coefficients}
Embedded contact homology can be defined over the integers by choosing a coherent orientation system for the moduli spaces. For its definition or construction we refer to
\cite{BM} and \cite[Section~9]{HT2}. Different choices of coherent orientation systems yield isomorphic chain complexes.

All results of this article carry over with integer coefficients, and with the same proofs, if there is a coherent orientation system such that:
\begin{itemize}
\item the holomorphic plane with positive asymptotics at $h'$ and the holomorphic plane with positive asymptotics at $e'$ and passing through a generic point count positively;
\item the holomorphic cylinders from $e'$ to $h$ and from $h'$ to $e$ count positively; and
\item the holomorphic cylinders from $e'$ to $h'$ and from $h$ to $e$ have opposite signs, so that they cancel each other in the differentials.
\end{itemize}
The first two items can be easily obtained by automorphisms of the complexes adjusting the signs of the generators $e', h', e, h$, and the third item follows from
the identification of orientations of moduli spaces of Morse trajectories with orientations of the corresponding moduli spaces of holomorphic maps, as sketched in the first paragraph of the proof of \cite[Lemma7.6]{Bo2}.
\end{rmk}

\section{Applications to sutured ECH}\label{section: sutured applications}

In this section we apply Theorem~\ref{thm: equivalence of ECHs} to sutured ECH.

\subsection{Sutured ECH}

In this subsection we briefly review sutured ECH, referring the reader to the paper~\cite{CGHH} for more details.

A {\em sutured manifold} is a pair $(M, \Gamma)$, where $M$ is a $3$-manifold with boundary and corners, $\Gamma\subset \bdry M$ is a possibly disconnected $1$-manifold,\footnote{In this section $\Gamma$ will denote a suture, not an orbit set.} $N(\Gamma)$ is an annular neighborhood of $\Gamma$, and $\partial M$ admits the following decomposition into two-dimensional strata
$$\partial M = R_+(\Gamma) \cup R_-(\Gamma) \cup N(\Gamma)$$
as in \cite[Definition~2.7]{CGHH}.  Note that our definition does not allow for ``torus sutures" as in Gabai's  original definition~\cite[Definition 2.6]{Ga}.

A {\em sutured contact form $\overline{\alpha}$ on $(M, \Gamma)$}\footnote{We use $\overline{\alpha}$ to denote an unspecified sutured contact form because $\alpha$ is reserved, in Section~\ref{section: proof of theorem equivalence of ECHs}, to the contact form on $N$. Such contact form will appear again later in this section.}  (cf.\ \cite[Definition~2.8]{CGHH}) is, roughly speaking, a contact form $\overline{\alpha}$ on $M$ whose Reeb vector field $R_{\overline{\alpha}}$ is positively transverse to $R_+(\Gamma)$, negatively transverse to $R_-(\Gamma)$, and tangent to $N(\Gamma)$, and such that the trajectories of $R_{\overline{\alpha}}|_{N(\Gamma)}$ are arcs from $\partial R_-(\Gamma)$ to $\partial R_+(\Gamma)$. One can easily verify that $(M, \Gamma)$ admits a sutured contact form if and only if it is {\em balanced}, i.e., $\chi(R_+(\Gamma)) = \chi(R_-(\Gamma))$.  A sutured contact manifold $(M,\Gamma, \overline{\alpha})$ admits a {\em completion} $(M^*, \overline{\alpha}^*)$; see \cite[Section~2.4]{CGHH}.

Let  $(M,\Gamma,\overline{\alpha})$  be a sutured contact manifold.  We now describe the {\em sutured ECH group} $ECH(M, \Gamma, \overline{\alpha}, J)$. Its chain group $ECC(M, \Gamma, \overline{\alpha}, J)$\footnote{We will often write $ECC(M, \Gamma, \overline{\alpha})$ and $ECH(M, \Gamma, \overline{\alpha})$ for simplicity.} is generated by orbit sets constructed
from simple Reeb orbits in $int(M)$ and the differential counts ECH index  one $J$-holomorphic maps in the symplectization of $(M^*, \overline{\alpha}^*)$ for an almost complex structure $J$ which is adapted to the symplectization and satisfies Properties (A$_0$)--(A$_2$) from \cite[Section~3.1]{CGHH}. Almost complex structures of this type are said to be {\em
tailored to $(M, \Gamma, \overline{\alpha})$}.

Completions are not necessary in dimension three by the following lemma:

\begin{lemma}\label{lemma: curves in int}
Let $J$ be tailored to $(M,\Gamma,\overline{\alpha})$. Then all $J$-holomorphic curves in $(M^*, \overline{\alpha}^*)$ which are asymptotic to closed Reeb orbits in $int(M)$ are contained in $\R \times int(M)$.
\end{lemma}

\begin{proof}
This follows from the proofs of \cite[Lemma~5.6]{CGHH} and \cite[Corollary~5.7]{CGHH}, and relies on the fact that $R_+(\Gamma)$ and $R_-(\Gamma)$ automatically admit Stein structures.
\end{proof}

We finish this review of sutured ECH by recalling a useful result from \cite{CGHH} and sketching a simpler proof in dimension three.

\begin{defn}[{\cite[Section~9]{CGHH}}]
Let $(M, \Gamma, \overline{\alpha})$ be a sutured contact manifold. An {\em interval-fibered extension} is a contact embedding
$$(M, \Gamma, \overline{\alpha}) \hookrightarrow (M', \Gamma', \overline{\alpha}')$$
such that $M' - int(M) = W \times [0,1]$, where:
\begin{itemize}
\item $W$ is a cobordism from $\Gamma'$ to $\Gamma$, and
\item $\overline{\alpha}'|_{W \times [0,1]}=c dt + \beta$ for a Liouville form $\beta$ on $W$ and $c>0$.
\end{itemize}
\end{defn}

\begin{lemma}[{\cite[Theorem~9.1]{CGHH}}] \label{lemma: i.-f. extensions}
Let $(M, \Gamma, \overline{\alpha}) \hookrightarrow (M', \Gamma', \overline{\alpha}')$ be an
interval-fibered extension. Then there is a canonical isomorphism of chain complexes
between $ECC(M, \Gamma, \overline{\alpha})$ and $ECC(M', \Gamma', \overline{\alpha}')$.
\end{lemma}

\begin{proof}
All closed Reeb orbits in $M'$ are contained in $M$ because all Reeb trajectories in
$M' - int(M)$ go from $R_-(\Gamma')$ to $R_+(\Gamma')$. Moreover,
$J$-holomorphic curves in $\R \times M'$ between orbit sets in $int(M)$ are contained in
$\R \times M$. In fact, if a $J$-holomorphic curve nontrivially intersects $\R \times (M'- M)
= \R \times W \times [0,1]$, then its projection to $W$ is surjective by the positivity of
intersections with the Reeb vector field. This implies that the curve touches $\R \times
\partial M'$, which is impossible by Lemma~\ref{lemma: curves in int}.
\end{proof}

\subsection{Topological invariance of sutured ECH}

In this subsection we pay off a debt from \cite{CGHH},  namely we sketch a proof that sutured ECH depends
only on the sutured manifold and the contact structure. A more detailed proof can be found in~\cite{KS}.
In view of \cite[Conjecture~1.5]{CGHH}, we expect sutured ECH to be independent also of the contact structure.

\begin{lemma}\label{lemma: folding of sutured manifolds}
Let $(M, \Gamma, \overline{\alpha})$ be a sutured contact manifold such that $\Gamma$ is
connected. Then for every $L \gg 0$ we can
embed $(M, \Gamma, \overline{\alpha})$ into a closed contact manifold $(\widetilde{M},
\widetilde{\alpha}^L)$ such that
$$ECH^{L'}(M, \Gamma, \overline{\alpha}) \simeq ECH^{L'}(\widetilde{M}, \widetilde{\alpha}^L)$$
for every $L' \le L$. Moreover $\widetilde{M}$, up to diffeomorphism, depends only on
$(M, \Gamma)$ and if $\overline{\alpha}_0$ and $\overline{\alpha}_1$ define isotopic contact structures on
$(M, \Gamma)$, then $\widetilde{\alpha}^L_0$ and $\widetilde{\alpha}^L_1$ define
isotopic contact structures on $\widetilde{M}$.
\end{lemma}

\begin{proof}
Since $(M, \Gamma)$ is balanced and $\Gamma$ is connected, $R_+(\Gamma)$ and $R_-(\Gamma)$ have the same genus and are diffeomorphic.  We identify $\bdry R_+(\Gamma)$ and $\bdry R_-(\Gamma)$ by  a diffeomorphism $\bdry \hh_0:\bdry R_+(\Gamma)\stackrel\sim \to \bdry R_-(\Gamma)$, which is defined by the Reeb flow on $N(\Gamma)$, and fix a diffeomorphism $\hh_0 : R_+(\Gamma) \stackrel\sim\to R_-(\Gamma)$ which extends $\bdry \hh_0$. Let us write $\beta_+= \overline{\alpha}|_{R_+(\Gamma)}$ and $\beta_-= \overline{\alpha}|_{R_-(\Gamma)}$. Then the contact form $\overline{\alpha}$, on a neighborhood $R_+(\Gamma)\times[1-\epsilon,1]$ or $R_-(\Gamma)\times[-1,-1+\epsilon]$ of $R_\pm(\Gamma)=R_\pm(\Gamma)\times\{\pm 1\}$ with coordinates $(x,t)$, has the form $cdt + \beta_\pm$ for some $c>0$ (see \cite[Definition~2.8]{CGHH}). Here $\epsilon>0$ is small.

By Moser's theorem and Lemma \ref{lemma: giroux}, there is a diffeomorphism $\hh : R_+(\Gamma)\stackrel\sim \to R_-(\Gamma)$ isotopic to $\hh_0$ relative to $\bdry \hh_0$, such that $\hh^* \beta_- - \beta_+ = df$ for some function $f: R_+(\Gamma) \to \R$ which is constant near $\bdry R_+(\Gamma)$.

Let us write $R = R_+(\Gamma)$. By repeating the proof of Lemma~\ref{lemma:construction}, we construct a contact form $f_tdt + \beta_t$ on $R \times [1,2]$ such that $f_t>0$,  $f_tdt + \beta_t = c dt + \beta_+$ on $R \times [1,1+ \epsilon]$, and $f_tdt + \beta_t = c dt + \hh^* \beta_-$ on $R \times [2 - \epsilon, 2]$. Pick a bump function $\varphi : [1,2] \to [1,2]$ and consider the contact forms $(f_t + C_L \varphi(t)) dt + \beta_t$ on $R \times [1,2]$ for some large positive constant $C_L$ to be determined later.

We obtain the manifold $M'$ by gluing $R \times \{ 1 \}$ to $R_+(\Gamma)$ by the identity and $R \times \{ 2 \}$ to $R_-(\Gamma)$ by $\hh$. The contact forms $\overline{\alpha}$ on $M$ and $(f_t + C_L \varphi(t)) dt + \beta_t$ on $R \times [1,2]$ match near the gluing region, so they define a contact form on $M'$. Finally we obtain $\widetilde{M}$ by gluing a solid torus $V$ to $M'$ along the boundary, so that a meridian of the solid torus is identified with a Reeb orbit on $\partial M'$. The contact form on $M'$ can be extended to a contact form $\widetilde{\alpha}^L$ on $\widetilde{M}$ by taking the contact form on $V$ as in Example~\ref{esempio utile}.

By taking $C_L$ sufficiently large, we ensure that Reeb trajectories from $R_+(\Gamma)$ to $R_-(\Gamma)$ and closed Reeb orbits in $V$ have action larger than $L$; for Reeb orbits in $V$ this is a simpler application of the arguments in the proof of Lemma~\ref{claim: extension of contact form wo orbits}. Hence $ECC^{L'}(M, \Gamma, \overline{\alpha}) = ECC^{L'}(\widetilde{M}, \widetilde{\alpha}^L)$ as abelian groups if $L' \le L$. Any tailored almost complex structure $J$ on $\R \times M$ can be extended to an almost complex structure $J$ on $\R \times \widetilde{M}$ which is adapted to the symplectization of $\widetilde{\alpha}^L$.

Next we claim that a $J$-holomorphic map $u: F\to \R \times \widetilde{M}$ which is asymptotic to orbit sets in $M$ has image in $\R \times M$. These orbit sets have trivial linking number with the core of $V$, so $\op{Im}(u)\subset \R \times M'$ by the Blocking Lemma. On the other hand, $\op{Im}(u)\cap (\R \times R \times [1,2])=\varnothing$:  Observe that $R_\pm(\Gamma)$ can be lifted to an family $v_s$, $s\in \R$, of $J$-holomorphic maps in $\R\times M'$ which foliate $\R\times R_\pm(\Gamma)$.  By the positivity of intersections, if $u$ intersects some $v_s$, then it intersects all $v_s$.  However $\op{Im}(u_{M'})\cap R_\pm(\Gamma)$ is compact and $u$ cannot intersect $v_s$ for $s\gg 0$, a contradiction.  Hence $\op{Im}(u)\subset \R\times M$.

The remaining claims in the statement are straightforward.
\end{proof}

\begin{thm}\label{invariance of sutured ECH}
Let $\overline{\alpha}_1$ and $\overline{\alpha}_2$ be sutured contact forms on a sutured three-manifold
$(M, \Gamma)$ and let $J_1$ and $J_2$ be almost complex structures on $\R \times M$
such that $J_i$ is tailored to $(M, \Gamma, \overline{\alpha}_i)$ for $i=1,2$. If $\xi_1 = \ker \overline{\alpha}_1$ and
$\xi_2 = \ker \overline{\alpha}_2$ are isotopic through contact structures adapted to the
sutures, then
$$ECH(M, \Gamma, \overline{\alpha}_1, J_1) \simeq ECH(M, \Gamma, \overline{\alpha}_2, J_2).$$
 Moreover this isomorphism preserves the decomposition of the sutured ECH groups as direct sums of subgroups indexed by homology classes in $H_1(M)$.
\end{thm}

\begin{proof}
We may assume that $\Gamma$ is connected, since otherwise we can make $\Gamma$ connected by gluing an
interval-fibered extension, which does not change the sutured ECH groups by Lemma~\ref{lemma: i.-f. extensions}. We extend
$(M, \Gamma, \overline{\alpha}_i)$ to $(\widetilde{M}, \widetilde{\alpha}_i^L)$ as in
Lemma~\ref{lemma: folding of sutured manifolds} and follow the proof of Theorem~\ref{prop: ECH of two contact forms in M} step-by-step.  The statement about the decomposition according homology classes follows from the fact that the isomorphism is supported on holomorphic buildings contained in $\R \times M$ in the sense of Theorem \ref{thm: Hutchings Taubes cobordism map}(i).
\end{proof}

\subsection{Applications}

If $M$ is a closed $3$-manifold and $B \subset M$ is an embedded open $3$-ball, we define the sutured manifold
$$M(1)= (M - B, \Gamma_0),$$
where $\Gamma_0$ is a connected simple closed curve in $\partial (M - B)$. If $K \subset M$ is a knot and $N(K)$ is an open tubular neighborhood of $K$, we define the sutured manifold
$$M(K) = (M -  N(K), \Gamma_K),$$
where $\Gamma_K$ consists of two disjoint copies of a meridian of $K$.  When considering $M(1)$, we will assume that $K \setminus B$ is connected and goes from $R_-(\Gamma)$ to $R_+(\Gamma)$. If $\overline{\alpha}$ is a contact form on $M - B$ or $M - N(K)$ satisfying the conditions in \cite[Definition~2.8]{CGHH}, then the sutured ECH groups $ECH(M(1), \overline{\alpha})$ and $ECH(M(K), \overline{\alpha})$ are defined.

\begin{thm} \label{hat and sutured}
$\widehat{ECH}(M) \simeq ECH(M(1), \overline{\alpha})$.
\end{thm}

This theorem concludes the proof of \cite[Theorem~1.6]{CGHH}.

\begin{proof}
Let $\xi$ be a contact structure on $M$  extending $\overline{\xi}= \ker \overline{\alpha}$ such that $K\subset M$ is a $\xi$-transverse knot.
Recall the decomposition
$$M=N\cup (T^2 \times [1,2]) \cup V$$
from previous sections, where we take $N_0(K)=(T^2 \times [1,2]) \cup V$ to be a neighborhood of $K$.

There exists a sequence of contact forms $\alpha'_i$, $i=0,1,\dots$, for $\xi$  (up to isotopy) and associated Reeb vector fields $R'_i$, satisfying Properties (1)--(8) of Section~\ref{subsec: description}.  Figure~\ref{fig: suture} depicts $R'_i$ on $N_0(K)\simeq D^2(2)\times S^1$ with cylindrical coordinates $(\rho,\phi,\theta)$, where $D^2(\rho_0)=\{\rho\leq \rho_0\}$ and $V\simeq D^2(1)\times S^1$. The Reeb vector field $R'_i$ is $\partial_\phi$-invariant and of the form $R'_i =Y+h_i(\rho)\partial_\phi$, where $Y$ is tangent to the slices $\{ \phi=\mbox{const}\}$ as given in Figure~\ref{fig: suture} and $h_i(\rho)>0$ for $\rho>0$.

Choose almost complex structures $J_i'$ adapted to $\alpha_i'$ as in Section~\ref{subsec: description} so that $J_i'$ is $\bdry_\phi$-invariant on $N_0(K)$ and is close to the almost complex structure $J_0$ from Proposition~\ref{prop: Wendl V} on $V$.

\begin{figure}[ht]
\begin{overpic}[height=4.5cm]{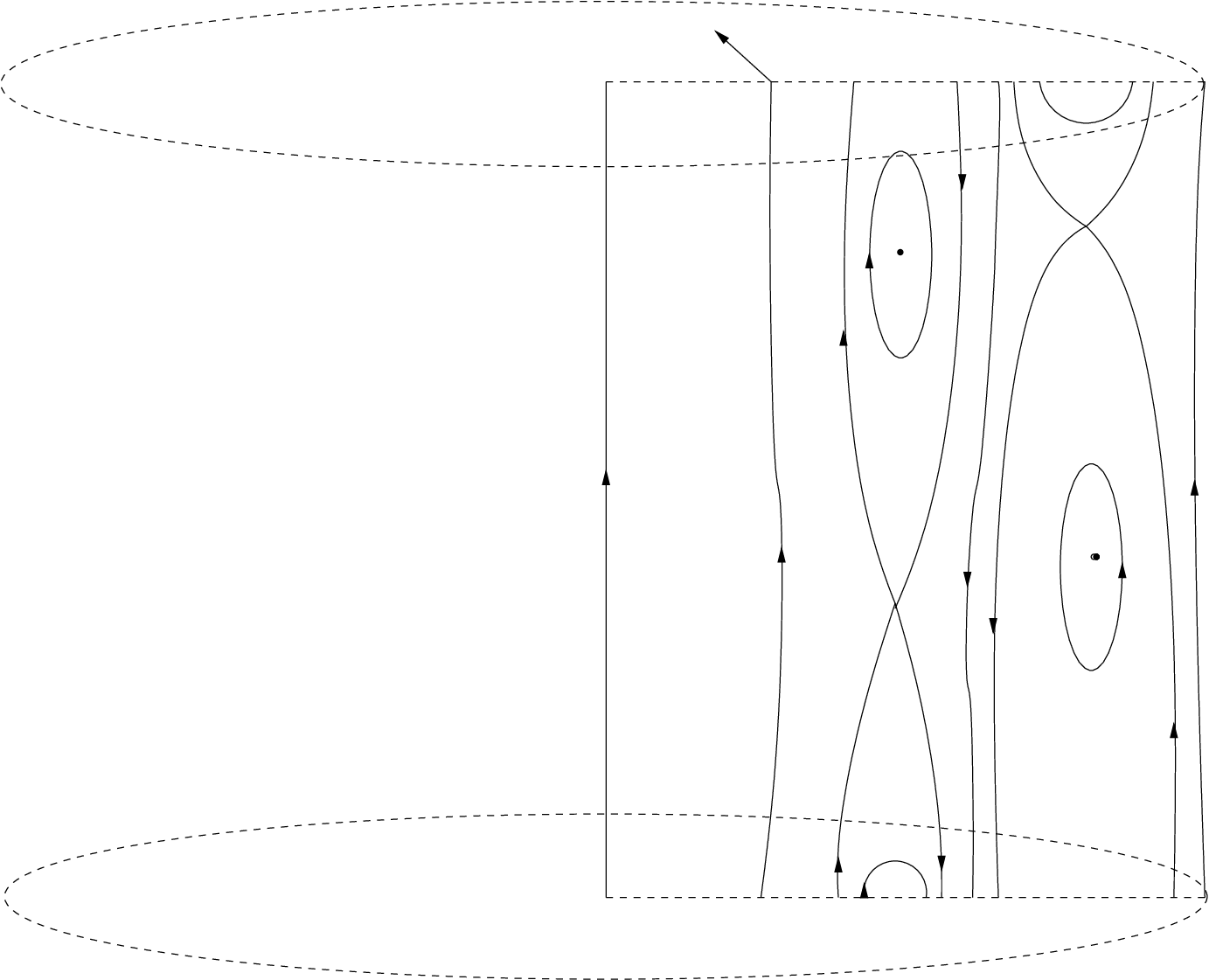}
\put(65.3,59){\tiny $e'$} \put(67.5,30.5) {\tiny $h'$}
\put(92,61){\tiny $h$} \put(84.2,33.6){\tiny $e$}
\put(44,45){\tiny $K$} \put(50,80){\tiny $h_i(\rho)\partial_\phi$}
\end{overpic}
\caption{The Reeb vector field $R'_i$ on $N_0(K)=(T^2 \times [1,2])\cup V$.  The top and the bottom are identified.} \label{fig: suture}
\end{figure}

We describe a concave ball $B$ in $M$ whose complement is $M(1)$; see Figure~\ref{fig: suture1}.
\begin{figure}[ht]
\begin{overpic}[height=4.5cm]{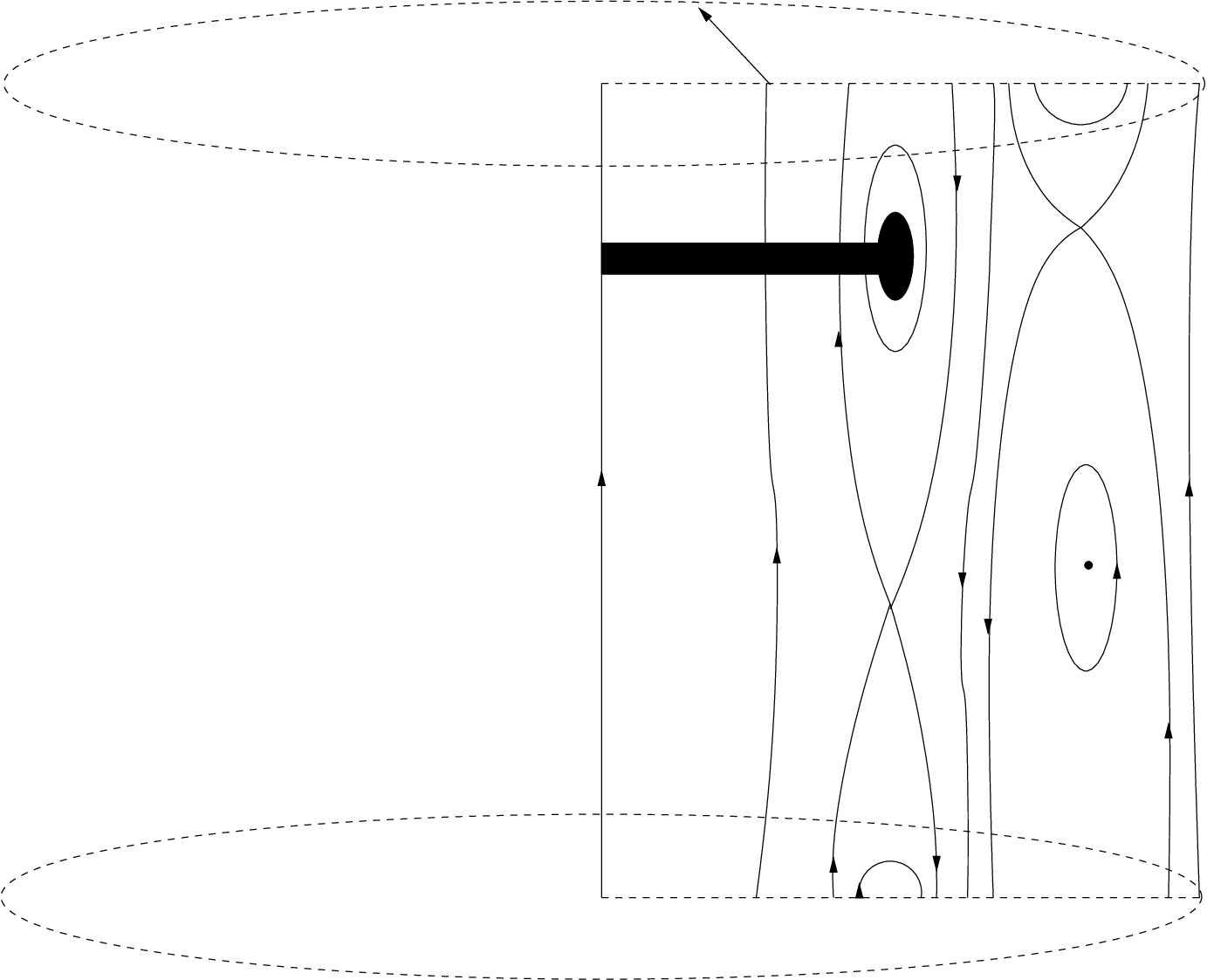}
\put(55,53){\tiny $B$}
\put(67.5,30.5) {\tiny $h'$}
\put(92,61){\tiny $h$} \put(84.2,33.6){\tiny $e$}
\end{overpic}
\caption{The concave ball $B$, obtained by rotating the shaded region about the vertical axis.} \label{fig: suture1}
\end{figure}
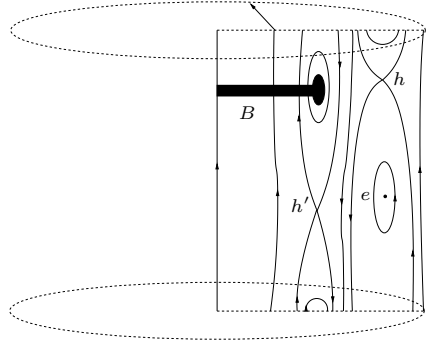
Let $D$ be a meridian disk in $V$ which bounds $e'$ and is the projection to $V$ of an $I=2$ $J_i'$-holomorphic plane $u$ asymptotic to $e'$ at the positive end. The plane $u$ corresponds to a leaf of the finite energy foliation $\mathcal{Z}_0$ of $\R\times V$ from Proposition~\ref{prop: Wendl V}. Let $N(e')$ be a neighborhood of $e'$ whose boundary is tangent to $R'_i$. We then set $B= N(D) \cup N(e')$, where $N(D)$ is a small neighborhood of $D$, chosen such that $\bdry B$ decomposes into three parts:
\begin{itemize}
\item two disks $R_\pm(\Gamma_0)$ transverse to $R'_i$ that are parallel copies of a small retract of $D$; and
\item an annulus $N(\Gamma_0 )\subset \partial N(e')$ tangent to $R_i$.
\end{itemize}
We assume that the $I=1$ $J_i'$-holomorphic plane asymptotic to $h'$ has image in $\R\times (V- B)$ and that $R_\pm(\Gamma_0)$ are also chosen to be restrictions of projections to $M$ of $I=2$ $J_i'$-holomorphic planes asymptotic to $e'$. The trajectories of $R_i$ flow from one boundary component of $N(\Gamma_0 )$ to the other.

The manifold $(M(1) , \Gamma_0 ,\alpha'_i )$ is a sutured contact manifold  and, by Theorem~\ref{invariance of sutured ECH}, $ECH(M(1) ,
\Gamma_0 ,\alpha'_i )$ is isomorphic to $ECH(M(1) , \Gamma_0 , \overline{\alpha})$. By construction, the orbit $e'$ does not belong to $M(1)$ and all the orbits in $V$ are now chords from $\partial M(1)$ to $\partial M(1)$. The Reeb orbits of $R'_i$ that are contained in $M(1)$ are:
\begin{enumerate}
\item all Reeb orbits in $N$;
\item $e$, $h$ and $h'$; and
\item orbits longer than $L_i$ in the no man's land.
\end{enumerate}

By taking direct limits as in Section~\ref{subsection: direct limit}, we can discard orbits in the no man's land. The use of direct limits in this context is justified by Theorem~\ref{invariance of sutured ECH}.

By our choice of $J_i'$, if $u$ is a holomorphic curve in $\R\times M$ between orbit sets constructed from orbits of type (1) and (2) in $M(1)$, then $\op{Im}(u)\subset \R\times M(1)$.  (The orbits of type (1) and (2) have the lowest $\mathcal{F}_i$-filtration level and we can use the Blocking and Trapping Lemmas.) In particular, there are exactly two $I=1$ curves that limit to $h'$ at the positive end, as it is in $\R \times N_0(K)$: one plane from $h'$ to $\varnothing$ and one cylinder from $h'$ to $e$. Therefore we obtain an identification
$$\lim \limits_{i \to \infty} (ECC^{L_i} (M(1) ,\Gamma_0 ,\alpha'_i ),\partial )\simeq (\widehat{ECC}^{\natural} (N,\alpha ), \widehat{\partial}^\natural ),$$
which in view of Lemma~\ref{aeroport} and Theorem~\ref{thm: equivalence of ECHs}(2) implies the theorem.
\end{proof}

If the contact form $\overline{\alpha}$ is chosen carefully, a null-homologous knot $K \subset M$ induces a filtration on the chain complex $ECC(M(1), \overline{\alpha})$ and the associated graded group is $ECC(M(K), \overline{\alpha})$. This construction was described in \cite[Section~7.2]{CGHH}. If $N = M - N_0(K)$ as above, there is a filtration ${\mathcal E}$ on $\widehat{ECC}(N, \partial N, \alpha)$ defined as follows: Let ${\mathcal P}$ be the set of simple Reeb orbits in $int(N)$. The generators of $\widehat{ECC}(N, \partial N, \alpha)$ are equivalence classes of orbit sets $\Gamma$ constructed from ${\mathcal P} \cup \{ h, e \}$, up to the equivalence relation $\Gamma \sim e \Gamma$. To the equivalence class of $\Gamma$ we can uniquely associate an orbit set $\Gamma'$ constructed from ${\mathcal P} \cup \{ h \}$. Then we define ${\mathcal E}(\Gamma)$ as  the algebraic intersection of $\Gamma'$ with a Seifert surface of $K$. The differential of $\widehat{ECC}(N, \partial N, \alpha)$ preserves ${\mathcal E}$ by the Trapping Lemma and it is easy to identify the graded group of this filtration with $ECC^{\sharp}(N, \alpha)$.

\begin{thm}
If $K \subset M$ is a null-homologous knot, then there is a contact form $\overline{\alpha}$ on $M$ for which the isomorphism in Theorem \ref{hat and sutured} preserves the filtrations and induces an isomorphism
$$ECH(M(K), \overline{\alpha}) \simeq ECH^{\sharp}(N, \alpha).$$
\end{thm}

\begin{proof}
Let $K \subset M$ be a null-homologous knot and $\Sigma$ a genus-minimizing Seifert surface for $K$. Following \cite{CH}, we construct a family of contact forms $\alpha'_i$ on $M$ as in the proof of Theorem~\ref{hat and sutured} on $N_0 (K)$, with the additional property that the Reeb vector fields $R_i'$ are positively transverse to $int(\Sigma)$. The construction is done in two steps: first on $N$ by a direct application of \cite{CH}, where we use $\Sigma$ as the first decomposing surface of a taut sutured hierarchy of $N$, and then on $N_0(K)$, where we extend the form by the explicit model already described in Section~\ref{subsubsection: Extension to M}.

We obtain a concave neighborhood $(N(K),\Gamma_K )$ of $K$ by taking
$N(K)=B \cup N_\epsilon (K)$, where $N_\epsilon (K)$ is a very small neighborhood of $K$
whose boundary is tangent to $R'_i$, as in Figure~\ref{fig: suture2}, and $B$ is the ball
constructed in the proof of Theorem~\ref{hat and sutured}.

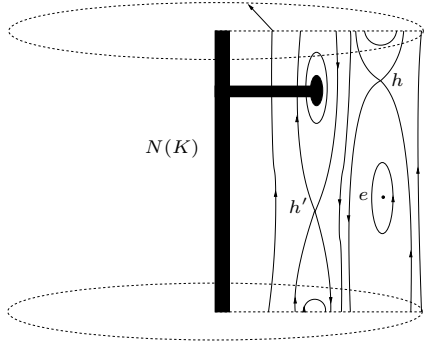
\begin{figure}[ht]
\begin{overpic}[height=4.5cm]{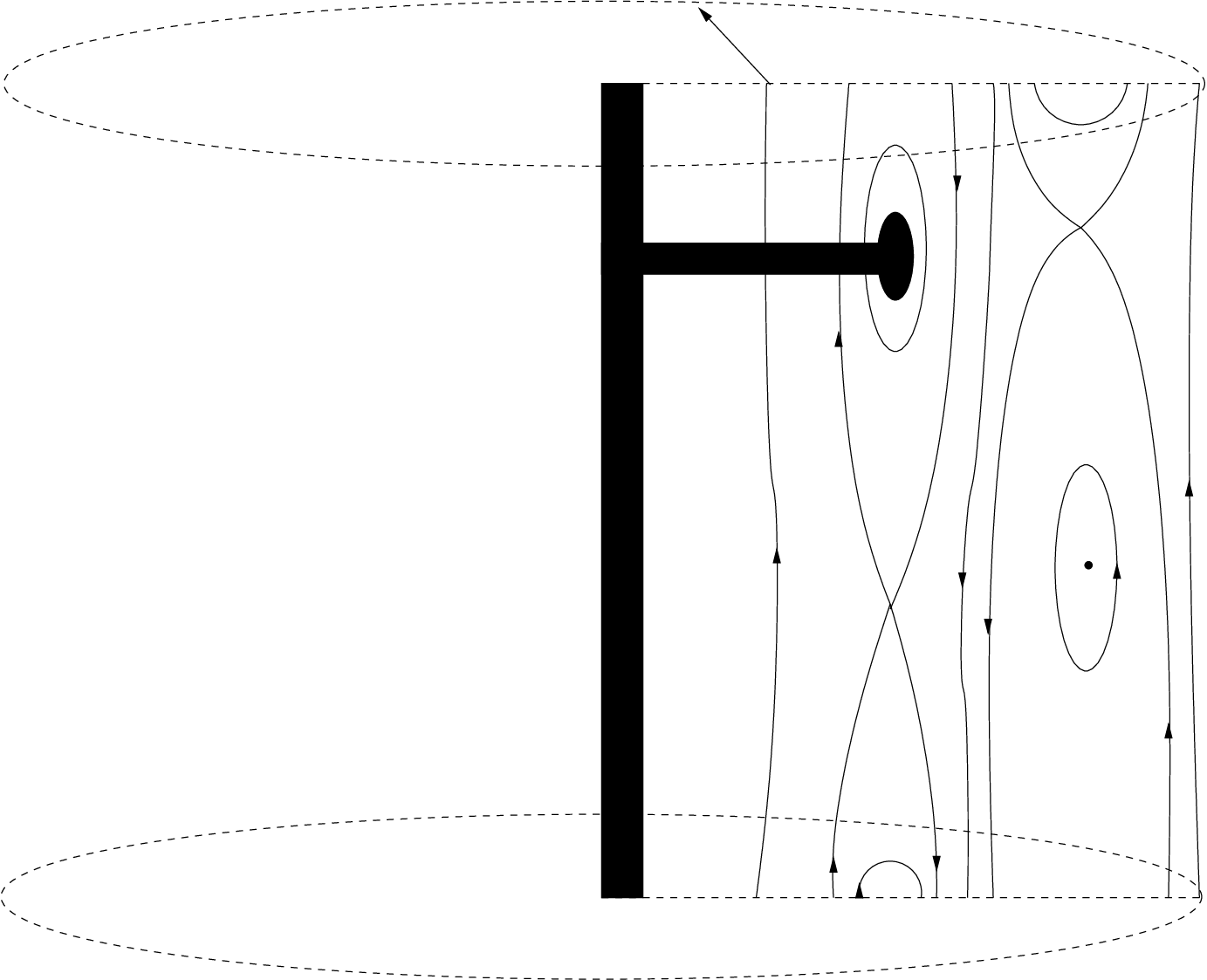}
\put(33,45){\tiny $N(K)$}
\put(67.5,30.5) {\tiny $h'$}
\put(92,61){\tiny $h$} \put(84.2,33.6){\tiny $e$}
\end{overpic}
\caption{Construction of the concave neighborhood  $(N(K),\Gamma_K )$, obtained by rotating the shaded region about the vertical axis} \label{fig: suture2}
\end{figure}

The suture $\Gamma_K$ corresponds to the core curves of the two annuli in $\partial N(K)$ tangent to $R'_i$. At this point, $(M- N(K), \Gamma_K )$ is not yet a convex sutured manifold, because $\partial N(K)$ is not convex for the dividing set given by the two curves of $\Gamma_K$. In fact, on the component $A$ of $\partial N(K)$ coming from $N_\epsilon (K)$, $\ker \alpha'_i|_A$ is negatively transverse to the core of $A$ (oriented as the boundary of $R_+(\Gamma_K)$). To correct this, we glue a collar of the form $(A\times [a,b], dt+ f(y)dx)$, ${\bdry f\over \bdry y}<0$, to $(M- N(K), \alpha'_i )$ along $A=A\times \{ a\}$, where $A\times [a,b]=[0,1]\times S^1\times[a,b]$ has coordinates $(t,x,y)$. Then the Reeb vector field remains $\bdry_t$ while the contact plane rotates until $\ker\alpha_i'|_{A\times\{b\}}$ is positively transverse to the core of $A$.

The positive transversality of the Reeb vector fields with the Seifert surface $\Sigma$ ensures that the isomorphism of Theorem~\ref{hat and sutured} preserves the filtrations
given by the linking number with $K$.

Passing from $M(1)$ to $M(K)$ has the effect of killing the ``meridian'' holomorphic disk from $h'$ which passes through $\R\times K$. After passing to direct limits, we obtain the desired isomorphism.
\end{proof}

\newpage
\appendix

\section{Morse-Bott gluing} \label{section: MB gluing}

\centerline{Vincent Colin, Paolo Ghiggini, Ko Honda, and Yuan Yao\footnote{YY address: University of California, Berkeley, Berkeley, CA 94720-3840. Email: \texttt{yuan\_yao@berkeley.edu}}
\footnote{YY acknowledges the support of the Natural Sciences and Engineering Research Council of Canada (NSERC), PGSD3-532405-2019.
Cette recherche a \'et\'e financ\'ee par le Conseil de recherches en sciences naturelles et en g\'enie du Canada (CRSNG), PGSD3-532405-2019.}
 \footnote{YY would like to thank his advisor Michael Hutchings for constant support. He would also like to thank Alexandru Oancea and Katrin Wehrheim for helpful discussions.}}

\s
The goal of this appendix is to prove Parts (2) and (3) of Theorem~\ref{thm: Morse-Bott perturbation of moduli spaces}. The proof of Part (4) is similar and will be omitted.  The proof involves working out Morse-Bott gluing in a special case, which easily generalizes to one-level cascades in ECH.  In \cite{Yao1, Yao2} the fourth author will prove the general ECH Morse-Bott gluing theorem in the presence of Morse-Bott tori and multiple-level cascades. There are slight differences in packaging, but our strategy and the one from \cite{Yao1, Yao2} for $1$-level cascades are essentially equivalent.

For simplicity we assume there is only one Morse-Bott torus $T_{\mathcal{N}}$ and that it is a negative Morse-Bott torus. It is generally acknowledged that the proof of Morse-Bott gluing in \cite{Bo2} is incomplete, but instead of fixing \cite{Bo2}, we carry out a different pregluing with a smaller error term. At first we will use a stable Hamiltonian structure whose hyperplane distribution is integrable near the Morse-Bott torus to simplify the gluing estimates in various ways. {\em In Section~\ref{subsection: reduction} we will explain how to derive a similar statement for contact structures from Theorem~\ref{thm: main theorem of appendix}.}

\subsection{Stable Hamiltonian structures, almost complex structures, and moduli spaces} \label{subsection: notation}

Let $[-1,1]\times T^2 =[-1,1]\times (\R^2/\Z^2)$ be a neighborhood of the negative Morse-Bott torus $T_{\mathcal{N}}$ with coordinates $(y,(\theta,t))$ such that $T_{\mathcal{N}}=\{0\}\times T^2$, and let $\mathcal{N}$ be the Morse-Bott family of simple orbits of the form $\{y=0,\theta=\mbox{const}\}$. Also let $A_{y_0}=[-y_0,y_0]\times \R/\Z$ be an annulus with coordinates $(y,\theta)$.

\s\n
{\em Morse-Bott perturbation of the stable Hamiltonian vector field.}
The construction will depend on parameters $c, a, b_0, b, \epsilon$  which will be made more specific during the course of this appendix and when we make specific choices they will be indicated by ($\boldsymbol{\dagger}_0$)--($\boldsymbol{\dagger}_3$). The parameters $c$, $a$ and $b_0$ (chosen in this order) will describe the data of the problem and will be chosen once and for all at the beginning so that  they satisfy 
$$0 < 4b_0< a < c <1.$$ 
The constant $c$ depends on the action level, $a$ depends on the Morse-Bott moduli spaces we want to glue { --- morally speaking it determines the region where the first nonconstant term in the Fourier expansion of the negative end is not dominated by the higher order terms; see ($\boldsymbol{\dagger}_1$) --- , and $b_0$ is arbitrary, as long as it is sufficiently smaller than $a$.  The perturbation of the Morse-Bott Reeb vector field  and pregluing will  depend on the parameters $b\in (0, b_0/2)$ and $\epsilon>0$ (chosen in this order). The parameter $b$ will determine the support of the perturbation and the parameter $\epsilon$ the size. Then, by the usual contraction mapping argument, 
 we will prove that for every $b$ sufficiently small and every $\epsilon$ sufficiently small compared to $b$, the preglued curve can be deformed to a holomorphic curve.
\s

On $[-c,c]\times T^2$ consider the stable Hamiltonian structure consisting of the $1$-form $dt$ and the $2$-form $\omega_H=dH\wedge dt+dy\wedge d\theta$, where $H: A_c\to \R$ is a function of $(y,\theta)$ (and is independent of $t$). The stable Hamiltonian vector field $R_H$ is then 
\begin{equation}\label{Hamilton equation}
R_H=\tfrac{\bdry}{\bdry t}+ X_H,\quad \mbox{where} \quad i_{X_H} dy\wedge d\theta=dH.
\end{equation}

Let $J_H$ be the adapted almost complex structure on $\R\times [-c,c]\times T^2$ which sends ${\bdry\over \bdry s}\mapsto R_H$, $R_H\mapsto -{\bdry\over \bdry s}$, ${\bdry\over \bdry y}\mapsto {\bdry\over \bdry \theta}$, and ${\bdry\over \bdry \theta}\mapsto -{\bdry\over \bdry y}$, where $s$ is the $\R$-coordinate.

We specialize the smooth function $H$ to:
\begin{equation}\label{eqn: f and f epsilon}
f(y,\theta)=\tfrac{1}{2} y^2 \quad \mbox{ or }\quad f_\epsilon(y,\theta)=\tfrac{1}{2} y^2+\epsilon\phi(y)\overline{g}_{\mathcal{N}}(\theta),
\end{equation}
where $\epsilon>0$ is small, the domain of $\overline{g}_{\mathcal{N}}(\theta)$ is $S^1$ viewed as the interval $[-\tfrac{1}{2}, \tfrac{1}{2}]$ with the endpoints identified, and the following hold:
\begin{itemize} 
\item[(P2)] $\overline{g}_{\mathcal{N}}:\R/\Z\to \R$ is a perfect Morse function with maximum at ${1\over 4}$ and minimum at $-{1\over 4}$. More specifically, we assume that $\overline{g}_{\mathcal{N}}'(\theta)=0$ on $\theta=\pm{1\over 4}$, is linear with positive slope on $[-{1\over 4},-{1\over 5}]$, is nondecreasing on $[-{1\over 5},-{1\over 6}]$, and is equal to $1$ on $[-\tfrac{1}{6}, \tfrac{1}{6}]$; and $\overline{g}_{\mathcal{N}}(\theta)$ is an odd function about $\theta=0$.
\item[(P3')] $\phi: [-c,c]\to [0,1]$ is an even function which has support on $[-2b_0,2b_0]$ and is equal to $1$ on $[-b_0 ,b_0]$. 
\end{itemize}
Here (P2) is exactly the same as (P2) from Section~\ref{subsection: Morse-Bott contact forms} and (P3') is a tweaking of (P3). We observe that $f_\epsilon\to f$ in $C^\infty$ as $\epsilon\to 0$.

The torus $T_{\mathcal{N}}$ is a negative Morse-Bott torus with respect to $R_f$. After perturbing to $R_{f_\epsilon}$, the Morse-Bott family of stable Hamiltonian orbits becomes a pair $e$ and $h$ of stable Hamiltonian orbits over $(0,-{1\over 4})$ and $(0,{1\over 4})$ in $A_c$. See Figure~\ref{figure: morse-function}.

\begin{figure}[ht]
\centering
\begin{overpic}[scale=.6]{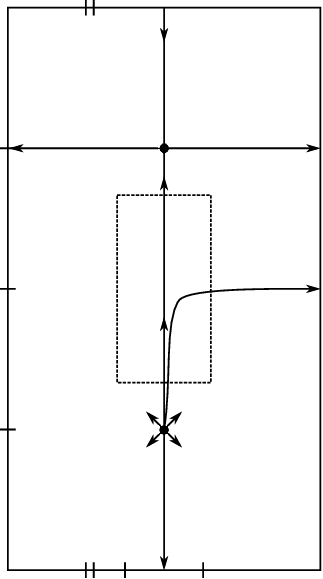}
	\put(-3,49){\tiny $0$}
	\put(-8,73){\tiny $1/4$}
	\put(-8,97){\tiny $1/2$}
	\put(-11.5,.5){\tiny $-1/2$}
	\put(-11.5,24.5){\tiny $-1/4$}
	\put(30.8,24.9){\tiny $e$}
	\put(29.5,70.2){\tiny $h$}
	\put(42,46){\tiny $\mathcal{T}_1^\epsilon$}
	\put(21.7,41){\tiny $\mathcal{T}_0^\epsilon$}
	\put(-1.3,-3.2){\tiny $-c$}
	\put(27.3, -3.2){\tiny $0$}
	\put(54.3,-3.2){\tiny $c$}
	\put(34.2,-3.2){\tiny $b_0$}
	\put(18,-3.2){\tiny $-b_0$}
	\put(15,55){\tiny $B$}
\end{overpic}
\caption{The annulus $A_c=[-c,c]\times \R/\Z$ with some gradient trajectories of $f_\epsilon$. The top and the bottom are identified. The dotted rectangle is the boundary of $B=[-b_0 ,b_0] \times[-\tfrac{1}{6},\tfrac{1}{6}]$, on which $\phi(y)=1$ and $\overline{g}'_{\mathcal{N}}(\theta)=1$.} 
\label{figure: morse-function}
\end{figure}

The Morse-Bott perturbation is performed below a fixed action $L$, which is later sent to infinity by a direct limit process. The action of a stable Hamiltonian orbit in $[-c,c] \times T^2$ depends on how many times it intersects an annulus $A_c \times \{ t \}$ and therefore, instead of working below an action level $L$, we work below an intersection number $N$. 
\be 
\item[($\boldsymbol{\dagger}_0$)] The constant $c$ is chosen so that $0<c<1$ and all closed orbits of $R_f$ in $A_c \times S^1$ that intersect $A_c \times \{ t \}$ at most $N$ times are covers of orbits in $T_{\mathcal N}$. The constant $\epsilon>0$ will always be small enough that all closed orbits of $R_{f_\epsilon}$ in $A_c \times S^1$ that intersect $A_c \times \{ t \}$ at most $N$ times are covers of $e$ and $h$.
\ee
  
The next lemma follows from the explicit constructions in Section~\ref{subsection: reduction} and Claim \ref{claim: stable Hamiltonian properties}.

\begin{lemma} \label{lemma: extension to rest of M} 
There exist stable Hamiltonian structures $(\alpha,\omega),(\alpha,\omega_{\epsilon})$ on $M$ and almost complex structures $J_f, J_{f_\epsilon}$ on  $\R \times M$ such that 
	\be
	\item On $[-a,a] \times T^2$, $(\alpha, \omega)=(dt, \omega_f)$ and $(\alpha, \omega_\epsilon)=(dt, \omega_{f_\epsilon})$.
	\item On $M-([-a,a]\times T^2)$, $\omega=\omega_\epsilon$ is a multiple of $d \alpha$ by a positive function (and therefore $\alpha$ is a contact form).
	\item On $M-( [-c,c] \times T^2)$, $\omega = \omega_\epsilon= d \alpha$. 
	\item $J_f$ and $J_{f_\epsilon}$ are adapted to $(\alpha,\omega_f)$ and $(\alpha,\omega_{f_\epsilon})$ respectively, and $J_f=J_{f_\epsilon}$ outside of $\R \times [-a,a] \times T^2$.
	\ee
\end{lemma}

\begin{simplification}
From now on we will consider only the case $N=1$ because it contains already all the relevant ideas.
\end{simplification}

\s\n
{\em Moduli spaces.} 
Let 
$$\mb:=\mb_{J_f}:=\mathcal{M}_{J_f}^{I=\op{ind}=1}(\boldsymbol{\gamma};\mathcal{N})$$ 
be the moduli space of (finite energy) $J_f$-holomorphic maps $u_+:(\dot F,j)\to \R\times M$ modulo domain automorphisms, where:
\be
\item[(C0)] $(\dot F,j)$ is a closed Riemann surface with a finite number of punctures removed and we are ranging over all complex structures $j$ with a fixed topological type $\dot F$;
\item[(C1)] $u_+$ limits to the orbit set $\boldsymbol{\gamma}$ at the positive end, where $\boldsymbol{\gamma}$ does not involve any orbits of the Morse-Bott family $\mathcal{N}$;
\item[(C2)] $u_+$ limits to some orbit in the Morse-Bott family $\mathcal{N}$ at the negative end; and
\item[(C3)] $u_+$ has ``unconstrained" Fredholm and ECH index $1$ (the negative end is unconstrained); cf.\ Section~\ref{subsub: automatic transversality} for more details.
\ee  
By (C3) we mean that if we concatenate $u_+$ with a cylinder corresponding to an upward gradient trajectory that starts at $(0,-\tfrac{1}{4})$ so that we have a map $C$ from $\gamma$ to $e$, then the Fredholm and ECH indices of $C$ are $1$. (C3) implies that curves of $\mb$ are isolated modulo $\R$-translation and are embedded.   

Next let 
$$\me:=\mathcal{M}_{J_{f_\epsilon}}:=\mathcal{M}^{I=\op{ind}=1}_{J_{f_\epsilon}}(\boldsymbol{\gamma},e)$$ 
be the moduli space of $J_{f_\epsilon}$-holomorphic maps $u:(\dot F,j)\to \R\times M$ modulo domain automorphisms, where (C0), (C1) (with $u$ instead of $u_+$) and the following hold:
\be
\item[(C2')] $u$ limits to the negative elliptic orbit $e$ obtained by perturbing the Morse-Bott family; and
\item[(C3')] $u$ has Fredholm and ECH index $1$. 
\ee

We also remark that the moduli spaces $\mb$ and $\me$ can be made Morse-Bott regular or regular by perturbing $J_f$ and $J_{f_\epsilon}$ outside of $[-c,c] \times T^2$. 

\s\n
{\em Holomorphic curves near the Morse-Bott torus.}

\begin{claim}\label{leftovers}
The equation $\overline\bdry_{J_{f_\epsilon}}u=0$ for a map\footnote{We abuse notation and use coordinates $(s,t)$ for both the cylindrical part of the domain and $\R\times S^1$. We also change the order of the coordinates from $(y, \theta, t)$ to $(t, y, \theta)$. This has no effect on the orientations of $M$ and $A_c$.} 
$$u \colon [\tilde s_0,\tilde s_1]\times S^1\to \R\times S^1_t\times A_c,\quad u(s,t) = (s,t, \eta(s,t)),$$
is equivalent to the equation
\begin{equation} \label{eqn: rewriting d bar equation}
\mathcal{D}_\epsilon\eta:= {\bdry\eta\over \bdry s} + j_0{\bdry\eta\over\bdry t}-\nabla f_\epsilon(\eta)=0,
\end{equation}
where $j_0=\begin{pmatrix} 0 & -1 \\ 1 & 0 \end{pmatrix}$ is the standard almost complex structure on $A_c$. 
\end{claim}

\begin{proof} 
We apply $\overline\bdry_{J_{f_\epsilon}} =\bdry_s + J_{f_\epsilon} \bdry_t$ to $(s,t,\eta(s,t))$ to obtain
\begin{equation}\label{eqn: calculation of d-bar equation for eta}
\begin{pmatrix}
 1 \\ 0 \\ \tfrac{\bdry\eta}{\bdry s}
 \end{pmatrix}
 + J_{f_\epsilon}  
 \begin{pmatrix}
 0 \\ 1 \\ \tfrac{\bdry\eta}{\bdry t}
 \end{pmatrix}
=
\begin{pmatrix}
 1 \\ 0 \\ \tfrac{\bdry\eta}{\bdry s}
 \end{pmatrix}
+ 
\begin{pmatrix}
- 1 \\ 0 \\ j_0 \tfrac{\bdry\eta}{\bdry t} - \nabla f_\epsilon(\eta)
\end{pmatrix}
.
\end{equation}
This is because $J_{f_\epsilon}(\partial_t)= - \partial_s- j_0 X_{f_\epsilon}$ and $j_0 X_{f_\epsilon}= \nabla f_\epsilon$ (recall the sign in Equation \eqref{Hamilton equation}).
Hence 
$$J_{f_\epsilon} \begin{pmatrix}
0 \\ 1 \\ 0
\end{pmatrix}
=
\begin{pmatrix}
-1\\ 0 \\ - \nabla f_\epsilon
\end{pmatrix}.$$
\vskip-.25in
\end{proof}

The claim holds also for $\epsilon=0$: the equation $\overline\bdry_{J_{f_\epsilon}} u=0$
for a map $u(s,t)=(s,t,\eta(s,t))$ as above is equivalent to ${\mathcal D}_0 \eta =0$, where
\[ {\mathcal D}_0 \eta = {\bdry\eta\over \bdry s} + j_0{\bdry\eta\over\bdry t} - \nabla f(\eta). \]

\begin{rmk}
To treat the case $N>1$ we need to consider maps $u \colon [\tilde s_0,\tilde s_1]\times S^1\to \R\times S^1_t\times A_c$ which wind $k$ times around $S^1_t$ for $k \le N$. In that case we should write $u(s,t)= (ks, kt, \eta(s,t))$, but all estimates on $\eta$ remain unchanged.
\end{rmk}

The following easy consequence of Claim \ref{leftovers} provides the link between gradient trajectories and holomorphic curves.  

\begin{lemma}\label{from wendl}
	Every gradient trajectory $\mathcal{T}$ of $f_\epsilon$ (here we are allowing $\epsilon=0$ and $f_0=f$) admits a unique lift to a simply-covered $J_{f_\epsilon}$-holomorphic cylinder $u_{\mathcal{T}}$ whose projection to $A_c$ is $\mathcal{T}$ modulo repara\-me\-trization of the domain and $\R$-translations of $u_{\mathcal{T}}$.
\end{lemma}

\begin{proof}
If $\eta \colon [\tilde s_0, \tilde s_1] \to A_c$ is a parametrization of ${\mathcal T}$ satisfying $\tfrac{d \eta}{ds}=\nabla f_\epsilon (\eta)$, then $u_{\mathcal T}(s,t):=(s,t, \eta(s))$ satisfies $\overline\bdry_{J_{f_\epsilon}}u_{\mathcal T}=0$ by Claim \ref{leftovers}. On the other hand, one can immediately check that a simply-covered map to $\R \times S^1 \times A_c$ that projects to ${\mathcal T}$ must be of the form $(s, t) \mapsto (s,t, \eta(s))$ for some $\eta$ up to reparametrizations and translations. 
\end{proof}

\subsection{Main result}\label{subsection: main result}

The main result of the appendix is the following:

\begin{thm}\label{thm: main theorem of appendix}
If $\mb_{J_f}$ is Morse-Bott regular, then for $a,b_0>0$ sufficiently small there exist:
\begin{itemize}
\item $J_f'$ that agrees with $J_f$ on $[-c,c] \times T^2$ and is arbitrarily close to $J_f$ on $M- ([-c,c]\times T^2)$, 
\item $\epsilon>0$ that is sufficiently small,  and
\item $J'_{f_\epsilon}$ that agrees with $J_{f_\epsilon}$ on $[-c,c]\times T^2$ and with $J_f'$ on $M-([-c,c]\times T^2)$,
\end{itemize}
such that $\mb_{J_f'}$ is Morse-Bott regular, $\mathcal{M}_{J'_{f_\epsilon}}$ is regular, and there is a bijection between $\mb_{J_f'}$ and $\mathcal{M}_{J'_{f_\epsilon}}$.
\end{thm}

\begin{rmk} 
In the case where $\mb_{J_f}$ satisfies (C0), (C1), and the unconstrained end is replaced by a constrained end in (C2) and (C3), i.e., the negative end limits to a hyperbolic orbit after perturbation, we can simply glue in a trivial cylinder at the said end, since having constrained index means not including $\tilde\partial_\theta$ in Equation~\eqref{eqn for D plus} and Morse-Bott gluing then reduces to standard gluing.
\end{rmk}

\s\n
{\em Brief discussion on regularity.} We will not prove that for all $\epsilon>0$ sufficiently small $\me$ is regular if $\mb$ is Morse-Bott regular, although that is true.  It suffices for our purposes to know that ``for some $\epsilon>0$ small and some $J_f'$ and $J_{f_\epsilon}'$, there is a bijection between $\mb_{J_f'}$ which is Morse-Bott regular and $\mathcal{M}_{J_{f_\epsilon}'}$ which is regular."

We will explain the existence of $J_f'$ and $J'_{f_\epsilon}$ such that $\mb_{J_f'}$ is Morse-Bott regular and $\mathcal{M}_{J'_{f_\epsilon}}$ is regular: Since $J_f$ is Morse-Bott regular for $\mb_{J_f}$, the same holds for all $J_f'$ that are sufficiently close to $J_f$ on $M- ([-c,c]\times T^2)$ and agree with $J_f$ on $[-c,c] \times T^2$. Next, we perturb $J_{f_\epsilon}$ to $J'_{f_\epsilon}$ on $M- ([-c,c] \times T^2)$ so that $\mathcal{M}_{J'_{f_\epsilon}}$ is regular. This is possible because the only Reeb orbits of $(\alpha, \omega_{J_\epsilon})$ inside $[-c,c] \times T^2$ come from the perturbation of the Morse-Bott torus, and therefore every holomorphic curve in $\mathcal{M}_{J'_{f_\epsilon}}$ intersects $M- ([-c,c] \times T^2)$, except for the two curves corresponding to the two flow lines on the Morse-Bott family, whose regularity can be easily checked by hand.

Let us fix an $\R$-invariant Riemannian metric on $\R\times M$ which agrees with the flat metric $ds^2+dt^2+ dy^2+d\theta^2$ on $\R\times [-1,1]\times T^2$. All distances will be measured with respect to this metric. 

\begin{defn}
	Let $\kappa>0$. A curve $u:\dot F\to \R\times M$ in $\me$ is {\em $\kappa$-close to breaking into $u_+: \dot F\to \R\times M$ in $\mb$ and $u_{\mathcal{T}^\epsilon}: (-\infty,0]\times S^1\to \R\times M$}, where $\mathcal{T}^\epsilon$ is an upward gradient trajectory of $f_\epsilon$, if:
	\begin{itemize}
		\item[(i)] on the complement of a negative cylindrical end $(-\infty, 0]\times S^1$ of $\dot F$, the maps $u$ and $u_+^*$ (obtained from $u_+$ by a suitable translation in the domain if $\dot F$ is a cylinder and a suitable $\R$-translation in the target) are a distance $\leq \kappa$ apart;
		\item[(ii)] on $(-\infty, 0]\times S^1$, the maps $u$ and $u_{\mathcal{T}^\epsilon}^*$ (obtained from $u_{\mathcal{T}^\epsilon}$ by a suitable $\R$-translation in the target) are a distance $\leq \kappa$ apart.
	\end{itemize}
\end{defn}

Let $u_+:(\dot F,j)\to \R\times M$ be an element of $\mb$.  In what follows we may assume without loss of generality that
\be
\item[(C4)] $u_+$ limits to the Morse-Bott orbit ${\frak o}$ over the point $(0,0)$ from the positive $y$-direction at the negative end.
\ee
This is justified as follows: The quotient $\mb/\R$ by $\R$-translations in the target is a finite set by (C3). Let $\mathcal{E}: \mb/\R\to \mathcal{N}$ be the map that sends $[u]$ to the orbit of $\mathcal{N}$ that $u$ limits to at the negative end. Since the image of ${\mathcal E}$ is a finite set, we can parametrize ${\mathcal N} \cong \R/\Z$ such that ${\mathcal E}([u]) \in [-\frac 16, \frac 16]$ for all 
$u \in \mb$.
Since our proof works in the same way as long as $\mathcal{E}([u])$ is in the interior of the interval $\{\theta\in \R/\Z~|~ \overline{g}'_{\mathcal{N}}(\theta)=1\}$ (refer to (P2) for the definition of $\overline{g}_{\mathcal{N}}$), we normalize $\mathcal{E}([u])=0$. Moreover, approaching $\theta=0$ from the positive $y$-direction and the negative $y$-direction can be treated in the same way.

\begin{noti}
Let $\mathcal{T}^\epsilon_{0}$ denote the (upward) gradient trajectory of $f_\epsilon$
that goes from $(0,-{1\over 4})$ to $(0,0)$.
\end{noti}

Theorem~\ref{thm: main theorem of appendix} is an immediate consequence of the following theorems, which are proved in Sections~\ref{subsection: gluing 1} and \ref{subsection: gluing 2}, together with the above discussion on regularity:

\begin{thm} \label{thm: 1}
Suppose $a,b_0 >0$ are small. If $\mb$ is Morse-Bott regular, then for all $\epsilon>0$ sufficiently small there exists $u\in \me$ that is $\kappa$-close to breaking into $u_+$ and $u_{\mathcal{T}^\epsilon_0}$. 
\end{thm}

\begin{thm} \label{thm: 2}
Suppose $a,b_0 >0$ are small. If $\mb$ is Morse-Bott regular and $\me$ is regular, then there exists $\kappa>0$ such that for all $\epsilon>0$ sufficiently small and $u, v\in \me$ that are $\kappa$-close to breaking into $u_+$ and $u_{\mathcal{T}^\epsilon_0}$, $u=v$ modulo $\R$-translation in the target and domain translation if the domain is $\R\times S^1$.
\end{thm}

\begin{rmk}	The assumptions 
\be
\item[(i)] there is only one Morse-Bott torus $T_{\mathcal N}$ and it is negative, and
\item[(ii)] $u_+$ limits to $\gamma$ at the positive end and $\mathcal{N}$ at the negative end,
\ee
are only to make the notation simpler, since gluing each pair of ends can be done more or less independently. This is due to the fact that the magnitude of the error that comes from a pair $\frak P$ of glued ends and needs to be inverted in the Newton iteration decays exponentially with respect to the distance to the gluing region of $\frak P$.
\end{rmk}

\subsection{Asymptotic operator}\label{subsection: asymptotic operator}

On $B:=[-b_0, b_0] \times [-\tfrac{1}{6}, \tfrac{1}{6}]\subset A_c$, we have $\overline{g}_{\mathcal{N}}'(\theta)=1$ by (P2) and $\phi(y)=1$ by (P3').  Then $\nabla f_\epsilon= y\bdry_y + \epsilon \bdry_\theta$ and the equation $\mathcal{D}_\epsilon \eta=0$ becomes the linear equation
$${\bdry\eta\over \bdry s} + j_0{\bdry\eta\over\bdry t}-\begin{pmatrix}  \eta_1 \\ \epsilon  \end{pmatrix}=0,$$
or
\begin{equation} \tag{$J_{f_\epsilon}$}
{\bdry\eta\over \bdry s} -{\bf A}\eta= \begin{pmatrix} 0 \\ \epsilon  \end{pmatrix}, \quad {\bf A}\eta= -j_0{\bdry \eta \over \bdry t} + \begin{pmatrix} 1 & 0\\ 0 & 0\end{pmatrix} \eta,
\end{equation}
where $j_0=\begin{pmatrix} 0 & -1 \\ 1 & 0 \end{pmatrix}$,  $\eta= (\eta_1, \eta_2)$, and ${\bf A}$ is the asymptotic operator for the negative end of $u_+$ that goes to the Morse-Bott family $\mathcal{N}$. Here we are regarding $S^1$ as $[-\tfrac{1}{2},\tfrac{1}{2}]\subset \R$ so that $\eta\subset A_c$ is regarded in $\R^2$ and the matrix multiplication by $\begin{pmatrix} 1 & 0\\ 0 & 0\end{pmatrix}$ makes sense.

Similarly, a $J_{f}$-holomorphic map $(s,t)\mapsto (s,t,\eta(s,t))$ with $\eta(s,t)\in A_c$ is equivalent to
\begin{equation}
  \tag{$J_{f}$}
  {\mathcal D}_0 \eta = {\bdry\eta\over \bdry s} -{\bf A}\eta= 0.
\end{equation}

\begin{rmk} \label{rmk: ansatz}
In the region where ${\mathcal D}_\epsilon \eta=0$ is equivalent to Equation ($J_{f_\epsilon}$), a solution of ($J_{f}$) can be converted to a solution of ($J_{f_\epsilon}$) by adding 
$\begin{pmatrix} 0 \\ \epsilon s+C \end{pmatrix}$. 
\end{rmk}

From now on we will write the components of $\eta$ as row vectors if there is no confusion. 
\begin{claim}\label{claim: eigenfunctions and eigenvalues}
The eigenfunctions of ${\bf A}$ can be arranged as:
$$\dots, g_{-2},g_{-1},g_0=(0,1), g_1=(1,0),g_2,\dots,$$
normalized to have unit $L^2$-norm, with corresponding eigenvalues
$$\dots\leq\lambda_{-2}\leq \lambda_{-1}< \lambda_0=0< \lambda_1=1< \lambda_2\leq \dots,$$
where if $\lambda$ is any of $\lambda_{2n}=\lambda_{2n+1}$, $\lambda_{-2n}=\lambda_{-2n+1}$, then $\lambda(\lambda-1)=(2\pi n)^2$ and
\begin{align*}
g_{2n} \quad \mbox{is a multiple of} \quad & (\tfrac{2\pi n}{\lambda_{2n}-1}\cos (2\pi n t), \sin (2\pi n t)), \\
g_{2n+1}\quad \mbox{is a multiple of} \quad & (\tfrac{2\pi n}{\lambda_{2n}-1}\sin (2\pi n t), - \cos (2\pi n t)).
\end{align*}
\end{claim}  

\begin{proof} 
If $\lambda$ is an eigenvalue of ${\bf A}$, then 
\begin{align*}
\begin{pmatrix}  0 & 1 \\ -1 & 0 \end{pmatrix} \begin{pmatrix} y'(t)\\ \theta'(t)\end{pmatrix} + \begin{pmatrix} y(t)\\ 0\end{pmatrix} = \lambda \begin{pmatrix} y(t)\\ \theta(t)\end{pmatrix},
\end{align*}
which is equivalent to $\theta'(t)= (\lambda-1)y(t)$, $-y'(t)=\lambda \theta(t)$. Hence $\theta''(t)= (1-\lambda)\lambda \theta(t)$. If $\theta(t)$ is to be $1$-periodic, $\lambda=0,1$, $\lambda>1$, or $\lambda<0$.  In the latter two cases, $\theta(t)$ is a translate of $\sin (2\pi nt)$ and $y(t)= \tfrac{2\pi n}{\lambda-1}$ times a translate of $\cos (2\pi nt)$ and $(2\pi n)^2= \lambda(\lambda-1)$. 
\end{proof}

We can then write a solution $\eta(s,t)$ of ($J_f$) as a Fourier series
\begin{equation}\label{Fourier series for eta}
\eta(s,t)=\sum_{i=-\infty}^\infty c_i e^{\lambda_i s} g_i(t).
\end{equation}

A clarification of the meaning of Equation~\eqref{Fourier series for eta} is in order: the eigenfunctions $g_i$ take values in $T_{(0,0)}A_c \cong\R^2$, while $\eta$ takes values in $A_c \cong [-c,c] \times S^1$. Thus in the equality we have tacitly identified a neighborhood of $(0,0)$ in $T_{(0,0)}A_c$ with a neighborhood of $(0,0)$ in $A_c$ using the identification of $S^1$ with a quotient of $[- \frac 12, \frac 12]$ fixed at the beginning of the appendix.

\subsection{Pregluing} \label{subsection: pregluing}

Let $u_+: \dot F \to \R \times M$ be a $J_f$-holomorphic map representing an element of $\mb$. 
We fix a cylindrical end $(-\infty,s_0]\times S^1$ of $\dot F$ corresponding to the orbit $\mathfrak{o}$ on which $u_+$ takes the form $u_+(s,t)=(s,t,\eta_+(s,t))$, $\eta_+(s,t)\in A_c$. In view of (C4) we can write
\begin{equation}\label{eqn: expantion of eta_+}
\eta_+(s,t)=\sum_{i=1}^\infty c_i e^{\lambda_i s} g_i(t),
\end{equation}
where $c_1>0$.  The condition $c_1\not=0$ holds for a generic $J_f$ because the moduli space $\mb$ is one-dimensional. This is proved in the same way as \cite[Theorem 4.1]{HT2}, which treats the contact case. We further assume that $c_1>0$ since the $c_1<0$ case can be treated in the same way. Finally, we can assume that \eqref{eqn: expantion of eta_+} has no $i=0$ term because we assumed that $\mathfrak{o}$ is the orbit over $(0,0)$.

\begin{defn}\label{defn: T_0 and T_1}
  Let $T_0=T_0(a)$ and $T_1=T_1(b)$ be real numbers such that
  \begin{equation}\label{eqn: T_0 and T_1}
     c_1e^{-\lambda_1 T_0}g_1=(a/2,0), \quad c_1e^{-\lambda_1 T_1}g_1=(b,0).
  \end{equation}
\end{defn}

Note that $T_1(b) \to \infty$ as $b \to 0$.
\be
\item[($\boldsymbol{\dagger}_1$)] We choose $a, b_0>0$, with $b_0< a/4$, to be sufficiently small such that $T_0>0$ and, for all $b<b_0$,
$$\eta_+|_{s\leq -T_0}\subset A_a, \quad \eta_+|_{-T_0\leq s\leq s_0}\subset A_c\cap \{y>2b_0\}, \quad \mbox{and} \quad \eta_+|_{s\leq -T_1}\subset B.$$
\ee
The choice of $a$ is made possible by the fact that $\sum_{i=2}^\infty c_i e^{\lambda_i s} g_i(t)$ decays exponentially at a rate which is faster than $c_1e^{\lambda_1 s}g_1(t)$.
 {\em From now on $a$ and $T_0$ are fixed constants, while $b$ and $T_1(b)$ are, for the moment, still  allowed to vary and will be fixed at a later time.} 
\begin{rmk}
 Since the perturbation  of $f$, and therefore of $R_f$, given in Equation \eqref{eqn: f and f epsilon}  depends on $a$ and $b_0$  by Conditions (P2) and (P3'),  it is important that $\mb$ is finite, so that we can find $a$ and $b_0$ which satisfy ($\boldsymbol{\dagger}_1$) for every $u_+ \in \mb$.
\end{rmk} 

Let $\eta_-^\epsilon \colon \R \to A_a$ be a parametrization of the gradient flow trajectory of $f_\epsilon$ from $(0, -\frac 14)$ to $(0, \frac 14 )$ solving the Cauchy problem
$$\begin{cases}
\frac{d \eta_-^\epsilon}{ds}= \nabla f_\epsilon(\eta_-^\epsilon), \\
\eta_-^\epsilon(-T_1)=(0,0)
\end{cases}$$
and let $u_-^\epsilon(s,t)=(s,t,\eta_-^\epsilon(s))$. We trivially extend $\eta_-^\epsilon$ to a function $\eta_-^\epsilon \colon \R \times S^1\to A_a$ by $\eta_-^\epsilon(s,t)=\eta_-^\epsilon(s)$.

\begin{defn}\label{definition of T_2}
 Let  $T_2=T_2(\epsilon)$ be a real number such that $T_1<T_2$ and $\eta_-^\epsilon|_{-T_2\leq s\leq -T_1}\subset B$.
\end{defn}

Note that $T_2(\epsilon) \to +\infty$  as $\epsilon \to 0$.

Let $\beta: \R\to [0,1]$ be a nondecreasing function such that $\beta(s)=0$ if $s\leq 0$ and $\beta(s)=1$ if $s\geq 1$.
The pregluing $u_*^{\epsilon, \aaa}$ (note the Fraktur symbol $\aaa$ is different from the parameter $a$) will depend on $\epsilon$ and an extra real parameter $\aaa \in [- \aaa_0, \aaa_0]$, where $\aaa_0$ is independent of $b$ and $\epsilon$, and small enough that $\eta_-^\epsilon(-T_1+ \aaa/\epsilon)$ is contained in $B$, where $\nabla f_\epsilon$ is constant. Then we define
\begin{equation}\label{eqn: definition of u_*}
u_*^{\epsilon, \aaa}(s,t):= \left\{
\begin{array}{cl}
u_+(s,t) & \mbox{ on $\dot F-(-\infty,-T_0]\times S^1$},\\
(s,t,\eta_*^{\epsilon, \aaa}(s,t)) & \mbox{ on $(-\infty,-T_0]\times S^1$},
\end{array}
\right.
\end{equation}
where
\begin{equation}\label{equation for eta star}
\eta_*^{\epsilon, \aaa}(s,t)= \left\{ \begin{array}{cl}
\eta_+(s,t) + \beta({s+T_0\over -T_1+T_0})(0, \epsilon (s+T_1)) + \beta(-s-T_0)(0, \aaa) & \mbox{ on $[-T_1,-T_0]\times S^1$},\\
\eta_-^\epsilon(s+ \aaa / \epsilon, t)+ \beta(s+T_2)\cdot \eta_+(s,t) & \mbox{ on $(-\infty,-T_1]\times S^1$}.
\end{array}
\right.
\end{equation}
Observe that $\eta_-^\epsilon(s+\aaa/\epsilon ,t)=(0,\epsilon(s+T_1) + \aaa)$ on $[-T_2,-T_1]\times S^1$ since
$\nabla f_\epsilon(y, \theta) = (y, \epsilon)$ and $\eta_-^\epsilon(-T_1, t)=(0,0)$. Hence the two definitions agree along $s=-T_1$.  Therefore $u_*^{\epsilon, \aaa}$ coincides with $u_+$ for $s \ge -T_0$, with the lift of a gradient trajectory of $f_\epsilon$ for $s \le -T_2$, and interpolates between the two for $s \in [-T_2, -T_0]$. The interpolation is performed in three steps: for $s \in [-T_0-1, -T_0]$ the holomorphic curve is pushed in the $\theta$-direction (i.e., along the Morse-Bott family) by a small amount $\aaa$; for $s \in [-T_1, -T_0]$ a perturbation corresponding to the gradient trajectory is slowly turned on and added to $u_+$; for $s \in [-T_2+1, -T_1]$ the preglued curve $u_*^{\epsilon, \aaa}$ is the sum of $u_+$ and the lift of a gradient trajectory of $f_\epsilon$; and for $s \in [-T_2, -T_2+1]$ the contribution of $u_+$ is turned off.   

\be
\item[($\boldsymbol{\dagger}_2$)] We choose $\epsilon= \epsilon(b)>0$ such that
$$\lim \limits_{b \to 0} \epsilon(b) e^{T_1(b)}T_1(b)=0.$$
\ee
Note that $T_0$ has become a constant after we fixed $a$, while $T_1$ depends on $b$ and $T_2$ depends on $\epsilon$.

\begin{lemma} \label{lemma: where d bar is zero}
  If $u_*^{\epsilon, \aaa}$ is defined by Equations \eqref{eqn: definition of u_*} and \eqref{equation for eta star}, then $\overline\bdry_{J_{f_\epsilon}}u_*^{\epsilon, \aaa}$ is supported on:
\be
\item $([- T_2, -T_2+1] \times S^1) \cup ([-T_1, -T_0] \times S^1)$ and 
\item the ``thick'' parts of the domain of $u_+$, i.e., $\dot{F}- (-\infty,T_0] \times S^1$, where the curve may still enter the region $y\in [-2b_0,2b_0]$.\footnote{In this case the error is extremely small, of total size $C\epsilon$, and we will not mention it further.}
\ee
\end{lemma}

\begin{proof}
  Note that $\overline\bdry_{J_{f_\epsilon}}u_*^{\epsilon, \aaa}=0$ 
\be
\item[(a)] on $\dot F-(-\infty,-T_0]\times S^1$, away from the region described in (2),  where $u_*^{\epsilon, \aaa}=u_+$ and $f_\epsilon=f$, and 
\item[(b)] on $(-\infty,-T_2]\times S^1$,  where $u_*^{\epsilon, \aaa}$ coincides with the $\overline\bdry_{J_{f_\epsilon}}$-holomorphic lift of a gradient trajectory of $f_\epsilon$.
\ee  
(For (a), note that $f_\epsilon$ and $f$ differ only when $y\in[-2b_0, 2b_0]$ by Equation~\eqref{eqn: f and f epsilon} and (P3'), but we are assuming ($\boldsymbol{\dagger}_1$), which ensures that $\eta_+(T_0,t)$ has $y$-coordinate $>2b_0$.)
  Therefore $\overline\bdry_{J_{f_\epsilon}}u_*^{\epsilon, \aaa}$ is supported in $[-T_2, -T_0] \times S^1$, where $\overline\bdry_{J_{f_\epsilon}}u_*^{\epsilon, \aaa}=0$ is equivalent to ${\mathcal D}_\epsilon \eta_*^{\epsilon, \aaa}=0$. We claim that ${\mathcal D}_\epsilon \eta_*^{\epsilon, \aaa}=0$ on $[-T_2+1, -T_1] \times S^1$. In fact, in that region,
  $$\eta_*^{\epsilon, \aaa}= \eta_-^{\epsilon, \aaa} + \eta_+$$
  by Equation \eqref{equation for eta star} and the definition of $\beta$. Moreover, $\eta_*^{\epsilon, \aaa}$ takes values in $B$ by Condition $(\dagger_1)$ and Definition \ref{definition of T_2}, and in $B$ we have
  $${\mathcal D}_\epsilon = {\mathcal D}_0 + (0, \varepsilon),$$
  where ${\mathcal D}_0$ is linear by Equation $(J_{f_\epsilon})$. Thus we have
  $${\mathcal D}_\epsilon(\eta_*^{\epsilon, \aaa})= {\mathcal D}_0(\eta_-^{\epsilon, \aaa})+ {\mathcal D}_0(\eta_+)+(0,\epsilon)= {\mathcal D}_\epsilon(u_-^{\epsilon, \aaa})+ {\mathcal D}_0(\eta_+)=0$$
  because ${\mathcal D}_\epsilon(u_-^{\epsilon, \aaa})= {\mathcal D}_0(\eta_+)=0$.
\end{proof}

\subsection{Function spaces} \label{subsection: function spaces}
Let us introduce the notation
\begin{equation}\label{I'm 49 :(}
  \eta_-^{\epsilon, \aaa}(s,t)= \eta_-^\epsilon(s+\aaa/\epsilon, t), \quad u_-^{\epsilon, \aaa}(s,t)=(s,t, \eta_-^{\epsilon, \aaa}(s,t)).
\end{equation}
In this subsection we describe the linearized $\overline\bdry$-operators $D_+$ and $D_-^{\epsilon, \aaa}$ for $u_+$ and $u_-^{\epsilon, \aaa}$.

Since we are assuming that the ECH and Fredholm indices of $u_+$ and $u_-^{\epsilon,\aaa}$ are both $1$, they are embedded and admit normal bundles. 
Let $N_+$ be a $J_f$-invariant normal bundle to $u_+$ in $\R\times M$ such that $N_+=TA_a$ on $(-\infty, -T_0]\times S^1$, let $N_-^{\epsilon, \aaa}=T A_a$ be the normal bundle to $u_-^{\epsilon, \aaa}$ in $\R\times[-a,a]\times T^2$, and let $N_*^{\epsilon, \aaa}$ be the normal bundle to $u_*^{\epsilon, \aaa}$ that agrees with $N_+$ on $\dot F- (-\infty, -T_0]\times S^1$ and with $T A_a$ on $(-\infty, -T_0]\times S^1$.

\subsubsection{Exponential maps} 

Let $\mathbb{D}_\kappa N_+$ denote the disk bundle of $N_+$ of radius $\kappa>0$, measured with respect to $g$.  
Writing an element of $N_+$ as $(x,\xi(x))$, where $x\in \dot F$ and $\xi(x)\in N_+(u_+(x))$, for $\kappa>0$ small we choose an exponential map
$$\op{exp}_{u_+}: \mathbb{D}_\kappa N_+\to \R\times M,$$
such that  $\op{exp}_{u_+}(x,0) = u_+(x)$, $d_{(x,0)}\op{exp}_{u_+}(0,\zeta)=\zeta(u_+(x))$ for a section $\zeta$ of $N^+$,  and 
$$\op{exp}_{u_+}(x,\xi(x))= (s(x),t(x),\eta_+(x)+ \xi(x))$$
when $u_+(x)=(s(x),t(x),\eta_+(x))$ and $x\in (-\infty, -T_0]\times S^1$. We also define
\begin{gather*}
\op{exp}_{u_-^{\epsilon, \aaa}}:\mathbb{D}_\kappa N_-^{\epsilon, \aaa} \to \R\times [-a,a]\times T^2,\\
(x,\xi(x))\mapsto (s(x),t(x),\eta_-^{\epsilon, \aaa}(x)+\xi(x)).
\end{gather*}
Finally we define  $\op{exp}_{u_*^{\epsilon, \aaa}}$  on $\mathbb{D}_\kappa N_*^{\epsilon, \aaa}$ such that it agrees with $\op{exp}_{u_+}$ on $\dot F-(-\infty,-T_0]\times S^1$ and satisfies
$$(x,\xi(x))\mapsto (s(x),t(x),\eta_*^{\epsilon, \aaa}(x)+\xi(x))$$
on $(-\infty,-T_0]\times S^1$. In particular $\op{exp}_{u_*^{\epsilon, \aaa}}$ coincides with 
$\op{exp}_{u_-^{\epsilon, \aaa}}$ on $(- \infty, -T_2) \times S^1$.  

\subsubsection{Normal $\overline \bdry$-equations}

Instead of using the full $\overline\bdry$-operator on sections of $u_+^*T(\R\times M)$ and $(u_-^{\epsilon, \aaa})^*T(\R \times M)$, following \cite{HT2} we will use the normal $\overline\bdry$-operators which act on sections of $N_+$ and $N_-^{\epsilon, \aaa}$. {\em The primary purpose of using the normal $\overline\bdry$-operators, assuming the curves are embedded, is to simplify the notation,} since the Teichm\"uller space parameters are automatically taken care of. More precisely, let $L$ be the total linearized $\overline\bdry$-operator --- this includes the Teichm\"uller space parameters --- and let $L_N$ be the normal linearized $\overline\bdry$-operator $L_N$. Then $\op{coker}{L}\simeq \op{coker}{L_N}$ and $\ker L_N\simeq (\ker L)/V$, where $V$ is subspace generated by the infinitesimal generators of the reparametrizations of the domain.

By standard local existence results of holomorphic disks, for $\kappa>0$ small there exists a foliation of $\op{exp}_{u_+}(\mathbb{D}_\kappa N_+)$ by $J_f$-holomorphic disks such that the holomorphic disk passing through $u_+(x)$ is tangent to $N_+(u_+(x))$. We can therefore adjust the map $\op{exp}_{u_+}$ such that the fibers of $\mathbb{D}_\kappa N_+$ are mapped to holomorphic disks, use local coordinates $(\sigma,\tau,\xi)$ on $\op{exp}_{u_+}(\mathbb{D}_\kappa N_+)$, where $\sigma+i\tau$ are holomorphic coordinates on $\dot F$ and $\xi$ is the fiber coordinate, and write
$$J_f(\sigma,\tau,\xi)= \begin{pmatrix} \tilde \jmath (\sigma,\tau,\xi) & 0 \\ X(\sigma,\tau,\xi) & j_0 \end{pmatrix},$$
where $\tilde \jmath (\sigma,\tau,0)=j_0$ and $X(\sigma, \tau, 0)=0$. Since $\tilde \jmath^2=-I$, we have
$$\tilde \jmath (\sigma,\tau,\xi)= \begin{pmatrix} a(\sigma,\tau,\xi)& c(\sigma,\tau,\xi) \\ b(\sigma,\tau,\xi) & -a(\sigma,\tau,\xi) \end{pmatrix}$$ 
and $\op{det}\tilde \jmath=1$. Also $X\tilde\jmath+j_0X=0$.

We derive the normal $\overline\bdry$-equation for a section $\xi$ of $N_+$ such that
\begin{equation} \label{pancakes} 
\overline \bdry_{J_f}\op{exp}_{u_+}\xi=0.
\end{equation} 
We recall that $\overline \bdry_{J_f}u= du+ J_f \circ du \circ j$, where $j$ is a complex structure on the domain of $u$, and therefore Equation \eqref{pancakes} is an equation for a pair $(j, \xi)$, where $j$ is a complex structure on $\dot F$. Then 
solving Equation~\eqref{pancakes} is equivalent to solving for $A(\sigma,\tau,\xi)$, $B(\sigma,\tau,\xi)$, and $\xi(\sigma,\tau)$ in:
\begin{equation} \label{eqn: towards normal d bar operator}
\left({\bdry \over \bdry \sigma} + J_f(\sigma,\tau,\xi)\left(A(\sigma,\tau,\xi){\bdry \over \bdry \sigma} + B(\sigma,\tau,\xi){\bdry \over \bdry \tau}\right)\right) \begin{pmatrix} \sigma\\ \tau \\ \xi\end{pmatrix}=0.
\end{equation}
Here the adjustment of the domain complex structure is equivalent to solving for $A(\sigma,\tau,\xi)$ and $B(\sigma,\tau,\xi)$.
One easily verifies that 
$$A(\sigma,\tau,\xi)=a(\sigma,\tau,\xi)\quad \mbox{and} \quad B(\sigma,\tau,\xi)=b(\sigma,\tau,\xi)$$ 
are the unique functions such that the $(\sigma,\tau)$-component of Equation~\eqref{eqn: towards normal d bar operator} holds. Then the $\xi$-component of Equation~\eqref{eqn: towards normal d bar operator} is the normal $\overline \bdry$-equation for the section $\xi$ of $N_+$:
\begin{equation}\label{eqn: normal d bar equation for u plus}
\overline\bdry_{N_+,f}\xi:={\bdry \xi\over \bdry \sigma} + j_0 \left( a(\sigma,\tau,\xi){\bdry \xi \over \bdry \sigma} + b(\sigma,\tau,\xi){\bdry \xi\over \bdry \tau}\right) + X(\sigma,\tau,\xi)\begin{pmatrix} a(\sigma,\tau,\xi)\\ b(\sigma,\tau,\xi)\end{pmatrix}=0,
\end{equation}
such that 
\begin{equation}\label{eqn: initial conditions}
a(\sigma,\tau,0)=0, \quad b(\sigma,\tau,0)=1,\quad \mbox{and} \quad X(\sigma,\tau,0)=0.
\end{equation}

Next we derive the normal $\overline\bdry$-equation for the section $\xi$ of $N_-^{\epsilon, \aaa}$ such that 
$$\overline \bdry_{J_{f_\epsilon}}\op{exp}_{u_-^{\epsilon, \aaa}}\xi=0.$$  
Recall that we write $\op{exp}_{u_-^{\epsilon, \aaa}}\xi= (s,t,\eta_-^{\epsilon, \aaa}+\xi)$ on $\R\times[-a,a]\times T^2$. Since $\overline\bdry_{J_{f_\epsilon}} u_-^{\epsilon, \aaa}=0$, the normal $\overline\bdry$-equation for $\xi$ has the following explicit expression:
\begin{align} \label{eqn: normal d bar equation for u minus}
\overline\bdry_{N_-^{\epsilon, \aaa},f_\epsilon}\xi :& =\mathcal{D}_\epsilon(\eta_-^{\epsilon, \aaa} + \xi)= \mathcal{D}_\epsilon(\eta_-^{\epsilon, \aaa}+ \xi)- \mathcal{D}_\epsilon(\eta_-^{\epsilon, \aaa})\\
\nonumber & = {\bdry\xi\over \bdry s} + j_0{\bdry\xi\over\bdry t}-\nabla f_\epsilon(\eta_-^{\epsilon, \aaa} +\xi)+\nabla f_\epsilon(\eta_-^{\epsilon, \aaa})=0,
\end{align}
by Claim~\ref{leftovers}.

Finally, $\overline\bdry_{N_*^{\epsilon, \aaa}, f_\epsilon}\xi$ for the section $\xi$ of $N_*^{\epsilon, \aaa}$ agrees with $\overline\bdry_{N_+,f}\xi$ on $\dot F-(-\infty,-T_0]\times S^1$ and with $\mathcal{D}_\epsilon(\eta_*+ \xi)$ on $(-\infty,-T_0]\times S^1$ by Claim~\ref{leftovers}.

The linearized operators for $\overline\bdry_{N_+, f}$, $\overline\bdry_{N_-^{\epsilon, \aaa}, f_\epsilon}$, and $\overline\bdry_{N_*^{\epsilon, \aaa}, f_\epsilon}$ will be denoted by $D_+$,  $D_-^{\epsilon, \aaa}$, and $D_*^{\epsilon, \aaa}$. Next we will describe the proper function-theoretic setup for these operators.

\subsubsection{Morrey spaces}

The function spaces that we use are {\em Morrey spaces}, following \cite[Section 5.5]{HT2}. Let $u:\dot F\to \R\times M$ be a finite energy holomorphic curve. On $\dot F$ we choose a Riemannian metric such that the ends are isometric to $\R/\Z\times[0,\infty)$ with the product metric. On $\R\times M$ we continue use the $\R$-invariant Riemannian metric from before.

The {\em Morrey space} $\mathcal{H}_0(\dot F, \wedge^{0,1} N_+)$ is the Banach space which is the completion of the compactly supported sections of $\wedge^{0,1} N_+$ with respect to the norm
\begin{equation}\label{morrey norm 0}
  \| \xi\|= \left( \int_{\dot F} |\xi|^2  \right)^{1/2} + \left( \sup_{x\in \dot F} \sup_{\rho\in(0,1]} \rho^{-1/2} \int_{B_\rho(x)} |\xi|^2  \right)^{1/2},
\end{equation}
where $B_\rho(x)\subset \dot F$ is the ball of radius $\rho$ about $x$.  Similarly, $\mathcal{H}_1(\dot F, N_+)$ is the completion of the compactly supported sections of $N_+$ with respect to 
\begin{equation}\label{morrey norm one}
  \| \xi \|_*= \|\nabla \xi\| + \|\xi\|.
\end{equation}

Although Morrey spaces are not used as frequently as Sobolev spaces, they satisfy the analog of the usual Sobolev embedding theorem (Lemma~\ref{lemma: sobolev embedding}) and have the advantage that {\em we only need to do elementary $L^2$-type estimates instead of more complicated $L^p$-type estimates.}

The analog of the usual Sobolev embedding theorem is the following:\footnote{The lemma is stated slightly differently from \cite[Lemma 5.3]{HT2}.} 

\begin{lemma} \label{lemma: sobolev embedding}
	There is a bounded linear map
	$$\mathcal{H}_1(\dot F, N_+)\to C^{0}(\dot F, N_+) \cap L^\infty(\dot F, N_+), \quad \xi\mapsto \xi.$$
\end{lemma}

\begin{proof}
If $\xi\in \mathcal{H}_1(\dot F, N_+)$ and $K\subset \dot F$ is a subdomain, then let us define 
$$|\xi|_{C^{0,1/4},K}= \op{sup}_{x\not=y\in K} {|\xi(x)-\xi(y)|\over |x-y|^{1/4}}.$$

The lemma is a consequence of \cite[Theorem 3.5.2]{Mo}, which implies\footnote{Morrey's theorem is stated for a Euclidean ball of radius $R$, but applies equally well to our setting. We take $p=2,\nu=2, \mu=\tfrac{1}{4}$ in the theorem.} that for any compact subdomain $K\subset \dot F$ there exists $C_K$ such that
\begin{equation}
|\xi|_{C^{0,1/4},K}\leq  C_K \left( \sup_{x\in \dot F} \sup_{\rho\in(0,1]} \rho^{-1/2} \int_{B_\rho(x)} |\nabla\xi|^2  \right)^{1/2} \leq C_K \|\xi\|_*.
\end{equation}
This implies that any $\xi\in \mathcal{H}_1(\dot F, N_+)$ is continuous. 

Since $\dot F$ has cylindrical ends, we can write $\dot F= K_0\cup K_1\cup K_2\cup \dots,$ where all the $K_i$ are compact connected subdomains and $K_1,K_2,\dots$ are annuli of the form $\R/\Z$ times a unit interval. For each $K_i$ and $x\not=y\in K_i$, we have
\begin{gather} \label{coffee 1}
|\xi(x)-\xi(y)|\leq  C_{K_i} \|\xi\|_*|x-y|^{1/4} \leq C (C')^{1/4} \|\xi\|_*,
\end{gather}
where $C=\max\{C_{K_0}, C_{K_1}\}$ since $C_{K_1}=C_{K_2}=\dots$ and $C'$ is the supremum of the diameters of $K_i$, $i=0,1,\dots$. Since $\xi$ is continuous, on each $K_i$ there exists $x_i$ such that 
\begin{gather} \label{coffee 2}
|\xi(x_i)|=\| \xi|_{K_i}\|_{L^2} / \op{vol}(K_i) \leq \|\xi\|_*/C'',
\end{gather} 
where $C''=\op{inf}_i \op{vol}(K_i)>0$. Inequalities~\eqref{coffee 1} and \eqref{coffee 2} together imply that there exists a constant $c>0$ which is independent of $\xi$ and such that 
$|\xi(x)|\leq c \|\xi\|_*$ for all $x\in \dot F$.
\end{proof}

Given $\delta>0$ sufficiently small, we define a smooth weight function 
\begin{gather}
g_\delta:\dot F\to \R^+,\\
\nonumber \left\{ \begin{array}{l} g_\delta(x)= 1 \quad \mbox{on} \quad \dot F-(-\infty,-T_0+1]\times S^1,\\   
	g_\delta(s,t)=e^{\delta |s+T_0|} \quad \mbox{for} \quad s\leq -T_0.\end{array} \right.
\end{gather} 
Also define the smooth weight function 
\begin{gather}
h_\delta:\R\times S^1\to \R^+,\\
\nonumber  (s,t)\mapsto e^{-\delta(s+T_0)}.
\end{gather} 
Note that $h_\delta$ agrees with $g_\delta$ for $s\leq -T_0$.  We recall that $T_0$ has been fixed once and for all in Definition \ref{defn: T_0 and T_1} and  ($\dagger_0$).  For our purposes we define $\lambda:=\min (\lambda_1,|\lambda_{-1}|)$ and take $\delta$ such that $5 \delta < \lambda$.

We also define the weighted Morrey spaces $\mathcal{H}_{1,g_\delta}(\dot F, N_+)$ and $\mathcal{H}_{0,g_\delta}(\dot F, \Lambda^{0,1} N_+)$ as the spaces of sections $\xi$ (of the respective bundles) such that the weighted Morrey norms
\begin{equation}
 \|\xi\|_{*,g_\delta}:= \|\xi \cdot g_\delta\|_*, \qquad  \|\xi\|_{g_\delta}:= \|\xi \cdot g_\delta\|
\end{equation}
are finite.  
Observe that since we are using normal bundles, it is not necessary to use weights except at the end which limits to the Morse-Bott orbit.  The Morrey spaces for $N_-^{\epsilon, \aaa}$ and $N_*^{\epsilon, \aaa}$ (with and without weights) are defined analogously.

\subsubsection{Linearized operators}\label{ss: linearized operators}

Let $\tilde\partial_\theta$ be a smooth section of $N_+$ which is equal to $\beta(-s-T_0)\partial_\theta$ on $s\leq -T_0$ and is zero elsewhere. 
We view $D_+$ as a bounded linear operator 
\begin{gather}
\label{eqn for D plus} D_+^\delta:\mathcal{H}_{1,g_\delta}(\dot F, N_+)\oplus \R\langle \tilde\partial_\theta\rangle \to \mathcal{H}_{0,g_\delta}(\dot F, \Lambda^{0,1} N_+),\\
\nonumber (\xi,c \tilde \partial_\theta) \mapsto D_+(\xi+c\tilde\partial_\theta).
\end{gather}
The term $\R\langle \tilde\partial_\theta\rangle$ is included since $D_+^\delta$ is the linearized operator for the Morse-Bott family $\mb=\mathcal{M}_{J_f}^{I=\op{ind}=1}(\boldsymbol{\gamma};\mathcal{N})$ with an unconstrained negative end but the infinitesimal deformations parallel to the Morse-Bott family do not belong to the Morrey space with weights.

We denote 
\begin{equation}
{\mathcal H}_{+,\delta}= \mathcal{H}_{1,g_\delta}(\dot F, N_+)\quad \mbox{and} \quad {\mathcal H}_{+,\delta}'= D_+^\delta({\mathcal H}_{+,\delta})\subset \mathcal{H}_{0,g_\delta}(\dot F, \Lambda^{0,1} N_+),
\end{equation} 
and let $\overline{D}_+^\delta:  {\mathcal H}_{+,\delta} \to {\mathcal H}_{+,\delta}'$ be the map induced by $D_+^\delta$ by restriction.

Let us denote 
\begin{equation}\label{definition of nu}
  \nu :=-\beta'(-s-T_0)\partial_\theta.
\end{equation}
Then $\nu$ has compact support in $\{-T_0-1 \le s \le -T_0 \}$ and satisfies
$$D_+^\delta(0,1) = D_+(\tilde \partial_\theta) = \nu.$$
Also observe} that $\nu \not \in {\mathcal H}_{+,\delta}'$ because $D_+^\delta$ is surjective and $\op{ind}(D_+^\delta)=1$ with $\ker D_+^\delta\subset \mathcal H_{+,\delta}$, and $\nu \neq 0$. Then we can define the projection 
\begin{equation}
\Pi \colon \mathcal{H}_{0,g_\delta}(\dot F, \Lambda^{0,1} N_+) \to {\mathcal H}_{+,\delta}'
\end{equation} 
with $\nu \in \ker \Pi$.
Note that $\nu$ has compact support in $\{-T_0-1 \le s \le -T_0 \}$, where it can be written as $\nu =-\beta'(-s-T_0)\partial_\theta$.

\begin{rmk}
The domain of $D_+^\delta$ is the tangent space to the Banach manifold
$$\mathcal{H}_{1,g_\delta}(\dot F, \R\times M):= \{ \exp_u(\xi)~|~ u\in \mathcal{C}, \xi\in \mathcal{H}_{1,g_\delta}(\dot F, N_+^u)\},$$
where $\mathcal{C}$ is the space of smooth embeddings $u: \dot F\to \R\times M$ that agree with holomorphic maps parametrizing trivial holomorphic half-cylinders near each of the punctures; the positive ends of $u$ and $u_+$ agree and the negative end of $u$ limits to $\mathcal{N}$; and $N_+^u$ is the $J_f$-invariant normal bundle to $u$.
\end{rmk}

Linearizing Equation~\eqref{eqn: normal d bar equation for u minus} we obtain: 
\begin{equation}
D_-^{\epsilon, \aaa} \xi= {\bdry \xi\over \bdry s} +j_0{\bdry \xi\over \bdry t} -({\bf H} {f_\epsilon})(\eta_-^{\epsilon, \aaa})\xi,
\end{equation}
where ${\bf H}  f_\epsilon$ is the Hessian of $f_\epsilon$.

We view $D_-^{\epsilon, \aaa}$ as a bounded linear operator
\begin{gather} \label{eqn for D minus}
D_-^{\epsilon, \aaa}: \mathcal{H}_1(\R\times S^1, N_-^{\epsilon, \aaa})\to \mathcal{H}_0(\R\times S^1, \Lambda^{0,1}N_-^{\epsilon, \aaa}).
\end{gather}
Since the normal bundles $N_-^{\epsilon, \aaa}$ are trivialized, we can identify the domains and codomains of $D^{\epsilon, \aaa}$ for different values of $\epsilon$ and $\aaa$. We abbreviate ${\mathcal H}_-= \mathcal{H}_1(\R\times S^1, N_-^{\epsilon, \aaa})$ and ${\mathcal H}_-'= \mathcal{H}_0(\R\times S^1, \Lambda^{0,1}N_-^{\epsilon, \aaa})$. Both $D_+^\delta$ and $D_-^{\epsilon, \aaa}$ are Fredholm of index $1$. 

We consider also operators 
$$D_-^{\epsilon, \aaa, \delta} \colon \mathcal{H}_{1, h_\delta}(\R\times S^1, N_-^{\epsilon, \aaa})\to \mathcal{H}_{0, h_\delta}(\R\times S^1, \Lambda^{0,1}N_-^{\epsilon, \aaa})$$ 
which have the same expression as $D_-^{\epsilon, \aaa}$ but act on the Morrey spaces with weights. 
We  abbreviate $\mathcal{H}_{-, \delta}= \mathcal{H}_{1, h_\delta}(\R\times S^1, N_-^{\epsilon, \aaa})$ and $\mathcal{H}_{-, \delta}'= \mathcal{H}_{0, h_\delta}(\R\times S^1, \Lambda^{0,1}N_-^{\epsilon, \aaa})$.
\begin{rmk}
Sections $\xi \in \mathcal{H}_{-, \delta}$ can diverge as $s \to + \infty$ and therefore $\exp_{u_-^{\epsilon, \aaa}} (\xi)$ may not be well defined. This makes the spaces $\mathcal{H}_{-, \delta}$ unsuitable for the nonlinear analysis of the  moduli space containing $u_-^{\epsilon, \aaa}$. However, they can still be used in the proof of Theorem \ref{thm: main theorem of appendix} because, for the purposes of gluing, what happens near the positive end of $u_-^{\epsilon, \aaa}$ is irrelevant. The reason we are using the operators $D_-^{\epsilon, \aaa, \delta}$ is so that we can take the limit of $D_-^{\epsilon, \aaa, \delta}$ as $\epsilon\to 0$ and obtain a Fredholm operator $D_-^{0, \aaa, \delta}$ of the same index in the limit. This would not be true if we worked without weights, as the operators $D^{\epsilon, \aaa}_-$ converge, for $\epsilon \to 0$, to an operator which is not Fredholm. 
\end{rmk}

\begin{lemma} \label{bounded inverses}
If $\delta$, $\aaa_0=\aaa_0(\delta)$, and $\epsilon_0=\epsilon_0(\delta,\aaa_0)$ are sufficiently small subject to $0 \le \epsilon_0 < \delta$, and $(\aaa ,\epsilon)\in [-\aaa_0, \aaa_0]\times [0,\epsilon_0]$, then the operators $D_-^{\epsilon, \aaa, \delta}$ are invertible. Moreover, for a fixed $\delta$ the norms of the inverse operators $(D_-^{\epsilon, \aaa, \delta})^{-1}$ are uniformly bounded on $[-\aaa_0,\aaa_0]\times[0,\epsilon_0]$.
\end{lemma}

\begin{proof}
The operators $D_-^{\epsilon, \aaa, \delta} \colon {\mathcal H}_{-, \delta} \to {\mathcal H}_{-, \delta}'$ (including for $\epsilon=\aaa=0$, which is well-defined because $\mathbf{H}_f$ is constant on $\{y=0\}$) are conjugated to the operators 
$$\widetilde{D}_-^{\epsilon, \aaa, \delta} = D_-^{\epsilon, \aaa} + \delta \op{Id} \colon {\mathcal H}_- \to {\mathcal H}_-'.$$
The operator
$$\widetilde{D}_-^{0,0, \delta}= \frac{\partial \eta}{\partial s} + j_0 \frac{\partial \eta}{\partial t} + \begin{pmatrix} -1+\delta & 0 \\ 0 & \delta \end{pmatrix}$$
is Fredholm because for $\delta$ small its asymptotic operators are invertible. Moreover 
$\widetilde{D}_-^{0,0, \delta}$ has no spectral flow, and therefore $\op{ind}(\widetilde{D}_-^{0,0, \delta})=0$. Hence $D_-^{0,0, \delta}$ is also a Fredholm operator of index zero.

By elliptic regularity all elements of $\ker D_-^{0, 0,\delta}$ are smooth solutions of Equation
($J_{f}$), and from the Fourier series expansion \eqref{Fourier series for eta} we see that no such solution has the correct growth for $s \to \pm \infty$ to belong to ${\mathcal H}_{-, \delta}$. Then  $D_-^{0, 0,\delta}$ is injective and therefore, having index zero, is invertible. Since 
$$\|D_-^{\epsilon, \aaa, \delta} \xi -  D^{0,0, \delta}_- \xi \|_{h_\delta} \le \epsilon C \| \xi \|_{*, h_\delta}$$ 
for a constant $C$ which is independent of $\epsilon$ and $\aaa$ and invertibility is an open condition, for a fixed $\delta$, all operators $D_-^{\epsilon, \aaa, \delta}$ are invertible when the conditions of the lemma are met. The uniform bound on the norms of $(D_-^{\epsilon, \aaa, \delta})^{-1}$ then follows by the continuity of taking the inverse.
\end{proof}

\subsection{Setting up the gluing}  \label{subsection: setting up gluing}

The gluing setup will follow \cite{BH}, which in turn is based on \cite{HT2}. 

Define smooth cutoff functions
\begin{equation} 
\beta_+,\beta_-: \R\to [0,1]
\end{equation} 
such that $\beta_+ + \beta_-=1$ and
$\beta_+(s)=0$ for $s\leq -T_1$ and $\beta_+(s)=1$ for $s\geq -T_0-1$. The cutoff functions $\beta_\pm$ will depend on the parameter $b$ and will be denoted by $\beta_\pm^b$ when we want to make the dependence explicit. Let us write $-T_1(b)$ for $-T_1$ viewed as a function of $b$. Then:
\be
\item[($\boldsymbol{\dagger}_3$)] If $a>0$ is fixed but we take $b\to 0$, then $-T_1(b)\to -\infty$ and we take $\beta_\pm^b$ such that $|(\beta^b)'_\pm|_{C^0}\to 0$ as $b\to 0$. 
\ee

Let $\psi_+$ and $\psi_-^{\epsilon,\aaa}$ be sections in ${\mathcal H}_{+, \delta}$ and ${\mathcal H}_{-, \delta}$ of sufficiently small norm. The goal is to deform the pregluing $u_*^{\epsilon, \aaa}$ to 
\begin{equation}
u^{\epsilon, \aaa}=\op{exp}_{u_*^{\epsilon, \aaa}}(\beta_+ \psi_+ + \beta_- \psi_-^{\epsilon, \aaa}),
\end{equation}
and solve for $\psi_+$ and $\psi_-^{\epsilon, \aaa}$ in the equation $\Pi\bdry_{N_*^{\epsilon, \aaa}, f_\epsilon}(\beta_+ \psi_+ + \beta_- \psi_-^{\epsilon, \aaa})=0$ when $\epsilon$ is sufficiently small. (Recall the identifications of the normal bundles made at the beginning of Section~\ref{subsection: function spaces} that justify writing $\beta_+ \psi_+ + \beta_- \psi_-^{\epsilon,\aaa}$.) The solutions will determine  functions $\mathfrak{p}_\epsilon \colon [-\aaa_0, \aaa_0] \to \R$ such that
\begin{equation}\label{eq: almost obstruction bundle}
  \bdry_{N_*, f_\epsilon}(\beta_+ \psi_+ + \beta_- \psi_-^{\epsilon, \aaa})=\mathfrak{p}_\epsilon(\aaa) \nu.
\end{equation}
Finally, we will solve the equation $\mathfrak{p}_\epsilon(\aaa)=0$.

In the following lemmas we will repeatedly use Taylor expansions of the form
\begin{align*}
\tag{$I$} \phi(\mathbf{x})= & \phi(\mathbf{0}) + \sum \limits_i \ell_i(\mathbf{x})x_i \\
\tag{$II$} \phi(\mathbf{x})= & \phi(\mathbf{0}) + \sum  \limits_i \partial_i \phi(\mathbf{0})x_i + \sum \limits_{j,k} q_{j,k}(\mathbf{x})x_jx_k
\end{align*}  
for a smooth function $\phi : \R^n \to \R$. 

\begin{lemma} \label{lemma: equivalent 1}
  Over the domain $(-\infty,-T_0]\times S^1$, we can expand
  \begin{align} \label{expansion}
    \overline \bdry_{N_*^{\epsilon, \aaa}, f_\epsilon} (\beta_+ \psi_+ + & \beta_- \psi_-^{\epsilon, \aaa}) = \\
    \nonumber & \mathcal{D}_\epsilon \eta_*^{\epsilon, \aaa} + \beta_+( D_+\psi_+ +  \mathcal{L}_+(\psi_+,\psi_-^{\epsilon, \aaa})+\mathcal{Q}_+ (\psi_+,\psi_-^{\epsilon, \aaa})) \\
    \nonumber & + \beta_- (D_-^{\epsilon, \aaa}\psi_-^{\epsilon, \aaa} +\mathcal{L}_-(\psi_+,\psi_-^{\epsilon, \aaa})+ \mathcal{Q}_- (\psi_+,\psi_-^{\epsilon, \aaa})),
  \end{align}
where:
\be
\item $D_+ \psi_+= {\bdry \psi_+ \over \bdry s} + j_0{\bdry \psi_+ \over \bdry t} - {\bf H}f(\eta_+) \psi_+$.
\item $D_-^{\epsilon, \aaa}\psi_-^{\epsilon, \aaa} = {\bdry \psi_-^{\epsilon, \aaa} \over \bdry s} + j_0{\bdry \psi_-^{\epsilon, \aaa}\over \bdry t}- {\bf H}f_\epsilon(\eta_-^{\epsilon, \aaa}) \psi_-^{\epsilon, \aaa}$.
\item $\mathcal{L}_\pm (\psi_+,\psi_-^{\epsilon, \aaa})$ are linear in $\psi_+$ and 
$\psi_-^{\epsilon, \aaa}$ with coefficients which are smooth coefficients of $\psi_+$ and $\psi_-^{\epsilon, \aaa}$,  are supported in $[-T_1, -T_0] \times S^1$ and $(- \infty, -T_0-1] \times S^1$ respectively,  and satisfy  
\begin{align}\label{eqn: linear term}
|\mathcal{L}_\pm (\psi_+(x),\psi_-^{\epsilon, \aaa}(x))| &< (c_1(a)\epsilon +c_2(b)) \cdot (|\psi_+(x)|+ |\psi_-^{\epsilon, \aaa}(x)|),
\end{align} 
at every point $x$ of the domain, $c_1(a)$ is a constant which depends only on $a$, $c_2(b)$ depends only on $b$, and $\lim_{b\to 0} c_2(b)=0$. 
\item ${\mathcal Q}_\pm$ are quadratic functions of $\psi_+$ and $\psi_-^{\epsilon, \aaa}$ with coefficients which are smooth functions of $\psi_+$ and $\psi_-^{\epsilon, \aaa}$, and there exists $C>0$ such that
\begin{equation} \label{eqn: quadratic term}
|\mathcal{Q}_\pm(\psi_+(x),\psi_-^{\epsilon, \aaa}(x))| < C (|\psi_+(x)|^2 +|\psi_-^{\epsilon, \aaa}(x)|^2)
\end{equation}
at every point $x$ of the domain.
\item $\mathcal{L}_+=0$ and $\mathcal{Q}_+=0$ for $s\leq -T_1$ and $\mathcal{L}_-, \mathcal{Q}_-$ can be extended smoothly to $\mathcal{L}_-=0$ and $\mathcal{Q}_-=0$ for $s\geq -T_0$. 
\ee
\end{lemma}

\begin{proof}
Over the domain $(-\infty,-T_0]\times S^1$, we have 
$$\overline \bdry_{N_*^{\epsilon, \aaa}, f_\epsilon} (\beta_+ \psi_+ + \beta_- \psi_-^{\epsilon, \aaa})=\mathcal{D}_\epsilon  (\eta_*^{\epsilon, \aaa}+ \beta_+ \psi_+ + \beta_- \psi_-^{\epsilon, \aaa})$$ 
by Claim \ref{leftovers}. 
Writing $\eta^{\epsilon, \aaa}=\eta_*^{\epsilon, \aaa}+ \beta_+ \psi_+ + \beta_- \psi_-^{\epsilon, \aaa}$ we expand
$$\mathcal{D}_\epsilon \eta^{\epsilon, \aaa}={\bdry\over \bdry s}(\eta_*^{\epsilon, \aaa} + \beta_+\psi_+ + \beta_-\psi_-^{\epsilon, \aaa}) + j_0 {\bdry \over \bdry t}(\eta_*^{\epsilon, \aaa} + \beta_+\psi_+ + \beta_-\psi_-^{\epsilon, \aaa})  - \nabla f_\epsilon(\eta_*^{\epsilon, \aaa} + \beta_+\psi_+ + \beta_-\psi_-^{\epsilon, \aaa} ).$$
Using the Taylor expansion of type ($II$) we write
\begin{align*}
\nabla f_\epsilon(\eta_*^{\epsilon, \aaa} + \beta_+\psi_+ + \beta_-\psi_-^{\epsilon, \aaa} )&= \nabla f_\epsilon(\eta_*^{\epsilon, \aaa}) + \beta_+\mathbf{H}f_\epsilon(\eta_*^{\epsilon, \aaa}) \psi_+ + \beta_- \mathbf{H} f_\epsilon(\eta_*^{\epsilon, \aaa}) \psi_-^{\epsilon, \aaa} \\
& \qquad \qquad  -\mathcal{Q}(\beta_+\psi_+,\beta_-\psi_-^{\epsilon, \aaa}),
\end{align*}
where $\mathcal{Q}$ is a quadratic function of $\beta_+\psi_+,\beta_-\psi_-^{\epsilon, \aaa}$ with coefficients which are smooth functions of $\beta_+\psi_+,\beta_-\psi_-^{\epsilon, \aaa}$. 
Then
\begin{align}  \label{bunch of terms}
\mathcal{D}_\epsilon \eta^{\epsilon, \aaa} = &  \left({\bdry \eta_*^{\epsilon, \aaa} \over \bdry s}+j_0{\bdry \eta_*^{\epsilon, \aaa} \over \bdry t}-\nabla f_\epsilon(\eta_*^{\epsilon, \aaa}) \right)\\
\nonumber &  + \beta_+\left( {\bdry \psi_+\over \bdry s}+j_0 {\bdry \psi_+\over \bdry t}- \mathbf{H} f(\eta_+) \psi_+   \right)\\
\nonumber &  + \beta_-\left( {\bdry \psi_-^{\epsilon, \aaa}\over \bdry s}+j_0 {\bdry \psi_-^{\epsilon, \aaa}\over \bdry t}- \mathbf{H} f_\epsilon (\eta_-^{\epsilon, \aaa}) \psi_-^{\epsilon, \aaa}  \right)\\
\nonumber &  +\beta_+(\mathbf{H} f(\eta_+) -\mathbf{H} f_\epsilon (\eta_*^{\epsilon, \aaa}))\psi_+  
+ \beta_-(\mathbf{H} f_\epsilon (\eta_-^{\epsilon, \aaa}) - \mathbf{H} f_\epsilon (\eta_*^{\epsilon, \aaa})) \psi_-^{\epsilon, \aaa} \\
\nonumber & \qquad \qquad \qquad  + \beta_+'(s) \psi_+ + \beta_-'(s) \psi_-^{\epsilon, \aaa}\\
\nonumber &  +\mathcal{Q}(\beta_+\psi_+,\beta_-\psi_-^{\epsilon, \aaa}).
\end{align}
The right-hand side of the first line is $\mathcal{D}_\epsilon \eta_*$ and the second line is $ \beta_+ D_+\psi_+ + \beta_- D_-^{\epsilon, \aaa}\psi_-^{\epsilon, \aaa}$. Let us define $g(y, \theta)= \phi(y)\overline{g}_{\mathcal N}(\theta)$. We denote 
\begin{equation}
c_1= | \mathbf{H}g |_{C^0} \quad \mbox{and} \quad  c_2(b)= \max \{|\beta_+'|_{C^0}, |\beta_-' |_{C^0} \}.
\end{equation} 
Using the fact that $\eta_*^{\epsilon, \aaa}$, $\eta_+$ and $\eta_-^{\epsilon, \aaa}$ take values in $A_a$, where $\mathbf{H}f$ is constant, for $s \le -T_0$, the terms of the third line can be bounded as follows:
\begin{align} \label{clarabella}
  |\beta_+'(s) \psi_+(x) + \beta_-'(s) \psi_-^{\epsilon, \aaa}(x)|  \leq &  c_2(b)(|\psi_+(x)| +|\psi_-^{\epsilon, \aaa}(x)|)  \\ \nonumber
  |\beta_+(\mathbf{H} f(\eta_+) -\mathbf{H} f_\epsilon (\eta_*^{\epsilon, \aaa}))\psi_+(x)|  \le & \epsilon c_1 \beta_+|\psi_+(x)|, \\ \nonumber
  |\beta_-(\mathbf{H} f_\epsilon(\eta_-^{\epsilon, \aaa}) -\mathbf{H} f_\epsilon (\eta_*^{\epsilon, \aaa}))\psi_-^{\epsilon, \aaa}(x)|  \le & 2 \epsilon c_1 \beta_- |\psi_-^{\epsilon, \aaa}(x)|.
\end{align}

We then set
\begin{align} 
\label{def of L} \mathcal{L}_+^{(1)}(\psi_+,\psi_-^{\epsilon, \aaa})&= (\mathbf{H} f(\eta_+) -\mathbf{H} f_\epsilon(\eta_*^{\epsilon, \aaa}))\psi_+  + (\beta_+'(s) \psi_+ + \beta_-'(s) \psi_-^{\epsilon, \aaa}),\\
\label{def of L minus} \mathcal{L}_-^{(1)}(\psi_+,\psi_-^{\epsilon, \aaa})&=(\mathbf{H} f_\epsilon(\eta_-^{\epsilon, \aaa}) -\mathbf{H} f_\epsilon (\eta_*^{\epsilon, \aaa}))\psi_-^{\epsilon, \aaa}+ (\beta_+'(s) \psi_+ + \beta_-'(s) \psi_-^{\epsilon, \aaa}),
\end{align}
and $ \beta_\pm \mathcal{L}_\pm^{(1)}(\psi_+,\psi_-^{\epsilon, \aaa})$ satisfies Inequality \eqref{eqn: linear term}. The terms $\mathcal{L}_+^{(1)}$ and $\mathcal{L}_-^{(1)}$ are not necessarily supported in $[-T_1, -T_0] \times S^1$ and $(- \infty, -T_0-1] \times S^1$ respectively, and therefore we rearrange
\begin{align*}
\beta_+\mathcal{L}_+^{(1)} +\beta_-\mathcal{L}_-^{(1)}&= \beta_+(\beta_++\beta_-)^2\mathcal{L}_+^{(1)}+  \beta_-(\beta_++\beta_-)^2\mathcal{L}_-^{(1)}\\
&= \beta_+\mathcal{L}_+(\psi_+,\psi_-^{\epsilon, \aaa})  + \beta_-\mathcal{L}_-(\psi_+,\psi_-^{\epsilon, \aaa}),
\end{align*}  
where
\begin{align}\label{eqn: defnitions of L plus and L minus}
\mathcal{L}_+(\psi_+,\psi_-^{\epsilon, \aaa})&=\beta_+^2 \mathcal{L}_+^{(1)} + 2\beta_+\beta_- \mathcal{L}_+^{(1)} + \beta_+\beta_- \mathcal{L}_-^{(1)},\\
\nonumber \mathcal{L}_-(\psi_+,\psi_-^{\epsilon, \aaa})&=\beta_+\beta_-\mathcal{L}_+^{(1)} + 2\beta_+\beta_-\mathcal{L}_-^{(1)} + \beta_-^2 \mathcal{L}_-^{(1)},
\end{align}
the same inequalities hold for $\mathcal{L}_\pm^{(1)}$ and $\mathcal{L}_\pm$ and (5) holds for $\mathcal{L}_\pm$. However, the constants $c_1$ and $c_2$ in the statement are closely related to the constants $c_1$ and $c_2$ in Equations \eqref{clarabella} but not exactly the same.
Finally we can decompose and rearrange so that 
$$\mathcal{Q}(\beta_+\psi_+,\beta_-\psi_-^{\epsilon, \aaa})= \beta_+ \mathcal{Q}_+(\psi_+,\psi_-^{\epsilon, \aaa}) + \beta_- \mathcal{Q}_-(\psi_+, \psi_-^{\epsilon, \aaa}),$$
Inequality~\eqref{eqn: quadratic term} holds, and (5) holds for $\mathcal{Q}_\pm$. This completes the proof of the lemma.
\end{proof}

In particular, Equation \eqref{eq: almost obstruction bundle} is satisfied on $(- \infty, - T_0] \times S^1$ if the pair of equations hold: 
\begin{gather} 
\label{eqn: plus} \mathcal{D}_\epsilon \eta_*^{\epsilon, \aaa} +  D_+ \psi_+   +\mathcal{L}_+(\psi_+,\psi_-^{\epsilon, \aaa}) + \mathcal{Q}_+ (\psi_+,\psi_-^{\epsilon, \aaa}) = \mathfrak{p}_\epsilon(\aaa)\nu,\\
\label{eqn: minus} \mathcal{D}_\epsilon \eta_*^{\epsilon, \aaa} +  D_-^{\epsilon, \aaa} \psi_-^{\epsilon, \aaa}   +\mathcal{L}_-(\psi_+,\psi_-^{\epsilon, \aaa}) + \mathcal{Q}_- (\psi_+,\psi_-^{\epsilon, \aaa}) =\aaa\nu.
\end{gather}
The term $\aaa\nu$ in the second equation is legitimate because $\nu$ is supported on $[-T_0-1,-T_0]\times S^1$ where $\beta_-=0$.  It was chosen to make $\mathcal{D}_\epsilon \eta_*^{\epsilon, \aaa}-\aaa\nu$ small (independent of $\aaa$) in the sense of Estimate~\eqref{bound for preglued curve}.

\begin{rmk} 
More in line with the obstruction gluing of \cite{HT1}, Equation \eqref{eq: almost obstruction bundle} can be split into:
\begin{gather*}
 \overline{D}^\delta_+ \psi_+   +\Pi(\mathcal{D}_\epsilon\eta_*^{\epsilon, \aaa} + \mathcal{L}_+(\psi_+,\psi_-^{\epsilon, \aaa}) + \mathcal{Q}_+ (\psi_+,\psi_-^{\epsilon, \aaa}))=0,\\
(1-\Pi)(\mathcal{D}_\epsilon\eta_*^{\epsilon, \aaa} + \mathcal{L}_+(\psi_+,\psi_-^{\epsilon, \aaa}) + \mathcal{Q}_+ (\psi_+,\psi_-^{\epsilon, \aaa}))= \mathfrak{p}_\epsilon(\aaa)\nu,\\
 D_-^{\epsilon, \aaa, \delta} \psi_-^{\epsilon, \aaa} +(\mathcal{D}_\epsilon \eta_*^{\epsilon, \aaa} -\aaa\nu)  +\mathcal{L}_-(\psi_+,\psi_-^{\epsilon, \aaa}) + \mathcal{Q}_- (\psi_+,\psi_-^{\epsilon, \aaa}) =0,
\end{gather*}
where $\Pi \colon \mathcal{H}_{0,g_\delta}(\dot F, \Lambda^{0,1} N_+) \to {\mathcal H}_{+,\delta}'$ is the projection from Section~\ref{ss: linearized operators} and the second equation is always satisfied.
\end{rmk}

We say that $\mathcal{Q}(\psi_+)$ is {\em type 1 quadratic} if it can be written as
$$\mathcal{Q}(\psi_+)= P(\psi_+) + Q(\psi_+)\cdot \nabla \psi_+,$$
where there exists a constant $C>0$ such that $|P(\psi_+(x))|< C|\psi_+(x)|^2$ and $|Q(\psi_+)(x)| \leq C|\psi_+(x)|$ at every point $x$ of the domain.

\begin{rmk}
 The reason for the different treatment of the term $\widetilde{\partial}_\theta$ compared to
  the other infinitesimal deformations of the map $u_+$ is that the term $\beta'(s)\widetilde{\partial}_\theta$ which would appear in Equation~\eqref{clarabella} cannot be made small in ${\mathcal H}_{0, g_\delta}(\dot F, \Lambda^{0,1}N_+)$ by choosing $b$ and $\epsilon$ small.
\end{rmk}

\begin{lemma}\label{lemma: equivalent 2}
Over the domain $\dot F- (-\infty,-T_0)\times S^1$ we can expand:
\begin{equation}\label{eqn: poire}
\overline \bdry_{N_*^{\epsilon, \aaa}, f_\epsilon} (\beta_- \psi_-^{\epsilon,\aaa} + \beta_+ \psi_+)= D_+ \psi_+ + \mathcal{Q}(\psi_+),
\end{equation}
where $\mathcal{Q}(\psi_+)$ is {\em type 1 quadratic}, and Equations~\eqref{expansion} and ~\eqref{eqn: poire} agree along $s=-T_0$.
\end{lemma} 

\begin{proof}
Over the domain $\dot F- (-\infty,-T_0)\times S^1$, $\beta_+=1$, $\beta_-=0$, $u_*^{\epsilon, \aaa}=u_+$ and $u^{\epsilon, \aaa}= \op{exp}_{u_*^{\epsilon, \aaa}}\psi_+=\op{exp}_{u_+} \psi_+$.  Hence $\psi_+$ satisfies Equation~\eqref{eqn: normal d bar equation for u plus} with $\psi_+$ instead of $\xi$ and $(\sigma,\tau)=(s,t)$. Equation~\eqref{eqn: poire} then follows from Equation~\eqref{eqn: normal d bar equation for u plus} together with \eqref{eqn: initial conditions} by applying the Taylor expansion of type ($I$) to $a$, $b$ and $X$. The agreement of Equations~\eqref{expansion} and ~\eqref{eqn: poire} along $s=-T_0$ is a consequence of the definition of $\op{exp}_{u_+}$ for $s\leq -T_0$.
\end{proof}

\subsection{Proof of Theorem~\ref{thm: 1}} \label{subsection: gluing 1}

In this subsection and the next, we use the convention that constants such as $C$, $c_1$, $c_2(b)$ may change from line to line when making estimates. Recall that 
\begin{equation}
\lambda:=\min (\lambda_1,|\lambda_{-1}|) > 5\delta.
\end{equation}

\begin{lemma}\label{lemma: estimate for D eta star}
There exists a constant $C>0$ such that
\begin{equation}\label{bound for preglued curve}
\|\mathcal{D}_\epsilon \eta_*^{\epsilon,\aaa}  - \aaa \nu \|_{g_\delta}\leq C (e^{(\delta-\lambda)T_2(\epsilon)}+\epsilon T_1e^{\delta T_1}).
\end{equation}
\end{lemma}

\begin{proof}
By Lemma~\ref{lemma: where d bar is zero}, it suffices to estimate $\|\mathcal{D}_\epsilon \eta_*^{\epsilon, \aaa} \|_{g_\delta}$ on $[-T_2,-T_2+1]\times S^1$ and $[-T_1,-T_0]\times S^1$.  We will use the simple fact that the Morrey norm of a continuous function on a compact domain is dominated by the $C^0$ norm.

 First we estimate $\mathcal{D}_\epsilon \eta_*^{\epsilon,\aaa}  - \aaa \nu$ on $[-T_2,-T_2+1]\times S^1$, where $\nu=0$. By the definition of $\eta_*^{\epsilon, \aaa}$ (Equation~\eqref{equation for eta star}) and Equation ($J_{f_\epsilon}$), and with the understanding that all maps and norms are  restricted to $[-T_2, -T_2+1] \times S^1$ (on which $f_\epsilon(y,\theta)=\tfrac{1}{2}y^2+\epsilon \overline{g}_{\mathcal{N}}(\theta)$, $\nabla f_\epsilon=y\bdry_y + \epsilon\bdry_\theta$, and $\mathcal{D}_0$ is linear), we have
\begin{equation*}
{\mathcal D}_\epsilon (\eta_*^{\epsilon, \aaa})= {\mathcal D}_\epsilon (\eta_-^{\epsilon,\aaa})+ \beta(s+T_2) {\mathcal D}_0 (\eta_+) + \beta'(s+T_2) \eta_+ = \beta'(s+T_2) \eta_+
\end{equation*}
and therefore 
\begin{equation}\label{farfalle}
\|\mathcal{D}_\epsilon \eta_*^{\epsilon, \aaa} \|_{g_\delta} \leq C\left \|\textstyle\sum_{i=1}^\infty c_i e^{\lambda_i s} g_i(t)\right \|_{g_\delta} \leq C e^{(\delta-\lambda)(T_2(\epsilon)-T_0)} \leq Ce^{(\delta-\lambda)T_2(\epsilon)}.
\end{equation}

Next we estimate $\mathcal{D}_\epsilon \eta_*^{\epsilon,\aaa}  - \aaa \nu$ on $[-T_1,-T_0] \times S^1$. 
The restriction of $\eta_*^{\epsilon, \aaa}$ to $[-T_1, -T_0] \times S^1$ takes values in the region where $\nabla f_\epsilon$ no longer has the simple expression which leads to Equation ($J_{f_\epsilon})$, but from Equation \eqref{eqn: f and f epsilon} we obtain 
\begin{equation} \label{bruchi2}
{\mathcal D}_\epsilon(\eta_*^{\epsilon, \aaa})=\mathcal{D}_0(\eta_*^{\epsilon, \aaa})- \epsilon \nabla g(\eta_*^{\epsilon, \aaa}),
\end{equation}
where $g(y, \theta)= \phi(y) \overline{g}_{\mathcal N}(\theta)$. By the definition of $\eta_*^{\epsilon, \aaa}$ and $\nu$, Equation \eqref{bruchi2}, and with the understanding that all maps and norms are restricted to $[-T_1, -T_0] \times S^1$, we have
\begin{align} \label{farfalle 2}
\|\mathcal{D}_\epsilon \eta_*^{\epsilon, \aaa}+ \beta'(-s-T_0)\cdot(0,\aaa) \|_{g_\delta} & \le \| \tfrac{\bdry}{\bdry s}(\beta(\tfrac{s+T_0}{-T_1+T_0})(0, \epsilon (s+T_1)))\|_{g_\delta} + \epsilon \| \nabla g (\eta_*^{\epsilon, \aaa}) \|_{g_\delta} \\ \nonumber & \le C \epsilon T_1 e^{\delta T_1}.
\end{align}

Estimates~\eqref{farfalle} and \eqref{farfalle 2} imply Estimate~\eqref{bound for preglued curve}.
\end{proof}

\begin{rmk}
  Since $\mathcal{D}_\epsilon \eta_*^{\epsilon, \aaa} + \beta'(-s-T_0)\cdot(0,\aaa)$ is supported in $(- \infty, -T_0] \times S^1$, where $g_\delta$ and $h_\delta$ coincide, we can also regard it as an element of ${\mathcal H}_{-, \delta}'$ with norm
 $$\|\mathcal{D}_\epsilon \eta_*^{\epsilon, \aaa} + \beta'(-s-T_0)\cdot(0,\aaa) \|_{h_\delta}= \|\mathcal{D}_\epsilon \eta_*^{\epsilon, \aaa} + \beta'(-s-T_0)\cdot(0,\aaa) \|_{g_\delta}.$$
\end{rmk}

Let $(\overline{D}_+^\delta)^{-1}$ be the inverse of $\overline{D}_+^\delta$, viewed as a map to the orthogonal complement $\mathcal{H}_{+,\delta}^\perp$ of $\ker \overline D_+^\delta$, and let $(D_-^{\epsilon, \aaa, \delta})^{-1}$ be the inverse of $D_-^{\epsilon, \aaa, \delta}$. Recall that the norm of $(D_-^{\epsilon, \aaa, \delta})^{-1}$ is uniformly bounded in $\epsilon$ and $\aaa$ by Lemma \ref{bounded inverses}. Let
\begin{align}
\mathcal{B}_+ &= \mbox{closed ball of radius $\widetilde\epsilon$ in $\mathcal{H}_{+,\delta}^\perp$}, \\
\mathcal{B}_- & = \mbox{closed ball of radius $\widetilde \epsilon$ in ${\mathcal H}_{-, \delta}$},
\end{align} 
where the small constant $\widetilde\epsilon>0$ is to be determined more precisely later.

Let $\mathcal{I}_+: \mathcal{B}_+\times \mathcal{B}_-\to \mathcal{H}_{+,\delta}^\perp$ and $\mathcal{I}_-: \mathcal{B}_+\times \mathcal{B}_-\to \mathcal{H}_{-,\delta}$ be maps given by
\begin{align}
\mathcal{I}_+(\psi_+,\psi_-^{\epsilon, \aaa}) &= - (\overline D_+^\delta)^{-1} \Pi (\mathcal{F}_+(\psi_+,\psi_-^{\epsilon, \aaa})),\\
\mathcal{I}_-(\psi_+,\psi_-^{\epsilon, \aaa}) &= - (D_-^{{\epsilon, \aaa}, \delta})^{-1}(\mathcal{F}_-(\psi_+,\psi_-^{\epsilon, \aaa})),
\end{align}
where 
$$\mathcal{F}_+(\psi_+,\psi_-^{\epsilon, \aaa}):= \left\{  
\begin{array}{cl} 
\mathcal{D}_\epsilon \eta_*^{\epsilon, \aaa} + \mathcal{L}_+(\psi_+,\psi_-^{\epsilon, \aaa}) +\mathcal{Q}_+(\psi_+,\psi_-^{\epsilon, \aaa} ) & \mbox{ on $(-\infty,-T_0]\times S^1$}\\
\mathcal{Q}(\psi_+) & \mbox{ on $\dot F- (-\infty,-T_0]\times S^1$},
\end{array} \right.
$$
$$\mathcal{F}_-(\psi_+,\psi_-^{\epsilon, \aaa}):= \left\{
\begin{array}{cl}
\mathcal{D}_\epsilon \eta_*^{\epsilon, \aaa} -\aaa\nu + \mathcal{L}_-(\psi_+,\psi_-^{\epsilon, \aaa}) +\mathcal{Q}_-(\psi_+,\psi_-^{\epsilon, \aaa} ) & \mbox{ on $(-\infty,-T_0]\times S^1$}\\
0 & \mbox{ on $[-T_0,\infty)\times S^1$}.
\end{array} \right.
$$ 
Here the definitions of $\mathcal{F}_+(\psi_+,\psi_-^{\epsilon, \aaa})$ agree on $\{-T_0\}\times S^1$ by Lemma~\ref{lemma: equivalent 2} and the definitions of $\mathcal{F}_-(\psi_+,\psi_-^{\epsilon, \aaa})$ agree on $\{-T_0\}\times S^1$ by Lemma~\ref{lemma: equivalent 1}(5) and Lemma~\ref{lemma: where d bar is zero}. 

Solving Equations \eqref{eqn: plus} and \eqref{eqn: minus} is then equivalent to solving the equations
\begin{gather} \label{fixed point plus}
\psi_+ =  {\mathcal I}_+(\psi_+, \psi_-^{\epsilon, \aaa}), \\
\label{fixed point minus} 
\psi_-^{\epsilon, \aaa} =  {\mathcal I}_-(\psi_+, \psi_-^{\epsilon, \aaa}).
\end{gather}

The following two lemmas provide the necessary estimates to apply the contraction mapping theorem to Equations \eqref{fixed point plus} and \eqref{fixed point minus}.

\begin{noti}
  We will sometimes write:
\begin{equation}
\| \cdot \|_{*, \delta}:= \| \cdot \|_{*,g_\delta} \quad \mbox{or} \quad \| \cdot \|_{*, h_\delta}, 
\end{equation}
depending on the context.
\end{noti}

\begin{lemma} \label{qui}
If $(\psi_+, \psi_-^{\epsilon, \aaa}) \in {\mathcal B}_+ \times {\mathcal B}_-$, then 
\begin{align} \label{eqn: estimate for I plus}
&\| \mathcal{I}_\pm(\psi_+,\psi_-^{\epsilon, \aaa})\|_{*,\delta}  \leq   C (e^{(\delta-\lambda)T_2(\epsilon)}+\epsilon T_1(b)e^{\delta T_1(b)}) \\
\nonumber & + (c_1\epsilon + c_2(b))(\|\psi_+\|_{*,g_\delta} + \|\psi_-^{\epsilon, \aaa}\|_{*,h_\delta}) + C( \|\psi_+\|_{*,g_\delta}^2+ \|\psi_-^{\epsilon, \aaa}\|_{*,h_\delta}^2),
\end{align}
where $c_1$ is constant and $\lim_{b\to 0} c_2(b)=0$. 
\end{lemma}

\begin{proof}
We will carry out estimates on the $(-\infty,-T_0]\times S^1$ portion, with the understanding that the norms are restricted to $(-\infty,-T_0]\times S^1$, where $g_\delta=h_\delta$. (This justifies the use of the weight $g_\delta$ throughout the proof, even where one should expect $h_\delta$.) The estimates on the $\dot F- (-\infty,-T_0]\times S^1$ portion, which involve only $\mathcal{I}_+$ and $\psi_+$, are straightforward and are left to the reader.

By the definitions of $\mathcal{I}_\pm(\psi_+, \psi_-^{\epsilon, \aaa}) $,
\begin{equation}\label{noia} 
\| {\mathcal I}_\pm(\psi_+, \psi_-^{\epsilon, \aaa}) \|_{*, g_\delta} \le C(\| \mathcal{D}_\epsilon \eta_*-\aaa \nu \|_{g_\delta} + \| \mathcal{L}_\pm(\psi_+,\psi_-^{\epsilon, \aaa}) \|_{g_\delta} + \|  \mathcal{Q}_\pm (\psi_+,\psi_-^{\epsilon, \aaa})\|_{g_\delta}),
\end{equation}
since $(\overline{D}_+^\delta)^{-1}$ is bounded, $(D_-^{{\epsilon, \aaa}, \delta})^{-1}$ are uniformly bounded, and $\Pi(\mathcal{D}_\epsilon \eta_*)= \Pi (\mathcal{D}_\epsilon \eta_*-\aaa\nu)$.

We will make frequent use of the estimate
\begin{equation} \label{estimate}
|\zeta|_{C^0}\leq C\|\zeta\|_{*, g_\delta},
\end{equation} 
which follows from Lemma~\ref{lemma: sobolev embedding} and $\| \zeta \|_* \le \| \zeta \|_{*, g_\delta}$ since $g_\delta >1$.

By Equation~\eqref{bound for preglued curve}, 
\begin{equation}\label{kyoho 1}
\|\mathcal{D}_\epsilon \eta_*^{\epsilon, \aaa}-\aaa \nu \|_{g_\delta}\leq C (e^{(\delta-\lambda)T_2(\epsilon)}+\epsilon T_1e^{\delta T_1}).
\end{equation}
Next, since $\mathcal{L}_\pm$ satisfies Estimate~\eqref{eqn: linear term},
\begin{align} \label{kyoho 2}
\|\mathcal{L}_\pm(\psi_+,\psi_-^{\epsilon, \aaa})\|_{g_\delta} & \leq (c_1\epsilon + c_2(b)) (\|\psi_+\|_{g_\delta} + \|\psi_-^{\epsilon, \aaa}\|_{g_\delta}).
\end{align}
Finally, since $\mathcal{Q}_\pm$ satisfies Estimate~\eqref{eqn: quadratic term},
\begin{align} \label{kyoho 3}
\|\mathcal{Q}_\pm(\psi_+,\psi_-^{\epsilon, \aaa})\|_{g_\delta} & \leq C(| \psi_+|_{C^0}\|\psi_+\|_{g_\delta} + | \psi_-^{\epsilon, \aaa}|_{C^0} \|\psi_-^{\epsilon, \aaa}\|_{g_\delta})\\
\nonumber & \leq C(\|\psi_+\|_{*,g_\delta}\|\psi_+\|_{g_\delta} + \|\psi_-^{\epsilon, \aaa}\|_{*,g_\delta} \|\psi_-^{\epsilon, \aaa}\|_{g_\delta}),
\end{align}
using Estimate~\eqref{estimate}. We explain the first line of Estimate~\eqref{kyoho 3}. The first term of the Morrey norm is the weighted $L^2$-norm, and we can bound:
\begin{align*}
\int_{\dot F}g_\delta^2 |\mathcal{Q}_+(\psi_+,\psi_-^{\epsilon, \aaa})|^2 & \leq \int_{\dot F} g_\delta^2 C^2 (|\psi_+|^4 +|\psi_-^{\epsilon, \aaa}|^4)\\
& \leq C^2 |\psi_+|_{C_0}^2 \int_{\dot F}g_\delta^2 |\psi_+|^2  + C^2 |\psi_-^{\epsilon, \aaa}|_{C_0}^2 \int_{\dot F} g_\delta^2 |\psi_-^{\epsilon, \aaa}|^2\\
\left(\int_{\dot F}g_\delta^2 |\mathcal{Q}_+(\psi_+,\psi_-^{\epsilon, \aaa})|^2\right)^{1/2}& \leq C|\psi_+|_{C_0}\left( \int_{\dot F}g_\delta^2 |\psi_+|^2   \right)^{1/2}
+C|\psi_-^{\epsilon, \aaa}|_{C_0} \left( \int_{\dot F} g_\delta^2 |\psi_-^{\epsilon, \aaa}|^2  \right)^{1/2}.
\end{align*}
The bound for the second term of the Morrey norm is similar, since it is of $L^2$ type. 

Estimates~\eqref{noia},\eqref{kyoho 1}--\eqref{kyoho 3}, together with Estimate~\eqref{noia}, give Estimate~\eqref{eqn: estimate for I plus}.  Here we are using the trivial observation  $\|\cdot \|_{g_\delta}\leq \|\cdot \|_{*,g_\delta}$.
\end{proof}

\begin{lemma} \label{quo} 
If $(\psi_+, \psi_-^{\epsilon, \aaa}), (\overline{\psi}_+, \overline{\psi}_-^{\epsilon, \aaa})\in {\mathcal B}_+ \times {\mathcal B}_-$, then
\begin{align} \label{second estimate} \|\mathcal{I}_\pm(\psi_+,\psi_-^{\epsilon, \aaa})-\mathcal{I}_\pm&(\overline{\psi}_+,\overline{\psi}_-^{\epsilon, \aaa})\|_{*,\delta} \leq  \\ \nonumber
  & (c_1\epsilon  + c_2(b)+C\widetilde{\epsilon}) (\|\psi_+-\overline{\psi}_+\|_{*,g_\delta} +
 \|\psi_-^{\epsilon, \aaa} -\overline{\psi}_-^{\epsilon, \aaa}\|_{*,h_\delta} ).
\end{align}
\end{lemma}

\begin{proof}
Again we carry out the estimate on the $(-\infty,-T_0]\times S^1$ portion, with the understanding that the norms are restricted to $(-\infty,-T_0]\times S^1$, and leave the estimate on the $\dot F- (-\infty,-T_0]\times S^1$ portion to the reader. 

In the following equation $\mathfrak{D}_\pm$ stands for either $\overline{D}_+^\delta$ or $D_-^{{\epsilon, \aaa}, \delta}$. We have:
\begin{align*}
\mathcal{I}_\pm(\psi_+,\psi_-^{\epsilon, \aaa})-\mathcal{I}_\pm(\overline{\psi}_+, \overline{\psi}_-^{\epsilon, \aaa}) &=   -\mathfrak{D}_\pm^{-1}((\mathcal{L}_\pm(\psi_+,\psi_-^{\epsilon, \aaa})- \mathcal{L}_\pm(\overline{\psi}_+, \overline{\psi}_-^{\epsilon, \aaa}))\\
& \qquad\qquad  + (\mathcal{Q}_\pm(\psi_+,\psi_-^{\epsilon, \aaa}) - \mathcal{Q}_\pm(\overline{\psi}_+, \overline{\psi}_-^{\epsilon, \aaa}))).
\end{align*}

By Equations~\eqref{def of L}, \eqref{def of L minus}, \eqref{eqn: defnitions of L plus and L minus}, as well as an analog of Estimate~\eqref{eqn: linear term}, we have
\begin{align*}
\|\mathcal{L}_\pm(\psi_+,\psi_-^{\epsilon, \aaa})- \mathcal{L}_\pm(\overline{\psi}_+, \overline{\psi}_-^{\epsilon, \aaa})\|_{g_\delta}&  \leq (c_1\epsilon  + c_2(b)) (\|\psi_+- \overline{\psi}_+\|_{g_\delta} + \|\psi_-^{\epsilon, \aaa} - \overline{\psi}_-^{\epsilon, \aaa} \|_{g_\delta}).
\end{align*}

By Lemma \ref{lemma: equivalent 1}(4)  
$\mathcal{Q}_\pm(\psi_+,\psi_-^{\epsilon, \aaa})$ is a quadratic function of $\psi_+,\psi_-^{\epsilon, \aaa}$ with uniformly bounded coefficients which are smooth functions of $\psi_+,\psi_-^{\epsilon, \aaa}$. Therefore
\begin{align*}
& \|\mathcal{Q}_\pm (\psi_+,\psi_-^{\epsilon, \aaa}) - \mathcal{Q}_\pm(\overline{\psi}_+, \overline{\psi}_-^{\epsilon, \aaa})\|_{g_\delta} \leq \\ & C(| \psi_+-\overline{\psi}_+|_{C^0}+ | \psi_-^{\epsilon, \aaa} -  \overline{\psi}_-^{\epsilon, \aaa}|_{C^0}) (\|\psi_+\|_{g_\delta}
 +\|\psi_-^{\epsilon, \aaa}\|_{g_\delta} + \|\overline{\psi}_+\|_{g_\delta}
 +\|\overline{\psi}_-^{\epsilon, \aaa}\|_{g_\delta}).
\end{align*}
Combining the two estimates and using \eqref{estimate} and $\|\cdot \|_{g_\delta}\leq \|\cdot \|_{*,g_\delta}$, we obtain Estimate~\eqref{second estimate}.
\end{proof}

\begin{prop} \label{wildfire}
There exists $\widetilde{\epsilon}>0$ sufficiently small such that, for all $b$, $\aaa_0$, and  $\epsilon_0=\epsilon_0(\aaa_0, b)$ sufficiently small  (in particular, satisfying ($\dagger_2$)) and for all $(\aaa,\epsilon) \in [-\aaa_0, \aaa_0]\times (0,\epsilon_0]$, there exists a unique $(\psi_+, \psi_-^{\epsilon, \aaa}) \in {\mathcal B}_+ \times {\mathcal B}_-$ satisfying 
\begin{equation}\label{this is the end}
    \mathcal{I}_+(\psi_+,\psi_-^{\epsilon, \aaa})= \psi_+, \quad \mathcal{I}_-(\psi_+,\psi_-^{\epsilon, \aaa})= \psi_-^{\epsilon, \aaa}.
\end{equation}
Moreover, the solutions of Equation~\eqref{this is the end} satisfy the estimate
\begin{equation} \label{trebisonda}
   \|\psi_+ \|_{*, g_\delta} + \| \psi_-^{\epsilon, \aaa} \|_{*, h_\delta} \le C(e^{(\delta- \lambda)T_2(\epsilon)} + \epsilon T_1(b)e^{\delta T_1(b)}).
\end{equation}
\end{prop}

\begin{proof}
Let ${\mathcal I}=({\mathcal I}_+, {\mathcal I}_-) : {\mathcal B}_+ \times {\mathcal B}_- \to {\mathcal H}_{+,\delta}^\perp \times {\mathcal H}_{-,\delta}$. Lemmas \ref{qui} and \ref{quo} imply that, for sufficiently small $\widetilde{\epsilon}$, there are sufficiently small constants $b$, $\aaa_0$, $\epsilon_0=\epsilon_0(\aaa_0)$ such that for all $(\aaa,\epsilon) \in [-\aaa_0, \aaa_0]\times (0,\epsilon_0]$ we have estimates
\begin{gather} \label{prima}
\|\mathcal I_{\pm}(\psi_+, \psi_-^{\epsilon, \aaa})\|_{*, \delta} \le  C(e^{(\delta- \lambda)T_2(\epsilon)} + \epsilon T_1(b)e^{\delta T_1(b)}) + \tfrac 14 (\| \psi_+ \|_{*, g_\delta} + \| \psi_-^{\epsilon, \aaa} \|_{*, h_\delta}) \\ 
\label{seconda}
\|\mathcal{I}_\pm(\psi_+,\psi_-^{\epsilon, \aaa})-\mathcal{I}_\pm(\overline{\psi}_+,\overline{\psi}_-^{\epsilon, \aaa})\|_{*,\delta} \leq
\tfrac 14  (\|\psi_+-\overline{\psi}_+\|_{*,g_\delta} + \|\psi_-^{\epsilon, \aaa} -\overline{\psi}_-^{\epsilon, \aaa}\|_{*,h_\delta} )
\end{gather}
and ${\mathcal I}$ is a contraction of ${\mathcal B}_+ \times {\mathcal B}_-$ (for the metric induced by the sum of the norms). Then the contraction mapping theorem implies that there is a unique pair $(\psi_+, \psi_-^{\epsilon,\aaa}) \in {\mathcal B}_+ \times {\mathcal B}_-$ such that ${\mathcal I}(\psi_+, \psi_-^{\epsilon, \aaa})=(\psi_+, \psi_-^{\epsilon, \aaa})$. Finally, Estimate~\eqref{trebisonda} is obtained by plugging Equation~\eqref{this is the end} in Equation~\eqref{prima}, rearranging the terms, and renaming the constant $C$.
\end{proof}

Proposition~\ref{wildfire} produces, provided $b$, $\aaa_0$, and $\epsilon_0=\epsilon_0(\aaa_0)$ are sufficiently small, a map
$$u^{\epsilon, \aaa}= \op{exp}_{u_*^{\epsilon, \aaa}}(\beta_+\psi_+ + \beta_-\psi_-^{\epsilon, \aaa})$$
for all $(\aaa,\epsilon) \in [-\aaa_0, \aaa_0]\times (0,\epsilon_0]$ and a continuous function $\mathfrak{p}_\epsilon : [-\aaa_0, \aaa_0] \to \R$ such that $\overline\bdry_{J_{f_\epsilon}}u^{\epsilon, \aaa}= \mathfrak{p}_\epsilon(\aaa) \nu$ for all $\epsilon\in(0,\epsilon_0]$. Moreover, by Equations~\eqref{bound for preglued curve} and~\eqref{trebisonda}, we have 
\begin{align*}
|\mathfrak{p}_\epsilon(\aaa)- \aaa| & \leq C(\| \mathfrak{p}_\epsilon(\aaa) \nu -\aaa \nu \|_{g_\delta}) \leq C( \|  \overline\bdry_{J_{f_\epsilon}}u^{\epsilon, \aaa} - \overline\bdry_{J_{f_\epsilon}}u_*^{\epsilon, \aaa}\|_{g_\delta} + \|\overline\bdry_{J_{f_\epsilon}}u_*^{\epsilon, \aaa} - \aaa \nu  \|_{g_\delta}) \\
& \leq C (\|\psi_+ \|_{*, g_\delta} + \| \psi_-^{\epsilon, \aaa} \|_{*, h_\delta}) + C( \|\mathcal{D}_\epsilon \eta_*^{\epsilon,\aaa} - \aaa \nu  \|_{g_\delta}) \\
& \leq  C(e^{(\delta- \lambda)T_2(\epsilon)} + \epsilon T_1(b)e^{\delta T_1(b)}).
\end{align*}
 In order to bound  $\|  \overline\bdry_{J_{f_\epsilon}}u^{\epsilon, \aaa} - \overline\bdry_{J_{f_\epsilon}}u_*^{\epsilon, \aaa}\|_{g_\delta}$ we used Lemma \ref{lemma: equivalent 1}, Lemma \ref{lemma: equivalent 2} and Equation \eqref{trebisonda}. 

If $\epsilon_0$ is sufficiently small, then $\mathfrak{p}_\epsilon(- \aaa_0)<0$ and $\mathfrak{p}_\epsilon(\aaa_0)>0$, and therefore $\mathfrak{p}_\epsilon$ has an odd number of zeros in the interval $[-\aaa_0, \aaa_0]$.

\subsection{Proof of Theorem~\ref{thm: 2}}\label{subsection: gluing 2}
 From now on, until the end of the appendix, we fix a $b$ such that Proposition \ref{wildfire} holds. {\em Therefore, from now on, $b$ and $T_1(b)$ are to be considered constants}. 

\begin{rmk} 
In \cite{Yao1}, the strategy for the proof of Theorem~\ref{thm: 2} is slightly different: One can actually differentiate the $1$-parameter family of functions $\psi_-^{\epsilon,\aaa}$ with respect to $\aaa$ to show that $\mathfrak{p}_\epsilon(\aaa)$ is $C^1$-close to $\aaa$ and hence that the zero is unique.  
\end{rmk}

Arguing by contradiction, suppose there are sequences $\{\kappa_i\}_{i=1}^\infty$, $\{\epsilon_i\}_{i=1}^\infty$, $\{u^{\epsilon_i}\}_{i=1}^\infty$, and $\{v^{\epsilon_i}\}_{i=1}^\infty$ (with $a,b,c,\epsilon'$ small but fixed), such that:
\be
\item $\kappa_i\to 0$ and $\epsilon_i\to 0$, 
\item $u^{\epsilon_i}, v^{\epsilon_i}: (\dot F,j_i)\to \R\times M$ are $\overline{\bdry}_{J_{f_{\epsilon_i}}}$-holomorphic and are not related by $\R$-translations in the target (and possibly the domain),  and
\item $u^{\epsilon_i}$ and $v^{\epsilon_i}$ are $\kappa_i$-close to breaking into $u_+$ and a cylinder over $\mathcal{T}^{\epsilon_i}_0$.
\ee

After translating the $u^{\epsilon_i}$ and $v^{\epsilon_i}$ in the target and possibly in the domain, we can find $T_1>0$ such that $u^{\epsilon_i}|_{(-\infty, -T_1]}$ and $v^{\epsilon_i}|_{(-\infty, -T_1]}$ have image in $\R\times [-b-\epsilon', b+\epsilon']\times T^2$ and $u^{\epsilon_i}|_{s=-T_1}$ and $v^{\epsilon_i}|_{s=-T_1}$ have image in $\R\times [b,b+\epsilon']\times T^2$. On $(-\infty, -T_1]\times S^1$ we write
$$u^{\epsilon_i}(s,t)=(s,t,\eta^{u^{\epsilon_i}}(s,t)) \quad \mbox{and} \quad v^{\epsilon_i}(s,t)=(s,t,\eta^{v^{\epsilon_i}}(s,t)).$$  

Recall that $\eta^{u^{\epsilon_i}}$ and $\eta^{v^{\epsilon_i}}$ satisfy Equation~\eqref{eqn: rewriting d bar equation}, which we repeat here:
$${\bdry \eta\over \bdry s} + j_0{\bdry \eta\over \bdry t}-\nabla f_\epsilon(\eta)=0.$$  
If we restrict $\eta^{u^{\epsilon_i}}$ and $\eta^{v^{\epsilon_i}}$ to any cylinder $[-T_1'(\epsilon_i),-T_1]\times S^1$ such that their images are in $B$, then Equation~\eqref{eqn: rewriting d bar equation} specializes to $(J_{f_{\epsilon_i}})$ and their difference $\zeta^{\epsilon_i}(s,t)=\eta^{v^{\epsilon_i}}(s,t)-\eta^{u^{\epsilon_i}}(s,t)$ satisfies the linear equation ($J_f$):
\begin{equation} \label{eqn: restated}
{\bdry \zeta^{\epsilon_i}\over \bdry s} -{\bf A}\zeta^{\epsilon_i}=0.
\end{equation}

Next, by the definition of $\kappa_i$-close to breaking and estimates on derivatives in the proof of Gromov-Hofer compactness, after applying the relevant translations in the domain or in the target and passing to a subsequence,  we can choose a sequence $T_2'(\epsilon_i)\to \infty$ such that $-T_2'(\epsilon_i)< -T_1$ and there are rough initial estimates
\begin{equation}\label{rough initiual estimates} \|\zeta^{\epsilon_i}_+(-T_{1}-1)\|_{L^2} \leq c \kappa_i, \quad \|\zeta^{\epsilon_i}_-(-T_2'(\epsilon_i))\|_{L^2}\leq c\kappa_i,
\end{equation}
where $c>0$ is independent of $\epsilon_i$ or $\kappa_i$. 



\s\n {\em Normalization.} We normalize $u^{\epsilon_i}$ so that, at $s=-T_1$, the $g_1$ term of the Fourier series of $\eta^{u^{\epsilon_i}}$ is equal to $(b,0)$ and the $g_0$ term is equal to $(0,h_{\epsilon_i})$, where $h_{\epsilon_i} \to 0$ as $i \to \infty$.  This is possible because $\eta^{u^{\epsilon_i}}(-T_1)$, before normalization, is close to $(b,0)$ and the $g_0$ term in the Fourier series for the negative end of $u_+$ vanishes.
Similarly we normalize $v^{\epsilon_i}$ by translating slightly in the target $\R$-direction, so that the $g_1$ term of $\zeta^{\epsilon_i}$ is zero.


\begin{defn}
An element of $\ker D_+^\delta$ or $\ker D_-^{\epsilon_i,\aaa, \delta}$ (from Section~\ref{ss: linearized operators}) is a {\em non-translation} element if it is nonzero and does not correspond to an $\R$-translation of the domain or target.  
\end{defn}

A sufficient condition for detecting a non-translation element of $\ker D_+^\delta$ is given in Lemma~\ref{non-translation element} in terms of the coefficient of $g_1$ in the Fourier expansion.

\s\n
{\em Idea of proof.}
The idea of the proof is to start with $v^{\epsilon_i}-u^{\epsilon_i}$ for $\epsilon_i>0$ small and construct a non-translation element of $\ker D_+^\delta$ or $\ker D_-^{\epsilon_i,0,\delta}$ (taking $\aaa=0$ suffices) by damping out and inverting the error.   The damping out occurs on a long neck region $[-T_3(\epsilon_i),-T_1]\times S^1$ (with $-T_3(\epsilon_i)$ defined later) that is mapped to $B$ by $\eta^{u^{\epsilon_i}}$ and $\eta^{v^{\epsilon_i}}$. The $D_+^\delta$ and $D_-^{\epsilon_i, 0, \delta}$ cases respectively correspond to Cases 1 and 2 below. (There is a slight complication in the $D_-^{\epsilon_i,0,\delta}$ case, which will be explained in Case 2.) The existence of a non-translation element is a contradiction.

\s
By Equation~\eqref{eqn: restated}, $\zeta^{\epsilon_i}|_{[-T_3(\epsilon_i),-T_1]\times S^1}$ can be written as a Fourier series
\begin{equation} \label{Fourier expansion} \zeta^{\epsilon_i}(s,t)=\sum_{j=-\infty}^\infty d^{\epsilon_i}_j e^{\lambda_j s} g_j(t).
\end{equation}
We write $(\cdot)_-$, $(\cdot)_0$, and $(\cdot)_+$ for the $L^2$-projections of $(\cdot)= \zeta^{\epsilon_i}$ etc.\ to the negative, null, and positive eigenspaces of ${\bf A}$, and write $(\cdot)(s_0)$ for $(\cdot)|_{s=s_0}$. By our normalization we may assume that $d_1^{\epsilon_i}=0$. 


\begin{lemma}
  Fix $T_1' \in [T_1+1, T_2'(\epsilon_i)]$. For
all $s\in [-T_1',-T_1]$,
\begin{align}
\|\zeta^{\epsilon_i}_+(s)\|_{L^2_1} & \leq \|\zeta^{\epsilon_i}_+(-T_1)\|_{L^2_1}\cdot e^{\lambda (s+T_1)}, \label{eqn:  basic estimate 1} \\
\|\zeta^{\epsilon_i}_-(s)\|_{L^2_1} &  \leq \|\zeta^{\epsilon_i}_-(-T_1')\|_{L^2_1}\cdot e^{-\lambda (s+T_1')}, \label{eqn:  basic estimate 2}
\end{align}
where $\lambda=\min (\lambda_{1},|\lambda_{-1}|)$ and $L^2_1$ refers to the $L^2$-Sobolev space with one derivative.  
\end{lemma}

\begin{proof}
  We prove the first inequality. By the Fourier expansion \eqref{Fourier expansion} and Parseval's identity we have
  $$\|\zeta^{\epsilon_i}_+(s)\|_{L^2_1}^2 = \sum_{j \ge 1} d_j^2(1+\lambda_j^2)e^{2 \lambda_j s},$$
  $$\|\zeta^{\epsilon_i}_+(-T_1)\|_{L^2_1}^2 \cdot e^{2\lambda (s+T_1)}= \sum_{j \ge 1} d_j^2(1+\lambda_j^2) e^{2(\lambda-\lambda_j)T_1+2 \lambda s}.$$
  Then Equation \eqref{eqn: basic estimate 1} follows from the inequality
  $$e^{\lambda_j s} \leq e^{(\lambda-\lambda_j)T_1+ \lambda s},$$
  which holds for $j>0$. To prove this inequality we divide the second term by the first and observe that $e^{(\lambda - \lambda_j)(T_1+s)} \ge 1$ because $\lambda - \lambda_j \le 0$ and $T_1 + s \le 0$.

  Now we prove the second inequality. We have:
  $$\|\zeta^{\epsilon_i}_-(s)\|_{L^2_1}^2 = \sum_{j<0} d_j^2 (1+\lambda_j^2) e^{2 \lambda_j s},$$
  $$\|\zeta^{\epsilon_i}_-(-T_1')\|_{L^2_1}^2\cdot e^{-2\lambda (s+T_1')} = \sum_{j<0} d_j^2 (1+\lambda_j^2) e^{-2(\lambda_j+ \lambda)T_1'- 2 \lambda s}.$$
  Then Equation  \eqref{eqn: basic estimate 1} follows from the inequality
  $$e^{\lambda_j s} \le e^{-(\lambda_j+ \lambda)T_1'- \lambda s},$$
  which holds for $j<0$ because $\lambda+ \lambda_j \le 0$ and $s+T_1' \ge 0$.
  \end{proof} 

There are two cases to consider:
\begin{itemize}
\item[(1)] $\|\zeta_0^{\epsilon_i}(-T_1-1)+ \zeta_+^{\epsilon_i}(-T_1-1)\|_{L^2_1}\geq \|\zeta_-^{\epsilon_i}(-T_2'(\epsilon_i))\|_{L^2_1}$ holds for infinitely many indices $i$, or 
\item[(2)] $\|\zeta_0^{\epsilon_i}(-T_1-1)+ \zeta_+^{\epsilon_i}(-T_1-1)\|_{L^2_1}\leq \|\zeta_-^{\epsilon_i}(-T_2'(\epsilon_i))\|_{L^2_1}$ holds for infinitely many indices $i$.
\end{itemize}
Note that the two cases are not mutually exclusive.

\s\n
{\bf Case 1.}  Up to extracting a subsequence, we assume that the following inequality holds for every 
$i$: 
\begin{equation}
\|\zeta_0^{\epsilon_i}(-T_1-1)+ \zeta_+^{\epsilon_i}(-T_1-1)\|_{L^2_1}\geq \|\zeta_-^{\epsilon_i}(-T_2'(\epsilon_i))\|_{L^2_1}.
\end{equation}
By Equation~\eqref{eqn: basic estimate 2} with $T_1'=T_2'(\epsilon_i)$ and $s=-T_1-1$ we have:
\begin{align} \label{poppycake}
\|\zeta_-^{\epsilon_i}(-T_1-1)\|_{L^2_1} & \leq \|\zeta_-^{\epsilon_i}(-T_2'(\epsilon_i))\|_{L^2_1}\cdot e^{\lambda(-T_2'(\epsilon_i)+T_1+1)}\\
\nonumber &\leq  \|\zeta_0^{\epsilon_i}(-T_1-1)+ \zeta_+^{\epsilon_i}(-T_1-1)\|_{L^2_1}\cdot e^{\lambda(-T_2'(\epsilon_i)+T_1+1)}.
\end{align}

Let $\xi^{\epsilon_i}=(\xi^{\epsilon_i}_1,\xi^{\epsilon_i}_2)\in \mathcal{H}_{1,g_\delta}(\dot F,N_+)\oplus \R\langle\tilde\bdry_\theta\rangle$ such that: \begin{itemize}
	\item on $\dot F- (-\infty, -T_1]\times S^1$, $\xi^{\epsilon_i}= \eta^{v^{\epsilon_i}}-\eta^{u^{\epsilon_i}}$, where $u^{\epsilon_i}=\op{exp}_{u_+} \eta^{u^{\epsilon_i}}$, $v^{\epsilon_i}=\op{exp}_{u_+} \eta^{v^{\epsilon_i}}$, and $\eta^{u^{\epsilon_i}}$ and $\eta^{v^{\epsilon_i}}$ are viewed as sections of the normal bundle $N_+$ to $u_+$;
	\item on $(-\infty,-T_1]\times S^1$, $\xi^{\epsilon_i}=\beta(s+T_1+1)\zeta_-^{\epsilon_i}(s,t)+\zeta_0^{\epsilon_i}(s,t)+\zeta_+^{\epsilon_i}(s,t)$ and $\xi^{\epsilon_i}_2= \zeta_0^{\epsilon_i}$.
\end{itemize}
Recall on the negative end of $u_+$ we are identifying $N_+\simeq TA_a\simeq \R^2$ with coordinates $y,\theta$.
As before $\beta:\R\to[0,1]$ is a nondecreasing function such that $\beta(s)=0$ if $s\leq 0$ and $\beta(s)=1$ if $s\geq 1$.  Informally, we are damping out the $\zeta_-^{\epsilon_i}$ term to zero for $s<-T_1$, under the condition that it is much smaller than $\zeta_0^{\epsilon_i}+ \zeta_+^{\epsilon_i}$ at $s=-T_1$.

\begin{noti}
  We denote the norm on $\R\langle\tilde\bdry_\theta\rangle$ by $\|\cdot\|_\circ$ and the norm on $\mathcal{H}_{1,g_\delta}(\dot F,N_+)\oplus \R\langle\tilde\bdry_\theta\rangle$ by $\|\xi^{\epsilon_i}\|_{\bullet} = \|\xi^{\epsilon_i}_1\|_{*,g_\delta} + \|\xi^{\epsilon_i}_2\|_\circ$. 
\end{noti}

\begin{lemma}\label{bagel}
There exist constants $C_i>0$ with $\lim \limits_{i \to + \infty} C_i=0$ such that
$$\| D_+^\delta \xi^{\epsilon_i}\|_{g_\delta} \le C_i ( \|\xi^{\epsilon_i}_1\|_{*,g_\delta}+ \|\xi^{\epsilon_i}_2\|_\circ)= C_i \|\xi^{\epsilon_i}\|_\bullet.$$
\end{lemma}

\begin{proof}
On $(-\infty, -T_1]\times S^1$, we use the fact that $D_+^\delta \zeta^{\epsilon_i}_*=0$ for $*=-,0,+$, to bound the contribution to $\|D_+^\delta \xi^{\epsilon_i}\|_{g_\delta} $ from above as follows: 
\begin{align}
\label{bagel inequality}\| D_+^\delta (\beta(s+&T_1+1)\zeta_-^{\epsilon_i}(s,t))\|_{g_\delta} = \| \beta'(s+T_1+1) \zeta_-^{\epsilon_i}(s,t)\| _{g_\delta} \\
 \nonumber &\leq C\left(\sup_{s\in [-T_1-1,-T_1]}g_\delta(s)\|\zeta_-^{\epsilon_i}(s)\|_{L^2} +\sup_{s\in [-T_1-1,-T_1]}g_\delta(s) |\zeta_-^{\epsilon_i}(s)|_{C^0}\right)\\
\nonumber &\leq C\sup_{s\in [-T_1-1,-T_1]}g_\delta(s)\|\zeta_-^{\epsilon_i}(s)\|_{L^2_1}\\
\nonumber &\leq C  g_\delta(-T_1-1)\|\zeta_-^{\epsilon_i}(-T_1-1)\|_{L^2_1}\\
\nonumber  & \leq C g_\delta(-T_1) e^{\lambda(-T_2'(\epsilon_i)+T_1+1)} \|\zeta_0^{\epsilon_i}(-T_1-1)+ \zeta_+^{\epsilon_i}(-T_1-1)\|_{L^2_1} \\
\nonumber & \leq Ce^{\lambda(-T_2'(\epsilon_i)+T_1)} ( \|\xi^{\epsilon_i}_1\|_{*,g_\delta}+ \|\xi^{\epsilon_i}_2\|_\circ).
\end{align}
The first line to the second follows from the definition of $\|\cdot \|$ (Equation \eqref{morrey norm 0}) and an easy $C^0$-bound of the right-hand term of the definition of $\|\cdot\|$; the second line to the third uses a standard Sobolev inequality (i.e., there is a bounded inclusion map $L^2_1(S^1)\to C^0(S^1)$); the third line to the fourth follows from Equation \eqref{eqn: basic estimate 2} applied to $T_1'=T_1+1$; and the fourth line to the fifth uses \eqref{poppycake}. The fifth line to the sixth follows from:
\begin{align*}
\|\xi^{\epsilon_i}_1\|_{*,g_\delta} &\geq \|\xi^{\epsilon_i}_1|_{[-T_1-1,-T_1]\times S^1}\|_{*,g_\delta} \geq  \| \zeta_+^{\epsilon_i}|_{[-T_1-1,-T_1]\times S^1}\|_{*,g_\delta} \\
& \geq C \left( \int_{[-T_1-1,-T_1]}\int_{S^1} g_\delta^2| \zeta_+^{\epsilon_i}|^2 \right)^{1/2}+C
\left( \int_{[-T_1-1,-T_1]}\int_{S^1} g_\delta^2 |\nabla \zeta_+^{\epsilon_i}|^2 \right)^{1/2}\\
&\geq C g_\delta(-T_1) (\| \ \zeta_+^{\epsilon_i}(-T_1-1)\|_{L^2} + \|\nabla \zeta_+^{\epsilon_i}(-T_1-1)\|_{L^2}).
\end{align*}

On the other hand, writing $\|\cdot \|_{g_\delta}'$ and $\|\cdot \|_{g_\delta}''$ for the restrictions of $\|\cdot \|_{g_\delta}$ to $\dot F - (- \infty, -T_1] \times S^1$ and $[-T_1, -T_0] \times S^1$, writing $v^{\epsilon_i}=\op{exp}_{u^{\epsilon_i}}(P^{-1}{\widetilde\xi}^{\epsilon_i})$ on $\dot F - (- \infty, -T_1] \times S^1$, where $P$ is the parallel transport of the appropriate bundles from $u^{\epsilon_i}$ to $u_+$, and using the fact that 
$$\widetilde \xi^{\epsilon_i}= \xi^{\epsilon_i} +\mathcal{B}(\eta^{u^{\epsilon_i}}, \xi^{\epsilon_i})+ \mathcal{Q}( \xi^{\epsilon_i}),$$ 
where $\mathcal{B}(\eta^{u^\epsilon_i}, \xi^{\epsilon_i})$ is bilinear in $\eta^{u^{\epsilon_i}}$ and $\xi^{\epsilon_i}$ and $ \mathcal{Q}( \xi^{\epsilon_i})$ is quadratic in $\xi^{\epsilon_i}$, both with coefficients which are smooth coefficients of $\eta^{u^{\epsilon_i}}$ and $\xi^{\epsilon_i}$, 
the contribution to $\|D_+^\delta\xi^{\epsilon_i}\|_{g_\delta}$ on $F'$ is bounded above as follows: 
\begin{align*}
\|D_+^\delta \xi^{\epsilon_i}\|_{g_\delta}'& = \|D_+^\delta \xi^{\epsilon_i}- P\overline\bdry_{J_{f_{\epsilon_i}}}\op{exp}_{u^{\epsilon_i}}(P^{-1}{\widetilde\xi}^{\epsilon_i})\|_{g_\delta}'\\
&\leq  \|D_+^\delta \xi^{\epsilon_i}- P D_{u^{\epsilon_i}}^\delta P^{-1}\xi^{\epsilon_i} \|_{g_\delta}' + \|P( D_{u^{\epsilon_i}}^\delta P^{-1}\xi^{\epsilon_i}- \overline\bdry_{J_{f_{\epsilon_i}}}\op{exp}_{u^{\epsilon_i}} P^{-1}\xi^{\epsilon_i})\|_{g_\delta}' \\
&\qquad \qquad \qquad +  \|P(\overline\bdry_{J_{f_{\epsilon_i}}}\op{exp}_{u^{\epsilon_i}}P^{-1} \xi^{\epsilon_i}-\overline\bdry_{J_{f_{\epsilon_i}}}\op{exp}_{u^{\epsilon_i}}P^{-1}{\widetilde \xi}^{\epsilon_i})\|_{g_\delta}' \\
& \leq C_i \|\xi^{\epsilon_i}\|_{*,g_\delta}' \leq C_i (\|\xi^{\epsilon_i}_1\|_{*,g_\delta} + \|\xi^{\epsilon_i}_2\|_{*,g_\delta}'') \leq C_i (\|\xi^{\epsilon_i}_1\|_{*,g_\delta} + e^{\delta(T_1-T_0)}\|\xi^{\epsilon_i}_2\|_\circ ),
\end{align*}
where $\lim_{i\to \infty} C_i=0$. (Recall that $b$ and $T_1(b)$ are constants that were fixed at the beginning of Section~\ref{subsection: gluing 2}.)

The two estimates together imply the lemma.
\end{proof}

In view of Lemma~\ref{bagel}, inverting the error using $(D_+^\delta)^{-1}$ (as before the image of $(D_+^\delta)^{-1}$ is $L^2$-orthogonal to $\ker D_+^\delta$) yields
$$(\xi')^{\epsilon_i}=\xi^{\epsilon_i}-(D_+^\delta)^{-1}(D_+^\delta\xi^{\epsilon_i}) \in\ker D_+^\delta,$$
such that $\| \xi^{\epsilon_i}\|_\bullet \gg \| (D_+^\delta)^{-1}(D_+^\delta\xi^{\epsilon_i}) \|_\bullet$, which implies that $(\xi')^{\epsilon_i}\not=0$. We define
$$\overline{\xi}^{\epsilon_i}= \frac{\xi^{\epsilon_i}}{\| (\xi')^{\epsilon_i}\|_\bullet}, \quad (\overline{\xi}')^{\epsilon_i}= \frac{(\xi')^{\epsilon_i}}{\| (\xi')^{\epsilon_i}\|_\bullet}.$$

\begin{lemma}\label{non-translation element}
There exists a non-translation element of $\op{ker}D_+^\delta$.
\end{lemma}  

\begin{proof}
So far we have constructed sequences $\{\overline{\xi}^{\epsilon_i}\}$ and $\{(\overline{\xi}')^{\epsilon_i}\}$ such that:
\be
\item the Fourier coefficient relative to $g_1$ is $\overline{d}_1^{\epsilon_i}=0$ for all $\overline{\xi}^{\epsilon_i}$; 
\item $\| (\overline{\xi}')^{\epsilon_i}\|_\bullet=1$;
\item $(\overline{\xi}')^{\epsilon_i}\in \op{ker} D_+^\delta$; and
\item $\| (\overline{\xi}')^{\epsilon_i}-\overline{\xi}^{\epsilon_i} \|_\bullet\to 0$ as $i \to \infty$. 
\ee
Since $\op{ker} D_+^\delta$ is finite-dimensional, the unit ball of $\op{ker} D_+^\delta$ is compact, and after possibly passing to a subsequence $(\overline\xi')^{\epsilon_i}$ converges to a nonzero $\xi'\in \op{ker} D_+^\delta$. Then (4) implies that $\|\overline{\xi}^{\epsilon_i} - \xi'  \|_\bullet\to 0$ and therefore, from Lemma~\ref{lemma: sobolev embedding}, we obtain $\overline{\xi}^{\epsilon_i}(-T_1) \to \xi'(-T_1)$ in $C^0$. This in turn implies that $\overline{d}_1^{\epsilon_i}\to d_1'$, where $\overline{d}_1^{\epsilon_i}$ and $d_1'$ are the Fourier coefficients of $g_1$ in $\overline{\xi}^{\epsilon_i}$ and $\xi'$; this is because the Fourier coefficients can be extracted by integration. Hence $d_1'=0$.

Finally we explain why $d_1'=0$ implies that $\xi'$ is a non-translation element: Recall that $u_+(s,t)=(s,t,\sum_{i=1}^\infty c_i e^{\lambda_i s} g_i(t))$ with $c_1>0$ at the negative end (cf.\ the beginning of Section~\ref{subsection: pregluing}). Let $u_+^\sigma$ be the translate of $u_+$ by $\sigma \in \R$ in the symplectization direction.  Then, at the negative end,
$$u_+^\sigma(s,t)= (s+\sigma, t,\sum_{i=1}^\infty c_i e^{\lambda_i s} g_i(t)),$$
or, after the change of coordinates $(s+\sigma, t) \mapsto (s,t)$ at the negative end of $\dot F$,
$$u_+^\sigma(s,t)= (s, t,\sum_{i=1}^\infty c_i e^{\lambda_i (s- \sigma)} g_i(t)).$$
Then a translation element is a nontrivial multiple of the projection of
$\left.\tfrac{\partial u_+^\sigma}{\partial \sigma}\right|_{\sigma=0}$
to the normal bundle $N_+$, i.e., $- \sum_{i=1}^\infty c_i \lambda_i e^{\lambda_i s} g_i(t)$,
and has nontrivial $g_1(t)$-coef\-ficient.
 
\end{proof}

The existence of a non-translation element of $\op{ker} D_+^\delta$ is a contradiction.

\s\n
{\bf Case 2.}  Up to extracting a subsequence, we assume that the following inequality holds for every $i$: 
\begin{equation}
\|\zeta_0^{\epsilon_i}(-T_1-1)+ \zeta_+^{\epsilon_i}(-T_1-1)\|_{L^2_1}\leq \|\zeta_-^{\epsilon_i}(-T_2'(\epsilon_i))\|_{L^2_1}.
\end{equation}
Let $-T_4(\epsilon_i)<- T_3(\epsilon_i)< -T'_2(\epsilon_i)$ such that $ T_3(\epsilon_i) -T'_2(\epsilon_i), T_4(\epsilon_i) -T_3(\epsilon_i)\to \infty$ as $i\to \infty$ and
\begin{gather*}
\op{Im}u^{\epsilon_i}|_{s=-T_3(\epsilon_i)}\subset \{ -\tfrac{1}{6}\leq \theta \leq -\tfrac{1}{6}+\kappa_i\},\\
\op{Im}u^{\epsilon_i}|_{s=-T_4(\epsilon_i)}\subset \{ -\tfrac{1}{5}-\kappa_i\leq \theta \leq -\tfrac{1}{5}\},
\end{gather*}
where $\op{Im}$ denotes the image. 
Using Equation~\eqref{eqn: basic estimate 2} with $T_1'=T_3(\epsilon_i)$ and $s=-T_2(\epsilon_i)$ we have:
\begin{equation} \label{eqn: dominating}
\|\zeta_0^{\epsilon_i}(-T_1-1)+ \zeta_+^{\epsilon_i}(-T_1-1)\|_{L^2_1}\leq C e^{\lambda(-T_3(\epsilon_i)+T_2'(\epsilon_i))}\cdot\|\zeta_-^{\epsilon_i}(-T_3(\epsilon_i))\|_{L^2_1}.
\end{equation}

\s\n
{\em Complication.}  There is one complication. By Equation~\eqref{eqn: rewriting d bar equation} the $\overline\bdry$-operator $\mathcal{D}_{\epsilon_i}$ is linear on $-\tfrac{1}{4}\leq \theta\leq -\tfrac{1}{5}$ (with respect to $(y,\theta+\tfrac 14))$  because $\overline{g}'_{\mathcal N}(\theta)= C(\theta+\tfrac{1}{4})$, and each of $\eta=\eta^{v^{\epsilon_i}}$, $\eta^{u^{\epsilon_i}}$ satisfies the equation
$$\mathcal{D}_{\epsilon_i}\eta=\frac{\bdry \eta}{\bdry s} + j_0\frac{\bdry \eta}{\bdry t} -  \begin{pmatrix} \eta_1 \\ \epsilon_i C (\eta_2+ \tfrac{1}{4})\end{pmatrix} =0,$$
 where $\eta= (\eta_1, \eta_2)$, i.e., $\eta_1$ is the $y$-coordinate and $\eta_2$ the $\theta$-coordinate of $\eta$. Hence $\zeta^{\epsilon_i}(s,t)$ admits a Fourier expansion at the negative end whose leading term has the form $(k_{i1} e^{s}, k_{i2} e^{\epsilon Cs} )$ for constants $k_{i1},k_{i2}$.  However, a section with growth rate $e^{\epsilon Cs}$ as $s\to\infty$ is not in ${\mathcal H}_{1,h_\delta}(\R \times S^1, N_-^{\epsilon_i,0})$ since we have been assuming that $0<\epsilon_i C<\delta$; in fact $\epsilon_i \to 0$ while $\delta$ is constant.  To circumvent this difficulty we switch to 
$$D^{\epsilon_i,0,-\delta}_-:{\mathcal H}_{1,h_{-\delta}}(\R \times S^1, N_-^{\epsilon_i,0})\to {\mathcal H}_{1,h_{-\delta}}(\R\times S^1, \Lambda^{0,1}N^{\epsilon_i,0}_-),$$ 
where $\delta>0$ is sufficiently small.  The analog of Lemma~\ref{bounded inverses} also holds for $h_{-\delta}$, i.e., the operators $D_-^{\epsilon_i,0,-\delta}$ are invertible with bounded inverses that are uniform with respect to $\epsilon_i$. 
\s

Let $\xi^{\epsilon_i}$ be the section of the normal bundle $N_-^{\epsilon_i,0}=TA_a$ to $u_-^{\epsilon_i,0}$ such that:
\begin{itemize}
	\item $\xi^{\epsilon_i}= \zeta^{\epsilon_i}$ on $(-\infty, -T_3(\epsilon_i)]\times S^1$;
	\item $\xi^{\epsilon_i}=(1-\beta(s+T_1+1))\zeta_-^{\epsilon_i}(s,t)+(1-\beta(s+T_3(\epsilon_i))(\zeta_0^{\epsilon_i}(s,t)+\zeta_+^{\epsilon_i}(s,t))$ on $[-T_3(\epsilon_i),-T_1]\times S^1$ (here we write $\zeta_+,\zeta_0,\zeta_-$ for the $L^2$-projections of $\zeta$ to the positive, null, and negative eigenspaces of ${\bf A}$); and
	\item $\xi^{\epsilon_i}=0$ on $[-T_1,\infty)\times S^1$.
\end{itemize}
Informally, we damp out $\zeta_-^{\epsilon_i}$ to zero for $s\geq -T_1$ and $\zeta_0^{\epsilon_i}+\zeta_+^{\epsilon_i}$ to zero for $s\geq -T_3(\epsilon_i)+1$ so that the damped out $\zeta_-^{\epsilon_i}$ dominates.  By the previous paragraph, $\xi^{\epsilon_i}\in \mathcal{H}_{1,h_{-\delta}}(\R\times S^1,N_-^{\epsilon_i,0})$.  Also $D_-^{\epsilon_i,0,-\delta}\xi^{\epsilon_i}$ has support on
\begin{equation}\label{support}
\{ -T_4(\epsilon_i)\leq s\leq -T_3(\epsilon_i)+1\}\cup \{-T_1-1\leq s\leq -T_1\}.
\end{equation}

One can compute using \eqref{support}, Estimate~\eqref{eqn: dominating}, the method of estimating \eqref{bagel inequality}, and the error estimate between $D_-^{\epsilon_i,0,{-\delta}}$ and the actual normal $\overline\bdry_{J_{f_{\epsilon_i}}}$-operator that:
$$\| D_-^{\epsilon_i,0,{-\delta}} \xi^{\epsilon_i}\|_{h_{-\delta}} \ll \|\xi^{\epsilon_i}\|_{*,h_{-\delta}}.$$
Hence inverting the error using $(D_-^{\epsilon_i,0, {-\delta}})^{-1}$ yields
$$(\xi')^{\epsilon_i}= \xi^{\epsilon_i} -(D_-^{\epsilon_i,0, {-\delta}})^{-1}(D_-^{\epsilon_i,0,{-\delta}} \xi^{\epsilon_i})\in \ker D_-^{\epsilon_i,0,{-\delta}},$$ 
such that $\| \xi^{\epsilon_i} \|_{*,h_{-\delta}} \gg \| (D_-^{\epsilon_i,0, {-\delta}})^{-1} (D_-^{\epsilon_i,0,{-\delta}} \xi^{\epsilon_i}) \|_{*,h_{-\delta}}$, which implies that $(\xi')^{\epsilon_i}\not=0$.

The existence of a nontrivial element of $\ker D_-^{\epsilon_i,0,{-\delta}}$ contradicts Lemma~\ref{bounded inverses}. This completes the proof of Theorem~\ref{thm: 2}.

\subsection{How to recover the contact case} \label{subsection: reduction}

In this subsection we explain how to recover the contact case from the stable Hamiltonian case.  The brief idea is to start with the stable Hamiltonian case for which Morse-Bott gluing holds, perturb it to the contact case, and use the bifurcation method to establish Morse-Bott gluing in the contact case.

Let $a$ and $c$ be the positive numbers satisfying $0 < b < a < c < 1$ introduced in Section \ref{subsection: notation} and subject to the conditions ($\boldsymbol{\dagger}_0$) and ($\boldsymbol{\dagger}_1$). Recall the smooth functions 
$$f,f_\epsilon: [-1,1]\times T^2 \to \R$$ 
given by Equation~\eqref{eqn: f and f epsilon}.

\begin{warn}
The following have different meanings in this subsection from the previous subsections of the appendix: the real parameter $\delta$ in this subsection is unrelated to the weight appearing in the Morrey norms, and the functions $g,h,g_\delta, h_\delta$ appearing in this subsection are unrelated to the functions with the same names appearing in the previous subsections of the appendix.
\end{warn}

We then define smooth functions  
$$g,h: [-1,1]\to \R,$$ 
such that (i) $g$ is odd,  (ii) $g(y)=0$ on $[-a,a]$, (iii) $g'(y)>0$ on $(a,1]$ and $[-1,-a)$, and (iv) $g(y)=y$ on $y\geq c$ and $y\leq -c$;  and (v) $h(0)=0$ and (vi) $h'(y)= g'(y) {\bdry f\over \bdry y}=g'(y) y$.  In particular, $h(y)=0$ on $[-a, a]$. 

We define  differential forms 
\begin{gather}
\label{widetilde alpha} \alpha = dt + h(y)dt + g(y) d\theta,\quad \omega= df  \wedge dt+dy \wedge d\theta,\\
\nonumber  \omega_\epsilon= df_\epsilon \wedge dt+dy \wedge d\theta
\end{gather}
on $[-1,1]\times T^2$. (Here without loss of generality we are suppressing some constants that appeared in Equation~\eqref{eqn: alpha and alpha epsilon}.)

\begin{claim} \label{claim: verification of stable Hamiltonian}
The pairs $(\alpha,\omega)$ and $(\alpha,\omega_\epsilon)$ on $[-1,1]\times T^2$ are  stable Hamiltonian structures.
\end{claim}

\begin{proof}
It is immediate that $d\omega=d\omega_\epsilon=0$. 

Next we show that $\ker d \alpha\supset \ker \omega$ and $\ker d \alpha\supset \ker \omega_\epsilon$.
$$d \alpha= h'(y)dy \wedge dt + g'(y) dy \wedge d\theta= g'(y)\tfrac{\bdry f}{\bdry y}dy \wedge dt + g'(y)dy \wedge d\theta.$$
On $-a\leq y\leq a$, $d \alpha=0$ and hence $\ker d \alpha\supset \ker \omega$.  Outside of $-a\leq y\leq a$, $g'(y)\not=0$ and 
$$\ker d\alpha=\R\langle \tfrac{\bdry}{\bdry t}-\tfrac{\bdry f}{\bdry y}\tfrac{\bdry} {\bdry\theta}    \rangle=\ker \omega.$$ 
Moreover, outside of $-a\leq y\leq a$, $f_\epsilon=f$ and $\omega=\omega_\epsilon$. 

Finally, 
$$ \alpha\wedge \omega=  \alpha\wedge \omega_\epsilon=dt\wedge dy \wedge d\theta>0$$
on $-a\leq y\leq a$ and 
\begin{align*}
 \alpha\wedge \omega =  \alpha\wedge \omega_\epsilon & =(1+h(y)-g(y)\tfrac{\bdry f}{\bdry y})dt\wedge dy \wedge d\theta\\
&= (1+h(y)-g(y)y)dt\wedge dy \wedge d\theta
\end{align*}
outside of $-a\leq y\leq a$.  By (vi), $(h(y)-g(y)y)'= -g(y)$.  Since $|g(y)|<1$ except when $y=1$, $1+h(y)-g(y)y>0$.  

Hence $(\alpha,\omega)$ and $(\alpha,\omega_\epsilon)$ are both stable Hamiltonian structures.
\end{proof}

\begin{claim} \label{claim: stable Hamiltonian properties} $\mbox{}$
\be
\item[(A1)] On $[-a,a]\times T^2$, 
\begin{gather*}
(\alpha,\omega)=(dt, df \wedge dt+dy \wedge d\theta),\\
(\alpha,\omega_\epsilon)=(dt, df_\epsilon \wedge dt+dy \wedge d\theta).
\end{gather*}
\item[(A2)] On $([-1,-a)\cup (a,1])\times T^2$, the stable Hamiltonian structures $(\alpha,\omega)$ and $(\alpha,\omega_\epsilon)$ agree and $d \alpha$ is a positive function $g'(y)$ times $\omega=\omega_\epsilon$.
\item[(A3)] On $([-1,-c)\cup (c,1])\times T^2$, $g'(y)=1$, and $(\alpha,\omega)$ and $(\alpha,\omega_\epsilon)$ are contact.
\ee
\end{claim}

\begin{proof}
Immediate from the definitions and the proof of Claim~\ref{claim: verification of stable Hamiltonian}.
\end{proof}

In view of Claim~\ref{claim: stable Hamiltonian properties}, there exists an extension of $(\alpha,\omega)=(\alpha,\omega_\epsilon)$ to $( \alpha, d\alpha)$ on $M-([-1,1]\times T^2)$. (In practice, we start with a contact form $ \alpha$ on all of $M$ and modify it to the stable Hamiltonian structures $(\alpha,\omega)$ and $(\alpha,\omega_\epsilon)$ on $[-1,1]\times T^2$.)

Let $J_f$ and $J_{f_\epsilon}$ be almost complex structures on $\R\times M$ such that:
\be
\item[(A4)] On the complement of $\R\times[-a,a]\times T^2$, $J_f$ and $J_{f_\epsilon}$ agree and are adapted to the same contact structure.
\item[(A5)] On $\R\times [-1,1]\times T^2$, $J_f$ and $J_{f_\epsilon}$ are adapted to the stable Hamiltonian structures $(\alpha,\omega)$ and $(\alpha,\omega_\epsilon)$. 

\item[(A6)] $J_f$ is Morse-Bott regular, $J_{f_\epsilon}$ is regular (at least for the moduli spaces that are involved in the Morse-Bott gluing), and the pair satisfies Morse-Bott gluing (i.e., Theorem~\ref{thm: Morse-Bott perturbation of moduli spaces} (2) and (3)).
\ee
The existence of such $J_f$ and $J_{f_\epsilon}$ follows from Theorem~\ref{thm: main theorem of appendix}.

The key point is the following lemma which allows us to perturb to the contact case:

\begin{lemma} \label{lemma: cx str}
	There exist almost complex structures $J'_{f,\delta}$ and $J_{f_\epsilon,\delta}'$ that are $C^1$-close to $J_f$ and $J_{f_\epsilon}$  and contact forms  $\alpha'_\delta$ and $\alpha'_{\epsilon,\delta}$ that are $C^1$-close to $ \alpha$ (here the size of the perturbations depend on $\delta>0$) such that:
	\be
	\item on the complement of $\R\times[-c,c]\times T^2$, $J_f$, $J_{f_\epsilon}$, $J'_{f,\delta}$, and $J_{f_\epsilon,\delta}'$ all agree;
	\item $J'_{f,\delta}$ and $J_{f_\epsilon,\delta}'$ are adapted to $\alpha'_\delta$ and $\alpha'_{\epsilon,\delta}$, respectively;
	\item the stable Hamiltonian vector fields corresponding to $J_f$ and $J'_{f,\delta}$ (and those corresponding to $J_{f_\epsilon}$ and $J_{f_\epsilon,\delta}'$) are parallel; and
	\item on $\R\times[-a,a]\times T^2$, $\alpha'_\delta$ and $\alpha'_{\epsilon,\delta}$ are as given in Equation~\eqref{eqn: alpha and alpha epsilon} and (P1)--(P4) in Section~\ref{subsection: Morse-Bott contact forms} with $C=1$ and  $J'_{f,\delta}$ and $J_{f_\epsilon,\delta}'$ satisfy (*) in Section~\ref{subsection: Morse-Bott regularity} and (**) in Section~\ref{subsection: Morse-Bott chain complex}.
	\ee
\end{lemma}

\begin{proof}
For $\delta>0$ small, let $g_\delta:[-1,1]\to \R$ be a smooth function which is a perturbation of $g$ such that (i) $g_\delta$ is odd, (ii') $g_\delta(y)=\delta y$ on $[-a,a]$, and (iii) and (iv) still hold.  We define
$$\alpha'_\delta= dt +\delta(fdt + yd\theta) \quad \mbox{ if $y\in [-a,a]$},$$
$$\alpha'_{\epsilon,\delta}=dt + \delta(f_\epsilon dt+yd\theta)\quad \mbox{ if $y\in[-a,a]$},$$
$$\alpha'_\delta= \alpha'_{\epsilon,\delta}= dt  + h_\delta(y)dt +g_\delta(y)d\theta \quad \mbox{ if $y\not\in [-a,a]$},$$
and $h_\delta'(y)=g_\delta'(y) {\bdry f\over \bdry y}$ for $y\not\in[-a,a]$ and $h_\delta(\pm a)= \delta f(\pm a)$.  If $c\gg a$, then it is not hard to see that we can choose $g_\delta$ such that $h_\delta(\pm c)= h(\pm c)$.
Then $( \alpha'_\delta,\omega)$ and $(\alpha'_{\epsilon,\delta},\omega_\epsilon)$ are stable Hamiltonian structures corresponding to contact structures, and are close to $(\alpha,\omega)$ and $(\alpha,\omega_\epsilon)$. (Strictly speaking, $d \alpha'_\delta=\phi_1\omega$ and $d \alpha'_{\epsilon,\delta}=\phi_2\omega_\epsilon$ for some functions $\phi_1$ and $\phi_2$.)

We verify the contact property for $ \alpha'_\delta$ and $ \alpha'_{\epsilon,\delta}$:  For $y\in [-a,a]$,
\begin{align*}
 \alpha'_\delta \wedge d \alpha'_\delta &= (1+\delta f)dt \wedge \delta dy d\theta + \delta y d\theta \wedge \delta df dt>0, 
\end{align*}
since we are assuming that $\delta>0$ is small. Similarly, $ \alpha'_{\epsilon,\delta}$ is contact by replacing $f$ by $f_\epsilon$ on $[-a,a]$.  For $y\not\in[-a,a]$, $ \alpha'_\delta=\alpha_{\epsilon,\delta}$ and 
\begin{align*}
 \alpha'_\delta \wedge d \alpha'_\delta &= (1+h_\delta)dt \wedge g_\delta'(y) dy d\theta + g_\delta(y) d\theta \wedge h'_\delta(y)dy dt\\
& = g'_\delta(y)[(1+h_\delta)- g_\delta(y)y ]dt dyd\theta >0 
\end{align*}
as in the proof of Claim~\ref{claim: verification of stable Hamiltonian}. 

Let $J'_{f,\delta}$ and $J'_{f_\epsilon,\delta}$ be the corresponding adapted almost complex structures that are close to $J_f$ and $J_{f_\epsilon}$ and subject to the condition that the projections of $J'_{f,\delta}|_{\ker {\alpha}'}$ and $J'_{f_\epsilon,\delta}|_{\ker \alpha'_\epsilon}$ to $[-1,1]\times (\R/\Z)$ with coordinates $(y,\theta)$ is the standard complex structure ${\bdry \over \bdry y}\mapsto {\bdry\over \bdry\theta}$.
	
The $C^1$-closeness and (1)--(4) are immediate from the construction.
\end{proof}

For the next lemma we introduce the following notation.

\begin{noti}
 If ${\mathcal M}$ is a moduli space of $J$-holomorphic curves in a symplectization for a cylindrical almost complex structure $J$, we denote by $\widehat{\mathcal M} := {\mathcal M}/\R$ the quotient of ${\mathcal M}$ by translations in the symplectization direction.
\end{noti}

\begin{lemma}
  There exist $\delta>0$ sufficiently small and $\epsilon_0=\epsilon_0(\delta)>0$ such that  Theorem~\ref{thm: Morse-Bott perturbation of moduli spaces} (2) and (3) hold for any $\epsilon$ satisfying $0<\epsilon<\epsilon_0$, with $J_0$ and $J_\epsilon$ replaced by $J_{f,\delta}'$ and $J'_{f_\epsilon,\delta}$.
\end{lemma}

\begin{proof}
	Consider the $I=\op{ind}=1$, unconstrained, Morse-Bott regular moduli space $\mathcal{M}_{J_f}^{I=\op{ind}=1}(\boldsymbol{\gamma};\mathcal{N})$ from Section~\ref{subsection: notation}. Then  $\widehat{\mathcal{M}}_{J_f}^{I=\op{ind}=1}(\boldsymbol{\gamma};\mathcal{N})$ consists of a finite number of holomorphic maps $u_+$.	If $\delta>0$ is small, then $J_{f,\delta}'$ is also Morse-Bott regular since it is close to $J_f$ and
	\begin{equation} \label{eqn: moduli same}
	\widehat{\mathcal{M}}_{J_f}^{I=\op{ind}=1}(\boldsymbol{\gamma};\mathcal{N})\simeq \widehat{\mathcal{M}}_{J'_{f,\delta}}^{I=\op{ind}=1}(\boldsymbol{\gamma};\mathcal{N}),
      \end{equation}
where $\simeq$ indicates a bijection.  (If signs were done carefully, they would be preserved by the bijection.)
	Next, there exists $\epsilon>0$ small such that $J_{f_\epsilon}$ is regular (after possibly perturbing $J_f$) and there exists $\delta=\delta(\epsilon)>0$ such that $J'_{f_\epsilon,\delta}$ is close to $J_{f_\epsilon}$ and hence is regular and
	\begin{equation} \label{eqn: moduli same 2}
	\widehat{\mathcal{M}}_{J_{f_\epsilon}}^{I=\op{ind}=1}(\boldsymbol{\gamma},e)\simeq \widehat{\mathcal{M}}_{J'_{f_\epsilon,\delta}}^{I=\op{ind}=1}(\boldsymbol{\gamma},e),
	\end{equation}
	where $e$ is the negative elliptic orbit obtained by perturbing the Morse-Bott family. Also for the same $\epsilon>0$ small the Morse-Bott gluing theorem in the stable Hamiltonian case (Theorem~\ref{thm: main theorem of appendix}) gives a bijection
	\begin{equation} \label{eqn: moduli same 3}
	\widehat{\mathcal{M}}_{J_f}^{I=\op{ind}=1}(\boldsymbol{\gamma};\mathcal{N})\simeq \widehat{\mathcal{M}}_{J_{f_\epsilon}}^{I=\op{ind}=1}(\boldsymbol{\gamma},e).
	\end{equation}
	Combining \eqref{eqn: moduli same}, \eqref{eqn: moduli same 2}, and \eqref{eqn: moduli same 3} gives
	\begin{equation} \label{eqn: moduli same 4}
	\widehat{\mathcal{M}}_{J'_{f,\delta}}^{I=\op{ind}=1}(\boldsymbol{\gamma};\mathcal{N}) \simeq \widehat{\mathcal{M}}_{J'_{f_\epsilon,\delta}}^{I=\op{ind}=1}(\boldsymbol{\gamma},e),
	\end{equation}
	for some $\epsilon>0$ small and $\delta=\delta(\epsilon)>0$ small.  The difficulty is that  we want $\epsilon$ to depend on $\delta$,  not the other way around.

  To remedy this we start with $\epsilon_1>0$ small, choose $\delta=\delta(\epsilon_1)>0$ small such that \eqref{eqn: moduli same 2} holds with $\epsilon=\epsilon_1$, and apply the bifurcation method to the $1$-parameter family $\{J'_{f_\epsilon,\delta}\}_{\epsilon\in (0,\epsilon_1]}$. We may assume that $J'_{f_{\epsilon_1},\delta}$ is regular and that $\{J'_{f_\epsilon,\delta}\}_{\epsilon\in(0, \epsilon_1]}$ is regular as a family.  By \eqref{eqn: moduli same}--\eqref{eqn: moduli same 4}, $\widehat{\mathcal{M}}_{J'_{ f_{\epsilon_1}, \delta}}^{I=\op{ind}=1}(\boldsymbol{\gamma},e)$  consists of a finite number of holomorphic maps (up to translation in the target) that are close to breaking and is in bijection with $\widehat{\mathcal{M}}_{J_f}^{I=\op{ind}=1}(\boldsymbol{\gamma};\mathcal{N})$.
	
We claim that for $\epsilon_1$ and $\delta$ sufficiently small,
$$\# \widehat{\mathcal{M}}_{J'_{f_\epsilon,\delta}}^{I=\op{ind}=1}(\boldsymbol{\gamma}, e) \equiv \# \widehat{\mathcal{M}}_{J'_{f_{\epsilon_1}, \delta}}^{I=\op{ind}=1}(\boldsymbol{\gamma}, e) \ \mbox{ mod 2}$$ 
for all $\epsilon\in  (0,\epsilon_1]$. To this end we consider the $1$-dimensional parametric moduli space which, slightly abusing notation, we denote by $\coprod_{\epsilon\in (0,\epsilon_1]}\widehat{\mathcal{M}}_{J'_{f_\epsilon, \delta}}^{I=\op{ind}=1}(\boldsymbol{\gamma}, e)$.  Note that the Reeb orbits do not vary as $\epsilon$ varies by Lemma~\ref{lemma: cx str} (2).
The claim is a consequence of the following claim: For $\epsilon_1$ and $\delta$ small there is no $u_{\tilde \epsilon} \in \bdry (\coprod_{\epsilon\in (0,\epsilon_1]}\widehat{\mathcal{M}}_{J'_{f_\epsilon, \delta}}^{I=\op{ind}=1}(\boldsymbol{\gamma}, e))$, where $u_{\tilde \epsilon}$ is a limit $J'_{f_{\tilde \epsilon},\delta}$-holomorphic curve/building for some $\tilde \epsilon\in (0,\epsilon_1)$.  
Arguing by contradiction, suppose there exist sequences $\{(\delta_i, \tilde \epsilon_i)\}_{i=1}^\infty$ and $\{u_{\tilde \epsilon_i}\}_{i=1}^\infty$ such that $(\delta_i,\tilde \epsilon_i)\to (0,0)$ and $u_{\tilde \epsilon_i}$ converges to a $J_f$-holomorphic limit curve $u$ which is:
\begin{itemize}
\item[(i)] a $2$-level holomorphic building $\tilde{u}_1 \cup \tilde{u}_2$, one of whose components --- say $\tilde{u}_1$ ---  satisfies $I(\tilde{u}_1)= \op{ind}(\tilde{u}_1)=0$; or
\item[(ii)] a multiple cover of a holomorphic map $\tilde{v}$ with $I(\tilde{v})= \op{ind}(\tilde{v}) =0$;
\end{itemize}
and neither can occur since $u$ is $J_f$-holomorphic and $J_f$ is regular.
\end{proof}

\end{document}